\documentclass[11pt,a4paper]{article}

\usepackage[T1]{fontenc}
\usepackage{amsmath,amsthm,amsopn,amsfonts,amssymb,dsfont,color}
\usepackage{mathrsfs}
\usepackage{statex}
\usepackage{graphicx}
\usepackage{float}
\usepackage{xcolor}
\usepackage{mathrsfs}
\usepackage{mathtools}
\usepackage{accents}
\usepackage{setspace}
\usepackage{times}
\usepackage{listings}
\usepackage{enumerate}
\usepackage{indentfirst}
\usepackage{a4,a4wide}
\usepackage{tikz}

\newtheorem{theorem}{\textbf{Theorem}}[section]
\newtheorem{lemma}[theorem]{\textbf{Lemma}}
\newtheorem{proposition}[theorem]{\textbf{Proposition}}
\newtheorem{claim}[theorem]{\textbf{Claim}}
\newtheorem{corollary}[theorem]{\textbf{Corollary}}

\newtheorem{remark}[theorem]{\textbf{Remark}}

\numberwithin{equation}{section}
\numberwithin{figure}{section}

\makeatletter
\g@addto@macro\th@plain{\thm@headpunct{}}
\makeatother

\newcommand\dps{\displaystyle}

\newcommand\bea{\begin{eqnarray}}
\newcommand\eea{\end{eqnarray}}
\newcommand\beaa{\begin{eqnarray*}}
\newcommand\eeaa{\end{eqnarray*}}

\title{Complete Classification of Traveling Waves and Resolution of Linear Conjecture in Monostable Systems}
\date{}
\author{}
\begin{document}

\maketitle
\vspace{-30pt}
\begin{center}
{\large\bf
Chang-Hong Wu
\footnote{Department of Applied Mathematics, National Yang Ming Chiao Tung University, Hsinchu, Taiwan.

e-mail: {\tt changhong@math.nctu.edu.tw}}
\footnote{National Center for Theoretical Sciences, Taipei, Taiwan},
Dongyuan Xiao\footnote{Advanced Institute for Materials Research, Tohoku University, Sendai, Japan.

e-mail: {\tt dongyuanx@hotmail.com}} and
Maolin Zhou\footnote{Chern Institute of Mathematics and LPMC, Nankai University, Tianjin, China.

e-mail: {\tt zhouml123@nankai.edu.cn}}
} \\
[2ex]
\end{center}

\begin{abstract}
In this paper, we present a complete classification of traveling wave solutions for monostable systems in a unified framework. To achieve this, we introduce a novel technique called the {\it{slicing method}}, which is based on the super- and sub-solution approach. Furthermore, it serves as a useful tool for addressing the linear conjecture in the Lotka–Volterra competition system, which remains a long-standing problem. 
\\


\end{abstract}

\newpage
\tableofcontents

\newpage
\section{Introduction}

The phenomenon of the front propagation into unstable states is a classical issue and has been discussed by many 
physicists in early works; see, for instance, 
\cite{Benguria Depassier1994,Benguria Depassier1996, van Saarloos1988,van Saarloos1989,van Saarloos2003}. A prototypical model to describe the transition from an unstable one to a stable one in reaction-diffusion equations is the well-known
Fisher-KPP equation
\begin{equation}\label{KPP eq}
w_t=w_{xx}+w(1-w),\ \ t>0,\ x\in\mathbb{R},
\end{equation}
which was independently proposed by Fisher \cite{Fisher} and Kolmogorov et al. \cite{KPP} in early 1937 to depict the spatial propagation of organisms such as dominant genes and invasive species in a homogeneous environment. It was shown in \cite{KPP} that for any $c\geq 2$, there exists a traveling wave solution with the particular form $W(\xi)=W(x-ct)=w(t,x)$ to \eqref{KPP eq} satisfying
\begin{equation}\label{KPP tw}
\left\{
\begin{aligned}
&W''+cW'+W(1-W)=0,\ \xi\in\mathbb{R},\\
&W(-\infty)=1,\ W(+\infty)=0,\ W'(\cdot)<0.
\end{aligned}
\right.
\end{equation}
Moreover, $c=2$ is not only the minimal traveling wave speed (denoted by $c^*$) of \eqref{KPP tw}, but also the propagation speed of \eqref{KPP eq} with compactly supported initial datum (see also \cite{Aronson Weinberger}). Furthermore, the minimal front has the following asymptotic behavior:
\beaa
W(\xi)\sim \xi e^{-\xi}\ \text{ as }\ \xi\to+\infty,\ \text{ if }\ c=2.
\eeaa

To illustrate a rich structure of decay rates in front propagation, we consider the following scalar reaction-diffusion equation with a parameter-dependent monostable nonlinearity
(see \cite{Hadeler Rothe}):
\begin{equation}\label{sample eq}
w_t=w_{xx}+w(1-w)(1+su),\ \ t>0,\ x\in\mathbb{R},
\end{equation}
where $s\geq0$ is a varying parameter. When $s=0$, the model reduces to the classical Fisher-KPP equation mentioned above. Thanks to the availability of explicit solutions, 
the propagation speed and the minimal traveling wave solution $W_*$ for \eqref{sample eq} can be explicitly characterized:
\begin{itemize}
  \item[(1)] in the case $0\leq s<2$, $c^*=2$ and $W_*(\xi)\sim \xi e^{-\xi}$ as $\xi\to +\infty$;
  \item[(2)] in the case $s=2$,  $c^*=2$ and $W_* (\xi)\sim e^{-\xi}$ as $\xi\to +\infty$;
  \item[(3)] in the case $s>2$, $c^*=\sqrt{\frac{2}{s}+\frac{s}{2}}$ and $W_*(\xi)\sim e^{-{\frac{c^*+\sqrt{(c^*)^2-4}}{2}\xi}}$ as $\xi\to +\infty$.
\end{itemize}
The emergence of the pure exponential decay rate $e^{-\xi}$ for $W_*$ exclusively at the critical case $s = 2$ is both surprising and significant. 
Although one can analyze this type of problem using phase plane techniques, the computation in the threshold case can become very intricate.

To better understand this interesting phenomenon, we introduce a novel technique referred to as the 
{\it{slicing method}}, which serves as an alternative to the classical phase plane approach. This method facilitates the construction of suitable comparison functions, especially in threshold regimes where standard techniques are less effective. The central idea is to perform a delicate “slicing” at the leading edge of the traveling front by introducing carefully designed perturbations, thereby enabling us to precisely capture subtle variations in the decay behavior.
In what follows, we consider three representative problems: two scalar equations with different diffusion operators and one classical competition system. Our main focus is to apply the aforementioned technique to understand the properties of the traveling wave solutions, which are the invariants of propagation phenomena, more deeply than before:
\begin{itemize}
\item Reaction-diffusion equation: to clarify why the pure exponential decay rate $e^{-\xi}$ arises only in the critical case in the above sample and generalize this observation;
\item Nonlocal diffusion problem: to classify all 
traveling wave solutions by their asymptotic decay rates;
\item Two species Lotka-Volterra competition system: to solve the long-standing linear speed selection problem completely.
\end{itemize}

\subsection{Traveling waves of three typical monostable systems}

\bigskip
In the following, we will introduce three types of problems mentioned above. Firstly, we consider the reaction-diffusion equation of monostable type
\begin{equation}\label{scalar equation}
\left\{
\begin{aligned}
&w_t=w_{xx}+f(w),\ \ t>0,\ x\in\mathbb{R},\\
&w(0,x)=w_0(x),\ \ x\in\mathbb{R},
\end{aligned}
\right.
\end{equation}
where 
$f$ satisfies
\begin{equation}\label{monostable cd}
f(0)=f(1)=0,\ f'(0)>0>f'(1),\ \text{and}\ f(w)>0\ \text{for all}\ w\in(0,1).
\end{equation}
It is well-known, as shown in \cite{Fisher,KPP}, that under the KPP condition:
\begin{equation}\label{kpp condition}
f'(0)w\geq f(w)\ \text{for all}\ w\in[0,1],
\end{equation}
the spreading speed of \eqref{scalar equation} can be directly derived from the linearization at the invading state $w=0$:
$$w_t=w_{xx}+w.$$
For a long time, it had been widely conjectured that nonlinear differential equations for population spread always have the same velocity as their linear approximation.
This so-called {\it ``linear conjecture''}, developed over more than $80$ years from numerous instances, is stated explicitly by Bosch et al. \cite{Bosch Metz Diekmann} and Mollison \cite{Mollison}. However, it is now known that the conjecture does not hold in general, particularly for nonlinearities that deviate from the standard KPP-type structure.
This motivates a more detailed investigation of the general monostable equation, with attention to the structure of its traveling wave solutions.

For the general monostable equation, it is well-known that the global dynamics of \eqref{scalar equation} are highly related to the properties of traveling wave solutions,
which are particular solutions in the form
$w(t, x)=W(x-ct)=W(\xi)$ satisfying
\begin{equation}\label{def of scalar tw}
\left\{
\begin{aligned}
&W''+cW'+f(W)=0,\ \xi\in\mathbb{R},\\
&W(-\infty)=1,\ W(+\infty)=0,\ W'(\cdot)<0.
\end{aligned}
\right.
\end{equation}
It has been proved that (see \cite{Aronson Weinberger,Volpert}) there exists
$$c^*\geq 2\sqrt{f'(0)}>0$$
such that \eqref{def of scalar tw} admits a solution if and only if $c\geq c^*$. Thus, $c^*$ is called the minimal traveling wave speed. Moreover, Aronson and Weinberger \cite{Aronson Weinberger} showed the existence of a speed $c_w=c^*$ indicating the 
spreading property
of the solution to the Cauchy problem \eqref{scalar equation}
as follows:
\begin{equation*}
\left\{
\begin{aligned}
&\lim_{t\to\infty}\sup_{|x|\ge ct}w(t,x)=0\ \ \mbox{for all} \ \ c>c_w;\\
&\lim_{t\to\infty}\sup_{|x|\le ct}|1-w(t,x)|=0\ \ \mbox{for all} \ \ c<c_w.
\end{aligned}
\right.
\end{equation*}
Therefore, the speed $c_w$ is called the asymptotic speed of spread (in short, spreading speed).
We remark that, in general, the value of the minimal speed $c^*$ depends on the shape of $f$ and cannot be characterized explicitly.

In the literature, the minimal traveling wave is classified into two types: {\em pulled front} and
{\em pushed front} \cite{Rothe1981,Stokes1976,van Saarloos2003}.
\begin{itemize}
    \item The minimal traveling wave $W$ with the speed $c^*$ is called a {\em pulled front} if $c^*=2\sqrt{f'(0)}$.
In this case, the front is pulled by the leading edge with speed determined by the linearized problem at the unstable state $w=0$. Therefore, the minimal speed $c^*$ is said to be linearly selected.
\item On the other hand, if $c^*>2\sqrt{f'(0)}$,
the minimal traveling wave $W$ with a speed $c^*$ is called a {\em pushed front} since
the spreading speed is determined by the whole wave, not only by the behavior of the leading edge. Thus
the minimal speed $c^*$ is said to be nonlinearly selected.
\end{itemize}

The asymptotic behavior of solutions to the Cauchy problem with compactly supported initial datum
differs significantly between these two cases. For pulled fronts, the wave speed coincides with the linear spreading speed, and the front location exhibits a logarithmic delay known as the Bramson correction (see, e.g., \cite{Bramson, Ebert van Saarloos, Giletti, Hamel etal, Lau, Uchiyama}). In contrast, pushed fronts propagate at faster speeds, and the solution converges to a traveling wave profile without any logarithmic correction (see \cite{Rothe1981}). We also refer to \cite{An etal2023-a, An etal2023-b, Avery Scheel} for results on convergence in shape to a traveling wave, and to recent works \cite{Alfaro Giletti Xiao, Berestycki etal} for discussions on the influence of the decay rate of the initial datum.

In the remarkable paper \cite{Aronson Weinberger}, the decay rates of pushed fronts and traveling wave solutions with speeds exceeding the minimal speed were studied using delicate phase plane analysis. However, such techniques are not applicable to nonlocal diffusion equations or competition systems. Subsequently, Hamel extended these results to spatially periodic media in \cite{Hamel}, and further generalizations were made by Guo in \cite{Guo}.

In \cite{Lucia Muratov Novaga},
Lucia, Muratov, and Novaga proposed a variational approach to rigorously establish a mechanism to determine  the linear selection and nonlinear selection on speed for the scalar monostable reaction-diffusion equations. Roughly speaking, the following two conditions are equivalent:
\begin{itemize}
\item[(i)] the minimal traveling wave speed of
$w_t=w_{xx}+f(w)$
is nonlinearly selected;
\item[(ii)] $\Phi_c[w]\leq 0$ holds for some $c>2\sqrt{f'(0)}$ and $w(\not\equiv 0)\in C_0^\infty(\mathbb{R})$, where
\beaa
\Phi_c[w]:=\int_\mathbb{R}e^{cx}\Big[\frac{1}{2}w_x^2-\int_0^wf(s)ds\Big]dx.
\eeaa
\end{itemize}
Roughly speaking, their result implies that the decay rate of the pulled and pushed front is crucial to fully understand the essence of the speed selection problem.

However, a precise description of the decay rate of pulled fronts—that is, the minimal traveling wave solutions in the case of linear selection—remains unclear in the absence of the KPP condition \eqref{kpp condition}. This gap motivates our study: to elucidate the behavior in the critical case. Once this is achieved, we will have a more complete understanding of all traveling wave solutions of \eqref{def of scalar tw}.

\bigskip
Secondly, in recent decades, the nonlocal diffusion problem
\begin{equation}\label{nl}
w_t= J\ast w-w +f(w)
\end{equation}
has appeared widely in various applications ranging from population dynamics to the Ising model as seen in \cite{nl 1, nl 2, nl 3, nl 4, nl 5}. 
Here $J$ is a nonnegative dispersal kernel defined on $\mathbb{R}$, and $J\ast w$ is defined as
$$J\ast w(x):=\int_{\mathbb{R}}J(x-y)w(y)dy.$$ 
For the simplicity of our discussion, throughout this paper, we always assume that the dispersal kernel satisfies
\begin{equation}\label{assumption on J}
J\ge 0\ \text{is compactly supported, symmetric, and}\ \int_{\mathbb{R}}J=1.
\end{equation}
When the nonlinear term $f(w)$ satisfies the KPP condition \eqref{kpp condition}, the traveling waves 
satisfying 
\begin{equation}\label{scalar nonlocal tw-parameter s}
\left\{
\begin{aligned}
&J\ast \mathcal{W}+c\mathcal{W}'+f(\mathcal{W})-\mathcal{W}=0,\quad  \xi\in\mathbb{R},\\
&\mathcal{W}(-\infty)=1,\ \mathcal{W}(+\infty)=0,\\
&\mathcal{W}'<0,\quad \xi\in\mathbb{R},
\end{aligned}
\right.
\end{equation}
have been constructed by \cite{nl 6,nl 7,nl 8} for any $c\ge c^*_{NL}$. 
Subsequently, the uniqueness of traveling wave solutions was established by Carr and Chmaj \cite{Carr Chmaj}, primarily through the application of Ikehara’s theorem. Building on this, Coville et al. \cite{Coville} extended the results to equations where the nonlinear term 
$f(\cdot)$ satisfies only the general monostable condition. They proved the existence of the minimal speed $c_{NL}^*$ such that equation \eqref{nl} admits a unique (up to translation) traveling wave solution $\mathcal{W}$ if and only if
$c\geq c_{NL}^*$. 
Furthermore, a lower bound for the minimal speed is given by $c_{NL}^*\geq c_0^*$,
where the critical speed $c_0^*$ is characterized by the following variational formula
\begin{equation}\label{formula of c_NL}
c_{0}^*:=\min_{\lambda>0}\frac{1}{\lambda}\Big(\int_{\mathbb{R}}J(x)e^{\lambda x}dx+f'(0)-1\Big),
\end{equation}
which derived from the linearization of \eqref{scalar nonlocal tw-parameter s} at the trivial state $\mathcal{W}= 0$.
If $f(\cdot)$ additionally satisfies the KPP condition \eqref{kpp condition}, then
$c_{NL}^*= c_0^*$. In this context, we call the case $c_{NL}^*= c_0^*$ as the linear selection on speed 
and the case $c_{NL}^*> c_0^*$ as the nonlinear selection on speed.

\begin{remark}\label{rm:lambda_0}
Let $h(\lambda)$ be defined by
$$h(\lambda):=\int_{\mathbb{R}}J(z)e^{\lambda z}dz-1+f'(0).$$
It is easy to check that $\lambda\mapsto h(\lambda)$ is an increasing, strictly convex, and sublinear function satisfying $h(0)=f'(0)>0$.
Therefore, there exist only one $\lambda_0>0$ satisfying $h(\lambda_0)=c_0^*\lambda_0$, and for $c>c_0^*$, the equation $h(\lambda)=c\lambda$ admits two different positive roots $\lambda^-(c)$ and $\lambda^+(c)$ satisfying $0<\lambda^-(c)<\lambda_0<\lambda^+(c)$.
\end{remark}

Additionally, it was shown in \cite{Carr Chmaj}, via Ikehara's theorem, that if $f(\cdot)$ satisfies the KPP condition \eqref{kpp condition}, then
\bea\label{decay-U-linear-1}
\mathcal{W}(\xi)=A\xi e^{-\lambda_0\xi}+Be^{-\lambda_0\xi}+o(e^{-\lambda_0\xi})\quad \mbox{as $\xi\to+\infty$},
\eea
where $A>0$ and $B\in \mathbb{R}$.
This asymptotic estimate 
has been extended to the general monostable case in
\cite{Coville}
with $A\geq0$ and $B\in \mathbb{R}$,
and $B>0$ if $A=0$.
However, we note that the proof provided in \cite[Theorem 1.6]{Coville} contains a gap, where the authors deduced that $A>0$ always holds in \eqref{decay-U-linear-1}.
We will fix the gap in Proposition~\ref{prop:correction-U-linear-decay} below. Moreover, in the general monostable case, the analysis of the decay rate for traveling waves with speed $c>c_0^*$ becomes substantially more intricate, as phase plane techniques are no longer applicable.

\bigskip
Finally, we turn our attention to the two-species Lotka-Volterra competition system
\begin{equation}\label{system}
\left\{
\begin{aligned}
&u_t=u_{xx}+u(1-u-av), & t>0,\ x\in\mathbb{R},\\
&v_t=dv_{xx}+rv(1-v-bu), & t>0,\ x\in \mathbb{R},
\end{aligned}
\right.
\end{equation}
where $u=u(t,x)$ and $v=v(t,x)$ represent the population densities of two competing species at the time $t$ and position $x$. Here, $d$ and $r$ represent the diffusion rate and intrinsic growth rate of $v$, respectively. $a$ and $b$ represent the competition coefficient of $v$ and $u$, respectively.

One of the main targets in this paper is to study the speed selection problem of \eqref{system} with the monostable structure, {\it i.e.}, 
$a$ and $b$ satisfy
\begin{itemize}
\item[{\bf(H)}] $0<a<1$ and $b>0$,
\end{itemize} 
which is of significant biological relevance \cite{Murray1993}.
In the long survey paper \cite{van Saarloos2003}, van Saarloos highlighted the practical significance of this problem, pointing out that it
 is not only esoteric from purely academic interest but also plays an important role in reality,
as there are numerous important experimental examples
for which the fronts propagate rapidly into an unstable state.
Among other things, he also
emphasized the importance of the connection 
between pulled fronts and pushed fronts, which is crucial in studying the speed selection problem of 
front propagation. 



Similar to the scalar equation, the spreading speed of the solution starting from the initial datum
\begin{equation}\label{initial datum}
u_0(x)\ge 0\ \ \text{compactly supported continuous function,}\ v_0(x)>0\ \ \text{uniformly positive},
\end{equation}
can be characterized by the minimal traveling wave speed $c_{LV}^{*}$ (see \cite{Lewis Li Weinberger 1}).
The linear and nonlinear selection of $c_{LV}^{*}$ can be defined as follows:
\begin{itemize}
\item
It is linearly selected if
$c_{LV}^{*}=2\sqrt{1-a}$ since the linearization of  \eqref{system}
at the unstable state $(u,v)=(0,1)$ results in the linear speed $2\sqrt{1-a}$.
This situation is also called {\em pulled front} case since the spreading speed is determined only by the leading edge of the distribution of the population.
\item
In the case $c_{LV}^{*}>2\sqrt{1-a}$, we say that the minimal traveling wave speed $c_{LV}^{*}$ is nonlinearly selected.
This situation is also called {\em pushed front} case since the spreading speed is not only determined by the behavior of the leading edge of the population distribution, but by the whole wave.
\end{itemize}
We also refer to the work of Roques et al. \cite{Roques et al 2015} that introduced another definition of the pulled front and the pushed front for \eqref{system}.


Sufficient conditions for linear or nonlinear selection mechanism for \eqref{system} with
$0<a<1<b$ have been investigated
widely.  Okubo et al. \cite{Okubo etal} used a heuristic argument to conjecture that the minimal speed $c_{LV}^{*}$ is linearly selected, and applied it to study the competition between gray squirrels and red squirrels.
Hosono \cite{Hosono 1998} suggested that $c_{LV}^{*}$ can be nonlinearly selected in some parameter regimes.
It has been proved by Lewis, Li and Weinberger \cite{Lewis Li Weinberger 1} that linear selection holds when
\bea\label{LLW-cond}
0<d<2\quad {\rm and}\quad  r(ab-1)\leq (2-d)(1-a).
\eea
An improvement for the sufficient condition for linear selection was found by Huang \cite{Huang2010}:
\begin{equation}\label{huang-condition}
\frac{(2-d)(1-a)+r}{rb}\ge \max\Big\{a,\frac{d-2}{2|d-1|}\Big\}.
\end{equation}
Note that \eqref{LLW-cond} and \eqref{huang-condition} are equivalent when $d\leq 2$.
Although Huang \cite{Huang2010} strongly believed that the condition \eqref{huang-condition} is optimal for linear determinacy,
Roques et al. \cite{Roques et al 2015} numerically reported that the region of the parameter for linear determinacy can still be improved. 
For the minimal speed $c_{LV}^{*}$ being nonlinearly selected, Huang and Han \cite{HuangHan2011} constructed examples in which linear determinacy fails to hold.
Holzer and Scheel \cite{Holzer Scheel 2012} showed that, for fixed $a$, $b$, and $r$,
 the minimal speed $c_{LV}^{*}$ becomes nonlinear selection as $d\to\infty$. 
For related discussions, we also refer to, e.g.,
\cite{Alhasanat Ou2019-1, Alhasanat Ou2019, Guo Liang, Hosono 1995, Hosono 2003} and the references cited therein.  Note that Proposition \ref{prop: implicit cond} in this paper implies that it may be impossible to solve the linear selection problem of Lotka-Volterra competiton system through explicit expression on parameters.

To the best of our knowledge,
the understanding of the sufficient and necessary condition of linear or nonlinear selection mechanism for \eqref{system}, under assumption {\bf(H)}, has not been completely achieved in the literature.
In particular, previous works on speed selection problems for \eqref{system} primarily focused on the strong-weak competition case ($0<a<1<b$). 
 However, as we will demonstrate in Remark~\ref{rm: c nonlinear selection} below, 
there are some cases that the speed $c_{LV}^{*}$ is nonlinearly selected for all $b>1$. These observations indicate that the speed selection problem for \eqref{system}
cannot be 
fully explained by considering only the strong-weak competition case.

In this paper we will fix $a$, $r$, and $d$, and set the competition rate $b\in\mathbb{R}^+$ as a continuously 
varying
parameter.
By analyzing
the asymptotic behavior of the minimal traveling wave at $+\infty$ and constructing novel super-solutions,
we can establish the threshold behavior between the linear selection and nonlinear selection with respect to $b$.
Our result reveals the fundamental mechanism underlying the transition from the linear selection to nonlinear selection for the system \eqref{system}.

\subsection{Intuitive explanation on the 
slicing method}
Before introducing the main results, we begin by providing an explanation of the core technique in this paper—the 
slicing method. First, let us recall the classification of traveling wavefronts for \eqref{def of scalar tw}. We summarize the well-known results as follows:
\begin{proposition}\label{prop: classification scalar}
Assume $f(\cdot)$ satisfies the monostable condition \eqref{monostable cd}. The traveling wavefronts $(c,W)$, defined as in \eqref{def of scalar tw},
satisfies
\begin{itemize}
\item[(1)] there exists $(A,B)\in \mathbb{R}^+\times\mathbb{R}$ or $A=0,B>0$  such that $W(\xi)=A\xi e^{-\xi}+B e^{-\xi}+o(e^{-\xi})$ as $\xi\to+\infty$,
if and only if $c=c^*=2$;
\item[(2)] there exists $A>0$ such that $W(\xi)=Ae^{-\lambda^+(c)\xi}+o(e^{-\lambda^+(c)\xi})$ as $\xi\to+\infty$,
if and only if $c=c^*>2$;
\item[(3)] there exists $A>0$ such that $W(\xi)=Ae^{-\lambda^-(c)\xi}+o(e^{-\lambda^-(c)\xi})$ as $\xi\to+\infty$,
if and only if $c>c^*$.
\end{itemize}
Here, $\lambda^{\pm}(c)$ are defined as
\bea\label{lambda + - scalar}
\lambda^{\pm}(c):=\frac{c\pm\sqrt{c^2-4}}{2}>0.
\eea
\end{proposition}
Proposition \ref{prop: classification scalar}, originally established by Aronson and Weinberger \cite{Aronson Weinberger} via a delicate phase plane analysis, can alternatively be proved 
using the slicing method. Moreover, we can distinguish whether the coefficient $A$ in ~\eqref{assumption 1} vanishes or not.

We now present a heuristic argument to illustrate how one may derive part~(2) of Proposition~\ref{prop: classification scalar}.
Consider the pushed front case governed by 
\begin{equation}\label{non mono eq}
w_t=w_{xx}+f(w),\ \quad t>0,\ x\in\mathbb{R},
\end{equation}
where $f$ satisfies the monostable condition \eqref{monostable cd}. 
For simplicity, we assume $f'(0)=1$. Then the minimal traveling wave speed is known to be $c^*\ge 2$. For any $c\geq c^*$, there exists a unique traveling wave solution $W$ up to translation. Let us consider the case $c>2$. By linearizing the equation satisfied by $W$ (the equation in \eqref{def of scalar tw}) around the unstable state $W=0$,
we obtain the following linearized equation
\begin{equation}\label{lin mono eq}
W''+cW'+W=0.
\end{equation}
It is easy to check that \eqref{lin mono eq} only admits two distinct single roots $\lambda^{\pm}(c)$ since $c>2$.

Next, we briefly explain, in the case $c^*>2$, how to prove the minimal traveling wave $W_*(\xi)\sim e^{-\lambda^+(c^*)\xi}$ through our slicing method. Assume by contradiction that $W_*(\xi)\sim e^{-\lambda^-(c^*)\xi}$. 
Neglecting all intermediate terms between $e^{-\lambda^-\xi}$ and $e^{-\lambda^+\xi}$, 
the asymptotic expansion then takes the form
\begin{equation}\label{assumption 1}
W_*(\xi)=A e^{-\lambda^-(c^*)\xi}+B e^{-\lambda^+(c^*)\xi}+o(1)e^{-\lambda^+(c^*)\xi},\ \text{where}\  A>0.
\end{equation}
Let us consider an auxiliary function of the form
$$\phi(\xi):=\max\{A e^{-\lambda^-(c^*)\xi}-C e^{-\lambda\xi},0\},$$
for $\xi\geq0$, where $C>0$ and $\lambda \in (\lambda^-(c^*),\lambda^+(c^*))$. Then we define the super-solution 
\beaa
\overline{w}(t,x):=W_*(x-(c^*-\delta)t)-\phi(x-(c^*-\delta)t)
\eeaa
for some small $\delta>0$, and thus
 $\overline{w}(t,x)\sim C e^{-\lambda\xi}$, where $\xi=x-(c^*-\delta)t$.
Note that $f'(0)=1$. Then we compute
\begin{equation*}
\overline{w}_t-\overline{w}_{xx}-f(\overline{w})\approx-Ce^{-\lambda\xi}[\lambda^2-(c^*-\delta)\lambda+1]\geq 0,\quad \xi\gg1,
\end{equation*}
which holds for sufficiently small $\delta\ll1$. Moreover, by the choice of $\lambda$, we have $W_*>\phi$. At this point, we have successfully constructed the front part of the super-solution $\overline{w}$, which is the most technically challenging part (i.e., $\xi\gg1$). Since the speed of $\overline{w}$ is equal to $(c^*-\delta)$, this construction implies that the propagation speed of the solution to \eqref{non mono eq} with compactly supported initial datum is at most $(c^*-\delta)$.  However, it is well known that the actual propagation speed equals the minimal traveling wave speed $c^*$. This contradiction 
shows the decay $e^{-\lambda^-\xi}$ must be excluded, 
 and the traveling wave necessarily decays at the faster decay $e^{-\lambda^+\xi}$.

We refer to this construction as a 
slicing perturbation since it effectively slices off a thin portion of 
$W_*$ near the leading edge by introducing a designed auxiliary function $\phi$. In addition,
since $\lambda$ is chosen within the interval $(\lambda^-,\lambda^+)$, a process we describe as “finding the root in the middle”, which reflects the fact that the perturbation lies strictly between the two admissible exponential decay modes.
Another crucial feature of $\phi$ is that its derivative $\phi'$ changes sign exactly once. This sign change plays a key role in enabling a smooth transition between the front part and the left part of $\phi$.

This construction method offers a new and intuitive explanation for why the minimal traveling wave $W_*$ exhibits the decay rate $e^{-\lambda^+\xi}$, which distinguishes it from all other traveling wave solutions with speeds strictly greater than the minimal speed. For any traveling wave $W$ with speed $c \ge c^* > 2$ and decay rate $e^{-\lambda^- \xi}$, a carefully designed slicing perturbation allows the construction of another traveling wave with a slightly slower speed. This implies that the minimal wave $W_*$ must necessarily decay as $e^{-\lambda^+ \xi}$. Thus, the slicing method offers a novel perspective on the classification of decay rates for traveling waves—one that fundamentally departs from the classical phase plane analysis.

A related, though not identical, observation was made by Roquejoffre in \cite{Roquejoffre}. Under the same assumption as in \eqref{assumption 1}, for the case $c=c^*>2$, he showed that there exists a solution to the perturbed equation
$W''+(c^*-\delta)W'+f(W)=0$ of the form 
$$\Phi(\xi)=W_*(\xi)-Ae^{-\lambda^-(c^*)\xi}+e^{\lambda^-(c^*-\delta)\xi}+\phi(\xi),$$
where the unknown auxiliary function $\phi(\xi)$ is obtained by applying the implicit function theorem in a suitable weighted space. However, with our direct method, we can go further by providing a classification of the decay rates in the speed linearly selected case $c=c^*=2$.

Since our construction method relies solely on the linearized roots and the comparison principle, it can be extended to more general monotone dynamical systems where traditional phase plane analysis is not applicable. This includes, for example, nonlocal diffusion problems and spatially periodic problems in cylindrical domains. In this paper, beyond the reaction-diffusion equation, we apply and verify the proposed method for two representative cases: the nonlocal diffusion problem and the Lotka–Volterra competition system, a classical model in population dynamics. Due to the increased complexity of spatially periodic problems in higher-dimensional settings, we do not address such cases in this work, although some progress has been made in that direction.

In practical applications, the auxiliary function 
$\phi$ may take on a form far more complex and delicate than the simple example given above. Moreover, the construction of super- and sub-solutions may involve dividing the interval into more than three distinct regions. Nonetheless, the core idea behind remains consistent.
 Roughly speaking, the front of the auxiliary function $\phi$ usually has two features: (1) its decay rate is between two linearized roots; (2) $\phi'$ changes sign in the front once to connect the next part smoothly.

\begin{remark}

There is a natural question: 
for which classes of monotone dynamical systems does the above observation remain valid?
For the porous medium equation and the $p$-Laplacian equation, the minimal traveling wave solution exhibits a free boundary rather than decaying exponentially to zero. In contrast, the fractional Laplacian equation does not admit a finite propagation speed. Consequently, the types of diffusion operators relevant to our analysis are limited to the classical Laplacian and nonlocal diffusion operators with continuous convolution kernels.
\end{remark}

\begin{remark}
In this paper, we focus exclusively on the nondegenerate case, in which the linearized equation admits only exponential-type eigenfunctions. For the degenerate case 
$$u_t=u_{xx}+u^p(1-u)^q,\ \ p,q>1,\ \ t>0,x\in\mathbb{R},$$
it was shown in \cite{Hou Li Meyer} that traveling wave solutions of the form \eqref{KPP tw} exist for all speeds $c\ge c^*(p,q)$. Moreover, the asymptotic behavior of such waves is classified as follows:
\begin{itemize}
    \item [(1)] $W(\xi)\sim e^{-c\xi}$ as $\xi\to +\infty$, if $c=c^*(p,q)$;
    \item [(2)] $W(\xi)\sim \Big(\frac{c}{(p-1)\xi}\Big)^{\frac{1}{p-1}}$ as $\xi\to +\infty$, if $c>c^*(p,q)$.
\end{itemize}
We believe that a similar classification holds for the nonlocal diffusion equation as well, and that our slicing method is effectively applied to address this problem.
\end{remark}

\subsection{Outline of the paper}
The rest of this paper is organized as follows.

Section 2 is to introduce all the main results.

Sections 3 and 4 are devoted to the speed selection problem for scalar equations. In Section 3, we extend our argument to the scalar reaction-diffusion equation and complete the proof of Theorem \ref{th: threshold scalar equation}. In Section 4, we extend our analysis to the scalar nonlocal diffusion equation and complete the proof of Theorem \ref{th: threshold scalar equation nonlocal}. The proof for Theorem \ref{th: threshold scalar equation nonlocal} is more involved since the minimal traveling wave speed can not be computed explicitly, but is given by a variational formula for the nonlocal diffusion problem.

Sections 5 and 6 are devoted to the speed selection problem for the Lotka-Volterra competition system.
Section 5 is devoted to the results for the existence of traveling waves, and the asymptotic behavior of traveling waves of \eqref{tw solution weak} under condition ({\bf H}).  Particularly, the asymptotic behaviors at $-\infty$ differ for the cases $0<b<1$, $b=1$, and $b>1$, leading to different constructions of super-solutions in Section 6.
In Section 6, we study the speed selection mechanism
for the Lotka-Volterra competition system, where Theorem \ref{th:threshold} is established. The construction of a super-solution to prove the sufficient condition is the most involved part, while the necessary condition is proved by applying the sliding method.

In Section 7, we conclude our observations regarding the speed selection problem and provide a complete classification of
the asymptotic behavior of the minimal traveling wavefronts, {\it i.e.},  Theorem~\ref{th: classification scalar nonlocal} and Theorem~\ref{th: classification}.

\section{Main results}

The first part of this paper is dedicated to the speed selection problem of the scalar equations. We begin by revisiting the speed selection problem
for the minimal traveling wave speed of the scalar monostable reaction-diffusion equations.
We establish a new sufficient and necessary condition for determining the linear or nonlinear selection mechanism by considering a family of continuously varying nonlinearities.
By varying the parameter within the nonlinearity, we obtain a full understanding of how the decay rate of the minimal traveling wave
at infinity affects the minimal speed.
This approach provides insight into the essence underlying the transition from the linear selection to nonlinear selection.
The propagation phenomenon and the inside dynamics
of the front for more general scalar equations have been widely discussed in the literature.
We may refer to, e.g., \cite{Berestycki Hamel2012, Fife McLeod1977, Gardner et al 2012, Liang Zhao 2007, Roques et al 2012, Rothe1981, Stokes1976} and references cited therein.

Furthermore,
as noted in \cite{van Saarloos2003}, many natural elements such as
advection, nonlocal diffusion, and periodicity need to be considered in the propagation problem.  The variational approach, 
as discussed in \cite{Lucia Muratov Novaga}, can 
treat homogeneous scalar equations with the standard Laplace diffusion,
but it is difficult to handle parabolic systems with different diffusion speeds.
In contrast, our method can be applied to equations and systems as long as the comparison principle holds.
In this paper, we also extend our observation on the threshold behavior between linear selection and nonlinear selection for the scalar integro-differential equation, a type of nonlocal diffusion equation.



\subsection{The scalar reaction-diffusion equation}

\noindent

The classification of traveling wavefronts for the scalar equation, provided in Proposition~\ref{prop: classification scalar}, is well-known. 
Our first main result concerns a refined understanding of 
(1) in Proposition \ref{prop: classification scalar}. In other words, we aim to determine under what conditions the coefficient 
$A$ vanishes. To do this, let us consider the following scalar equation
\beaa
w_t= w_{xx} +f(w;s),
\eeaa
where
$\{f(\cdot;s)\}\subset C^2$ is a one-parameter family of nonlinear functions satisfying monostable condition and varies continuously and monotonously on the parameter $s\in[0,\infty)$. 
The assumptions on $f$ are as follows:
\begin{itemize}
    \item[(A1)](monostable condition) $f(\cdot;s)\in C^2([0,1])$, $f(0;s)=f(1;s)=0$, $f'(0;s):=\gamma_0>0>f'(1;s)$,  and $f(w;s)>0$ for all $s\in\mathbb{R}^+$ and $w\in(0,1)$.
    \item[(A2)] (Lipschitz continuity) $f(\cdot;s)$, $f'(\cdot;s)$, and $f''(\cdot;s)$ are Lipschitz continuous on $s\in\mathbb{R}^+$
    uniformly in $w$. In other words,
    there exists $L_0>0$ such that
    \beaa
    |f^{(n)}(w;s_1)-f^{(n)}(w;s_2)|\leq L_0|s_1-s_2|\quad \mbox{for all $\ $ $w\in[0,1]$ and $n=0,1,2$,}
    \eeaa
    where $f^{(n)}$ mean the $n$th derivative of $f$ with respect to $w$ for $n\in\mathbb{N}$, {\it i.e.}, $f^{(0)}=f$, $f^{(1)}=f'$, and $f^{(2)}=f''$.
    \item[(A3)] (monotonicity condition) $f(w;\hat{s})>f(w;s)$ for all $w\in (0,1)$ if $\hat{s}>s$, and $f''(0;\hat{s})> f''(0;s)$ if $\hat{s}>s$.
\end{itemize}

\begin{remark}
Without loss of generality, we assume $\gamma_0=1$ in the assumption (A1) for the part concerned with the scalar reaction-diffusion equation, such that the linearly selected spreading speed is equal to $2$.
\end{remark}

\begin{remark}
Note that, in this paper, we always assume $\{f(\cdot;s)\}\subset C^2$ as that in the assumption (A1) for the simplicity of the proof. As a matter of fact,
our approach still works for weaker regularity of $f$, say
$\{f(\cdot;s)\}\subset C^{1,\alpha}$ for some $\alpha\in(0,1)$.
If we consider a higher degree of regularity for $f$, such as ${f(\cdot;s)}\subset C^k$ for some $k>2$, then the condition in the assumption  (A3) for $f''(0;\cdot)$ will be replaced by $f^{(i)}(0;\cdot)$ for some $1<i\le k$.
\end{remark}


Thanks to the assumption  (A1), there exists the minimal traveling wave speed for all $s\in[0,\infty)$, denoted by $c^*(s)$, such that
the system
\begin{equation}\label{scalar tw-parameter s}
\left\{
\begin{aligned}
&W''+cW'+f(W;s)=0,\quad  \xi\in\mathbb{R},\\
&W(-\infty)=1,\ W(+\infty)=0,\\
&W'<0,\quad \xi\in\mathbb{R},
\end{aligned}
\right.
\end{equation}
admits a unique (up to translations) solution $(c,W)$ if and only if  $c\geq c^*(s)$.

We further assume that linear (resp., nonlinear) selection mechanism can occur at some $s$.
More precisely, $f(\cdot;s)$ satisfies 
\begin{itemize}
    \item[(A4)] there exists $s_1>0$ such that $f(w;s_1)$ satisfies KPP condition \eqref{kpp condition},
   and thus $c^*(s_1)=2$.
    \item[(A5)]  there exists  
     $s_2>s_1$ such that $c^*(s_2)>2$.
\end{itemize}

\begin{remark}\label{rk:a4+a5}
In view of the assumption (A3), a simple comparison yields that $c^*(\hat{s})\geq c^*(s)$ if $\hat{s}\geq s$.
Together with assumptions (A4), (A5) and the fact $c^*(s)\geq 2$ for all $s\geq0$, we see that:
\begin{itemize}
\item[(1)] $c^*(s)=2$ for all $0\leq s\leq s_1$;
\item[(2)] $c^*(s)>2$ for all $s\geq s_2$.
\end{itemize}
\end{remark}

\begin{remark}
It is easy to check that \eqref{sample eq} satisfies assumptions  (A1)-(A5) (see Figure \ref{Figure hr}). The minimal speed  $c^*(s)$ is linearly selected for $0<s\le2$, while it is nonlinearly selected for $s>2$. Note particularly that, for $s\in(1,2]$, the minimal speed $c^*(s)$ is still linearly selected
even though the KPP condition \eqref{kpp condition} is not satisfied.
In addition, we see that the pulled-to-pushed transition front for \eqref{sample eq} occurs when $s=2$.
\end{remark}

\begin{figure}
\begin{center}
\begin{tikzpicture}[scale = 0.8]
\draw[thick](-3,6)-- (-3,0)-- (3,0) node[right] {$w$};
\draw[dashed] [thick](-3,0)--(2.8,5.8) node[above] {$f'(0;s)$};
\node[right] at (-3,6) {$f(w;s)$};
\draw[dotted] [thick] (-3,0) to [out=40,in=180] (0,1.2) to [out=0,in=125] (3,0);
\node[below] at (0.5,2) {$f(w;0)$};
\node[below] at (0,1) {KPP equation};
\draw [thick] (-3,0) to [out=45, in=220] (-0.5,1.7) to [out=45,in=220] (0.7,3.7) to [out=40,in=130] (3,0);
\node[below] at (1,3.4) {$f(w;1)$};
\draw [ultra thick] (-3,0) to [out=35, in=220] (-1.5,2) to [out=45,in=220] (0.5,5) to [out=40,in=120] (3,0);
\node[below] at (0.8,4.8) {$f(w;2)$};
\draw [thick] (-3,0) to [out=35, in=220] (-1,4) to [out=40,in=220] (0.7,6) to [out=40,in=110] (3,0);
\node[below] at (1.1,6) {$f(w;3)$};
\draw[dashed] [thick] (2,4)--(4,4) node[right] {threshold};
\end{tikzpicture}
\caption{The transition from linear selection to nonlinear selection of \eqref{sample eq}.}\label{Figure hr}
\end{center}
\end{figure}

It is well known (\cite{Aronson Weinberger}) that if $c^*(s)=2$, then
\bea\label{decay-W-linear}
W_s(\xi)=A\xi e^{-\xi}+Be^{-\xi}+o(e^{-\xi})\quad \mbox{as $\xi\to+\infty$},
\eea
where 
$A\geq0$ and $B\in \mathbb{R}$,
and $B>0$ if $A=0$.
As we will see, the key point to understanding the speed selection problem is to determine the leading order of the decay rate of $W_s(\xi)$, {\it i.e.},
whether $A>0$ or $A=0$ in \eqref{decay-W-linear}.

\begin{theorem}\label{th: threshold scalar equation}
Assume that assumptions (A1)-(A5) hold. Then there exists the threshold value $s^*\in[s_1,s_2)$ such that the minimal traveling wave speed of \eqref{scalar tw-parameter s} satisfies
\begin{equation}\label{def of threshold scalar}
c^*(s)=2\quad\text{for all}\ s\in[0,s^*];\quad c^*(s)>2\quad\text{for all}\ s\in(s^*,\infty).
\end{equation}
Moreover,
the minimal traveling wave $W_s(\xi)$ satisfies 
\begin{equation}
\mbox{(1) if }s<s^*,\ W_s(\xi)=A\xi e^{-\xi}+o(\xi e^{-\xi})\quad \mbox{as}\quad \xi\to+\infty\quad\text{for some}\quad A>0;
\end{equation}
 \begin{equation}\label{asy tw threshold scalar}
\mbox{(2) if }s=s^*,\ W_s(\xi)=Be^{-\xi}+o(e^{-\xi})\quad \quad \mbox{as} \quad \xi\to+\infty\quad\text{for some}\quad B>0.
\end{equation}

\end{theorem}

\begin{remark}
\begin{itemize}
\item[(1)]
Note that \eqref{asy tw threshold scalar} in Theorem~\ref{th: threshold scalar equation} indicates that, as $\xi\to+\infty$, the leading order of the decay rate of $W_s(\xi)$ switches from $\xi e^{-\xi}$ to $e^{-\xi}$ as $s\to s^*$ from below.
\item[(2)] In our proof of \eqref{def of threshold scalar} and the sufficient condition for \eqref{asy tw threshold scalar},
the condition in the assumption  (A3) that $f''(0;\hat{s})> f''(0;s)$ for $\hat{s}>s$ is not required.
\end{itemize}
\end{remark}

\begin{remark}
The asymptotic behaviors of the pushed front are crucial for understanding the long-time behavior of the solution of the Cauchy problem (see \cite{Rothe1981} for the scalar reaction-diffusion equation and \cite{Wu Xiao Zhou} for the Lotka-Volterra competition-diffusion system).
\end{remark}

\medskip

%

\subsection{The scalar nonlocal equation}\label{sec: intro nonlocal}

\noindent

Next, we  consider the following scalar integro-differential equation
\beaa
w_t= J\ast w-w +f(w;q),
\eeaa
where
$\{f(\cdot;q)\}\subset C^2$ is a one-parameter family of nonlinear functions satisfying assumptions (A1)-(A3) defined in \S 1.2.1 with $s=q$, 

We further assume that a linear (resp., nonlinear) selection mechanism can occur at some $q$.
More precisely, $f(\cdot;q)$ satisfies 
\begin{itemize}
    \item[(A6)] there exists $q_1>0$ such that $f(w;q_1)$ satisfies KPP condition \eqref{kpp condition},
   and thus $c_{NL}^*(q_1)=c^*_0$.
    \item[(A7)]  there exists  
     $q_2>q_1$ such that $c_{NL}^*(q_2)>c^*_0$.
\end{itemize}

\begin{remark}\label{rk:a6+a7}
In view of the assumption  (A3), a simple comparison yields that $c_{NL}^*(\hat{q})\geq c_{NL}^*(q)$ if $\hat{q}\geq q$.
Together with assumptions (A6), (A7) and the fact
$c_{NL}^*(q)\geq c^*_0\ \ \text{for all}\ \ q\geq0,$
we see that
$$c_{NL}^*(q)=c^*_0\ \ \text{for all}\ \ 0\leq q\leq q_1\ \ \text{and}\ \ c_{NL}^*(q)>c^*_0\ \ \text{for all}\ \ q\geq q_2.$$
\end{remark}

One of our main results is 
the complete classification of the decay rates of the traveling wave solutions. 
We first establish a result for the nonlocal diffusion equation that is analogous to the classical diffusion case in Proposition~\ref{prop: classification scalar}.

\begin{theorem}\label{th: classification scalar nonlocal}
Assume that $f(\cdot)$ satisfies the monostable condition \eqref{monostable cd}. The traveling wavefronts $(c,\mathcal W)$, defined as \eqref{scalar nonlocal tw-parameter s}, satisfies
\begin{itemize}
\item[(1)] there exists $(A,B)\in \mathbb{R}^+\times\mathbb{R}$ or $A=0,B>0$  such that $\mathcal W(\xi)=A\xi e^{-\lambda_0\xi}+B e^{-\lambda_0\xi}+o(e^{-\lambda_0\xi})$ as $\xi\to+\infty$,
if and only if $c=c^*_{NL}=c^*_0$;
\item[(2)] there exists $A>0$ such that $\mathcal W(\xi)=Ae^{-\lambda^+(c)\xi}+o(e^{-\lambda^+(c)\xi})$ as $\xi\to+\infty$,
if and only if $c=c^*_{NL}>c^*_0$;
\item[(3)] there exists $A>0$ such that $\mathcal W(\xi)=Ae^{-\lambda^-(c)\xi}+o(e^{-\lambda^-(c)\xi})$ as $\xi\to+\infty$,
if and only if $c>c^*_{NL}$.
\end{itemize}
Here, $\lambda^{\pm}(c)$ are defined as that in Lemma \ref{rm:lambda_0} but independent on $q$.
\end{theorem}

The second result is concerned with how the {\em pulled front} 
evolves to the pulled-to-pushed transition front in terms of the varying parameter $q$. Similar to Theorem \ref{th: threshold scalar equation}, the key point is to completely characterize the evolution of the decay rate of the minimal traveling wave $\mathcal{W}_q(\xi)$ with respect to $q$. It is natural to expect that, as $\xi\to+\infty$, the leading order of the decay rate of $\mathcal{W}_q(\xi)$ switches from $\xi e^{-\lambda_0\xi}$ to $e^{-\lambda_0\xi}$ as $q\to q^*$ from below. However, establishing this result requires a completely different construction of the super-solution.

\begin{theorem}\label{th: threshold scalar equation nonlocal}
Assume that assumptions (A1)-(A3) and (A6)-(A7) hold. Then there exists the threshold value $q^*\in[q_1,q_2)$ such that the minimal traveling wave speed of \eqref{scalar nonlocal tw-parameter s} satisfies
\begin{equation}\label{def of threshold scalar nonlocal}
c_{NL}^*(q)=c_0^*\quad\text{for all}\ q\in[0,q^*];\quad c_{NL}^*(q)>c_0^*\quad\text{for all}\ q\in(q^*,\infty).
\end{equation}
Moreover,
the minimal traveling wave $U_s(\xi)$ satisfies 
\begin{equation*}
\mbox{(1) if }q<q^*,\ \mathcal{W}_q(\xi)=A\xi e^{-\lambda_0\xi}+o(\xi e^{-\lambda_0\xi})\quad \mbox{as}\quad \xi\to+\infty\quad\text{for some}\quad A>0;
\end{equation*}
\begin{equation}\label{asy tw threshold scalar nonlocal}
\mbox{(2) if }q=q^*,\ \mathcal{W}_q(\xi)=Be^{-\lambda_0\xi}+o(e^{-\lambda_0\xi})\quad \quad \mbox{as}\quad \xi\to+\infty\quad\text{for some}\quad B>0.
\end{equation}
\end{theorem}

\begin{remark}
In our proof of \eqref{def of threshold scalar nonlocal} and the sufficient condition for \eqref{asy tw threshold scalar nonlocal},
the condition in the assumption  (A3) that $f''(0;\hat{q})> f''(0;q)$ for $\hat{q}>q$ is not required.
\end{remark}


\begin{remark}
We remark that (3) of Theorem~\ref{th: classification scalar nonlocal} indicates that the pushed front always has a fast decay, which answers an open problem given in \cite{Bonnefon Coville Garnier Roques2014}.
More importantly, our approach is applicable to establish the decay rate of pushed fronts for more general monostable equations and systems as long as the comparison principle holds.
\end{remark}

\subsection{The Lotka-Volterra competition system}

\noindent

In this subsection, we focus on the two-species Lotka-Volterra competition system \eqref{system}. That is,
\begin{equation*}
\left\{
\begin{aligned}
&u_t=u_{xx}+u(1-u-av), & t>0,\ x\in\mathbb{R},\\
&v_t=dv_{xx}+rv(1-v-bu), & t>0,\ x\in \mathbb{R},
\end{aligned}
\right.
\end{equation*}
where all parameters are assumed to be positive, and 
$a$ and $b$ satisfy
\begin{itemize}
\item[{\bf(H)}] $0<a<1$ and $b>0$.
\end{itemize}

Depending on the different dynamics of the related ODE systems, the assumption {\bf(H)} can be classified into three cases:
\begin{itemize}
\item[(I)] $0<a<1<b$ {\em{(the strong-weak competition case)}};
\item[(II)] $0<a<1$ and $0<b<1$ {\em{(the weak competition case)}};
\item[(III)] $0<a<1$ and $b=1$ {\em{(the critical case)}}.
\end{itemize}
Regarding the traveling wave solution of \eqref{system} for the case (I),
Kan-on  \cite{Kan-on1997} showed
that there exists
the minimal traveling wave speed $c_{LV}^{*}\in[2\sqrt{1-a},2]$
such that \eqref{system} admits a positive solution
$(u,v)(x,t)=(U,V)(x-ct)$
 satisfying
\begin{equation*}
\left\{
\begin{aligned}
&U''+cU'+U(1-U-aV)=0,\\
&dV''+cV'+rV(1-V-bU)=0,\\
&(U,V)(-\infty)=(1,0),\ (U,V)(\infty)=(0,1),\\
&U'<0,\ V'>0,
\end{aligned}
\right.
\end{equation*}
if and only if $c\ge c_{LV}^{*}$.
For the case (II),
it has been showed in \cite[Example 4.2]{Lewis Li Weinberger 2} that
there exists
the minimal traveling wave speed $c_{LV}^{*}>0$
such that \eqref{system} admits a positive solution
$(u,v)(x,t)=(U,V)(x-ct)$, connecting
$$(U,V)(-\infty)=(\frac{1-a}{1-ab},\frac{1-b}{1-ab})\quad \text{and}\quad (U,V)(+\infty)=(0,1),$$
if and only if $c\ge c_{LV}^{*}$.
Additionally, the existence of the minimal wave speed for Case (III) can be established by a certain approximation argument.
Moreover, by fixing parameters $a$, $d$, and $r$, the minimal traveling wave speed $c_{LV}^{*}$ is continuous on $b>0$.
Further details
are given in Section 2.
Note that,  we define
$$(u^*,v^*)=(1,0)\ \text{if}\ b\ge 1,\quad (u^*,v^*)=(\frac{1-a}{1-ab},\frac{1-b}{1-ab})\ \text{if}\ b<1,$$
and use
\begin{equation}\label{tw solution weak}
\left\{
\begin{aligned}
&U''+cU'+U(1-U-aV)=0,\\
&dV''+cV'+rV(1-V-bU)=0,\\
&(U,V)(-\infty)=(u^*,v^*),\ (U,V)(\infty)=(0,1),\\
&U'<0,\ V'>0,
\end{aligned}
\right.
\end{equation}
to indicate traveling wave solutions of \eqref{system} throughout this paper
whenever we consider the case (I), (II), or (III).

As seen in the literature, the minimal traveling wave speed depends on system parameters $d, r, a,$ and $b$, but whether linear selection holds is not completely understood until now.
In this paper, we always assume {\bf(H)} and fix $d,r>0$ and $a\in(0,1)$.
We choose the competition rate $b\in\mathbb{R}^+$
as a continuously varying parameter and establish a threshold behavior between the linear and nonlinear selection in terms of $b$.
To emphasize the dependence on parameter $b$, we denote the minimal traveling wave as $(c^*_{LV}(b), U_b,V_b)$.
We will show that there exists $b^*\in(0,+\infty)$ such that
$c_{LV}^{*}(b)$ is linearly selected for $0<b\leq b^*$ and is nonlinearly selected for $b>b^*$. 


A key role in characterizing the transition from linear selection to nonlinear selection is
the asymptotic behavior of the {\em pulled-to-pushed transition front} $U_{b^{*}}$ at $+\infty$. 
It is well known that (see  \cite{Girardin Lam} or \cite{MoritaTachibana2009}) that, if $c^*_{LV}(b)=2\sqrt{1-a}$, then
\begin{equation}\label{1}
U_{b}(\xi)=A\xi e^{-\lambda_u\xi}+B e^{-\lambda_u\xi}+o(e^{-\lambda_u\xi})\ \ \mbox{as}\ \ \xi\to+\infty,
\end{equation}
where $\lambda_u:=\sqrt{1-a}>0$, $A\geq0$, $B\in\mathbb{R}$, and
if $A=0$, then $B>0$. We gain a full understanding of how the decay rate of $U$-fronts at infinity impacts the mechanism of speed selection by showing that $A=0$ occurs if and only if  $b=b^*$.
Namely, the leading order term of the decay rate of $U_{b^{*}}(\xi)$ at $\xi=+\infty$ is $e^{-\lambda_u\xi}$.


We state our main result on the speed selection problem as follows.

\begin{theorem}\label{th:threshold}
For any $d>0$, $r>0$ and $a\in(0,1)$, there exists $b^*\in(0,+\infty)$
such that
\begin{equation*}
c^{*}_{LV}(b)=2\sqrt{1-a}\ \ \mbox{for}\ \ b\in(0,b^*];\quad c^{*}_{LV}(b)>2\sqrt{1-a} \ \ \mbox{for}\ \  b\in(b^*,+\infty).
\end{equation*}
Furthermore,
for the minimal traveling wave $(c^{*}_{LV}(b),U_b,V_b)$ 
satisfying \eqref{tw solution weak},
the following three conditions are equivalent:
\begin{itemize}
\item[(i)] $b=b^*$;
\item[(ii)] $U_b(\xi)=Be^{-\lambda_u\xi}+o(e^{-\lambda_u\xi})$ as $\xi\to+\infty$ for some $B>0$;
\item[(iii)] $\int_{-\infty}^{\infty} e^{\lambda_u\xi} U_b(\xi)[a(1-V_b)-U_b](\xi)d\xi=0$,
\end{itemize}
where $\lambda_u=\sqrt{1-a}$. 
\end{theorem}

Note that the sub-solution for $U$-component constructed in \cite{Huang2010} has the asymptotic behavior
$\xi e^{-\lambda_u\xi}$ as $\xi\to\infty$, which cannot capture the transition front
$U_{b^*}$ with the asymptotic behavior $e^{-\lambda_u\xi}$ as $\xi\to\infty$ reported in Theorem~\ref{th:threshold}.
This observation gives a natural reason for why the condition \eqref{huang-condition} for linear selection can still be improved (see, e.g., \cite{Alhasanat Ou2019-1,Roques et al 2015}).
We formulate this as a corollary as follows.
\begin{corollary}
The condition \eqref{huang-condition} for linear selection is not optimal.
\end{corollary}

\begin{remark}
We should not expect an explicit formula for the speed selection problem of the system \eqref{system}, as found in \cite{Huang2010,Lewis Li Weinberger 1}. Indeed, statements (1) and (2) of Theorem \ref{th:threshold} already suggest that the transition between linear and nonlinear selection is influenced by the entire traveling wave profile 
$(U,V)$, rather than solely by its leading edge. This dependence on the full structure of the wave makes it unlikely that an explicit expression for the speed can be obtained.
\end{remark}

Our second result provides a complete classification of the traveling wavefronts, which improved the related results given in \cite{Kan-On,MoritaTachibana2009}.
\begin{theorem}\label{th: classification}
Assume $d>0$, $r>0$, $a\in(0,1)$, and $b>0$. The traveling wavefronts $(c,U,V)$, defined as \eqref{tw solution weak}, satisfies
\begin{itemize}
\item[(1)] there exists $(A,B)\in \mathbb{R}^+\times\mathbb{R}$ or $A=0,B>0$  such that $U(\xi)=A\xi e^{-\lambda_u\xi}+B e^{-\lambda_u\xi}+o(e^{-\lambda_u\xi})$ as $\xi\to+\infty$,
if and only if $c=c^*_{LV}=2\sqrt{1-a}$;
\item[(2)] there exists $A>0$ such that $U(\xi)=Ae^{-\lambda^+_u(c)\xi}+o(e^{-\lambda^+_u(c)\xi})$ as $\xi\to+\infty$,
if and only if $c=c^*_{LV}>2\sqrt{1-a}$;
\item[(3)] there exists $A>0$ such that $U(\xi)=Ae^{-\lambda^-_u(c)\xi}+o(e^{-\lambda^-_u(c)\xi})$ as $\xi\to+\infty$,
if and only if $c>c^*_{LV}$.
\end{itemize}
Here, $\lambda^{\pm}_u(c)$ are eigenvalues defined in Lemma \ref{lm: behavior around + infty}.
\end{theorem}


Theorem~\ref{th:threshold} indicates that $(U_{b^*},V_{b^*})$ is the
{\em pulled-to-pushed transition front}.
 Furthermore, with Theorem \ref{th: classification},
we can fully understand how the decay rates of the minimal traveling wave solution depend on $b$ and completely classify propagation fronts.


\begin{remark}\label{rm: c nonlinear selection}
It is the first time to provide a sufficient and necessary condition for the speed selection problem of the Lotka-Volterra competition system under {\bf(H)}. 
We have improved the understanding of this problem by considering a wide range of competition coefficients $0<a<1$ and $0<b<+\infty$,
not just the previously studied case of $0<a<1<b$.
In addition, we expect that
in some cases, $c^*_{LV}(b)>2\sqrt{1-a}$ for all $b>1$, indicating that the threshold $b^*$ may not be well-defined by only considering $b>1$.

For instance,  numerical simulations suggest that
for any fixed $0<a<1$ and $r>0$, there exists $d_0>0$ sufficiently large such that
$$c^*_{LV}(b)>2\sqrt{1-a}\ \ \text{for all}\ \ b>1\ \ \text{if}\ \ d>d_0.$$
In Figure \ref{Figure0},
we consider \eqref{system} with
$$a=b=1/2,\ r=1,\ v_0(x)\equiv2/3$$
 and $u_0(x)$ satisfying
$$u_0(x)=1\ \ \text{for}\ \ x\le 10,\quad u_0(x)=0\ \ \text{for}\ \ x> 10.$$
Set
$x(t):=\sup_{x\geq 0}\{x>0|\, u(t,x)=1/2\}$.
A numerical simulation suggests that
$$\liminf_{t\to\infty}[x(t)/t]>2\sqrt{1-a}=\sqrt{2}$$
when $d=50$. Together with the comparison principle, it indicates that
the spreading speed should be nonlinearly selected for all $b>1/2$ when $a=1/2$, $r=1$, and $d=50$.
\begin{figure}
\centering
\includegraphics[width=0.74\textwidth]{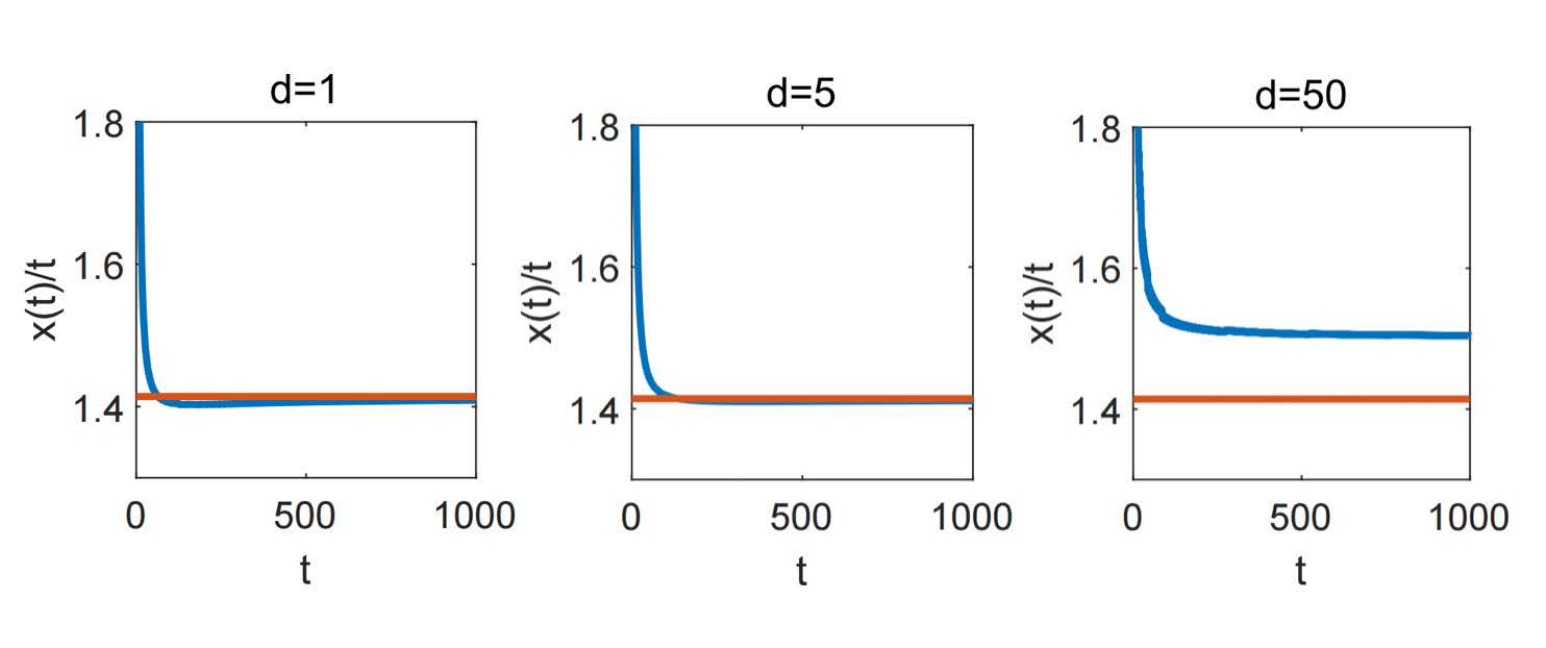}
\caption{The horizontal axis is the time; the vertical axis represents $x(t)/t$;
the orange line indicates the value $2\sqrt{1-a}=\sqrt{2}$, and
the blue curve represents the evolution of $x(t)/t$ on different $d$.} \label{Figure0}
\end{figure}

On the other hand, numerical simulations suggest that for any fixed $0<a<1$ and $d>0$,  there exists $r_0>0$ sufficiently small such that
$$c^*_{LV}(b)>2\sqrt{1-a}\ \ \text{for all}\ \ b>1\ \text{if}\ r<r_0.$$
In Figure \ref{Figure00}, we consider \eqref{system} with $a=b=1/2$, $d=1$,
and the initial datum $(u_0,v_0)$ is taken as the same as the one in Figure~\ref{Figure0}.
Together with the comparison principle, 
it suggests that the wave speed should be
nonlinearly selected for all $b>1/2$ when $a=1/2$, $d=1$ and $r=0.00001$.
\begin{figure}
\centering
\includegraphics[width=0.78\textwidth]{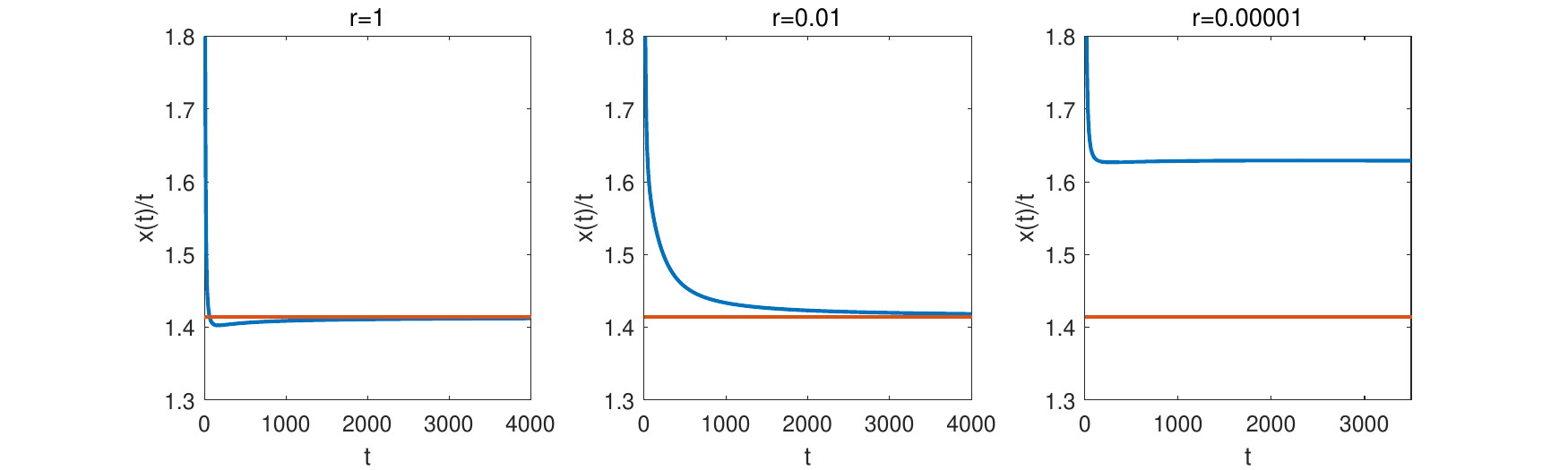}
\caption{
the blue curve represents the evolution of $x(t)/t$ on different $r$.} \label{Figure00}
\end{figure}
\end{remark}

\section{Threshold of the reaction-diffusion equation}\label{sec:threshold scalar}

In this section, we aim to prove Theorem~\ref{th: threshold scalar equation}.
First, it is well known that for each $s\geq0$, under the assumption  (A1), the minimal traveling wave is unique (up to a translation).
Together with the assumption (A2), one can use the standard compactness argument to conclude that $c^*(s)$ is continuous for all $s\ge 0$.
It follows from assumptions (A3)-(A5) and Remark~\ref{rk:a4+a5} that $c^*(s)$ is nondecreasing in $s$. Thus, we immediately obtain the following result.

\begin{lemma}\label{lem: th1-part1}
Assume that assumptions (A1)-(A5) hold. Then there exists a threshold $s^*\in[s_1,s_2)$ such that \eqref{def of threshold scalar} holds.
\end{lemma}

Thanks to Lemma~\ref{lem: th1-part1}, to prove Theorem~\ref{th: threshold scalar equation}, it suffices to show that
\eqref{asy tw threshold scalar} holds if and only if $s=s^*$.
Let $W_{s^*}$ be the minimal traveling wave satisfying \eqref{scalar tw-parameter s} with $s=s^*$
and $c^*(s^*)=2$. For simplicity, we denote $W_*:=W_{s^*}$.
The first and the most involved step is to show that if $s=s^*$, then \eqref{asy tw threshold scalar} holds.
To do this, we shall use a contradiction argument.
Assume that \eqref{asy tw threshold scalar} is not true. Then, 
it holds that (cf. \cite{Aronson Weinberger})
\bea\label{AS-U-infty-for-contradiction scalar}
\lim_{\xi\rightarrow+\infty}\frac{W_{*}(\xi)}{\xi e^{-\xi}}=A_0\quad \mbox{for some $A_0>0$.}
\eea


Under the condition \eqref{AS-U-infty-for-contradiction scalar},
we shall prove the following proposition.

\begin{proposition}\label{Prop-supersol}
Assume that assumptions (A1)-(A5) hold. In addition, if \eqref{AS-U-infty-for-contradiction scalar} holds,
then there exists an auxiliary continuous function $R_w(\xi)$ defined in $\mathbb{R}$ satisfying
\begin{equation}\label{decay rate of Rw}
R_w(\xi)=O(\xi e^{-\xi}) \quad \mbox{as $\xi\to\infty$},
\end{equation}
such that
$$\overline{W}(\xi):=\min\{W_{*}(\xi)-R_w(\xi),1\}\geq (\not\equiv) 0$$
is a super-solution
satisfying
\bea\label{tw super solution scalar}
&N_0[\overline W]:=\overline W''+2\overline W'+f(\overline W;s^*+\delta_0)\le0,\quad  \mbox{a.e. in $\mathbb{R}$},
\eea
for some small $\delta_0>0$,
where ${\overline W}'(\xi_0^{\pm})$ exists and
${\overline W}'(\xi_0^+)\leq {\overline W}'(\xi_0^-)$
if $\overline W'$ is not continuous at $\xi_0$.

\end{proposition}

Next, we shall go through a lengthy process to prove Proposition~\ref{Prop-supersol}.
Hereafter, assumptions (A1)-(A5) are always assumed.



From the assumption (A1), by shifting the coordinates, we can immediately obtain
the following lemma.

\begin{lemma}\label{lm: divide R to 3 parts}
Let $\nu_{1}>0$ be 
an arbitrary constant. Then there exist
$$-\infty<\xi_2<0<\xi_1<+\infty\ \ \text{with}\ \ |\xi_1|,|\xi_2|\ \text{very large},$$
such that
the following hold:
\begin{itemize}
    \item[(1)]
    $\dps f(W_{*}(\xi);s^*)=W_{*}(\xi)+\frac{f''(0;s^*)}{2}W^2_{*}(\xi)+o(W^2_{*}(\xi))$ for all $\xi\in[\xi_1,\infty)$;
    \item[(2)]  $f'(W_{*}(\xi);s^*)<0$ for all $\xi\in(-\infty,\xi_2]$.
\end{itemize}
\end{lemma}

\subsection{Construction of the super-solution}\label{subsec-2-1}

\noindent

Let us define $R_w(\xi)$ as (see Figure \ref{Figure scalar})
\begin{equation}\label{definition of Rw}
R_w(\xi)=\begin{cases}
\varepsilon_1\sigma(\xi) e^{-\xi},&\ \ \mbox{for}\ \ \xi\ge\xi_{1}+\delta_1,\\
\varepsilon_2 e^{\lambda_1\xi},&\ \ \mbox{for}\ \ \xi_{2}+\delta_2\le\xi\le\xi_{1}+\delta_1,\\
\varepsilon_3\sin(\delta_4(\xi-\xi_2)),&\ \ \mbox{for}\ \ \xi_2-\delta_3\le\xi\le\xi_2+\delta_2,\\
-\varepsilon_4e^{\lambda_2\xi},&\ \ \mbox{for}\ \ \xi\le\xi_2-\delta_3,
\end{cases}
\end{equation}
where $\delta_{i=1,\cdots,4}>0$, $\lambda_{n=1,2}>0$, and $\sigma(\xi)>0$ will be
determined such that $\overline W(\xi)$ satisfies \eqref{decay rate of Rw} and \eqref{tw super solution scalar}. Moreover, we should choose
positive $\varepsilon_{j=1,\cdots,4}\ll A_0$ ($A_0$ is defined in \eqref{AS-U-infty-for-contradiction scalar}) such that $R_w(\xi)\ll W_*(\xi)$ and $\overline W(\xi)$ is continuous for all $\xi\in\mathbb{R}$.

Since $f(\cdot;s^*)\in C^2$, there exist $K_1>0$ and $K_2>0$ such that
\bea\label{K1K2}
|f''(W_*(\xi);s^*)|< K_1,\quad |f'(W_*(\xi);s^*)|< K_2\quad \mbox{for all}\quad\xi\in\mathbb{R}.
\eea
We set $\lambda_1>0$ 
large enough such that
\begin{equation}\label{condition on lambda 1}
-2\lambda_1-\lambda_1^2+K_2<0\ \text{and}\ \lambda_1>K_2.
\end{equation}
Furthermore, there exists $K_3>0$ such that
\begin{equation}\label{condition on xi_2 scalar}
f'(W_*(\xi); s^*)\le -K_3<0\quad\text{for all}\quad\xi\le \xi_2.
\end{equation}
We set
$$0<\lambda_2<\lambda_w:=\sqrt{1-f'(1;s^*)}-1$$
sufficiently small
such that
\bea\label{lambda 2}
\lambda_2^2+2\lambda_2-K_3<0.
\eea


\begin{figure}
\begin{center}
\begin{tikzpicture}[scale = 1.1]
\draw[thick](-6,0) -- (7,0) node[right] {$\xi$};
\draw [semithick] (-6, -0.1) to [ out=0, in=150] (-4,-0.4) to [out=30, in=230] (-3.5,0)  to [out= 40, in=180] (-3,0.2) to [out=15, in=220] (1,1.6) to [out=70,in=190] (1.5,2) to [out=0,in=170] (7,0.3);
\node[below] at (1,1.2) {$\xi_{1}+\delta_1$};
\draw[dashed] [thick] (-3.5,0)-- (-3.5,0.3);
\node[above] at (-3.5,0.3) {$\xi_2$};
\draw[dashed] [thick] (-3,0.2)-- (-3,-0.5);
\draw[dashed] [thick] (-4,0)-- (-4,1);
\node[above] at (-4,0.8) {$\xi_2-\delta_3$};
\draw[dashed] [thick] (1,0)-- (1,0.65);
\draw[dashed] [thick] (1,1.15)-- (1,1.6);
\node[below] at (-3,-0.4) {$\xi_2+\delta_2$};
\draw [thin] (-3.85,-0.31) arc [radius=0.2, start angle=40, end angle= 140];
\node[above] at (-4,-0.36) {$\alpha_3$};
\draw [thin] (-2.8,0.255) arc [radius=0.2, start angle=30, end angle= 175];
\node[above] at (-3,0.3) {$\alpha_2$};
\draw [thin] (1.1,1.8) arc [radius=0.2, start angle=70, end angle= 220];
\node[above] at (1,1.8) {$\alpha_1$};
\end{tikzpicture}
\caption{the construction of $R_w(\xi)$.}\label{Figure scalar}
\end{center}
\end{figure}


\medskip

We now divide the proof into several steps.

\noindent{\bf{Step 1}:} We consider $\xi\in[\xi_1+\delta_1,\infty)$ where $\delta_1>0$ is small enough and will be determined in Step 2.
In this case, we have
\beaa
R_w(\xi)=\varepsilon_1\sigma(\xi)\ e^{-\xi}
\eeaa
for some small $\varepsilon_1\ll A_0$
such that $\overline W=W_*-R_w>0$ for $\xi\geq \xi_1+\delta_1$.

Note that $W_*$ satisfies \eqref{scalar tw-parameter s} with $c=2$.
By some  straightforward computations, we have
\begin{equation}\label{N0 inequality step 1}
\begin{aligned}
N_0[\overline W]=&-R''_w-2 R'_w-f(W_*;s^*)+f(W_*-R_w;s^*+\delta_0)\\
=&-R''_w-2 R'_w-f(W_*;s^*)+f(W_*-R_w;s^*)\\
&-f(W_*-R_w;s^*)+f(W_*-R_w;s^*+\delta_0).
\end{aligned}
\end{equation}
By the assumption (A1) and the statement (1) of Lemma \ref{lm: divide R to 3 parts}, since $W_*\ll 1$ and $R_w\ll W_*$ for $\xi\in[\xi_1+\delta_1,\infty)$, we have
\begin{equation}\label{estimate 1}
-f(W_*;s^*)+f(W_*-R_w;s^*)= -R_w+f''(0;s^*)(\frac{R^2_w}{2}-W_*R_w)+o((W_*)^2).
\end{equation}
By the assumption (A2) and the statement (1) of Lemma \ref{lm: divide R to 3 parts}, there exists $C_1>0$ such that
\begin{equation}\label{estimate 2}
-f(W_*-R_w;s^*)+f(W_*-R_w;s^*+\delta_0)\le C_1\delta_0(W_*-R_w)^2+o((W_*)^2).
\end{equation}
From \eqref{K1K2}, \eqref{N0 inequality step 1}, \eqref{estimate 1}, \eqref{estimate 2}, 
we have
\bea\label{estimate 3}
N_0[\overline W]\le -\varepsilon_1\sigma''e^{-\xi}+K_1(\frac{R_w^2}{2}+W_*R_w)+C_1\delta_0W_*^2+o((W_*)^2).
\eea

Now, we define
$$\sigma(\xi):=4e^{-\frac{1}{2}(\xi-\xi_1)}-4+4\xi-4\xi_1$$
which satisfies
$$\sigma(\xi_1)=0,\ \sigma'(\xi)=4-2e^{-\frac{1}{2}(\xi-\xi_1)},\ \sigma''(\xi)=e^{-\frac{1}{2}(\xi-\xi_1)}.$$
Moreover, $\sigma(\xi)= O(\xi)$ as $\xi\to\infty$ implies that $R_w$ satisfies \eqref{decay rate of Rw}.

Due to \eqref{AS-U-infty-for-contradiction scalar} and the equation of $W_*$, we may also assume
\bea\label{W-M}
W_*(\xi) \leq 2A_0\xi e^{-\xi}\quad \mbox{for all}\quad\xi\geq \xi_1.
\eea
Then, from \eqref{estimate 3}, up to enlarging $\xi_1$ if necessary, we always have 
\beaa
N_0[\overline W]&\le &-\varepsilon_1e^{-\frac{1}{2}(\xi-\xi_1)}e^{-\xi}+K_1(\frac{R_w^2}{2}+W_*R_w)+C_1\delta_0W_*^2+o((W_*)^2)\\
&\le& -\frac{\varepsilon_1}{2}e^{-\frac{1}{2}(\xi-\xi_1)}e^{-\xi}+C_1\delta_0W_*^2\\
\eeaa
for any $\delta_0>0$
since $R_w^2(\xi)$, $W_*R_w(\xi)$, and  $W_*^2(\xi)$ are $o(e^{-\frac{3}{2}\xi})$
as $\xi\to\infty$
by \eqref{W-M} and the definition of $R_w$.
Consequently, we find some $\delta_0(\varepsilon_1)\ll \varepsilon_1$, not depending on $\xi_1$ such that $N_0[\overline W]\leq 0$ for $\xi\ge \xi_1$.

\medskip

\noindent{\bf{Step 2:}} We consider $\xi\in[\xi_2+\delta_2,\xi_1+\delta_1]$ for $\xi_1+\delta_1$ fixed by Step 1, and small $\delta_1>0$ satisfying
\bea\label{cond delta 1 scalar}
1+3(1-e^{-\frac{\delta_1}{2}})-2\delta_1>0.
\eea
In this case, we have $R_w(\xi)=\varepsilon_2 e^{\lambda_1\xi}$
for some large $\lambda_1>0$ satisfying \eqref{condition on lambda 1}.
Note that, $\xi_1$ is decided in Step 1, and it is easy to check that $R'_w((\xi_1+\delta_1)^+)>0$ under the condition \eqref{cond delta 1 scalar}.

We first choose
\begin{equation}\label{condition on epsilon 2-RD}
\varepsilon_2=\varepsilon_2(\varepsilon_1,\delta_1)=\varepsilon_1\Big(4e^{-\frac{\delta_1}{2}}-4+4\delta_1\Big)e^{-(1+\lambda_1)(\xi_1+\delta_1)}
\end{equation}
such that $R_w(\xi)$ is continuous at $\xi=\xi_{1}+\delta_1$. Then, from \eqref{condition on epsilon 2-RD}, we have
$$R'_w((\xi_1+\delta_1)^{+})=\varepsilon_1\sigma'(\xi_1+\delta_1)e^{-(\xi_1+\delta_1)}- R_w(\xi_1+\delta_1)>R'_w((\xi_1+\delta_1)^{-})=\lambda_1R_w(\xi_1+\delta_1)$$
is equivalent to
\begin{equation}\label{d2}
1+(3+2\lambda_1)(1-e^{-\frac{\delta_1}{2}})>2(1+\lambda_1)\delta_1,
\end{equation}
which holds by taking $\delta_1$ sufficiently small.
This implies that $\angle\alpha_1<180^{\circ}$.

From now on, $\delta_1$ satisfying \eqref{cond delta 1 scalar} and \eqref{d2} is fixed. By some  straightforward computations, we have
\begin{equation*}
\begin{aligned}
N_0[\overline W]=&-(2\lambda_1+\lambda_1^2)R_w-f(W_*;s^*)+f(W_*-R_w;s^*+\delta_0)\\
=&-(2\lambda_1+\lambda_1^2)R_w-f(W_*;s^*)+f(W_*-R_w;s^*)\\
&-f(W_*-R_w;s^*)+f(W_*-R_w;s^*+\delta_0).
\end{aligned}
\end{equation*}
Thanks to \eqref{K1K2}, we have
$$-f(W_*;s^*)+f(W_*-R_w;s^*)< K_2R_w.$$
Moreover, by assumption (A2), 
$$-f(W_*-R_w;s^*)+f(W_*-R_w;s^*+\delta_0)\le L_0\delta_0.$$
Then, 
since
$\lambda_1$ satisfies \eqref{condition on lambda 1}, we have
\begin{equation}\label{d1}
L_0\delta_0<\varepsilon_2(\lambda_1^2+2\lambda_1-K_2) e^{\lambda_1(\xi_2+\delta_2)}
\end{equation}
for all sufficiently small $\delta_0>0$. Note that, since $\delta_1$ is fixed by the discussion in Step 2.
As a result, in Step 2, we find some $\delta_0(\varepsilon_1,\xi_2+\delta_2)$ such that $N_0[\overline W]\le 0$ for all $\xi\in[\xi_2+\delta_2,\xi_1+\delta_1]$.

\medskip

\noindent{\bf{Step 3:}} We consider $\xi\in[\xi_2-\delta_3,\xi_2+\delta_2]$ with $\xi_2+\delta_2$ fixed by Step 2 and some small $\delta_2,\delta_3>0$. $\delta_2$ is determined in Step 3, and $\delta_3$ will be determined in Step 4. In this case, $R_w(\xi)=\varepsilon_3\sin(\delta_4(\xi-\xi_2))$.
We first verify the following Claim.
\begin{claim}\label{cl scalar}
For any $\delta_2$ with $\delta_2>\frac{1}{\lambda_1}$, there exist
$\varepsilon_3>0$ and small $\delta_4>0$
such that
$$R_w((\xi_2+\delta_2)^+)=R_w((\xi_2+\delta_2)^-)$$
and $\angle\alpha_2<180^{\circ}$.
\end{claim}
\begin{proof}
Note that
$$R_w((\xi_2+\delta_2)^+)=\varepsilon_2e^{\lambda_1(\xi_2+\delta_2)},\ R_w((\xi_2+\delta_2)^-)=\varepsilon_3\sin(\delta_4\delta_2).$$
Therefore, by \eqref{condition on epsilon 2-RD}, we may take
\bea\label{epsilon 3}
\varepsilon_3= \varepsilon_3(\varepsilon_1,\delta_2,\delta_4)=\frac{\varepsilon_2e^{\lambda_1(\xi_2+\delta_2)}}{\sin(\delta_4\delta_2)}=\varepsilon_1\Big(4e^{-\frac{\delta_1}{2}}-4+4\delta_1\Big)\frac{e^{\lambda_1(\xi_2+\delta_2)-(1+\lambda_1)(\xi_1+\delta_1)}}{\sin(\delta_4\delta_2)}>0
\eea
such that $R_w((\xi_2+\delta_2)^+)=R_w((\xi_2+\delta_2)^-)$.

By some  straightforward computations, we have
$R'_w((\xi_2+\delta_2)^+)=\lambda_1\varepsilon_2e^{\lambda_1(\xi_2+\delta_2)}$ and
\beaa
R'_w((\xi_2+\delta_2)^-)=\varepsilon_3\delta_4\cos(\delta_4\delta_2)=\frac{\varepsilon_2e^{\lambda_1(\xi_2+\delta_2)}}{\sin(\delta_4\delta_2)}\delta_4\cos(\delta_4\delta_2),
\eeaa
which yields that
$$R'_w((\xi_2+\delta_2)^-)\rightarrow \varepsilon_2e^{\lambda_1(\xi_2+\delta_2)}/\delta_2\ \ \text{as}\ \ \delta_4\to0.$$
In other words, as $\delta_4\to0$,
\bea\label{delta 2}
R'_w((\xi_2+\delta_2)^+)>R'_w((\xi_2+\delta_2)^-)\ \ \text{is equivalent to}\ \ \delta_2>\frac{1}{\lambda_1}.
\eea
Therefore, we can choose $\delta_4>0$ sufficiently small so that $\angle\alpha_2<180^{\circ}$.
This completes the proof of Claim~\ref{cl scalar}.
\end{proof}

Next, we verify the differential inequality of $N_0[\overline W]$ for $\xi\in[\xi_2-\delta_3,\xi_2+\delta_2]$. By some  straightforward computations, we have
\begin{equation*}
\begin{aligned}
N_0[\overline W]=&\delta^2_4R_w-2\varepsilon_3\delta_4\cos(\delta_4(\xi-\xi_2))\\
&-f(W_*;s^*)+f(W_*-R_w;s^*)-f(W_*-R_w;s^*)+f(W_*-R_w;s^*+\delta_0).
\end{aligned}
\end{equation*}
The same argument as in Step 2 implies that
$$-f(W_*;s^*)+f(W_*-R_w;s^*)\le K_2R_w\ \ \text{and}\ \ -f(W_*-R_w;s^*)+f(W_*-R_w;s^*+\delta_0)\le L_0\delta_0,$$
which yields that
\beaa
N_0[\overline W]\leq\delta^2_4R_w-2\varepsilon_3\delta_4\cos(\delta_4(\xi-\xi_2))+K_2R_w+L_0\delta_0.
\eeaa
We first focus on $\xi\in[\xi_2,\xi_2+\delta_2]$. From now on, we fix 
$\delta_2\in(1/\lambda_1,1/K_2)$.
Then, by \eqref{epsilon 3}, \eqref{delta 2}, and the definition of $\lambda_1$ (see \eqref{condition on lambda 1}), 
\beaa
\min_{\xi\in[\xi_2,\xi_2+\delta_2]}\delta_4\varepsilon_3\cos(\delta_4(\xi-\xi_2))\to\frac{\varepsilon_2e^{\lambda_1(\xi_2+\delta_2)}}{\delta_2}
=\frac{R_w(\xi_2+\delta_2)}{\delta_2}> K_2R_w(\xi_2+\delta_2)
\quad\text{as}\ \delta_4\to0.
\eeaa
Thus, we have
\beaa
\min_{\xi\in[\xi_2,\xi_2+\delta_2]}\Big[\delta_4\varepsilon_3\cos(\delta_4(\xi-\xi_2))-(K_2+\delta_4^2)R_w(\xi)\Big]>0,
\eeaa
for all sufficiently small $\delta_4>0$.
Then, for all sufficiently small $\delta_0(\varepsilon_1,\delta_2,\delta_4)>0$, 
we see that $N_0[\overline W]\le 0$ on the bounded interval $\xi\in[\xi_2,\xi_2+\delta_2]$.

For $\xi\in[\xi_2-\delta_3,\xi_2]$,
by setting $\delta_3>0$ small enough,
$N_0[\overline W]\le 0$ can be verified easier by the same argument since $R_w<0$. 
As a result, we find some $\delta_0(\varepsilon_1)>0$ such that $N_0[\overline W]\le 0$ for $\xi\in[\xi_2-\delta_3,\xi_2+\delta_2]$, by recalling that $\delta_2,\delta_3, \delta_4$ are fixed by the discussion in Step 3.

\medskip

\noindent{\bf{Step 4:}} We consider $\xi\in(-\infty,\xi_2-\delta_3]$ with $\xi_2-\delta_3$ determined in Step 3. In this case, we have $R_w(\xi)=-\varepsilon_4e^{\lambda_2\xi}<0$, and $\delta_3$ is decided in Step 3.
Recall that we choose $0<\lambda_2<\lambda_w$
and
$$1-W_*(\xi)\sim C_2 e^{\lambda_{w}\xi}\ \ \text{as}\ \ \xi\to-\infty.$$
Then, there exists $M>0$ such that
$$\overline W:=\min\{W_*-R_w,1\}\equiv 1\ \ \text{for all} \ \ \xi\le -M,$$
and thus $N_0[\overline W]\le 0$ for all $\xi\le -M$. Therefore, we only need to show
$$N_0[\overline W]\le 0\ \ \text{for all}\ \ -M\le\xi\le \xi_2-\delta_3.$$

From now on, we fix $\xi_2\le \delta_3-M$. Then, by \eqref{epsilon 3}, we choose
$$\varepsilon_4=\varepsilon_4(\varepsilon_1)=\varepsilon_3\frac{\sin(\delta_4\delta_3)}{e^{\lambda_2(\xi_2-\delta_3)}}=\varepsilon_1\Big(4e^{-\frac{\delta_1}{2}}-4+4\delta_1\Big)\frac{e^{\lambda_1(\xi_2+\delta_2)-(1+\lambda_1)(\xi_1+\delta_1)}\sin(\delta_4\delta_3)}{e^{\lambda_2(\xi_2-\delta_3)}\sin(\delta_4\delta_2)}$$
such that $R_w$ is continuous at $\xi_2-\delta_3$. It is easy to check that
$$R'_w((\xi_2-\delta_3)^+)>0>R'_w((\xi_2-\delta_3)^-),$$
 and hence $\angle\alpha_3< 180^{\circ}$.

By some  straightforward computations, we have
\begin{equation*}
\begin{aligned}
N_0[\overline W]=&-(\lambda_2^2+2\lambda_2)R_w-f(W_*;s^*)+f(W_*-R_w;s^*+\delta_0)\\
=&-(\lambda_2^2+2\lambda_2)R_w-f(W_*;s^*)+f(W_*-R_w;s^*)\\
&-f(W_*-R_w;s^*)+f(W_*-R_w;s^*+\delta_0).
\end{aligned}
\end{equation*}
From \eqref{condition on xi_2 scalar}, we have
$$-f(W_*;s^*)+f(W_*-R_w;s^*)< K_3R_w<0.$$
Together with the assumption  (A2), we have
\beaa
N_0[\overline W]\leq -(\lambda_2^2+2\lambda_2-K_3)R_w+L_0\delta_0\quad\text{for all}\quad \xi\in[-M,\xi_2-\delta_3].
\eeaa

In view of  \eqref{lambda 2},
we can assert that
$$N_0[\overline W]\le 0\ \ \text{for all}\ \ \xi\in[-M,\xi_2-\delta_3],$$
 provided that $\delta_0$ is sufficiently small. 
 We note that, from Step 1 to Step 4, the choice of 
$\delta_0$ depends on several parameters. However, all of them, except for
$\varepsilon_1$, are fixed throughout construction. Hence, in the end, it suffices to choose $\delta_0\ll \varepsilon_1$. 
This completes Step 4, and thereby the entire construction of $R_w$.

\subsection{Proof of Theorem \ref{th: threshold scalar equation}}

\noindent

We first complete the proof of Proposition~\ref{Prop-supersol}.
\begin{proof}[Proof of Proposition~\ref{Prop-supersol}]
From the discussion from Step 1 to Step 4 in \S \ref{subsec-2-1},
we are now equipped with a suitable function $R_w(\xi)$ 
defined as in \eqref{definition of Rw} such that
$$\overline W (\xi)=\min \{W_*(\xi)-R_w(\xi),1\},$$
 which is independent of the choice of all sufficiently small $\delta_0>0$, forms
a super-solution satisfying \eqref{tw super solution scalar}.
Therefore, we complete the proof of Proposition~\ref{Prop-supersol}.
\end{proof}

Now, we are ready to prove Theorem~\ref{th: threshold scalar equation} as follows.

\begin{proof}[Proof of Theorem~\ref{th: threshold scalar equation}]
In view of Lemma~\ref{lem: th1-part1}, we have obtained \eqref{def of threshold scalar}.
It suffices to show that \eqref{asy tw threshold scalar} holds if and only if $s=s^*$.
First, we show that
\bea\label{th1:goal-1}
s=s^* \quad \Longrightarrow \quad \mbox{\eqref{asy tw threshold scalar} holds}.
\eea
Suppose that \eqref{asy tw threshold scalar} does not hold. Then $W_*$ satisfies \eqref{AS-U-infty-for-contradiction scalar}. In view of Proposition~\ref{Prop-supersol}, we can
choose $\delta_0>0$ sufficiently small such that
$$\overline W(\xi)= \min\{W_*(\xi)-R_w(\xi),1\}\geq(\not\equiv) 0$$
satisfies \eqref{tw super solution scalar}. Next,
we consider the following Cauchy problem with compactly supported initial datum $0\le w_0(x)\leq \overline W(x)$:
\begin{equation}\label{scalar cauchy problem}
\left\{
\begin{aligned}
&w_t=w_{xx}+f(w;s^*+\delta_0),\ t\ge 0,\ x\in\mathbb{R},\\
&w(0,x)=w_0(x),\ x\in\mathbb{R}.
\end{aligned}
\right.
\end{equation}
Then, in view of \eqref{def of threshold scalar}, we see that $c^*(s^*+\delta_0)>2$ (the minimal speed is nonlinearly selected). Therefore, we can apply Theorem 2 of \cite{Rothe1981}
to conclude that the spreading speed of the Cauchy problem \eqref{scalar cauchy problem} is strictly greater than $2$.

On the other hand, we
define $\overline{w}(t,x):=\overline W(x-2 t)$, and hence
$$\overline{w}(0,x)=\overline W(x)\geq w_0(x)\ \ \text{for all}\ \ x\in\mathbb{R}.$$
Since $\overline W$ satisfies \eqref{tw super solution scalar},
$\overline{w}$ forms a super-solution of \eqref{scalar cauchy problem}.
This immediately implies that the spreading speed of the solution, 
namely $w(t,x)$, of \eqref{scalar cauchy problem} is slower than or equal to $2$, due to the comparison principle.
This contradicts the spreading speed of the Cauchy problem \eqref{scalar cauchy problem}, which is strictly greater than $2$.
Thus, we obtain \eqref{th1:goal-1}.


Finally, we prove that
\bea\label{th1:goal-2}
\mbox{\eqref{asy tw threshold scalar} holds} \quad \Longrightarrow \quad s=s^*.
\eea
Note that for $s>s^*$, from \eqref{def of threshold scalar} we see that $c^{*}(s)>2$; so the asymptotic behavior of $W_s$ at $\xi\approx+\infty$ in Proposition \ref{prop: classification scalar} implies that
\eqref{asy tw threshold scalar} does not hold for any $s>s^*$. Therefore,
we only need to show that if $s<s^*$, then \eqref{asy tw threshold scalar} does not hold.
We assume by contradiction that there exists $s_0\in(0,s^*)$ such that
the corresponding minimal traveling wave satisfies
\bea\label{W-S0+infty}
W_{s_0}(\xi)=B_0 e^{-\xi}+o(e^{-\xi})\quad\text{as}\quad \xi\to+\infty
\eea
for some $B_0>0$. For $\xi\approx -\infty$, we have
\bea\label{W-S0-infty}
1-W_{s_0}(\xi)=C_0 e^{\hat\lambda\xi}+o(e^{\hat\lambda\xi})\quad\text{as}\quad \xi\to-\infty
\eea
for some $C_0>0$, where $\hat\lambda:=\sqrt{1-f'(1;s_0)}-1$.
Recall that the asymptotic behavior of $W_s^*$ at $\pm\infty$ satisfies
\begin{equation}\label{W-S*-pm-infty}
\begin{aligned}
W_{s^*}(\xi)=B e^{-\xi}+o(e^{-\xi})\ \text{as}\ \xi\to+\infty,\\
1-W_{s^*}(\xi)=C e^{\lambda_w\xi}+o(e^{\lambda_w\xi})\ \text{as}\ \xi\to-\infty,
\end{aligned}
\end{equation}
for some $B,C>0$, where $\lambda_w=\sqrt{1-f'(1;s^*)}-1$. In view of the assumption (A3), we have $\lambda_w> \hat\lambda$.
Combining \eqref{W-S0+infty}, \eqref{W-S0-infty}, and \eqref{W-S*-pm-infty}, there exists $L>0$ sufficiently large such that
$$W_{s^*}(\xi-L)> W_{s_0}(\xi)\ \ \text{for all}\ \ \xi\in\mathbb{R}.$$
Now, we define
\beaa
L^*:=\inf\{L\in\mathbb{R}\ |\ W_{s^*}(\xi-L)\ge W_{s_0}(\xi)\ \text{for all}\ \xi\in\mathbb{R}\}.
\eeaa
By the continuity, we have
$$W_{s^*}(\xi-L^*)\geq W_{s_0}(\xi)\ \ \text{for all}\ \ \xi\in\mathbb{R}.$$
 If there exists $\xi^*\in\mathbb{R}$ such that
$W_{s^*}(\xi^*-L^*)= W_{s_0}(\xi^*)$, by the strong maximum principle, we have $W_{s^*}(\xi-L^*)=W_{s_0}(\xi)$ for $\xi\in\mathbb{R}$,
which is impossible since $W_{s^*}(\cdot-L^*)$ and $W_{s_0}(\cdot)$ satisfy different equations. Consequently,
$$W_{s^*}(\xi-L^*)> W_{s_0}(\xi)\ \ \text{for all}\ \ \xi\in\mathbb{R}.$$
In particular, we have
\beaa
\lim_{\xi\to\infty}\frac{W_{s^*}(\xi-L^*)}{W_{s_0}(\xi)}\geq1.
\eeaa
Furthermore, we can claim that
\bea\label{limit=1}
\lim_{\xi\to\infty}\frac{W_{s^*}(\xi-L^*)}{W_{s_0}(\xi)}=1.
\eea
Otherwise, if the limit in \eqref{limit=1} is strictly bigger than 1, together with
\beaa
\lim_{\xi\to-\infty}\frac{1-W_{s^*}(\xi-L^*)}{1-W_{s_0}(\xi)}=0,
\eeaa
we can easily  find $\varepsilon>0$ sufficiently small such that
$$W_{s^*}(\xi-(L^*+\varepsilon))> W_{s_0}(\xi)\ \ \text{for}\ \ \xi\in\mathbb{R},$$
which contradicts the definition of $L^*$.
As a result, from \eqref{W-S0+infty}, \eqref{W-S*-pm-infty} and \eqref{limit=1}, we obtain $B_0=Be^{L^*}$.

On the other hand, we set $\widehat{W}(\xi)=W_{s^*}(\xi-L^*)-W_{s_0}(\xi)$. Then $\widehat{W}(\xi)$ satisfies
\bea\label{W-hat-eq2}
\widehat{W}''+2\widehat{W}'+\widehat{W}+J(\xi)=0, \quad \xi\in\mathbb{R},
\eea
where
$$J(\xi)=f(W_{s^*};s^*)- W_{s^*}-f(W_{s_0};s_0)+ W_{s_0}.$$
By the assumption  (A1) and Taylor's Theorem, there exist $\eta_1\in(0, W_{s^*})$ and $\eta_2\in(0,W_{s_0})$ such that
\beaa
J(\xi)&=&f''(\eta_1;s^*)W^2_{s^*}-f''(\eta_2;s_0)W^2_{s_0}\\
      &=&f''(\eta_1;s^*)(W^2_{s^*}-W^2_{s_0})+[f''(\eta_1;s^*)-f''(\eta_2;s_0)]W^2_{s_0}\\
      &=&f''(\eta_1;s^*)(W_{s^*}+W_{s_0})\widehat{W}+[f''(\eta_1;s^*)-f''(\eta_2;s_0)]W^2_{s_0}.
\eeaa
Define
$$J_1(\xi):=f''(\eta_1;s^*)(W_{s^*}+W_{s_0})\widehat{W},$$
$$J_2(\xi):=[f''(\eta_1;s^*)-f''(\eta_2;s_0)]W^2_{s_0}.$$
 It is easy to see that $J_1(\xi)=o(\widehat{W})$ for $\xi\approx+\infty$. Next, we will show $J_2(\xi)=o(\widehat{W})$ for $\xi\approx+\infty$.

Since $f''(0;s^*)>f''(0;s_0)$ (from the assumption (A3)), we can find small $\delta>0$ such that
$$\min_{\eta\in[0,\delta]}f''(\eta;s^*)>\max_{\eta\in[0,\delta]}f''(\eta;s_0)$$
and thus
there exist $\kappa_1,\kappa_2>0$ such that
\bea\label{J-lower bound}
\kappa_1e^{-2\xi}\ge J_2(\xi)=[f''(\eta_1;s^*)-f''(\eta_2;s_0)]W^2_{s_0}(\xi)\ge \kappa_2 e^{-2\xi}\quad \mbox{for all large $\xi$}.
\eea

We now claim that $J_2(\xi)=o(\widehat{W})$ as $\xi\to+\infty$.
For contradiction, we assume that it is not true. Then there exists $\{\xi_n\}$ with
$\xi_n\to+\infty$ as $n\to\infty$ such that for some $\kappa_3>0$,
\bea\label{kappa3}
\frac{J_2(\xi_n)}{\widehat{W}(\xi_n)}\geq \kappa_3\quad \mbox{for all $n\in\mathbb{N}$.}
\eea
Set $\widehat{W}(\xi)=\alpha(\xi)e^{-2\xi}$, where $\alpha(\xi)>0$ for all $\xi$.
By substituting it into \eqref{W-hat-eq2},
we have
\bea\label{alpha-eq}
L(\xi):=(\alpha''(\xi)-2\alpha'(\xi)+\alpha(\xi))e^{-2\xi}+J_1(\xi)+J_2(\xi)=0\quad \mbox{for all large $\xi$}.
\eea
By \eqref{J-lower bound} and \eqref{kappa3}, we have
\bea\label{alpha-bdd}
0<\alpha(\xi_n)\leq \frac{\kappa_1}{\kappa_3}\quad  \mbox{for all $n\in\mathbb{N}$.}
\eea
Now, we will reach a contradiction by dividing the behavior of $\alpha(\cdot)$ into two cases:
\begin{itemize}
    \item[(i)] $\alpha(\xi)$ oscillates for all large $\xi$;
    \item[(ii)] $\alpha(\xi)$ is monotone for all large $\xi$.
\end{itemize}
For case (i), there exist local minimum points $\eta_n$ of $\alpha$ with $\eta_n\to\infty$ as $n\to\infty$ such that
\beaa
\alpha(\eta_n)>0,\quad \alpha'(\eta_n)=0,\quad \alpha''(\eta_n)\geq 0\quad  \mbox{for all $n\in\mathbb{N}$.}
\eeaa
Together with \eqref{J-lower bound} and $J_1(\xi)=o(\hat{\mathcal{W}}(\xi))$,
from \eqref{alpha-eq} we see that
\beaa
0=L(\eta_n)\geq \alpha(\eta_n)e^{-2\eta_n}+o(1)\alpha(\eta_n)e^{-2\eta_n}+\kappa_2e^{-2\eta_n}>0
\eeaa
for all large $n$, which reaches a contradiction.

For case (ii),
due to \eqref{alpha-bdd}, there exists $\alpha_0\in[0, \kappa_1/\kappa_3]$
such that $\alpha(\xi)\to \alpha_0$ as $\xi\to\infty$. Hence, we can find subsequence $\{\eta_j\}$ that tends to $\infty$ such that $\alpha'(\eta_j)\to0$, $\alpha''(\eta_j)\to0$ and
$\alpha(\eta_j)\to \alpha_0$ as $n\to\infty$.
From \eqref{alpha-eq} we deduce that
\beaa
0=L(\eta_j)\geq (o(1)+\alpha(\eta_j)+ \kappa_2)e^{-2\eta_j}>0
\eeaa
for all large $j$, which reaches a contradiction.
Therefore, we have proved that
$J_2(\xi)=o(\widehat{W})$ as $\xi\to\infty$.
Consequently, we have
\bea\label{J-small o}
J(\xi)=J_1(\xi)+J_2(\xi)=o(\widehat{W}(\xi))\quad \mbox{as $\xi\to\infty$.}
\eea

Thanks to \eqref{J-small o}, we can apply
\cite[Chapter 3, Theorem 8.1]{CoddingtonLevison} to assert that
the asymptotic behavior of $\widehat{W}(\xi)$ at $\xi=+\infty$ satisfies
\beaa
\widehat{W}(\xi)=(C_1\xi+C_2)e^{-\xi}+o(e^{-\xi})\quad \mbox{as $\xi\to\infty$},
\eeaa
where  $C_1\geq0$, and $C_2>0$ if $C_1=0$. From \eqref{W-S0+infty} and \eqref{W-S*-pm-infty}, we see that $C_1=0$, and $C_2>0$. On the other hand, $B_0=Be^{ L^*}$ implies that $C_2=0$, which reaches a contradiction.
Therefore, \eqref{th1:goal-2} holds, and the proof is complete.
\end{proof}


\section{Threshold of the nonlocal diffusion equation}\label{sec: threshold scalar nonlocal}

In this section, we aim to prove Theorem~\ref{th: threshold scalar equation nonlocal}.
The main idea follows the approach used in the proof of Theorem \ref{th: threshold scalar equation}, but the analysis here is more involved for two main reasons. First, the linearly selected speed $c^*_0$ cannot be computed explicitly, as it is characterized by a variational formula. Second, due to the nature of nonlocal diffusion, it is no longer possible to construct the super-solution pointwisely. In particular, when constructing the super-solution within a given interval, one must also account for its behavior outside that interval. Moreover, since the kernel $J$ has compact support, we may assume without loss of generality that $J \ge 0$ on $[-L, L]$ and $J = 0$ for $x \in (-\infty, -L] \cup [L, \infty)$. In fact, we believe that this approach can be extended to kernels with exponential decay.

\subsection{Preliminary}
We first introduce some propositions concerned with the asymptotic behavior of the minimal traveling wave of \eqref{scalar nonlocal tw-parameter s} as $\xi\to+\infty$ and $\xi\to-\infty$.
To obtain the asymptotic behavior at $\xi\to+\infty$, we will use specific linearized results established in \cite{Chen Fu Guo, Zhang_etal2012}.
\begin{proposition}[Proposition 3.7 in \cite{Zhang_etal2012}]\label{prop:Zhang-etal}
Assume that $c>0$ and $B(\cdot)$ is a continuous function having finite limits at infinity $B(\pm \infty):=\lim_{\xi\to\pm\infty}B(\xi)$.
Let $z(\cdot)$ be a measurable function satisfying
\beaa
c z(\xi)=\int_{\mathbb{R}}J(y)e^{\int_{\xi-y}^{\xi}z(s)ds}dy+B(\xi),\quad \xi\in\mathbb{R}.
\eeaa
Then $z$ is uniformly continuous and bounded. Furthermore, $\omega^{\pm} = \lim_{\xi\to\pm\infty} z(\xi)$ exist and are real roots of the
characteristic equation
\beaa
c\omega=\int_{\mathbb{R}}J(y)e^{\omega y}dy+B(\pm\infty).
\eeaa
\end{proposition}

\begin{proposition}\label{prop:correction-U-linear-decay}
Assume that $c=c_{NL}^*(q)=c^*_0$. Let $\lambda_0$ be defined as that in Remark~\ref{rm:lambda_0}.  Then the minimal traveling wave $\mathcal{W}_q(\xi)$ satisfies
\bea\label{decay-U-linear}
\mathcal{W}_q(\xi)=A\xi e^{-\lambda_0\xi}+Be^{-\lambda_0\xi}+o(e^{-\lambda_0\xi})\quad \mbox{as $\xi\to+\infty$},
\eea
where $A\geq0$ and $B\in \mathbb{R}$,
and $B>0$ if $A=0$.
\end{proposition}
\begin{proof}
For convenience, we write $\mathcal{W}$ instead of $\mathcal{W}_q(\xi)$.
Let $z(\xi):=-\mathcal{W}'(\xi)/\mathcal{W}(\xi)$. Then, from \eqref{scalar nonlocal tw-parameter s} we have
\beaa
cz(\xi)=\int_{\mathbb{R}}J(y)e^{\int_{\xi-y}^{\xi}z(s)ds}dy+B(\xi),
\eeaa
where $B(\xi)=f(\mathcal{W})/\mathcal{W}-1$.
Since $\mathcal{W}(+\infty)=0$, we have $B(+\infty)=f'(0)-1$.
It follows from Proposition~\ref{prop:Zhang-etal} and Remark~\ref{rm:lambda_0} that
\bea\label{limit-z}
\lim_{\xi\to+\infty}\frac{\mathcal{W}'(\xi)}{\mathcal{W}(\xi)}=-\lim_{\xi\to+\infty} z(\xi)=-\lambda_0.
\eea

With \eqref{limit-z}, we can correct the proof of \cite[Theorem 1.6]{Coville} and obtain the desired result. To see this, we set
\bea\label{mathcal F}
\mathcal{F}(\lambda)=\int_0^{\infty}\mathcal{W}(\xi) e^{-\lambda \xi}d\xi.
\eea
Because of \eqref{limit-z},
$\mathcal{F}$ is well-defined for $\lambda\in\mathbb{C}$ with $-\lambda_0<{\rm Re}\lambda<0$.
From \eqref{scalar nonlocal tw-parameter s}, we can rewrite it as
\beaa
(c\lambda+h(\lambda))\int_{\mathbb{R}}\mathcal{W}(\xi)e^{-\lambda \xi}d\xi=
\int_{\mathbb{R}} e^{-\lambda \xi}[f'(0)\mathcal{W}(\xi)-f(\mathcal{W}(\xi))]d\xi=:Q(\lambda),
\eeaa
where $h(\lambda)=h(-\lambda)$ is defined in Remark~\ref{rm:lambda_0}.
Moreover, we see that $Q(\lambda)$ is well-defined for $\lambda\in\mathbb{C}$ with
$-2\lambda_0<{\rm Re} \lambda<0$
since
$$f(w)=f'(0)w+O(w^2)\ \ \text{as}\ \ w\to0.$$
Then, we have
\bea\label{F-form-nonlocal}
\mathcal{F}(\lambda)=\frac{Q(\lambda)}{c\lambda+h(\lambda)}-\int_{-\infty}^{0}\mathcal{W}(\xi)e^{-\lambda \xi}d\xi,
\eea
as long as $\mathcal{F}(\lambda)$ is well-defined.

To apply Ikehara's Theorem (Proposition~\ref{prop:ikehara}), we
rewrite \eqref{F-form-nonlocal} as
\beaa
\mathcal{F}(\lambda)=\frac{H(\lambda)}{(\lambda+\lambda_0)^{p+1}},
\eeaa
where $p\in\mathbb{N}\cup\{0\}$ and
\bea\label{Ikehara-form-nonlocal}
H(\lambda)=\frac{Q(\lambda)}{(c\lambda+h(\lambda))/(\lambda+\lambda_0)^{p+1}}
-(\lambda+\lambda_0)^{p+1}\int_{-\infty}^{0}e^{-\lambda \xi}\mathcal{W}(\xi) d\xi.
\eea
It is well known from (cf. \cite[p.2437]{Carr Chmaj}) that all roots of $c\lambda+h(\lambda)=0$ must be real.  Together with the assumption $c_{NL}^*=c^*_0$ and Remark~\ref{rm:lambda_0},
we see that $\lambda = -\lambda_0$ is the only (double) root of $c\lambda+h(\lambda)=0$.

Next, we will show $H$ is analytic in the strip $\{-\lambda_0\leq  {\rm Re}\lambda<0\}$ and $H(-\lambda_0)\neq 0$ with some $p\in\mathbb{N}\cup\{0\}$.
Note that the second term on the right-hand side of \eqref{Ikehara-form-nonlocal} is analytic on $\{{\rm Re}\lambda<0\}$.
Consequently, it is enough to deal with the first term.
\begin{itemize}
    \item[(i)] Assume that
$Q(-\lambda_0)\neq 0$. Then
 by setting $p=1$, we obtain $H(-\lambda_0)\neq 0$ (since $c\lambda+h(\lambda)=0$ has the double root $\lambda_0$), and thus
\beaa
\lim_{\xi\to+\infty}\frac{\mathcal{W}(\xi)}{\xi e^{-\lambda_0\xi}}=C_1
\eeaa
for some $C_1>0$ by Ikehara's Theorem (Proposition~\ref{prop:ikehara}).
\item[(ii)] Assume that
$Q(-\lambda_0)=0$. This means that $\lambda=-\lambda_0$ is a root of $Q(\lambda)$.
One can observe from \eqref{F-form-nonlocal} that the root $\lambda=-\lambda_0$ of $Q$ must be simple; otherwise, $\mathcal{F}(\lambda)$ has a removable singularity at $\lambda=-\lambda_0$ and thus can
be extended to exist over $\{ -\lambda_0-\epsilon\leq  {\rm Re} \lambda<0\}$
for some $\epsilon>0$. However, by \eqref{limit-z} and \eqref{mathcal F}, we see that $\mathcal{F}(\lambda)$ is divergent for
$\lambda$ with ${\rm Re} \lambda<-\lambda_0$,
which leads to a contradiction. Therefore, $\lambda=-\lambda_0$ is a simple root of $Q$.
By taking $p=0$ in \eqref{Ikehara-form-nonlocal}, we obtain $H(\lambda_0)\neq 0$, and thus
\beaa
\lim_{\xi\to+\infty}\frac{\mathcal{W}(\xi)}{e^{-\lambda_0\xi}}=C_2
\eeaa
for some $C_2>0$ by Ikehara's Theorem (Proposition~\ref{prop:ikehara}).
\end{itemize}
As a result, we obtain \eqref{decay-U-linear} in which $A$ and $B$ cannot be equal to $0$ at the same time.
\end{proof}

The third proposition provides the asymptotic behavior of the minimal traveling wave as $\xi\to-\infty$,
\begin{proposition}\label{prop: asy tw - infty nonlocal}
Let $\mathcal{W}_{q,c}$ be the traveling wave satisfying \eqref{scalar nonlocal tw-parameter s} with speed $c\ge c^*_0$ and $q\ge 0$.
We define $\mu_{q,c}$ as the unique positive root of
\begin{equation}\label{nonlocal root}
-c\mu=I_1(\mu):=\int_{\mathbb{R}}J(y)e^{-\mu y}dy+f'(1;q)-1.
\end{equation}
Then it holds
\begin{equation*}
1-\mathcal{W}_{q,c}(\xi)=O(e^{\mu_{q,c}\xi})\quad\text{as}\quad \xi\to-\infty.
\end{equation*}
\end{proposition}
By linearizing the equation of \eqref{scalar nonlocal tw-parameter s} near $\mathcal{W}=1$ and changing $1-\mathcal{W}=\hat{\mathcal{W}}$, we have
$$J\ast\hat{\mathcal{W}}-\hat{\mathcal{W}}+c\hat{\mathcal{W}}'+f'(1;q)\hat{\mathcal{W}}=0.$$
Define
$I_2(\mu)=\int_{\mathbb{R}}\hat{\mathcal{W}}e^{-\mu\xi}d\xi.$
Then, by multiplying $e^{-\mu \xi}$ and integral on $\mathbb{R}$, we obtain
$$I_2(\mu)\Big(1-f'(1;q)-\mu c-\int_{\mathbb{R}}J(y)e^{-\mu y}dy\Big)=0.$$
Notice that, $I_1(\mu)$ is a symmetric and convex function.
Since $\int_{\mathbb{R}}J(y)e^{-\mu y}dy=1$ when $\mu=0$, $\int_{\mathbb{R}}J(y)e^{-\mu y}dy\to \infty$ as $\mu\to\infty$, and $f'(1;q)<0$, \eqref{nonlocal root} admits the unique positive root. Then, the proof of Proposition \ref{prop: asy tw - infty nonlocal} follows from the similar argument as Theorem 1.6 in \cite{Coville}.

\subsection{Construction of the super-solution}\label{subsec-3-1}

\noindent

Under the assumption (A1) and \eqref{assumption on J}, from Theorem 1.6 in \cite{Coville}, for each $q\geq0$, there exists a unique minimal traveling wave(up to a translation), and the minimal speed $c_{NL}^*(q)$ is continuous for all $q\ge 0$ by the assumption (A2).
Moreover, it follows from the assumption (A3) that $c_{NL}^*(q)$ is nondecreasing on $q$. Thus, we immediately obtain the following result by assumptions (A6),(A7), and Remark~\ref{rk:a6+a7}.

\begin{lemma}\label{lem: th1-part1 nonlocal}
Assume that assumptions (A1)-(A3), (A6), and (A7) hold. Then there exists a threshold $q^*\in[q_1,q_2)$ such that \eqref{def of threshold scalar nonlocal} holds.
\end{lemma}

Thanks to Lemma~\ref{lem: th1-part1 nonlocal}, to prove Theorem~\ref{th: threshold scalar equation nonlocal}, it suffices to show that
\eqref{asy tw threshold scalar nonlocal} holds if and only if $q=q^*$.
Let $\mathcal{W}_{q^*}$ be the minimal traveling wave of \eqref{scalar nonlocal tw-parameter s} with $q=q^*$
and 
speed $c_{NL}^*(q^*)=c_0^*$ defined as \eqref{formula of c_NL}. For simplicity, we denote $\mathcal{W}_*:=\mathcal{W}_{q^*}$.
Similar as the proof of Theorem \ref{th: threshold scalar equation}, the first and the most involved step is to show that if $q=q^*$, then \eqref{asy tw threshold scalar nonlocal} holds.
To do this, we shall use the contradiction argument again.
Assume that \eqref{asy tw threshold scalar nonlocal} is not true. Then, from \eqref{decay-U-linear}  
it holds that
\bea\label{AS-U-infty-for-contradiction scalar nonlocal}
\lim_{\xi\rightarrow+\infty}\frac{\mathcal{W}_{*}(\xi)}{\xi e^{-\lambda_0\xi}}=A_0\quad \mbox{for some $A_0>0$,}
\eea
where $\lambda_0$ is defined in Remark \ref{rk:a6+a7}.

Under the condition \eqref{AS-U-infty-for-contradiction scalar nonlocal},
we shall prove the following proposition.

\begin{proposition}\label{Prop-supersol nonlocal}
Assume that assumptions (A1)-(A3), (A6), and (A7) hold. In addition, if \eqref{AS-U-infty-for-contradiction scalar nonlocal} holds,
then there exists an auxiliary continuous function $\mathcal{R}_w(\xi)$ defined in $\mathbb{R}$ satisfying
\begin{equation}\label{decay rate of Rw nonlocal}
\mathcal{R}_w(\xi)=O(\xi e^{-\lambda_0\xi}) \quad \mbox{as $\xi\to\infty$},
\end{equation}
such that  $\overline{\mathcal{W}}(\xi):=\min\{\mathcal{W}_{*}(\xi)-\mathcal{R}_w(\xi),1\}\geq (\not\equiv) 0$
satisfies
\bea\label{tw super solution scalar nonlocal}
\mathcal{N}_0[\overline{\mathcal{W}}]:=J\ast\overline{\mathcal{W}}-\overline{\mathcal{W}}+c^*_0\overline{\mathcal{W}}'+f(\overline{\mathcal{W}};q^*+\delta_0)\le0,\quad  \mbox{a.e. in $\mathbb{R}$},
\eea
for all sufficiently small $\delta_0>0$. 

\begin{remark}\label{rm: glue}
Unlike the definition of the super-solution for the reaction-diffusion equation given in Proposition \ref{Prop-supersol}, in the present setting, the discontinuity in the derivative at $\xi_0$, i.e., ${\overline{\mathcal{W}}}'(\xi_0^+)\neq {\overline{\mathcal{W}}}'(\xi_0^-)$, does not pose a problem. It suffices that ${\overline{\mathcal{W}}}(\xi)\in W^{1,1}(\mathbb{R})$ (see \S 2.2.1 in \cite{nonlocal book} ). This means that the angle at the junction is irrelevant. However, during the construction, we find that an appropriate choice of angle can significantly simplifies the computation near the gluing points such as $\xi_1-\delta_1$ and $\xi_2$ in \eqref{definition of Rw nonlocal}. Assume we have $\mathcal{N}_0[\overline{\mathcal{W}}_1]=\mathcal{N}_0[\mathcal{W}_*-\mathcal R_1]\le 0$ for $\xi\in [\xi_1-\delta_1,+\infty)$ and $\mathcal{N}_0[\overline{\mathcal{W}}_2]=\mathcal{N}_0[\mathcal{W}_*-\mathcal R_2]\le 0$ for $\xi\in [\xi_2,\xi_1-\delta_1]$. If further $\mathcal R_1\le \mathcal R_2$ for $\xi\in[\xi_1-\delta_1-L,\xi_1-\delta_1]$ and $\mathcal R_2\le \mathcal R_1$ for $\xi\in[\xi_1-\delta_1,\xi_1-\delta_1+L]$, then by $J\ge 0$ we obtain
\beaa
&&\mathcal{N}_0[\overline{\mathcal{W}}](\xi_1-\delta_1)\\
&=&\int_0^L J(y)\overline{\mathcal{W}}_2(\xi_1-\delta_1-y)dy+\int_{-L}^0 J(y)\overline{\mathcal{W}}_1(\xi_1-\delta_1-y)dy-\overline{\mathcal{W}}_1+c^*_0\overline{\mathcal{W}}_1'+f(\overline{\mathcal{W}}_1;q^*+\delta_0)\\
&\le & \int_{-L}^L J(y)\overline{\mathcal{W}}_1(\xi_1-\delta_1-y)dy-\overline{\mathcal{W}}_1+c^*_0\overline{\mathcal{W}}_1'+f(\overline{\mathcal{W}}_1;q^*+\delta_0)\le 0,
\eeaa
which implies $\overline{\mathcal{W}}_1$ and $\overline{\mathcal{W}}_2$ can be smoothly glued at $\xi_1-\delta_1$. 
\end{remark}

\end{proposition}

\begin{figure}
\begin{center}
\begin{tikzpicture}[scale = 1.1]
\draw[thick](-6,0) -- (7,0) node[right] {$\xi$};
\draw [semithick] (-6, -0.1) to [ out=0, in=150] (-4,-0.7) to [out=20, in=180] (0,-0.3)   to [out=70,in=190] (1.5,2) to [out=0,in=170] (7,0.3);
\node[below] at (0,1.5) {$\xi_{1}+\delta_1$};
\draw[dashed] [thick] (-4,-0.7)-- (-4,1);
\node[above] at (-4,0.9) {$\xi_2$};
\draw[dashed] [thick] (0,-0.3)-- (0,1.3);
\end{tikzpicture}
\caption{the construction of $\mathcal{R}_w(\xi)$.}\label{Figure scalar nl}
\end{center}
\end{figure}

In the following, assumptions (A1)-(A3), (A6), and (A7)  are always assumed.
We shall construct the auxiliary function $\mathcal{R}_w(\xi)$,  which differs in structure from the function $R_w$ in \S \ref{subsec-2-1}, as follows (see Figure \ref{Figure scalar nl}) :
\begin{equation}\label{definition of Rw nonlocal}
\mathcal{R}_w(\xi)=\begin{cases}
\mathcal{R}_1(\xi):=\varepsilon_1\sigma(\xi) e^{-\lambda_0\xi},&\ \ \mbox{for}\ \ \xi\ge\xi_{1}-\delta_1,\\
\mathcal{R}_2(\xi):=-\varepsilon_2 \Psi(\xi-\xi_1+\delta_1+\frac{L^*}{2}),&\ \ \mbox{for}\ \ \xi_{2}\le\xi\le\xi_{1}-\delta_1,\\
\mathcal{R}_3(\xi):=-\varepsilon_3e^{\lambda_1\xi},&\ \ \mbox{for}\ \ \xi\le\xi_2.
\end{cases}
\end{equation}
Here $\Psi(\xi)> 0$ is the eigenfunction corresponding to the principal eigenvalue $\nu_0>0$ of the following linear operator on the bounded interval $[-L^*, L^*]$:
\begin{equation}\label{eigenvalue}
-J\ast \Psi+\Psi-c^*_0\Psi'-f'(\mathcal{W}_*)\Psi=\nu_0 \Psi\ \ \text{for}\ \ \xi\in[-L^*,L^*].
\end{equation}
Since $\nu_0\to 0$ and $\Psi(\xi)\to -\mathcal W'_*(\xi)$ uniformly as $L^*\to\infty$, we choose sufficiently large $L^*$ such that 
\begin{equation}\label{est Psi'}
\Psi(\xi)\sim K_0\xi e^{-\lambda_0\xi}\ \ \text{and}\ \ \Psi'(\xi)\sim -\lambda_0K_0\xi e^{-\lambda_0\xi}\ \ \text{for}\ \ \xi\in[\frac{L^*}{4}-L,\frac{L^*}{2}+L],
\end{equation}
where $[-L,L]$ is the support of $J$.
We fix $\xi_1-\delta_1-\xi_2=L^*/4$. $\delta_1>0$ and $\sigma(\xi)>0$ will be
determined such that $\overline{\mathcal{W}}(\xi)$ satisfies \eqref{tw super solution scalar nonlocal}. Moreover, we should choose
$\varepsilon_{j=1,2,3}\ll A_0$ ($A_0$ is defined in \eqref{AS-U-infty-for-contradiction scalar nonlocal}) such that $\mathcal{R}_w(\xi)\ll {\mathcal{W}}_*(\xi)$ and $\overline{\mathcal{W}}(\xi)$ is continuous for all $\xi\in\mathbb{R}$.

Since $f(\cdot;q^*)\in C^2$, there exist $K_1>0$ and $K_2>0$ such that
\bea\label{K1K2 nonlocal}
|f''({\mathcal{W}}_*(\xi);q^*)|< K_1,\quad |f'({\mathcal{W}}_*(\xi);q^*)|< K_2\quad \mbox{for all}\quad\xi\in\mathbb{R}.
\eea
Furthermore, there exists $K_3<-f'(1; q^*)$ such that
\begin{equation}\label{condition on xi_2 scalar nonlocal}
f'({\mathcal{W}}_*(\xi); q^*)< -K_3<0\quad\text{for all}\quad\xi\le \xi_2.
\end{equation}
Then, by setting $\lambda_1\in(0,\mu_0)$, where $\mu_0=\mu_{q^*,c^*_0}$ is the unique positive root obtained from Proposition \ref{prop: asy tw - infty nonlocal} with $q=q^*$ and $c=c^*_0$,
 sufficiently small, we have
\bea\label{lambda 2 nonlocal}
1+K_3-e^{\lambda_1L}-c_0^*\lambda_1>0.
\eea

We now divide the proof into several steps. 

\noindent{\bf{Step 1}:} We consider $\xi\in[\xi_1-\delta_1,+\infty)$ where $\delta_1>0$ is determined in the end of this step.
In this case, we have
\beaa
\mathcal{R}_w(\xi)=\mathcal{R}_1(\xi)=\varepsilon_1\sigma(\xi)\ e^{-\lambda_0\xi}
\eeaa
for some small $\varepsilon_1\ll A_0$. 


Note that $\mathcal{W}_*$ satisfies \eqref{scalar nonlocal tw-parameter s} with $c=c_0^*$.
By some  straightforward computations, we have
\begin{equation}\label{N0 inequality step 1 nonlocal}
\begin{aligned}
\mathcal{N}_0[\overline{\mathcal{W}}]=&-J\ast \mathcal{R}_w+\mathcal{R}_w-c_0^*\mathcal{R}'_w-f(\mathcal{W}_*;q^*)+f(\mathcal{W}_*-\mathcal{R}_w;q^*+\delta_0)\\
=&-J\ast \mathcal{R}_w+\mathcal{R}_w-c_0^*\mathcal{R}'_w-f(\mathcal{W}_*;q^*)+f(\mathcal{W}_*-\mathcal{R}_w;q^*)\\
&-f(\mathcal{W}_*-\mathcal{R}_w;q^*)+f(\mathcal{W}_*-\mathcal{R}_w;q^*+\delta_0).
\end{aligned}
\end{equation}
By assumptions (A1) and (A2), and the statement (1) of Lemma \ref{lm: divide R to 3 parts}, since $\mathcal{R}_w\ll \mathcal{W}_*\ll 1$ for $\xi\in[\xi_1+\delta_1,\infty)$, we have
\begin{equation}\label{estimate 1 nonlocal}
-f(\mathcal{W}_*;q^*)+f(\mathcal{W}_*-\mathcal{R}_w;q^*)= -f'(0;q^*)\mathcal{R}_w+f''(0;q^*)(\frac{\mathcal{R}^2_w}{2}-\mathcal{W}_*\mathcal{R}_w)+o((\mathcal{W}_*)^2),
\end{equation}
\begin{equation}\label{estimate 2 nonlocal}
-f(\mathcal{W}_*-\mathcal{R}_w;q^*)+f(\mathcal{W}_*-\mathcal{R}_w;q^*+\delta_0)\le C_1\delta_0(\mathcal{W}_*-\mathcal{R}_w)^2+o((\mathcal{W}_*)^2).
\end{equation}

For $\xi\in [\xi_1+\delta_1+L,\infty)$, from \eqref{formula of c_NL}, \eqref{K1K2 nonlocal}, \eqref{N0 inequality step 1 nonlocal}, \eqref{estimate 1 nonlocal}, \eqref{estimate 2 nonlocal}, we have
\begin{equation}\label{estimate 4 nonlocal}
\begin{aligned}
\mathcal{N}_0[{\mathcal{W}}_*-\mathcal R_1]\le &-\varepsilon_1e^{-\lambda_0\xi}\Big(\int_{\mathbb{R}}J(y)[\sigma(\xi-y)-\sigma(\xi)] e^{\lambda_0y}dy\Big)-c^*_0\sigma'e^{-\lambda_0\xi}\\
&+K_1(\frac{\mathcal{R}_w^2}{2}+\mathcal{W}_*\mathcal{R}_w)+C_1\delta_0\mathcal{W}_*^2+o((\mathcal{W}_*)^2).
\end{aligned}
\end{equation}
Let $h(\lambda)$ be defined as that in Remark \ref{rm:lambda_0}. Since $(h(\lambda)/\lambda)'=0$ when $\lambda=\lambda_0$, from \eqref{formula of c_NL}, we get
\begin{equation}\label{c*0 formula}
c^*_0=\int_{\mathbb{R}}yJ(y)e^{\lambda_0y}dy.
\end{equation}
Then, it follows from \eqref{estimate 4 nonlocal} and \eqref{c*0 formula} that
\begin{equation}\label{estimate 3 nonlocal}
\begin{aligned}
\mathcal{N}_0[{\mathcal{W}}_*-\mathcal R_1]\le &-\varepsilon_1e^{-\lambda_0\xi}\int_{\mathbb{R}}J(y)[\sigma(\xi-y)-\sigma(\xi)+y\sigma'(\xi)] e^{\lambda_0y}dy\\
&+K_1(\frac{\mathcal{R}_w^2}{2}+\mathcal{W}_*\mathcal{R}_w)+C_1\delta_0\mathcal{W}_*^2+o((\mathcal{W}_*)^2).
\end{aligned}
\end{equation}

Now, we define
$$\sigma(\xi):=\frac{1}{\lambda_0^2}e^{-\frac{\lambda_0}{2l}(\xi-\xi_1)}-\frac{1}{\lambda_0^2}+\frac{\xi-\xi_1}{\lambda_0l}$$
which satisfies
$$\sigma(\xi_1)=0,\ \sigma'(\xi)=\frac{1}{\lambda_0l}-\frac{1}{2\lambda_0l}e^{-\frac{\lambda_0}{2l}(\xi-\xi_1)}.$$
 Moreover, $\sigma(\xi)= O(\xi)$ as $\xi\to\infty$ implies that $\mathcal{R}_w$ satisfies \eqref{decay rate of Rw nonlocal}.

By some straightforward computation, we have
\begin{equation*}
\int_{\mathbb{R}}J(y)[\sigma(\xi-y)-\sigma(\xi)+y\sigma'(\xi)] e^{\lambda_0y}dy=\frac{1}{\lambda_0^2}e^{-\frac{\lambda_0}{2l}(\xi-\xi_1)}\int_{\mathbb{R}}J(y)e^{\lambda_0 y}[e^{\frac{\lambda_0y}{2l}}-1-\frac{\lambda_0y}{2l}]dy.
\end{equation*}
Notice that, the function
$$g(y):=e^{\frac{\lambda_0y}{2l}}-1-\frac{\lambda_0y}{2l}\ge 0$$
 is convex and obtains minimum at $y=0$, and $J(y)=0$ for $|y|>L$. Therefore, we assert that there exists $K_4>0$ independent on $\xi_1$ such that
\begin{equation}\label{estimate K4 nonlocal}
-\varepsilon_1e^{-\lambda_0\xi}\int_{\mathbb{R}}J(y)[\sigma(\xi-y)-\sigma(\xi)+y\sigma'(\xi)] e^{\lambda_0y}dy\le -\varepsilon_1K_4e^{-\lambda_0\xi}e^{\frac{-\lambda_0(\xi-\xi_1)}{2l}}.
\end{equation}

Then, from \eqref{estimate 3 nonlocal} and \eqref{estimate K4 nonlocal}, up to enlarging $\xi_1$ if necessary, we always have 
\beaa
\mathcal{N}_0[{\mathcal{W}}_*-\mathcal R_1]\le -\varepsilon_1K_4e^{-\frac{\lambda_0}{2l}(\xi-\xi_1)}e^{-\lambda_0\xi}+K_1(\frac{\mathcal{R}_w^2}{2}+\mathcal{W}_*\mathcal{R}_w)+C_1\delta_0\mathcal{W}_*^2+o((\mathcal{W}_*)^2)\le 0
\eeaa
for all sufficiently small $\delta_0\ll\varepsilon_1$
since $\mathcal{R}_w^2(\xi)$, $\mathcal{W}_*\mathcal{R}_w(\xi)$, and  $\mathcal{W}_*^2(\xi)$ are $o(e^{-\frac{(2l+1)\lambda_0}{2l}\xi})$ for $\xi\ge \xi_1-\delta_1$ from \eqref{AS-U-infty-for-contradiction scalar nonlocal} and the definition of $\mathcal{R}_w$.

The rest of Step 1 devotes to the verification $\mathcal{N}_0[\overline{\mathcal{W}}]\le 0$ for $\xi\in [\xi_1-\delta_1,\xi_1-\delta_1+L]$, where $\mathcal{R}_2$ defined on $[\xi_2,\xi_1-\delta_1]$ is also involved in the computation. From Remark \ref{rm: glue}, it suffices to find a $\delta_1$ such that
$\mathcal R_1\le \mathcal R_2$ for $\xi\in[\xi_1-\delta_1-L,\xi_1-\delta_1]$ and $\mathcal R_2\le \mathcal R_1$ for $\xi\in[\xi_1-\delta_1,\xi_1-\delta_1+L]$. 

From now on, we fix $\xi_1-\delta_1$. To make sure that $\overline{\mathcal{W}}$ is continuous at $\xi_1-\delta_1$ where $\xi_1$ is decided by the above discussion,  we set
\begin{equation}\label{epsilon 2 nl}
\varepsilon_2=\varepsilon_2(\varepsilon_1,\delta_1,l)=-\frac{\varepsilon_1}{\Psi(L^*/2)}\Big(\frac{1}{\lambda_0^2}(e^{\frac{\lambda_0\delta_1}{2l}}-1)-\frac{\delta_1}{\lambda_0l}\Big)e^{-\lambda_0(\xi_1-\delta_1)},
\end{equation}
where $\Psi(L^*/2)=K_0\frac{L^*}{2}e^{-\lambda_0L^*/2}$.  From \eqref{est Psi'}, we assert the following:

\begin{claim}\label{cl: delta1}
There exists a small $\delta_1>0$ such that $\mathcal R_2\le \mathcal R_1$ for $\xi\in[\xi_1-\delta_1,\xi_1-\delta_1+L]$.
\end{claim}
\begin{proof}
Note that $\mathcal R_2\le 0\le \mathcal R_1$ for $\xi\in[\xi_1,\xi_1-\delta_1+L]$, so it suffices to show 
$\mathcal R_2\le \mathcal R_1$ for $\xi\in[\xi_1-\delta_1,\xi_1]$. From \eqref{est Psi'} and \eqref{epsilon 2 nl}, we know that $\mathcal R_2\le \mathcal R_1$ is equivalent to
 $\mathcal R'_2\le \mathcal R'_1$, which leads to the inequality
$$-(e^{\frac{\lambda_0\delta_1}{2l}}-1-\frac{\lambda_0\delta_1}{l})(\frac{2\tilde L}{L^*}+1)\le \frac{1}{l}-\frac{1}{2l}e^{-\frac{\lambda_0}{2l}(\tilde L-\delta_1)}-e^{-\frac{\lambda_0}{2l}(\tilde L-\delta_1)}+1-\frac{\lambda_0(\tilde L-\delta_1)}{l},$$
where $\xi=\xi/1-\delta_1+\tilde L$, $\tilde L\in (0,\delta_1]$.
By setting $\delta_1$ sufficiently small, the asymptotic on the left-hand side is 
$$\frac{\lambda_0\delta_1}{2l}(\frac{2\tilde L}{L^*}+1)\to 0\ \ \text{as} \ \ \delta_1\to 0.$$
The asymptotic on the right-hand side is
$$\frac{1}{2l}\Big(1-\frac{\lambda_0}{2l}(\delta_1-\tilde L)+\lambda_0(\delta_1-\tilde L)\Big)\to \frac{1}{2l}\ \ \text{as} \ \ \delta_1\to 0.$$
Therefore, we can choose a small $\delta_1$ such that $\mathcal R'_2\le \mathcal R'_1$
remains valid on $[\xi_1-\delta_1,\xi_1]$. Thus, $\mathcal R_2\le-\varepsilon_2 \underline\Psi\le \mathcal R_1$ for $\xi\in[\xi_1-\delta_1,\xi_1]$. 
\end{proof}

\begin{claim}\label{cl: l}
There exists a small $\delta_1>0$ and a large $l>0$ such that $\mathcal R_2\ge \mathcal R_1$ for $\xi\in[\xi_1-\delta_1-L,\xi_1-\delta_1]$.
\end{claim}
\begin{proof}
From \eqref{epsilon 2 nl}, we know that, for $\xi\in [\xi_1-\delta_1-L,\xi_1-\delta_1]$, $\mathcal R_2\ge \mathcal R_1$ is equivalent to $\mathcal R'_2\le \mathcal R'_1$,  which leads to the inequality:
$$-(e^{\frac{\lambda_0\delta_1}{2l}}-1-\frac{\lambda_0\delta_1}{l})(-\frac{2\tilde L}{L^*}+1)\le \frac{1}{l}-\frac{1}{2l}e^{\frac{\lambda_0}{2l}(\tilde L+\delta_1)}-e^{\frac{\lambda_0}{2l}(\tilde L+\delta_1)}+1+\frac{\lambda_0(\tilde L+\delta_1)}{l},$$
where $\xi=\xi/1-\delta_1-\tilde L$, $\tilde L\in (0,L]$.
By setting $l$ sufficiently large, the asymptotic on the left-hand side is 
$$I_1\sim \frac{\lambda_0\delta_1}{2l}(-\frac{2\tilde L}{L^*}+1).$$
The asymptotic on the right-hand side is
$$I_2\sim\frac{1}{2l}\Big(1-\frac{\lambda_0}{2l}(\delta_1+\tilde L)+\lambda_0(\delta_1+\tilde L)\Big).$$
By further setting $\delta_1$ sufficiently small,  we ensure that $I_1\le \frac{1}{4l}\le I_2$ for all $\tilde L\in [0,L]$.
Thus, $\mathcal R_2\ge \mathcal R_1$ for $\xi\in[\xi_1-\delta_1-L,\xi_1-\delta_1]$. 
\end{proof}

Now, we let $\delta_1$ and $l$ be determined in Claim \ref{cl: delta1} and Claim \ref{cl: l}. Follow the discussion in Remark \ref{rm: glue}, $\mathcal{N}_0[\overline{\mathcal{W}}]\le 0$ for $\xi\in [\xi_1-\delta_1,\xi_1-\delta_1+L]$. Consequently, we find some $\delta_0(\varepsilon_1)\ll \varepsilon_1$, not depending on $\xi_1-\delta_1$, such that $N_0[\overline{\mathcal{W}}]\leq 0$ for $\xi\ge \xi_1-\delta_1$.

\medskip

\noindent{\bf{Step 2:}} We consider $\xi\in[\xi_2,\xi_1-\delta_1]$ with $\xi_1,\delta_1>0$ fixed by Step 1, and $\xi_1-\delta_1-\xi_2=L^*/4$. In this case, we have $$\mathcal{R}_w(\xi)=\mathcal R_2(\xi)=-\varepsilon_2\Psi(\xi-\xi_1+\delta_1+L^*/2)<0,$$
where $\Psi$ is the eigenfunction defined in \eqref{eigenvalue} and $\varepsilon_2(\varepsilon_1)$ is fixed by \eqref{epsilon 2 nl}.
Note that the magnitude of $\mathcal{R}_2$ can be made arbitrarily small by reducing $\varepsilon_2$, which in turn can be achieved by taking a smaller value of $\varepsilon_1$ from \eqref{epsilon 2 nl}.

By \eqref{eigenvalue}, \eqref{K1K2 nonlocal} , and \eqref{N0 inequality step 1 nonlocal}, we have
\begin{equation*}
\begin{aligned}
&\mathcal{N}_0[{\mathcal W}_*-\mathcal R_2]\\
\le&-\varepsilon_2\nu_0\Psi-f(\mathcal{W}_*;q^*)+f(\mathcal{W}_*-\mathcal{R}_2;q^*)-f(\mathcal{W}_*-\mathcal{R}_2;q^*)+f(\mathcal{W}_*-\mathcal{R}_2;q^*+\delta_0)\\
\le &-\varepsilon_2\nu_0\Psi+o(\varepsilon_2\Psi)+K_2\delta_0\le 0
\end{aligned}
\end{equation*}
on the bounded interval $[\xi_2,\xi_1-\delta_1]$, after possibly reducing $\delta_0(\varepsilon_1,\xi_2,L^*)$ if necessary. Combining  with Claim \ref{cl: delta1}, we obtain $\mathcal{N}_0[\overline{\mathcal W}]\le 0$ for some $\delta_0(\varepsilon_1,L^*)\ll \varepsilon_1$ on $[\xi_2+L,\xi_1-\delta_1]$. 
Thus, in the rest of Step 2, we only need to compare $\mathcal R_2$ and $\mathcal R_3$ for $\xi\in[\xi_2-L,\xi_2+L]$.

From now on, we fix $L^*$ and $\xi_2$. To make sure $\overline{\mathcal W}$ is continuous at $\xi_2$, decided by the discussion above, we set 
\begin{equation}\label{epsilon 3 nl}
\varepsilon_3=\varepsilon_3(\varepsilon_1)=\varepsilon_2\Psi(L^*/4)e^{-\lambda_1\xi_2}=-\frac{\varepsilon_1\Psi(L^*/4)}{\Psi(L^*/2)}\Big(\frac{1}{\lambda_0^2}(e^{\frac{\lambda_0\delta_1}{2l}}-1)-\frac{\delta_1}{\lambda_0l}\Big)e^{-\lambda_0(\xi_1-\delta_1)-\lambda_1\xi_2},
\end{equation}
where $\Psi(L^*/4)=K_0\frac{L^*}{4}e^{-\lambda_0L^*/4}$. 
Then we assert the following:
\begin{claim}\label{cl:xi 2}
 $\mathcal R_3\le \mathcal R_2$ for $\xi\in[\xi_2,\xi_2+L]$ and  $\mathcal R_3\ge \mathcal R_2$ for $\xi\in[\xi_2-L,\xi_2]$.
\end{claim}
\begin{proof}
From \eqref{epsilon 3 nl}, we know that, for $\xi\in[\xi_2,\xi_2+L]$, $\mathcal R_2\ge \mathcal R_3$ is equivalent to
$\mathcal R'_2\ge \mathcal R'_3$. On the other hand,  for $\xi\in[\xi_2-L,\xi_2]$, $\mathcal R_2\le \mathcal R_3$ is also equivalent to
$\mathcal R'_2\ge \mathcal R'_3$.
By \eqref{est Psi'}, we have $\mathcal R'_2\ge 0\ge \mathcal R'_3$ for $\xi\in[\xi_2-L,\xi_2+L]$.
Thus, the proof of Claim \ref{cl:xi 2} is complete.
\end{proof}

Follow the discussion in Remark \ref{rm: glue}, $\mathcal{N}_0[\overline{\mathcal{W}}]\le 0$ for $\xi\in [\xi_2,\xi_2+L]$. Consequently, we find some $\delta_0(\varepsilon_1)\ll \varepsilon_1$ such that $N_0[\overline{\mathcal{W}}]\leq 0$ for $\xi\in [\xi_2,\xi_1-\delta_1]$. The choice of $\delta_0$ is only depending $\varepsilon_1$ by recalling that $\xi_2,L^*$ are fixed by $\xi_1-\delta_1-\xi_2=L^*/4$.

\medskip

\noindent{\bf{Step 3:}} We consider $\xi\in(-\infty,\xi_2]$. In this case, we have $$\mathcal{R}_w(\xi)=\mathcal R_3=-\varepsilon_3e^{\lambda_1\xi}<0.$$
Recall that we choose $0<\lambda_1<\mu_0$
and
$$1-\mathcal{W}_*(\xi)\sim C_2 e^{\mu_0\xi}\ \ \text{as}\ \ \xi\to-\infty.$$
Then, there exists $M_1>0$ such that
$$\overline{\mathcal W}=\min\{\mathcal{W}_*-\mathcal{R}_w,1\}\equiv 1\ \ \text{for all}\ \ \xi\le -M_1,$$
and thus
$$\mathcal{N}_0[\overline{\mathcal W}]\le 0\  \ \text{for all}\ \ \xi\le -M_1.$$ Therefore, we only need to show
$$\mathcal{N}_0[{\mathcal W}_*-\mathcal R_3]\le 0\ \  \text{for all}\ \ -M_1\le\xi\le -\xi_2-\delta_3.$$

Since the kernel $J$ is trivial outside of $[-L,L]$, by some straightforward computations, we have
\begin{equation*}
\begin{aligned}
\mathcal{N}_0[{\mathcal W}_*-\mathcal R_3]
\le&-(e^{\lambda_1L}+c_0^*\lambda_1-1)\mathcal{R}_3-f(\mathcal{W}_*;q^*)+f(\mathcal{W}_*-\mathcal{R}_3;q^*)\\
&-f(\mathcal{W}_*-\mathcal{R}_3;q^*)+f(\mathcal{W}_*-\mathcal{R}_3;q^*+\delta_0).
\end{aligned}
\end{equation*}
From \eqref{condition on xi_2 scalar nonlocal} and $\mathcal{R}_3\le 0$, we have
$$-f(\mathcal{W}_*;q^*)+f(\mathcal{W}_*-\mathcal{R}_3;q^*)< K_3\mathcal{R}_3<0.$$
Together with the assumption (A2), we have
\beaa
\mathcal{N}_0[{\mathcal W}_*-\mathcal R_3]\leq -(e^{\lambda_1L}+c_0^*\lambda_1-1-K_3)\mathcal{R}_3+L_0\delta_0\quad\text{for all}\quad \xi\in[-M,\xi_2-\delta_3].
\eeaa
In view of \eqref{lambda 2 nonlocal} and Claim \ref{cl:xi 2},
we can assert that
$$\mathcal{N}_0[\overline{\mathcal W}]\le 0\ \ \text{for all}\ \ \xi\in[-M,\xi_2-\delta_3],$$
provided that $\delta_0(\varepsilon_1)$ is sufficiently small.
This completes the construction of Step 3.

\subsection{Proof of Theorem \ref{th: threshold scalar equation nonlocal}}

We are ready to prove Theorem~\ref{th: threshold scalar equation nonlocal} as follows.

\begin{proof}[Proof of Theorem~\ref{th: threshold scalar equation nonlocal}]
In view of Lemma~\ref{lem: th1-part1 nonlocal}, we have obtained \eqref{def of threshold scalar nonlocal}.
It suffices to show that \eqref{asy tw threshold scalar nonlocal} holds if and only if $q=q^*$.
From the discussion from Step 1 to Step 4 in \S \ref{subsec-3-1},
we are now equipped with an auxiliary function $\mathcal{R}_w(\xi)$ 
defined as in \eqref{definition of Rw nonlocal} such that
$$\overline{\mathcal W} (\xi)=\min \{\mathcal{W}_*(\xi)-\mathcal{R}_w(\xi),1\},$$
 which is independent of the choice of all sufficiently small $\delta_0>0$, forms
a super-solution satisfying \eqref{tw super solution scalar nonlocal}. By the comparison argument used in the proof of Theorem \ref{th: threshold scalar equation}, similarly we can show
\beaa
q=q^* \quad \Longrightarrow \quad \mbox{\eqref{asy tw threshold scalar nonlocal} holds}.
\eeaa
Therefore, it suffices to prove
\bea\label{th1:goal-2 nonlocal}
\mbox{\eqref{asy tw threshold scalar nonlocal} holds} \quad \Longrightarrow \quad q=q^*
\eea
by the sliding method.

We assume by contradiction that there exists $q_0\in(0,q^*)$ such that
the corresponding minimal traveling wave satisfies
\bea\label{W-S0+infty nonlocal}
\mathcal{W}_{q_0}(\xi)=B_0 e^{-\lambda_0\xi}+o(e^{-\lambda_0\xi})\quad\text{as}\quad \xi\to+\infty
\eea
for some $B_0>0$. For $\xi\approx -\infty$, from Proposition \ref{prop: asy tw - infty nonlocal}, we have
\bea\label{W-S0-infty nonlocal}
1-\mathcal{W}_{q_0}(\xi)=C_0 e^{\tilde\mu_0\xi}+o(e^{\tilde\mu_0\xi})\quad\text{as}\quad \xi\to-\infty
\eea
for some $C_0>0$, where $\tilde\mu_0=\mu_{s_0,c_0^*}$.
Recall that the asymptotic behavior of $\mathcal{W}_{q^*}$ at $\pm\infty$ satisfies
\bea\label{W-S*-pm-infty nonlocal}
\mathcal{W}_{q^*}(\xi)=B e^{-\lambda_0\xi}+o(e^{-\lambda_0\xi})\ \text{as}\ \xi\to+\infty;\ \ 1-\mathcal{W}_{q^*}(\xi)=C e^{\mu_0\xi}+o(e^{\mu_0\xi})\ \text{as}\ \xi\to-\infty
\eea
for some $B,C>0$, where $\mu_0=\mu_{q^*,c^*_0}$. In view of the assumption (A3), we have $\mu_0>\tilde\mu_0$ since $q^*>q_0$.
Combining \eqref{W-S0+infty nonlocal}, \eqref{W-S0-infty nonlocal} and \eqref{W-S*-pm-infty nonlocal}, there exists $0<L<\infty$ sufficiently large such that
$\mathcal{W}_{q^*}(\xi-L)> \mathcal{W}_{q_0}(\xi)$ for all $\xi\in\mathbb{R}$. Now, we define
\beaa
L^*:=\inf\{L\in\mathbb{R}\ |\ \mathcal{W}_{q^*}(\xi-L)\ge \mathcal{W}_{q_0}(\xi)\ \text{for all}\ \xi\in\mathbb{R}\}.
\eeaa
By the continuity, we have
$$\mathcal{W}_{q^*}(\xi-L^*)\geq \mathcal{W}_{q_0}(\xi)\ \text{for all}\ \ \xi\in\mathbb{R}.$$
 If there exists $\xi^*\in\mathbb{R}$ such that
$\mathcal{W}_{q^*}(\xi^*-L^*)= \mathcal{W}_{q_0}(\xi^*)$, by the strong maximum principle, we have
$$\mathcal{W}_{q^*}(\xi-L^*)=\mathcal{W}_{q_0}(\xi)\ \ \text{for all}\ \ \xi\in\mathbb{R},$$
which is impossible since $\mathcal{W}_{q^*}(\cdot-L^*)$ and $\mathcal{W}_{q_0}(\cdot)$ satisfy different equations. Consequently,
$$\mathcal{W}_{q^*}(\xi-L^*)> \mathcal{W}_{q_0}(\xi)\ \ \text{for all}\ \ \xi\in\mathbb{R}.$$
In particular, we have
\beaa
\lim_{\xi\to+\infty}\frac{\mathcal{W}_{q^*}(\xi-L^*)}{\mathcal{W}_{q_0}(\xi)}\geq1.
\eeaa
Furthermore, we can claim that
\bea\label{limit=1 nonlocal}
\lim_{\xi\to+\infty}\frac{\mathcal{W}_{q^*}(\xi-L^*)}{\mathcal{W}_{q_0}(\xi)}=1.
\eea
Otherwise, if the limit in \eqref{limit=1 nonlocal} is strictly bigger than 1, together with $\mu_0>\tilde \mu_0$ and
\beaa
\lim_{\xi\to-\infty}\frac{1-\mathcal{W}_{q^*}(\xi-L^*)}{1-\mathcal{W}_{q_0}(\xi)}=0,
\eeaa
we can easily  find $\varepsilon>0$ sufficiently small such that
$$\mathcal{W}_{q^*}(\xi-(L^*+\varepsilon))> \mathcal{W}_{q_0}(\xi)\ \ \text{for all}\ \ \xi\in\mathbb{R},$$
which contradicts the definition of $L^*$.
As a result, from \eqref{W-S0+infty nonlocal}, \eqref{W-S*-pm-infty nonlocal} and \eqref{limit=1 nonlocal}, we obtain $B_0=Be^{L^*}$.

On the other hand, we set $\widehat{\mathcal{W}}(\xi)=\mathcal{W}_{q^*}(\xi-L^*)-\mathcal{W}_{s_0}(\xi)$. Then $\widehat{\mathcal{W}}(\xi)$ satisfies
\bea\label{W-hat-eq2 nonlocal}
J\ast\widehat{\mathcal{W}}+c^*_0\widehat{\mathcal{W}}'+(f'(0)-1)\widehat{\mathcal{W}}+J(\xi)=0, \quad \xi\in\mathbb{R},
\eea
where
$$J(\xi)=f(\mathcal{W}_{s^*};s^*)- f'(0)\mathcal{W}_{s^*}-f(\mathcal{W}_{s_0};s_0)+ f'(0)\mathcal{W}_{s_0}.$$
By the assumption  (A1) and Taylor's Theorem, there exist $\eta_1\in(0, W_{s^*})$ and $\eta_2\in(0,W_{s_0})$ such that
\beaa
J(\xi)=J_1(\xi)+J_2(\xi)
\eeaa
where
$$J_1(\xi):=f''(\eta_1;q^*)(\mathcal{W}_{q^*}+\mathcal{W}_{q_0})\widehat{\mathcal{W}},$$
$$J_2(\xi):=[f''(\eta_1;q^*)-f''(\eta_2;q_0)]\mathcal{W}^2_{q_0}.$$
 It is easy to see that $J_1(\xi)=o(\widehat{\mathcal{W}})$ for $\xi\approx+\infty$. Next, we will show $J_2(\xi)=o(\widehat{\mathcal{W}})$ for $\xi\approx+\infty$.

Since $f''(0;s^*)>f''(0;s_0)$ (from the assumption (A3)), we can find small $\delta>0$ such that
$$\min_{\eta\in[0,\delta]}f''(\eta;q^*)>\max_{\eta\in[0,\delta]}f''(\eta;q_0)$$
and thus
there exist $\kappa_1,\kappa_2>0$ such that
\bea\label{J-lower bound nonlocal}
\kappa_1e^{-2\lambda_0\xi}\ge J_2(\xi)=[f''(\eta_1;q^*)-f''(\eta_2;q_0)]\mathcal{W}^2_{q_0}(\xi)\ge \kappa_2 e^{-2\lambda_0\xi}\quad \mbox{for all large $\xi$}.
\eea

We now claim that $J_2(\xi)=o(\widehat{\mathcal{W}})$ as $\xi\to+\infty$.
For contradiction, we assume that it is not true. Then there exists $\{\xi_n\}$ with
$\xi_n\to+\infty$ as $n\to\infty$ such that for some $\kappa_3>0$,
\bea\label{kappa3 nonlocal}
\frac{J_2(\xi_n)}{\widehat{\mathcal{W}}(\xi_n)}\geq \kappa_3\quad \mbox{for all $n\in\mathbb{N}$.}
\eea
Set $\widehat{\mathcal{W}}(\xi)=\alpha(\xi)e^{-2\lambda_0\xi}$, where $\alpha(\xi)>0$ for all $\xi$.
By substituting it into \eqref{W-hat-eq2 nonlocal},
we have
\begin{equation}\label{alpha-eq nonlocal}
\begin{aligned}
L(\xi):=&\Big(\int_{\mathbb{R}}J(y)\alpha(\xi-y)e^{2\lambda_0y}dy+(f'(0)-1-2\lambda_0c^*_0)\alpha(\xi)+ c^*_0\alpha'(\xi)\Big)e^{-2\lambda_0\xi}\\
&+J_1(\xi)+J_2(\xi)=0
\end{aligned}
\end{equation}
for all large $\xi$.
By \eqref{J-lower bound nonlocal} and \eqref{kappa3 nonlocal}, we have
\bea\label{alpha-bdd nonlocal}
0<\alpha(\xi_n)\leq \frac{\kappa_1}{\kappa_3}\quad  \mbox{for all $n\in\mathbb{N}$.}
\eea
Now, we will reach a contradiction by dividing the behavior of $\alpha(\cdot)$ into two cases:
\begin{itemize}
    \item[(i)] $\alpha(\xi)$ oscillates for all large $\xi$;
    \item[(ii)] $\alpha(\xi)$ is monotone for all large $\xi$.
\end{itemize}

For case (i), there exist local minimum points $\eta_n$ of $\alpha$ with $\eta_n\to\infty$ as $n\to\infty$ such that
\beaa
\alpha(\eta_n)>0\quad\text{and}\quad \alpha'(\eta_n)=0\quad  \mbox{for all $n\in\mathbb{N}$.}
\eeaa
Without loss of generality, we also assume that
\bea\label{local mini L}
\alpha(\eta_n)\ge\alpha(\xi)\quad\text{for all}\quad \xi\in[\eta_n-L,\eta_n+L].
\eea
Then from \eqref{formula of c_NL},  \eqref{alpha-eq nonlocal} yields that
\begin{equation*}
L(\eta_n)>\Big(\int_{\mathbb{R}}J(y)(\alpha(\eta_n-y)-\alpha(\eta_n))e^{2\lambda_0y}dy\Big)e^{-2\lambda_0\eta_n}+J_1(\xi_n)+J_2(\eta_n)
\end{equation*}
Together with \eqref{J-lower bound nonlocal} and $J_1(\xi)=o(\widehat{\mathcal{W}}(\xi))$,
from \eqref{alpha-eq nonlocal} and \eqref{local mini L}, we see that
\beaa
0=L(\eta_n)\geq o(1)\alpha(\eta_n)e^{-2\lambda_0\eta_n}+\kappa_2e^{-2\lambda_0\eta_n}>0
\eeaa
for all large $n$, which reaches a contradiction.

For case (ii),
due to \eqref{alpha-bdd nonlocal}, there exists $\alpha_0\in[0, \kappa_1/\kappa_3]$
such that $\alpha(\xi)\to \alpha_0$ as $\xi\to\infty$. Hence, we can find subsequence $\{\eta_j\}$ that tends to $\infty$ such that $\alpha'(\eta_j)\to0$ and
$\alpha(\eta_j)\to \alpha_0$ as $n\to\infty$.
From \eqref{alpha-eq nonlocal} we deduce that
\beaa
0=L(\eta_j)\geq (o(1)+ \kappa_2)e^{-2\lambda_0\eta_j}>0
\eeaa
for all large $j$, which reaches a contradiction.
Therefore, we have proved that
$J_2(\xi)=o(\widehat{\mathcal{W}})$ as $\xi\to\infty$.
Consequently, we have
\beaa
J(\xi)=J_1(\xi)+J_2(\xi)=o(\widehat{\mathcal{W}}(\xi))\quad \mbox{as $\xi\to\infty$.}
\eeaa

Now, by the proof of Proposition \ref{prop:correction-U-linear-decay}, we can assert that the asymptotic behavior of $\widehat{\mathcal{W}}(\xi)$ at $\xi=+\infty$ satisfies
\beaa
\widehat{\mathcal{W}}(\xi)=(C_1\xi+C_2)e^{-\beta \xi}+o(e^{-\beta\xi})\quad \mbox{as $\xi\to\infty$},
\eeaa
in which $C_1$ and $C_2$ can not be equal to $0$ simultaneously.
However, by $B_0=Be^{L^*}$, the asymptotic behaviors \eqref{W-S0+infty nonlocal} and \eqref{W-S*-pm-infty nonlocal} yield  $C_1=0$ and $C_2=0$, which reaches a contradiction.
Therefore, \eqref{th1:goal-2 nonlocal} holds, and the proof is complete.
\end{proof}

\section{Preliminary for the Lotka-Volterra competition system}

\subsection{Existence of traveling waves for \eqref{system} under {\bf(H)}}
\begin{proposition}\label{prop:existence-appendix}
Assume that {\bf(H)} holds. There exists the minimal speed $c_{LV}^{*}\in[2\sqrt{1-a},2]$
such that \eqref{system} admits a positive solution
$(u,v)(x,t)=(U,V)(x-ct)$ satisfying
\begin{equation}\label{TW-appendix}
\left\{
\begin{aligned}
&U''+cU'+U(1-U-aV)=0,\\
&dV''+cV'+rV(1-V-bU)=0,\\
&(U,V)(-\infty)=\omega,\ (U,V)(\infty)=(0,1),\\
&U'<0,\ V'>0,
\end{aligned}
\right.
\end{equation}
if and only if $c\ge c_{LV}^{*}$, where
\beaa
\omega=\begin{cases}
(1,0)&\quad \mbox{if $b\geq 1$},\\
\dps (u^*,v^*):=\Big(\frac{1-a}{1-ab},\frac{1-b}{1-ab}\Big)&\quad \mbox{if $0<b<1$}.
\end{cases}
\eeaa
Moreover, the minimal traveling wave speed $c_{LV}^{*}(b)$ is continuous and monotone increasing on $b\in(0,\infty)$.
\end{proposition}
\begin{proof}
For the existence of the minimal speed $c_{LV}^{*}$, it suffices to deal with the critical case $b=1$ since the case $b>1$ and $0<b<1$ have been proved in \cite{Kan-on1997} and \cite[Example 4.2]{Lewis Li Weinberger 2}, respectively.

\medskip
\begin{claim}\label{claim 1 AP}
Suppose that, for each $n\in \mathbb{N}$, $(\hat{c}, U_n, V_n)$ is a solution of \eqref{TW-appendix} with $b=b_n$ and
$b_n\searrow 1$ as $n\to\infty$. Then
\eqref{TW-appendix} has a monotone solution with $b=1$ and $c=\hat{c}$.
\end{claim}
\begin{proof}[Proof of Claim \ref{claim 1 AP}]
First, by translation, we may assume that $V_n(0)=1/2$ for all $n$. Also, by transferring the equation into integral equations (using a variation of the constants formula),
it is not hard to see that $U'_n$ and  $V'_n$ are uniformly bounded. Together with the fact that
$0\leq U_n(\xi), V_n(\xi)\leq 1$ for all $\xi\in\mathbb{R}$
and $n\in\mathbb{N}$, Arzel\`{a}-Ascoli Theorem allows us to take a subsequence that converges to a pair of limit functions $(U,V)\in [C(\mathbb{R})]^2$ with
$0\leq U,V\leq 1$, locally uniformly in $\mathbb{R}$. Moreover,
using Lebesgue's dominated convergence theorem to integral equations, we can conclude that
$(\hat{c},U,V)$ satisfies \eqref{TW-appendix} with $b=1$ (since $b_n\searrow 1$).
Moreover, we can see from the equations satisfied by $U$ and $V$ that $(U,V)\in [C^2(\mathbb{R})]^2$ and $U'\leq 0$ and $V'\geq0$ (since $U'_n\leq 0$ and $V'_n\geq0$ for all $n$), which implies that $(U,V)(\pm\infty)$ exists.

It remains to show that
\bea\label{BC-cond-Appendix}
(U,V)(-\infty)=(1,0),\quad  (U,V)(+\infty)=(0,1).
\eea
 Note that we must have
\bea\label{bdry-cond}
U(\pm\infty)[1-U(\pm\infty)-aV(\pm\infty)]=0,\quad V(\pm\infty)[1-V(\pm\infty)-U(\pm\infty)]=0.
\eea
Hence, $U(\pm\infty)$, $V(\pm\infty)\in\{0,1\}$. Since $V_n(0)=1/2$ for all $n$, we have $V(0)=1/2$ and thus
\bea\label{lim-of-V}
V(-\infty)=0,\quad V(+\infty)=1.
\eea
Also, note that from \eqref{bdry-cond} we see that $V(+\infty)=1$ implies that
\bea\label{lim-of-U}
U(+\infty)=0.
\eea

If $U(-\infty)=0$, then $U\equiv0$ due to $U'\leq 0$. However, by integrating the equation of $V$ over $(-\infty,+\infty)$,
it follows that
\beaa
\hat{c}+r\int_{-\infty}^{\infty}V(\xi)(1-V(\xi))d\xi=0,
\eeaa
which implies that $\hat{c}<0$. This contradicts with $\hat{c}>0$
(more precisely, from \cite{Kan-on1997} we see that $2\sqrt{1-a}\leq \hat{c}\leq 2$).
As a result, we have $U(-\infty)=1$, which together with \eqref{lim-of-V} and \eqref{lim-of-U} implies
\eqref{BC-cond-Appendix}.
We, therefore, obtain a monotone solution with $b=1$ and $c=\hat{c}$.
\end{proof}

Let us define
\beaa
c_{LV}^{*}:=\min\{\hat{c}>0|\, \mbox{\eqref{TW-appendix}  has a solution with $c=\hat{c}$}\}.
\eeaa
We write $c_{LV}^{*}=c_{LV}^{*}(b)$ to emphasize the dependence of $c_{LV}^{*}$ on $b$.
It follows from \cite{Kan-on1997} and \cite[Example 4.2]{Lewis Li Weinberger 2} that $c_{LV}^{*}(b)$ is well defined for all $b>0$ except $b=1$. We next prove the existence of $c_{LV}^{*}(1)$, {\it i.e.}, $c_{LV}^{*}(b)$ is continuous from both $b\to 1^+$ and $b\to 1^-$.

Let us define
$$\lim_{b\to 1^+}c_{LV}^{*}(b)=\overline c\ \ \text{and}\ \ \lim_{b\to 1^-}c_{LV}^{*}(b)=\underline c.$$
 Note that, by simple comparison argument, it holds
$$\underline c\le c_{LV}^{*}(1)\le \overline c.$$
Therefore, to complete the proof of Proposition \ref{prop:existence-appendix}, we only need to show $\underline c=\overline c$.

\begin{claim}\label{claim 2 AP}
It holds $\underline c=\overline c$.
\end{claim}
\begin{proof}[Proof of Claim \ref{claim 2 AP}]
Assume by contradiction that $\underline c<\overline c$, and hence by the continuity argument, there exists a traveling wave satisfying
 \begin{equation}\label{traveling wave c1}
    \left\{
        \begin{aligned}
        &U''+c_1U'+U(1-U-aV)=0,\\
        &dV''+c_1V'+rV(1-V-U)=0,\\
        &(U,V)(-\infty)=(1,0),\\
        &(U,V)(+\infty)=(0,1),
        \end{aligned}
        \right.
    \end{equation}
with $\underline c\le c_1<\overline c$. Remark that, in the following proof, we will use certain asymptotic estimates of the traveling wave $(U, V)$ with speed $c_1$, as defined in \eqref{traveling wave c1}. These estimates are provided in Lemma~\ref{lem:AS-infty:b=1} and Corollary~\ref{lm: behavior around - infty b=1}, and notably, they do not depend on the specific value of the wave speed. The proofs of these results will be given at the end of \S4.3.

We aim to find $(R_u,R_v)(\xi)$ like Figure \ref{Figure ap} such that
$$(\overline U,\underline{V})(\xi):=\Big(\min\{(U_1-R_u)(\xi),1\},\max\{(V_1+R_v)(\xi),0\}\Big)$$ become a super-solution
satisfying
\begin{equation}\label{def: super solution 2}
\left\{
\begin{aligned}
N_1[\overline U,\underline V]&:=\overline U''+c_2\overline U'+\overline U(1-\overline U-a\underline V)\le 0,\\
N_{2}[\overline U,\underline V]&:=d\underline V''+c_2\underline V'+r\underline V(1-\underline V-(1+\delta_0)\overline U)\ge 0,
\end{aligned}
\right.
\end{equation}
for some small $\delta_0>0$ and $c_1<c_2<\overline c$.
Moreover,
$\overline U'(\xi_0^{\pm})$
(resp. $\underline V'(\xi_0^{\pm})$) exists and
\begin{equation*}
\overline U'(\xi_0^+)\leq \overline U'(\xi_0^-)\quad (resp.\, \underline V'(\xi_0^+)\geq \underline V'(\xi_0^-))
\end{equation*}
if  $\overline U'$ (resp., $\underline V'$) is not continuous at $\xi_0$.

\begin{figure}
\begin{center}
\begin{tikzpicture}[scale = 1.1]
\draw[thick](-6,0) -- (6,0) node[right] {$\xi$};
\draw [ultra thick] (-6,-0.7)to [out=0,in=140] (-3,-1.5) -- (3,-1.5) to [out=60,in=185] (6,-0.5);
\draw [semithick] (-6, -0.2) to [ out=0, in=150] (-3,-0.7) -- (3,-0.7)  to [out= 40, in=180] (6,-0.1);
\node[above] at (3,0.5) {$M$};
\draw[dashed] [thick] (3,0)-- (3,0.5);
\node[above] at (-3,0.5) {$-M$};
\draw[dashed] [thick] (-3,0)-- (-3,0.5);
\draw [thin] (-2.8,-1.5) arc [radius=0.2, start angle=0, end angle= 135];
\node[above] at (-2.8,-1.4) {$\alpha_3$};
\draw [thin] (-2.8,-0.7) arc [radius=0.2, start angle=0, end angle= 155];
\node[above] at (-2.8,-0.5) {$\alpha_4$};
\draw [thin] (3.2,-1.25) arc [radius=0.2, start angle=60, end angle= 200];
\node[above] at (2.7,-1.4) {$\alpha_1$};
\draw [thin] (3.2,-0.55) arc [radius=0.2, start angle=50, end angle= 180];
\node[above] at (2.7,-0.5) {$\alpha_2$};
\node[below] at (0,-0.9) {$R_u$};
\node[below] at (0,-0.1) {$R_v$};
\end{tikzpicture}
\caption{ $(R_u,R_v)$.}\label{Figure ap}
\end{center}
\end{figure}

We now define $(R_u,R_v)(\xi)$ as following:
\begin{equation}\label{def rurv}
(R_u,R_v)(\xi):=\begin{cases}
(-\varepsilon_1e^{-\lambda_1\xi},-\eta_1e^{-\lambda_1\xi}),&\ \ \mbox{for}\ \ M\le\xi,\\
(-\varepsilon_1e^{-\lambda_1 M},-\eta_1e^{-\lambda_1 M}),&\ \ \mbox{for}\ \ -M\le\xi\le M,\\
(-\varepsilon_2(-\xi)^{1/2}[1-U_1(\xi)],-\eta_2(-\xi)^{1/2}V_1(\xi)),&\ \ \mbox{for}\ \ \xi\le -M,
\end{cases}
\end{equation}
where $\lambda_{1}>\max\{\Lambda(c_1),\lambda_v^-(c_1)\}>0$ which is defined in Lemma \ref{lm: behavior around + infty}. Here $\varepsilon_{1,2}>0$ and $\eta_{1,2}>0$, very small such that $|R_u|,|R_v|\ll 1$ , will be determined later.

\noindent{\bf{Step 1}}
We consider $\xi\in[M,\infty)$.
In this case, we have $(R_u,R_v)=(-\varepsilon_1e^{-\lambda_1\xi},-\eta_1e^{-\lambda_1\xi})$ with $\lambda_{1}>\max\{\Lambda(c_1),\lambda_v^-(c_1)\}$.

Recall that, $(U_1,V_1)$ is the minimal traveling wave satisfying \eqref{traveling wave c1}.
By some  straightforward computations,  we have
\begin{equation}\label{eq 1}
\begin{aligned}
N_1[\overline U,\underline V]=&(c_2-c_1)U'_1-(\lambda_1^2+c_2\lambda_1)R_u\\
&-R_u(1-2U_1+R_u-aV_1-aR_v)-aR_vU_1,
\end{aligned}
\end{equation}
and
\begin{equation*}
\begin{aligned}
N_{2}[\overline U,\underline V] = &(c_2-c_1)V_1'+ (d\lambda_1^2+c_2\lambda_1)R_v+ rR_v(1-2V_1-R_v-U_1+R_u)\\
&-r\delta_0(V_1+R_v)(U_1-R_u).
\end{aligned}
\end{equation*}

By Lemma \ref{lm: behavior around + infty}, there exists $C_1>0$ such that
$$(c_2-c_1)U_1'\le -C_1(c_2-c_1)U_1\ \ \text{for all}\ \ \xi\in[M,\infty)$$
up to enlarging $M$ if necessary.
Then, from the definition of $(R_u,R_v)$ and $\lambda_{1}>\max\{\Lambda(c_1),\lambda_v^-(c_1)\}$. As $\xi\to+\infty$, we have
$$-(\lambda_1^2+c_2\lambda_1)R_u=o(U_1),\ -R_u(1-2U_1+R_u-aV_1-aR_v)=o(U_1),\ -aR_vU_1=o(U_1).$$
Therefore, by setting  $\varepsilon_1\ll 1$ and $\eta_1\ll 1$, from \eqref{eq 1}, we have $N_1[\overline U,\underline V]\le 0$ for all $\xi\in[M,\infty)$.

Next, we deal with the inequality of $N_{2}[W_u, W_v]$. Since $\lambda_{1}>\max\{\Lambda(c_1),\lambda_v^-(c_1)\}$, as $\xi\to+\infty$ we have
$$ (d\lambda_1^2+c_2\lambda_1)R_v=o(1-V_1), rR_v(1-2V_1-R_v-U_1+R_u)=o(1-V_1).$$
From Lemma \ref{lm: behavior around + infty}, there exists $C_2>0$ such that
$$(c_2-c_1)V'_1\ge C_2(c_2-c_1)(1-V_1).$$
By the asymptotic behavior of $U_1$ in Lemma \ref{lm: behavior around + infty}, we have
$$(c_2-c_1)V'_1-r\delta_0(V_1+R_v)(U_1-R_u)>(c_2-c_1)V'_1-r\delta_0V_1(U_1-R_u)>0,$$
provided that $\delta_0\ll (c_2-c_1)$ is sufficiently small.
Then, we have $N_{2}[W_u,W_v]\geq0$
for $\xi\geq M$ up to enlarging $M$ if necessary. The choice of $\delta_0$ is not depending on $M$.

\noindent{\bf{Step 2}}
We consider $\xi\in[-M,M]$. In this case, $(R_u,R_v)$ are constants. By the definition \eqref{def rurv}, $(R_u.R_v)$ is continuous at $\xi=M$. Moreover, it is easy to verify that
$$\lim_{\xi\to M^+}R'_u(\xi)>0=\lim_{\xi\to M^-}R'_u(\xi)\quad\text{and}\quad \lim_{\xi\to M^+}R'_v(\xi)>0=\lim_{\xi\to M^-}R'_v(\xi),$$
which implies $\angle \alpha_1,\angle \alpha_2<180^{\circ}$.

By some straightforward computation, we have
 \begin{equation*}
N_1[\overline U,\underline V]=(c_2-c_1)U'_1-R_u(1-2U_1+R_u-aV_1-aR_v)-aR_vU_1,
\end{equation*}
and
\begin{equation*}
N_{2}[\overline U,\underline V] \ge (c_2-c_1)V_1'+ rR_v(1-2V_1-R_v-U_1+R_u)-r\delta_0V_1(U_1-R_u).
\end{equation*}
Since $[-M,M]$ is a bounded interval and $|R_u|,|R_v|\ll 1$, by reducing  $\delta_0(M,\eta_1)\ll (c_2-c_1)$, $\varepsilon_1$,  $\eta_1$ if necessary, we have $N_1[\overline U,\underline V]\le 0$ and
$N_{2}[\overline U,\underline V]\ge 0$ for $\xi\in[-M,M]$. Note that the choice of $\delta_0$ is unaffected by reducing $\eta_1$.

\noindent{\bf{Step 3}}
We consider $\xi\in(-\infty,-M]$. In this case, we have
$$(R_u,R_v)=(-\varepsilon_2(-\xi)^{1/2}[1-U_1(\xi)],-\eta_2(-\xi)^{1/2}V_1(\xi)).$$
We take
$$\varepsilon_2=\frac{\varepsilon_1e^{-\lambda_1 M}}{M^{1/2}(1-U_1(-M))}\quad\text{and}\quad\eta_2=\frac{\eta_1e^{-\lambda_1 M}}{M^{1/2}V_1(-M)}$$
such that $(R_u.R_v)$ is continuous at $\xi=-M$. It is easy to verify that
$$\lim_{\xi\to -M^+}R'_u(\xi)=0>\lim_{\xi\to -M^-}R'_u(\xi)\quad\text{and}\quad \lim_{\xi\to -M^+}R'_v(\xi)=0>\lim_{\xi\to M^-}R'_v(\xi),$$
which implies $\angle \alpha_3,\angle \alpha_4<180^{\circ}$.


Note that, from the definition of $(R_u,R_v)$, by adjusting $\varepsilon_1/\eta_1$,
we have $\varepsilon_2=\eta_2=(M')^{-\frac{1}{2}}$, which implies
 $\overline U(\xi)<1$, $\underline V(\xi)>0$ for all $\xi\in(-M',-M]$ and
$\overline U(\xi)=1$, $\underline V(\xi)=0$ for all $\xi\in(-\infty,-M']$, which implies that
$$N_1[\overline U,\underline V]\le 0\ \ \text{and}\ \ N_{2}[\overline U,\underline V]\ge 0\ \ \text{for}\ \ \xi\in(-\infty,-M'].$$

It suffices to consider $\xi\in[-M',-M]$.
By $R_u,R_v<0$ in $[-M',-M]$ and $U'_1<0$, we have
\begin{equation*}
\begin{aligned}
N_1[\overline U,\underline V]=&(c_2-c_1)U'_1+\varepsilon_2(-\xi)^{1/2}\Big(-U_1''-c_2U_1'-\frac{1}{4}(-\xi)^{-2}(1-U_1)+(-\xi)^{-1}U_1'\\&-\frac{c_2}{2}(-\xi)^{-1}(1-U_1)\Big)
-R_u(1-2U_1+R_u-a(V_1+R_v))-aU_1R_v\\
\leq&(c_2-c_1)(U'_1-\varepsilon_2(-\xi)^{1/2}U'_1)+\varepsilon_2(-\xi)^{1/2}\Big(U_1(1-U_1-aV_1)-\frac{c_2}{2}(-\xi)^{-1}(1-U_1)\Big)\\&-R_u(1-2U_1+R_u-a(V_1+R_v))-aU_1R_v.
\end{aligned}
\end{equation*}

Note that,  $1-U_1\ge -R_u$ in $\xi\in[-M',-M]$ and $1-U_1=-R_u$ on $\xi=-M'$.
By the asymptotic behavior in Lemma \ref{lem:AS-infty:b=1}, we have
$(1-U_1)'\ge -R'_u$ for $\xi\in[-M',-M]$, which implies
$$(c_2-c_1)(U'_1-\varepsilon_2(-\xi)^{1/2}U'_1)\le -\frac{(c_2-c_1)}{2}(-\xi)^{-1}R_u.$$
Then, by using
$\varepsilon_2(-\xi)^{1/2}U_1(1-U_1)=-R_uU_1$ and $\varepsilon_2=\eta_2$, from the computation above,
\beaa
N_1[\overline U,\underline V]&\le&-R_uU_1-a\varepsilon_2(-\xi)^{1/2}U_1V_1+\frac{c_2}{2}(-\xi)^{-1}R_u-R_u(1-2U_1-aV_1)
\\
&&-R_u^2+aR_uR_v+a\varepsilon_2(-\xi)^{1/2}U_1V_1 -\frac{(c_2-c_1)}{2}(-\xi)^{-1}R_u\\
&=&\frac{c_1}{2}(-\xi)^{-1}R_u-R_u(1-U_1-aV_1)-R_u^2+aR_uR_v.
\eeaa

Denote that
\beaa
I_1:=\frac{c_1}{2}(-\xi)^{-1}R_u,\quad
I_2:=-R_u(1-U_1-aV_1),\quad
I_3:=-R_u^2+aR_uR_v.
\eeaa
From Corollary \ref{lm: behavior around - infty b=1}, we have $1-U_1-aV_1>0$ for all $\xi\leq -M$. Therefore,
\beaa
I_3=-R_u^2+aR_uR_v\leq R_u\varepsilon_2(-\xi)^{1/2}(1-U_1-aV_1)(\xi)<0\quad \mbox{for}\quad \xi\in[-M',-M].
\eeaa
Moreover, in view of Corollary~\ref{lm: behavior around - infty b=1} again, we have
$I_2=o(I_1)$ as $\xi\to-\infty$.
Then, up to enlarging $M$ if necessary, we have $N_1[\overline U,\underline V]\le 0$ for $\xi\in[-M',-M]$. From now on, we fix $M$.

On the other hand, by some straightforward computations, we have
\beaa
N_{2}[\overline U,\underline V]&=&d\Big(V_1''+\frac{\eta_2}{4}(-\xi)^{-2/3}V_1+\eta_2(-\xi)^{-1/2}V'_1-\eta_2(-\xi)^{1/2}V''_1\Big)\\
&&+c_2\Big(V_1'+\frac{\eta_2}{2}(-\xi)^{-1/2}V_1-\eta_2(-\xi)^{1/2}V'_1\Big)\\&&
+r(V_1+R_v)(1-V_1-R_v-(1+\delta_0)(U_1-R_u)).\\
&\geq& r\eta_2(-\xi)^{1/2}V_1\Big((\eta_2(-\xi)^{1/2}-1)(1-U_1-V_1)
+\frac{c_2}{2r}(-\xi)^{-1}\Big)\\
&&+(c_2-c_1)(V_1'-\eta_2(-\xi)^{1/2}V_1')-r(U_1-R_u)(V_1+R_v)\delta_0.
\eeaa

In the discussion above, we fixed $\varepsilon_1/\eta_1$ to get $\varepsilon_2=\eta_2$. Now. we further reduce $\eta_1$ but keep $\varepsilon_1/\eta_1$ unchanged. Then by Corollary \ref{lm: behavior around - infty b=1}, we have
$(\eta_2(-\xi)^{1/2}-1)(1-U_1-V_1)>0$ for $\xi\le-M$.
Note that $V_1\ge -R_v$ in $\xi\in[-M',-M]$ and $V_1=-R_v$ on $\xi=-M'$. By Corollary \ref{lm: behavior around - infty b=1} again, we have
$V_1'\ge -R'_v$ in $\xi\in[-M',-M]$, which implies
$$(c_2-c_1)(V'_1-\eta_2(-\xi)^{1/2}V'_1)\ge \frac{(c_2-c_1)}{2}(-\xi)^{-1}R_v.$$
Thus, on the bounded interval $[-M',-M]$,
$$N_{2}[\overline U,\underline V]\ge -\frac{c_1}{2}(-\xi)^{-1}R_v-r(U_1-R_u)(V_1+R_v)\delta_0\ge 0$$
for all small $\delta_0(\eta_1)>0$.
Therefore
the construction of $(R_u,R_v)$ is complete.

Now we are equipped with a super-solution satisfying \eqref{def: super solution 2}. Let us consider the spreading speed of the solution of
\begin{equation}\label{eq 2 1}
    \left\{\begin{aligned}\relax
    & u_t=u_{xx}+u(1-u-av) \\
        & v_t=dv_{xx}+rv(1-v-(1+\delta_0)u),
    \end{aligned}\right.
\end{equation}
with initial datum \eqref{initial datum}.  It is known that the spreading speed is greater than or equal to $\overline c$.

On the other hand, it is easy to check that $(\overline u, \underline v)(t,x):=(\overline U,\underline V)(x-c_2t-x_0)$ is a super-solution of \eqref{eq 2 1}. Moreover, by setting $x_0>0$ large, one has $\overline u(0,x)\ge u_0(x)$ and $\underline v(0,x)\le v_0(x)$. Then, by the comparison principle, the spreading speed is smaller than or equal to $c_2$, which is impossible since $c_2<\overline c$.
\end{proof}
By Claim \ref{claim 1 AP} and Claim \ref{claim 2 AP}, we can assert that $c^*_{LV}(1)$ is well defined and $c^*_{LV}(b)$ is continuous for all $b\in(0,+\infty)$. The proof of Proposition \ref{prop:existence-appendix} is complete.
\end{proof}
\subsection{Asymptotic behavior of traveling waves of \eqref{system} near $+\infty$}\label{subsec:Asymp behavior}

\noindent

In this subsection, we provide the asymptotic behavior of $(U_c,V_c)$ near $\pm\infty$
for $0<a<1$ and $b>0$, where $(U_c,V_c)$ satisfies either \eqref{tw solution weak} with speed $c$.
Some results are reported in \cite{MoritaTachibana2009}.

Hereafter, we denote
\beaa
&&\lambda_u^{\pm}(c):=\frac{c\pm\sqrt{c^2-4(1-a)}}{2}>0,\\ &&\lambda_v^{+}(c):=\frac{c+\sqrt{c^2+4rd}}{2d}>0>\lambda_v^{-}(c):=\frac{c-\sqrt{c^2+4rd}}{2d}.
\eeaa
The asymptotic behavior of $(U,V)$ near $+\infty$ for $0<a<1$ and $b>1$ can be found in \cite{MoritaTachibana2009}.
Note that the conclusions presented in \cite{MoritaTachibana2009} are still applicable for $b>0$ since $b$ is not present in the linearization at the unstable equilibrium $(0,1)$. Therefore, we have the following result.

\begin{lemma}[\cite{MoritaTachibana2009}]\label{lm: behavior around + infty}
Assume that $0<a<1$ and $b>0$.
Let $(c,U,V)$ be a solution of the system (\ref{tw solution weak}).
Then there exist positive constants $l_{i=1,\cdots,8}$ such that the following hold:
\begin{itemize}
\item[(i)]
For $c>2\sqrt{1-a}$,
\beaa
&&\lim_{\xi\rightarrow+\infty}\frac{U(\xi)}{e^{-\Lambda(c)\xi}}=l_1,\\
&&\lim_{\xi\rightarrow+\infty}\frac{1-V(\xi)}{e^{-\Lambda(c)\xi}}=l_2\quad \mbox{if $\lambda_v^{+}(c)>\Lambda(c)$},\\
&&\lim_{\xi\rightarrow+\infty}\frac{1-V(\xi)}{\xi e^{-\lambda_v^{+}(c)\xi}}=l_3\quad
\mbox{if $\lambda_v^{+}(c)=\Lambda(c)$},\\
&&\lim_{\xi\rightarrow+\infty}\frac{1-V(\xi)}{e^{-\lambda_v^{+}(c)\xi}}=l_4\quad
\mbox{if $\lambda_v^{+}(c)<\Lambda(c)$},
\eeaa
where $\Lambda(c)\in\{\lambda_u^+(c),\lambda_u^-(c)\}$. 
\item[(ii)]
For $c=2\sqrt{1-a}$,
\beaa
&&\lim_{\xi\rightarrow+\infty}\frac{U(\xi)}{\xi^{p} e^{-\Lambda(c)\xi}}=l_5,\\
&&\lim_{\xi\rightarrow+\infty}\frac{1-V(\xi)}{\xi^p e^{-\Lambda(c)\xi}}=l_6\quad \mbox{if $\lambda_v^{+}(c)>\Lambda(c)$},\\
&&\lim_{\xi\rightarrow+\infty}\frac{1-V(\xi)}{\xi^{p+1} e^{-\Lambda(c)\xi}}=l_7\quad
\mbox{if $\lambda_v^{+}(c)=\Lambda(c)$},\\
&&\lim_{\xi\rightarrow+\infty}\frac{1-V(\xi)}{ e^{-\lambda_v^{+}(c)\xi}}=l_8\quad
\mbox{if $\lambda_v^{+}(c)<\Lambda(c)$},
\eeaa
where  $\Lambda(c)=\lambda_u^{\pm}(c)=\sqrt{1-a}$ 
and  $p\in\{0,1\}$.
\end{itemize}
\end{lemma}


When $c=2\sqrt{1-a}$, it is not clear whether $p=0$ or $p=1$.
By applying a similar argument used in \cite{Guo Wu} that considered the discrete version of \eqref{system},
we can derive an implicit criterion for determining whether $p=0$ or $p=1$, which is given in the following proposition.

\begin{proposition}\label{prop: implicit cond}
Assume that $0<a<1$ and $b>0$.
Let $(c,U,V)$ be a solution of \eqref{tw solution weak} with
$c=2\sqrt{1-a}$ and
$p$ be given in (ii) of Lemma~\ref{lm: behavior around + infty}.
Then
\bea\label{cond-p}
 p=
\begin{cases}
1 & \quad \mbox{if and only if}\quad  \int_{-\infty}^{\infty}e^{\Lambda(c)\xi}U(\xi)[a(1-V(\xi))-U(\xi)]d\xi\neq0,\\
0 & \quad \mbox{if and only if}\quad  \int_{-\infty}^{\infty}e^{\Lambda(c)\xi}U(\xi)[a(1-V(\xi))-U(\xi)]d\xi=0,
\end{cases}
\eea
where $\Lambda(c)=\lambda_u^{\pm}(c)=\sqrt{1-a}$.
\end{proposition}

The proof of Proposition~\ref{prop: implicit cond} is based on
a modified version of Ikehara's Theorem, which is given as follows:

\begin{proposition}[see Proposition 2.3 in \cite{Carr Chmaj}]\label{prop:ikehara}
For a positive non-increasing function $U$, we define
\begin{eqnarray*}
\dps F(\lambda):=\int_0^{+\infty} e^{-\lambda\xi} U(\xi) d\xi,\quad
\mbox{$\lambda\in\mathbb{C}$ with ${\rm Re}\lambda<0$}.
\end{eqnarray*}
If $F$ can be written as
$F(\lambda)={H(\lambda)}/{(\lambda+\gamma)^{p+1}}$
for some constants $p>-1, \gamma>0$, and some analytic function $H$
in the strip $-\gamma\leq {\rm Re}\lambda<0$, then
\begin{eqnarray*}
\lim_{\xi\rightarrow +\infty}
\frac{U(\xi)}{{\xi}^{p}e^{-\gamma\xi}}=\frac{H(-\gamma)}{\Gamma(\gamma+1)}.
\end{eqnarray*}
\end{proposition}

\begin{proof}[Proof of Proposition~\ref{prop: implicit cond}]
In fact, by modifying the process used in \cite{Guo Wu}, we can prove Lemma~\ref{lm: behavior around + infty} and
\eqref{cond-p} independently; however, the proof is quite long.
Instead of giving detailed proof,
we simply assume that Lemma~\ref{lm: behavior around + infty} hold and derive
\eqref{cond-p} by using Proposition~\ref{prop:ikehara}.

Let us  define the bilateral Laplace transform of $U$ as
\beaa
\mathcal{L}(\lambda):=\int_{-\infty}^{+\infty}e^{-\lambda \xi}U(\xi) d\xi,
\eeaa
which is well-defined for $-\Lambda(c)<{\rm Re} \lambda<0$
(since we have assumed that Lemma~\ref{lm: behavior around + infty} holds).
Using the equation of $U$ and integration by parts several times, we have
\bea\label{Phi-eq}
\Phi(\lambda)\mathcal{L}(\lambda)+I(\lambda)=0,\quad -\Lambda(c)<{\rm Re} \lambda<0,
\eea
where
\beaa
\Phi(\lambda):=c\lambda+\lambda^2+1-a,\quad I(\lambda):=\int_{-\infty}^{\infty}e^{-\lambda \xi}U[a(1-V)-U](\xi)d\xi.
\eeaa
To apply Ikehara's Theorem, we
rewrite \eqref{Phi-eq} as
\beaa
F(\lambda):=\int_{0}^{+\infty}e^{-\lambda \xi}U(\xi) d\xi
=-\frac{I(\lambda)}{\Phi(\lambda)}-\int_{-\infty}^{0}e^{-\lambda \xi}U(\xi) d\xi,
\eeaa
as long as $\Phi(\lambda)$ does not vanish.
Also, we define
\bea\label{Ikehara-from}
H(\lambda):=Q(\lambda)-[\lambda+\Lambda(c)]^{p+1}\int_{-\infty}^{0}e^{-\lambda \xi}U(\xi) d\xi,
\eea
where $\Lambda(c)=\sqrt{1-a}$, $p\in\mathbb{N}\cup\{0\}$, and
\bea\label{Q-func}
Q(\lambda):=-\frac{I(\lambda)}{\Phi(\lambda)/[\lambda+\Lambda(c)]^{p+1}}.
\eea

We now prove that $H$ is analytic in the strip $S:=\{-\Lambda(c)\leq {\rm Re}\lambda<0\}$.
Since the second term on the right-hand side of \eqref{Ikehara-from} is always analytic for ${\rm Re}\lambda<0$,
it suffices to show that $Q$ is analytic in the strip $S$.
Since $\mathcal{L}$ is well-defined for $-\Lambda(c)< {\rm Re}\lambda<0$, we see that
$Q$ is analytic for $-\Lambda(c)< {\rm Re}\lambda<0$. Therefore, it suffices to prove the
analyticity of $Q$ on $\{{\rm Re}\lambda=-\Lambda(c)\}$.
For this, we claim that the only root of $\Phi(\lambda)=0$ is the real root $\lambda=-\Lambda(c)$.
To see this, let $\lambda=\alpha +\beta i$ for $\alpha,\beta\in\mathbb{R}$ and $i:=\sqrt{-1}$. If
$\Phi(\alpha +\beta i)=0$, then by simple calculations we see that $\beta=0$ and $\alpha=-\Lambda(c)$.
Therefore, from \eqref{Q-func} we see that $Q$ is analytic on $\{{\rm Re}\lambda=-\Lambda(c)\}$ and is also analytic in $S$. Then, Ikehara's Theorem can be applied to assert that
\beaa
\lim_{\xi\to+\infty}\frac{U(\xi)}{{\xi}^{p}e^{\Lambda(c)\xi}}=\frac{H(-\Lambda(c))}{\Gamma(\Lambda(c)+1)}=
\frac{Q(-\Lambda(c))}{\Gamma(\Lambda(c)+1)}.
\eeaa

Finally, we need to prove $Q(-\Lambda(c))\neq0$ by taking suitable $p$.
To do so, note that \eqref{Q-func} and the fact that $\Phi(\lambda)=0$ imply that $\lambda=-\Lambda(c)$.
We see that, if $I(-\Lambda(c))\neq0$, then $Q(-\Lambda(c))\neq0$ if and only if $p=1$.
On the other hand, when $I(-\Lambda(c))=0$, then $\lambda=-\Lambda(c)$ must be simple root of $I(\lambda)=0$. Otherwise,
we have $Q(-\Lambda(c))=0$ for any $p\in\mathbb{N}\cup\{0\}$, which contradicts
the conclusion (ii) of Lemma~\ref{lm: behavior around + infty}. Therefore, when $I(-\Lambda(c))=0$,
we have $Q(-\Lambda(c))\neq0$ if and only if $p=0$, so \eqref{cond-p} holds.
This completes the proof.
\end{proof}

\subsection{Asymptotic behavior of traveling waves of \eqref{system} near $-\infty$}
To describe the asymptotic behavior of $(U,V)$ near $-\infty$, we define
\beaa
&&\mu_{u}^{-}(c):=\frac{-c-\sqrt{c^2+4}}{2}<0<\mu_{u}^{+}(c):=\frac{-c+\sqrt{c^2+4}}{2},\\
&&\mu_v^{-}(c):=\frac{-c-\sqrt{c^2+4rd(b-1)}}{2d}<0<\mu_v^{+}(c):=\frac{-c+\sqrt{c^2+4rd(b-1)}}{2d}.
\eeaa

\begin{lemma}[\cite{MoritaTachibana2009}]\label{lem:AS-infty:b>1}
Assume that $0<a<1$ and $b>1$.
Let $(c,U,V)$ be a solution of the system (\ref{tw solution weak}). Then
there exist two positive constants $l_{i=9,\cdots,12}$ such that
\beaa
&& \lim_{\xi\rightarrow-\infty}\frac{V(\xi)}{e^{\mu_v^+(c)\xi}}=l_9,\\
&&\lim_{\xi\rightarrow-\infty}\frac{1-U(\xi)}{e^{\mu_v^+(c)\xi}}=l_{10}\quad \mbox{if $\mu_{u}^{+}(c)>\mu_v^+(c)$},\\
&&\lim_{\xi\rightarrow-\infty}\frac{1-U(\xi)}{|\xi|e^{\mu_v^+(c)\xi}}=l_{11}\quad
\mbox{if $\mu_{u}^{+}(c)=\mu_v^+(c)$},\\
&&\lim_{\xi\rightarrow-\infty}\frac{1-U(\xi)}{e^{\mu_u^+(c)\xi}}=l_{12}\quad \mbox{if $\mu_{u}^{+}(c)<\mu_v^+(c)$}.
\eeaa
\end{lemma}

\bigskip
\begin{lemma}\label{lem:AS-infty:b<1}
Assume that $0<a,b<1$.
Let $(c,U,V)$ be a solution of the system (\ref{tw solution weak}).
Then there exist two positive constants $l_{13}$ and $l_{14}$ such that
\beaa
\lim_{\xi\to-\infty}\frac{u^*-U(\xi)}{e^{\nu\xi}}=l_{13},\quad
\lim_{\xi\to-\infty}\frac{V(\xi)-v^*}{e^{\nu\xi}}=l_{14}
\eeaa
where
$\nu$ is the smallest positive zero of
\bea\label{chara-poly}
\rho(\lambda):=(\lambda^2+c\lambda-u^*)(d\lambda^2+c\lambda-rv^*)-rabu^*v^*.
\eea
\end{lemma}
\begin{proof}
Set
$g_u(\lambda):=\lambda^2+c\lambda-u^*$ and $g_v(\lambda):=d\lambda^2+c\lambda-rv^*$.
Then $g_u$ (resp., $g_v$) has two zeros $\mu^{u}_{\pm}$ (resp. $\mu^{v}_{\pm}$) with
$\mu^{u}_{-}<0< \mu^{u}_{+}$ (resp., $\mu^{v}_{-}<0< \mu^{v}_{+}$). More precisely, we have
\beaa
\mu^u_{\pm}=\frac{-c\pm\sqrt{c^2+4u^*}}{2},\quad \mu^v_{\pm}=\frac{-c\pm\sqrt{c^2+4drv^*}}{2d}.
\eeaa

Note that
$\rho(\lambda)=g_u(\lambda)g_v(\lambda)-rabu^*v^*$. Since $\rho(\pm\infty)=+\infty$,
$\rho(\mu^{u}_{\pm})<0$, $\rho(\mu^{v}_{\pm})<0$, and $\rho(0)=ru^*v^*(1-hk)>0$, we see that
$\rho$ has exactly four distinct real zeros $\lambda=\nu_i$ $(i=1,2,3,4)$, two negative and two positive zeros, such that
\beaa
\nu_4<\min\{\mu^{u}_{-},\mu^{v}_{-} \}\leq \max\{\mu^{u}_{-},\mu^{v}_{-} \}<\nu_3<0<\nu_2
<\min\{\mu^{u}_{+},\mu^{v}_{+} \}\leq \max\{\mu^{u}_{+},\mu^{v}_{+} \}<\nu_1.
\eeaa

Set $P=U'$ and $Q=V'$.
Then from \eqref{system}, we have
\bea\label{X-sys}
\qquad U'= P, \ P'=-c P-U(1-U-a V),\ V'= Q,
\ Q'=-\frac{c}{d}Q-\frac{r}{d}V(1-V-bU).
\eea

Linearizing \eqref{X-sys} at $(U,P,V,Q)=(u^*,0,v^*,0)$ yields that ${\bf Y}'=J {\bf Y}$, where
${\bf Y}=(Y_1,Y_2,Y_3,Y_4)^T$ and
\beaa
J:=
\left(\begin{array}{cccc} 0 & 1 & 0 & 0\\
                       u^* & -c & au^* & 0\\
                       0 & 0 & 0 & 1 \\
                       -\frac{rb}{d}v^* & 0 & \frac{r}{d}v^* & -\frac{c}{d}
                       \end{array}\right).
\eeaa
Using cofactor expansions, one has ${\rm det}(J-\lambda I)=\rho(\lambda)$,
where $\rho(\lambda)$ is defined in \eqref{chara-poly}. Hence,
$J$ has four distinct real eigenvalues $\nu_4<\nu_3<0<\nu_2<\nu_1$.
By straightforward calculations, for each eigenvalue $\nu_i$, the corresponding eigenvector ${\bf w}_i$ is given by
\beaa
{\bf w}_i:=\Big(1,\nu_i,\frac{g_u(\nu_i)}{au^*},\nu_i \frac{g_u(\nu_i)}{au^*}\Big)^T,\quad i=1,2,3,4.
\eeaa
Therefore, the general solution of ${\bf Y}'=J {\bf Y}$ with ${\bf Y}(-\infty)={\bf 0}$ is given by
${\bf Y}(\xi)=\sum_{i=1}^{2}K_i  e^{\nu_i\xi}{\bf w}_i$ for some constants $K_i\in\mathbb{R}$, $i=1,2$.
By standard ODE theory, as $\xi\to-\infty$,
\bea\label{expansion}
\left(\begin{array}{c} U(\xi)\\ U'(\xi) \\ V(\xi) \\V'(\xi) \end{array}\right)=
\left(\begin{array}{c} u^*+K_1  e^{\nu_1\xi}+K_2  e^{\nu_2\xi} \\
                            K_1  \nu_1 e^{\nu_1\xi}+ K_2 \nu_2 e^{\nu_2\xi}\\
                         v^*+K_1 \frac{g_u(\nu_1)}{au^*}   e^{\nu_1\xi}+K_2 \frac{g_u(\nu_2)}{au^*}   e^{\nu_2\xi}\\
                              K_1 \nu_1\frac{g_u(\nu_1)}{au^*} e^{\nu_1\xi}+K_2 \nu_2\frac{g_u(\nu_2)}{au^*}   e^{\nu_2\xi}
                              \end{array}\right)+h.o.t.
\eea
Clearly, $K_1^2+K_2^2\neq0$. If $K_2=0$, then
$K_1\neq0$  and it follows from
\eqref{expansion} that
$$U'(\xi)\sim K_1  \nu_1 e^{\nu_1\xi}\ \ \text{and}\ \ V'(\xi)\sim K_1 \nu_1\frac{g_1(\nu_1)}{au^*} e^{\nu_1\xi}\ \ \text{as}\ \ \xi\to-\infty.$$
Since $\nu_1>\max\{\mu^u_+, \mu^v_+\}$, we see that $g_1(\nu_1)>0$.
This implies that $U'$ and $V'$ have the same sign as $\xi\to -\infty$,
which is impossible since
$U'<0$ and $V'>0$ in $\mathbb{R}$. Therefore, we obtain $K_2\neq0$. Moreover,
we have $K_2<0$ due to the monotonicity of $U$ and $V$.
The proof is thus complete by taking $\nu=\nu_2$, $l_{13}=-K_2$ and $l_{14}=K_2 g_u(\nu_2)/au^*$.
\end{proof}

For the strong-weak competition case ($b>1$) (resp.,
the weak competition case ($b<1$)), Lemma \ref{lem:AS-infty:b>1} and Lemma~\ref{lem:AS-infty:b<1} show that $(U,V)(\xi)$ converges to $(1,0)$ (resp., $(u^*,v^*)$) exponentially as $\xi\to-\infty$. 
However, in the critical case ($b=1$), the convergence rates may be of polynomial orders due to the degeneracy of the principal eigenvalue.

We now apply the center manifold theory to establish the decay rate of $U$ and $V$ at $\xi=-\infty$ when $b=1$. Let $W(\xi)=1-U(\xi)$. Then by simple calculations, $(W,V)$ satisfies
\bea\label{WVsys-b=1}
\begin{cases}
W''+cW'-(1-W)(W-aV)=0,&\quad \xi\in\mathbb{R},\\
dV''+cV'+rV(W-V)=0&\quad \xi\in\mathbb{R},\\
(W, V)(-\infty)=(0,0), \quad(W,V)(+\infty)=(1,1)
\end{cases}
\eea
To reduce \eqref{WVsys-b=1} to first-order ODEs, we introduce
\beaa
X_1(\xi)=V(\xi),\quad X_2(\xi)=V'(\xi), \quad X_3(\xi)=W(\xi),\quad X_4(\xi)=W'(\xi).
\eeaa
Then $X:=(X_1,X_2,X_3,X_4)(\xi)$ satisfies
$X'=G(X)$, which is described as
\bea\label{X1X2X3X4sys}
\begin{cases}
\dps X_1'=X_2,&\quad \xi\in\mathbb{R},\\
\dps X_2'=-\frac{c}{d}X_2-\frac{r}{d}X_1(X_3-X_1),&\quad \xi\in\mathbb{R},\\
\dps X_3'=X_4,&\quad \xi\in\mathbb{R},\\
\dps X_4'=-cX_4+(1-X_3)(X_3-aX_1),&\quad \xi\in\mathbb{R},
\end{cases}
\eea
By linearizing \eqref{X1X2X3X4sys} at $(0,0,0,0)$, we obtain
${\bf Y}'=J {\bf Y}$, where
${\bf Y}=(Y_1,Y_2,Y_3,Y_4)^T$ and
\beaa
J:=
\left(\begin{array}{cccc} 0 & 1 & 0 & 0\\
0 & -\frac{c}{d} & 0 & 0\\
0 & 0 & 0 & 1 \\
-a & 0 & 1 & -c
\end{array}\right).
\eeaa
It is easy to calculate that $J$ has four eigenvalues
\beaa
\mu_1=0,\quad \mu_2=-\frac{c}{d}, \quad \mu_3:=\frac{-c-\sqrt{c^2+4}}{2}<0,\quad \mu_4:=\frac{-c+\sqrt{c^2+4}}{2}>0,
\eeaa
and the corresponding eigenvector $v_i$ with respect to $\mu_i$ is given by
\beaa
v_1=(1,0,a,0)^T,\ v_2=\Big(\omega, -\frac{c}{d}\omega,-ad,ac\Big)^T,\
v_3=(0,0,1,\mu_3)^T,\ v_4=(0,0,1,\mu_4)^T,
\eeaa
where
\bea\label{omega}
\omega:=-d-c^2+\frac{c^2}{d}.
\eea
To reduce \eqref{X1X2X3X4sys} into the normal form, we set $Z=Q^{-1}X$, where $Z:=(Z_1,Z_2,Z_3,Z_4)^T$ and
$Q:=(v_1\ v_2\ v_3\ v_4)\in \mathbb{R}^{4\times 4}$.
Through some tedious computations, we have
\bea\label{X-Z}
\begin{cases}
X_1=Z_1+\omega Z_2,&\quad X_2=-\frac{c}{d}\omega Z_2,\\
X_3=aZ_1-adZ_2+Z_3+Z_4,&\quad X_4=acZ_2+\mu_3Z_3+\mu_4 Z_4,
\end{cases}
\eea
and
\bea\label{Q-1}
Q^{-1}:=
\left(\begin{array}{cccc}
1 & \frac{d}{c} & 0 & 0\\
0 & -\frac{d}{c\omega} & 0 & 0\\
\frac{a\mu_4}{\mu_3-\mu_4} & \frac{da(c+\omega\mu_4 +d\mu_4)}{\omega c (\mu_3-\mu_4)} & -\frac{\mu_4}{\mu_3-\mu_4} & \frac{1}{\mu_3-\mu_4} \\
-\frac{a\mu_3}{\mu_3-\mu_4} &
-\frac{da(c+\omega\mu_3 +d\mu_3)}{\omega c (\mu_3-\mu_4)} & \frac{\mu_3}{\mu_3-\mu_4} & -\frac{1}{\mu_3-\mu_4}
\end{array}\right),
\eea
where $\omega$ is defined in \eqref{omega}.
By \eqref{X1X2X3X4sys}, \eqref{X-Z} and \eqref{Q-1}, some tedious computations yield that
\bea\label{Zsys}
\begin{cases}
Z_1'= g_1(Z),\\
Z_2'=-\frac{c}{d}Z_2+g_2(Z),\\
Z_3'=\mu_3Z_3+g_3(Z),\\
Z_4'=\mu_4Z_4+g_4(Z),
\end{cases}
\eea
where
\beaa
&&g_1(Z):=-\frac{r}{c}(Z_1+\omega Z_2)h_1(Z),\quad g_2(Z):=\frac{r}{\omega c}(Z_1+\omega Z_2)h_1(Z),\\
&&g_3(Z):=-q_{32}\frac{r}{d}(Z_1+\omega Z_2)h_1(Z)+q_{34}h_2(z)h_3(z),\\
&&g_4(Z):=-q_{42}\frac{r}{d}(Z_1+\omega Z_2)h_1(Z)+q_{44}h_2(z)h_3(z),\\
&&h_1(Z):=(a-1)Z_1-(\omega+ad)Z_2+Z_3+Z_4,\\
&&h_2(Z):=aZ_1-adZ_2+Z_3+Z_4,\\
&&h_3(Z):=a(\omega+d)Z_2-Z_3-Z_4.\\
\eeaa
Here $q_{ij}$ is defined as the $i,j$ entry of the matrix $Q^{-1}$.
Note from the definition of $g_i$ and $h_i$, 
we see that $g_i$ does have no linear term of $Z_i$ for $i=1,2,3,4$, and thus
\beaa
g_i({\bf0})=0,\quad Dg_i({\bf0})={\bf0},\quad i=1,2,3,4.
\eeaa
Therefore, we can apply the center manifold theory (see \cite[Chapter 18]{Wiggins}) to conclude that there exists a one-dimensional center manifold for \eqref{Zsys}, and $Z_i$, $i=2,3,4$ can be represented by a smooth function $Z_i=H_i(Z_1)$, $i=2,3,4$, for small $Z_1$.
We assume that
\beaa
H_i(Z_1)=C_iZ_i^2+o(|Z_1|^2),\quad i=2,3,4,
\eeaa
for some $C_i\in\mathbb{R}$. Indeed, $C_i$ is determined such that
\bea
&&H_2'(Z_1)g_1(Z)-\Big[-\frac{c}{d}Z_2+g_2(Z)\Big]=o(|Z_1|^2),\label{H2-eq}\\
&&H_3'(Z_1)g_1(Z)-(\mu_3Z_3+g_3(Z))=o(|Z_1|^2),\label{H3-eq}\\
&&H_4'(Z_1)g_1(Z)-(\mu_4Z_4+g_4(Z))=o(|Z_1|^2).\label{H4-eq}
\eea
By comparing the coefficients in front of $Z_1^2$ on the both sides of \eqref{H2-eq},
we need $C_2=-\frac{rd}{\omega c^2}(1-a)$.
Also, from \eqref{H3-eq} and \eqref{H4-eq},  with some tedious computations, we see that $C_3=C_4=0$.
Moreover, the flow on the center manifold is defined by
\beaa
Z_1'=g_1(Z_1,H_2(Z_1),H_3(Z_1),H_4(Z_1))=\frac{r}{c}(1-a)Z_1^2+o(|Z_1|^2),
\eeaa
for sufficiently small $Z_1(\xi)$,
which implies that
\beaa
Z_1(\xi)=\frac{c}{r(1-a)}|\xi|^{-1}+o(|\xi|^{-1})\quad \mbox{as $\xi\to-\infty$}.
\eeaa
Therefore, the center manifold theory yields that if $0<Z_1(\xi)\ll1$, we have
\beaa
Z_1(\xi)\sim \frac{c}{r(1-a)}|\xi|^{-1},\quad
Z_2(\xi)\sim-\frac{d}{r\omega(1-a)}|\xi|^{-2}\quad
\mbox{as $\xi\to-\infty$.}
\eeaa
Therefore,
in view of \eqref{X-Z} and the definition of $X_i$,
together with the fact that $0<U,V<1$ in $\mathbb{R}$,
we see that there exists  $l_{15}>0$ such that
\bea\label{AS-infty:b=1}
\lim_{\xi\to-\infty}\frac{V(\xi)}{|\xi|^{-1}}=l_{15},\quad \lim_{\xi\to-\infty}\frac{1-U(\xi)}{|\xi|^{-1}}=a l_{15},
\eea
Furthermore, it holds that
\bea\label{AS-infty2:b=1}
\lim_{\xi\to-\infty}\frac{1-U(\xi)}{V(\xi)}=a<1.
\eea
Combining \eqref{AS-infty:b=1} and \eqref{AS-infty2:b=1}, we have the following result.

\begin{lemma}\label{lem:AS-infty:b=1}
Assume that $0<a<1$ and $b=1$.
Let $(c,U,V)$ be a solution of the system (\ref{tw solution weak}).
Then there exist a positive constant $l_{15}$ such that
\beaa
\lim_{\xi\to-\infty}\frac{V(\xi)}{|\xi|^{-1}}=l_{15},\quad \lim_{\xi\to-\infty}\frac{1-U(\xi)}{|\xi|^{-1}}=a l_{15}, \quad \lim_{\xi\to-\infty}\frac{1-U(\xi)}{V(\xi)}=a<1.
\eeaa
\end{lemma}


Hence, we immediately obtain a Lemma as follows:

Thanks to Lemma~\ref{lem:AS-infty:b>1}, Lemma~\ref{lem:AS-infty:b<1} and Lemma~\ref{lem:AS-infty:b=1}, we immediately obtain

\begin{corollary}\label{lm: behavior around - infty b=1}
Assume that $0<a<1$ and $b>0$.
Let $(U,V)$ be a solution of the system (\ref{tw solution weak}) with speed $c$. Then it holds that
\beaa
1-U(\xi)-aV(\xi)=o(|\xi|^{-1}).
\eeaa
In particular, for the case $b=1$,
there exists $\xi_0$ near $-\infty$ such that $(1-U-V)(\xi)< 0$ for all $\xi\in(-\infty,\xi_0]$.
\end{corollary}

\section{Threshold of the Lotka-Volterra competition system}\label{sec:threshold-system}

This section is devoted to the proof of Theorem~\ref{th:threshold}.
Let us fix the parameters $a\in(0,1)$, $d>0$, and $r>0$.
It is well known (cf. \cite[Lemma 5.6]{Kan-on1997}) that the minimal traveling wave speed $c_{LV}^{*}(b)$ is a continuous function on $(0,+\infty)$.
Moreover, by Theorem 1.1 of \cite{Wu Xiao Zhou} and a simple comparison argument, we see that $c_{LV}^{*}(b)$ is nondecreasing on $b$.
We first introduce a crucial proposition which implies $0<b^*<\infty$ is well-defined.


\subsection{The well-defined threshold}

\noindent

Let us start by briefly recalling the competitive comparison principle.
Consider a domain $
\Omega:=(t_1,t_2)\times(x_1,x_2)$ with  $0\le t_1<t_2\le + \infty$ and $-\infty\le x_1<x_2\le+\infty$. A (classical) super-solution is a pair  $(\overline{u},\underline{v})\in \Big[C^1\Big((t_1,t_2),C^2((x_1,x_2))\Big)\cap C_b\left(\overline \Omega\right)\Big]^2$
satisfying
\begin{equation*}
\overline u_t-\overline u_{xx}-\overline u(1-\overline u-a\underline v)\ge 0\quad \text{ and } \quad \overline v_t-d\overline v_{xx}-r\overline v(1-\overline v-b\overline u)\le 0\ \ \text{in}\ \ \Omega.
\end{equation*}
Similarly, a (classical) sub-solution $(\underline u, \overline v)$ requires
\begin{equation*}
\overline u_t-\overline u_{xx}-\overline u(1-\overline u-a\overline v)\le 0\quad \text{ and } \quad \overline v_t-d\overline v_{xx}-r\overline v(1-\overline v-b\underline u)\ge 0\ \ \text{in}\ \ \Omega.
\end{equation*}

\begin{proposition}[Comparison Principle]
Let $(\overline{u},\underline{v})$ and $(\underline{u},\overline{v})$ be a super-solution and sub-solution of system (\ref{system}) in $\Omega$, respectively. If
$$
\overline{u}(t_1,x)\ge \underline{u}(t_1,x) \quad \text{and} \quad \underline{v}(t_1,x)\le\overline{v}(t_1,x),\quad\text{for all } x\in (x_1,x_2),
$$
and, for $i=1,2$,
$$
\overline{u}(t,x_i)\ge \underline{u}(t,x_i) \quad \text{and} \quad \underline{v}(t,x_i)\le\overline{v}(t,x_i),\quad\text{for all } t\in(t_1,t_2),
$$
then, it holds
$$
\overline{u}(t,x) \ge\underline{u}(t,x) \quad \text{ and } \quad \underline{v}(t,x)\le\overline{v}(t,x),\quad  \text{for all } (t,x)\in\Omega.
$$
If $x_1=-\infty$ or $x_2=+\infty$, the hypothesis on the corresponding boundary condition can be omitted.
\end{proposition}

We refer to the clear exposition of  {\it generalized} sub- and super-solutions in \cite[\S 2.1]{Girardin Lam} for more details. In particular,  if $(\underline {u}_1,\overline{v})$  and $(\underline {u}_2, \overline{v})$ are both classical sub-solutions, then $(\max(\underline {u}_1,\underline{u}_2),\overline v)$ is a generalized sub-solution. Also,
if $(\underline u,\overline{v}_1)$    and $(\underline u, \overline{v}_2)$ are both classical sub-solutions, then $(\underline u,\min(\overline{v}_1,\overline{v}_2))$ is a generalized sub-solution.

\begin{proposition}\label{prop:0<b*<infty}
For any fixed $a\in(0,1)$, $d>0$, and $r>0$, there exists $b_1>0$ very small such that $c_{LV}^{*}(b)=2\sqrt{1-a}$ for all $0\le b\le b_1$. On the other hand, there exists $b_2>0$ sufficiently large such that $c_{LV}^{*}(b)>2\sqrt{1-a}$ for all $b>b_2$.
\end{proposition}
\begin{proof}
We first show $b_2<\infty$ by applying the continuity argument. 
To do this, we assume by contradiction that $b_2=\infty$.
Due to the monotonicity of $c_{LV}^{*}(b)$, we have $c_{LV}^{*}(b)=2\sqrt{1-a}$ for all $b>0$.
To reach a contradiction, we take a sequence $b_n\uparrow \infty$ and write $(U_n,V_n)$ as the solution of \eqref{tw solution weak} with
$$c=c_{LV}^{*}(b_n)=2\sqrt{1-a}\ \ \text{and}\ \ b=b_n.$$
By a translation, we may assume that $U_n(0)=1/2$ for all $n$.
Since $0\leq U_n,V_n\leq 1$ in $\mathbb{R}$, by standard elliptic estimates, we have
$\|U_n\|_{C^{2+\alpha}(\mathbb{R})}\leq C$ for some $C>0$ independent of $n$.

We now fix $R>0$. Then there exists $\varepsilon>0$ such that
\bea\label{U-lower-bdd}
U_n(\xi)\geq \varepsilon\quad \mbox{for all $\xi\in[-R,R]$ and $n\in\mathbb{N}$}.
\eea
Next, we define an auxiliary function
\beaa
\overline{V}_n(\xi)=\frac{e^{-\lambda_n(\xi+2R)}+e^{\lambda_n(\xi-2R)}}{1+e^{-4\lambda_nR}},\quad \xi\in[-2R,2R],
\eeaa
where
\beaa
\lambda_n:=\frac{-c+\sqrt{c^2+4dr(\varepsilon b_n-1)}}{2d}\to\infty\quad\mbox{as $n\to\infty$ and $c=2\sqrt{1-a}$.}
\eeaa
Clearly, $\overline{V}_n(\pm 2R)=1$, $0\leq \overline{V}_n(\xi)\leq 1$ for all $\xi\in[-2R,2R]$ and $n\in\mathbb{N}$, and $\overline{V}_n\to0$ uniformly in $[-R,R]$ as $n\to\infty$.
Furthermore, by direct computation, for all large $n$ we have
\beaa
c\overline{V}'_n+d\overline{V}''_n+r\overline{V}_n(1-\overline{V}_n)-rb_n\varepsilon \overline{V}_n\leq 0,\quad \xi\in[-2R,2R].
\eeaa
Together with \eqref{U-lower-bdd}, one can apply the comparison principle to conclude that $V_n\leq \overline{V}_n$ in $[-2R,2R]$ for all large $n$. In particular, we have
\bea\label{Vn-to-0}
\sup_{\xi\in[-R,R]}|V_n(\xi)|\to0\quad \mbox{as $n\to\infty$}.
\eea
Thanks to \eqref{Vn-to-0} and the $C^{2+\alpha}$ bound of $U_n$, up to subtract a subsequence, we may assume that
$U_n\to U_{R}$ uniformly in $[-R,R]$ as $n\to\infty$,
where $U_{R}$ is defined in $[-R,R]$ and satisfies
$U_{R}(0)=1/2$, $U'_{R}\leq0$ in $[-R,R]$ and
\beaa
cU_R'+U_R''+U_R(1-U_R)=0, \quad \xi\in[-R,R].
\eeaa
Next, by standard elliptic estimates and taking $R\to\infty$, up to subtract a subsequence, we may assume that $U_R\to U_{\infty}$ locally uniformly in $\mathbb{R}$ as $n\to\infty$, where $U_{\infty}$ satisfies
\beaa
cU_{\infty}'+U_{\infty}''+U_{\infty}(1-U_{\infty})=0, \quad \xi\in\mathbb{R},\quad U_{\infty}(0)=1/2,\quad  U'_{\infty}\leq 0.
\eeaa
It is not hard to see that $U_{\infty}(-\infty)=1$ and $U_{\infty}(+\infty)=0$. Therefore, $U_{\infty}$ forms a traveling front with speed $c=2\sqrt{1-a}$,
which is impossible since such solutions exist only for $c\geq 2$ (see \cite{KPP}). This contradiction shows that $b_2<\infty$.

\medskip
Next, we prove $b_1>0$. To do this, we assume by contradiction that $b_1=0$ and let $W_*(\xi)$ be the minimal traveling wave satisfying
\beaa
\begin{cases}
W_*''+2\sqrt{1-a}W_*'+W_*(1-a-W_*)=0,\quad \xi\in\mathbb{R}\\
W_*(-\infty)=1,\quad W_*(+\infty)=0.
\end{cases}
\eeaa

We look for continuous functions $(R_u(\xi),R_v(\xi))$
defined in $\mathbb{R}$,
such that
\beaa
(W_u,W_v)(\xi):=\Big(\min\{(W_*-R_u)(\xi),1\},1+R_v(\xi)\Big)
\eeaa
forms a super-solution satisfying
\begin{equation}\label{super solution b=0}
\left\{
\begin{aligned}
&N_1[W_u,W_v]:=W_u''+cW_u'+W_u(1-W_u-aW_v)\le0,\ \mbox{a.e. in $\mathbb{R}$},\\
&N_2[W_u,W_v]:=dW_v''+cW_v'+rW_v(1-W_v-\delta_0W_u)\ge 0,\ \mbox{a.e. in $\mathbb{R}$},
\end{aligned}
\right.
\end{equation}
for $c=2\sqrt{1-a}$ and some sufficiently small $\delta_0>0$. By some straightforward computations, we have
\bea\label{N3 b=0}
N_1[W_u,W_v]=-R_u''-cR_u'-R_u(1-a-2W_*+R_u-aR_v)-aW_*R_v,
\eea
and
\bea\label{N4 b=0}
N_2[W_u,W_v]=dR_v''+cR_v'-rR_v(1+R_v)-\delta_0r(1+R_v)(W_*-R_u).
\eea
\begin{figure}
\begin{center}
\begin{tikzpicture}[scale = 1.1]
\draw[thick](-6,0) -- (6,0) node[right] {$\xi$};
\draw [semithick] (-6,-0.5)--  (1,-0.5) to [out=10,in=180] (6,-0.1);
\draw [ultra thick] (-6, -1) -- (-4,-1) to [out=30, in=260] (-3.5,0)  to [out= 70, in=180] (-3,0.6) to [out=20, in=220] (1,2.5) to [out=70,in=180] (2,3)to [out=0,in=170] (6,0.2);
\node[below] at (1,1) {$\xi_{1}+\delta_1$};
\draw[dashed] [thick] (-3.5,0)-- (-3.5,0.4);
\node[below] at (-3.5,0.9) {$\xi_2$};
\draw[dashed] [thick] (-3,0.)-- (-3,0.6);
\node[above] at (-2.8,-0.5) {$\xi_2+\delta_2$};
\draw[dashed] [thick] (-4,-0.5)-- (-4,1);
\node[above] at (-4,0.8) {$\xi_2-\delta_4$};
\draw[dashed] [thick] (1,0)-- (1,0.55);
\draw[dashed] [thick] (1,0.9)-- (1,2.5);
\draw [thin] (1.2,-0.46) arc [radius=0.2, start angle=20, end angle= 170];
\node[above] at (1,-0.42) {$\alpha_2$};
\draw [thin] (-3.82,-0.83) arc [radius=0.2, start angle=50, end angle= 185];
\node[above] at (-4,-0.87) {$\alpha_4$};
\draw [thin] (-2.8,0.655) arc [radius=0.2, start angle=30, end angle= 175];
\node[above] at (-3,0.7) {$\alpha_3$};
\draw [thin] (1.1,2.7) arc [radius=0.2, start angle=70, end angle= 220];
\node[above] at (0.6,2.5) {$\alpha_1$};
\node[above] at (2,1.8) {\Large{$R_u$}};
\node[above] at (2,-1) {\Large{$R_v$}};
\end{tikzpicture}
\caption{$(R_u,R_v)$ to prove Proposition \ref{prop:0<b*<infty}.}\label{Figure prop}
\end{center}
\end{figure}

We consider $(R_u,R_v)(\xi)$ defined  as (see Figure \ref{Figure prop})
\begin{equation*}
(R_u,R_v)(\xi):=\begin{cases}
(\varepsilon_1\sigma(\xi) e^{-\lambda_u\xi},-\eta_1 e^{-\lambda_1\xi}),&\ \ \mbox{for}\ \ \xi\ge\xi_1+\delta_1,\\
(\varepsilon_2e^{\lambda_2\xi},-\delta_v),&\ \ \mbox{for}\ \ \xi_2+\delta_2\le\xi\le\xi_1+\delta_1,\\
(\varepsilon_3\sin(\delta_3(\xi-\xi_2)),-\delta_v),&\ \ \mbox{for}\ \ \xi_2-\delta_4\le\xi\le\xi_{2}+\delta_2,\\
(-\delta_u,-\delta_v),&\ \ \mbox{for}\ \ \xi\le\xi_2-\delta_4,
\end{cases}
\end{equation*}
where $\lambda_u:=\sqrt{1-a}$, and $\xi_1>M$ and $\xi_2<-M$ are fixed points. Since $|R_u|,|R_v|\ll 1$, up to enlarging $M$, for all $\xi\in(-\infty,\xi_2]$, it holds
\bea\label{xi 2 b=0}
1-2W_*-a<-1+a+\rho,
\eea
with arbitrarily small $\rho>0$. We also set $0<\lambda_{1}<\lambda_u$ satisfies
\begin{equation}\label{condition on lambda 1 b=0}
d\lambda_1^2-2\sqrt{1-a}\,\lambda_1-r=:-C_0<0,
\end{equation}
and $\lambda_{2}$ very large satisfies
\begin{equation}\label{condition on lambda 2 b=0}
\lambda_2^2+2\sqrt{1-a}\,\lambda_2-3=:C_1>0.
\end{equation}
Here
$\varepsilon_{1,2,3}>0$ and $\eta_1>0$ make $(R_u,R_v)$ continuous on $\mathbb{R}$, while $\delta_{1,\cdots,4}>0$ will be determined later such that $(W_u,W_v)$ satisfies \eqref{super solution b=0}.  Moreover, we set
\bea\label{delta u v b=0}
\delta_u=\varepsilon_3\sin(\delta_3\delta_4)\quad\text{and}\quad \delta_v=\eta_1e^{-\lambda_1(\xi_1+\delta_1)},
\eea
which yield $(R_u,R_v)(\xi)$ are continuous on $\mathbb{R}$.

Next, we will divide the construction into several steps.

\noindent{\bf{Step 1}:} We consider $\xi\in[\xi_{1}+\delta_1,\infty)$ with $\xi_1>M$ and some small $\delta_1$ satisfying
\bea\label{cond delta 1 b=0}
0<\delta_1<\frac{1}{2(\lambda_2+\lambda_u)}.
\eea
In Step 1, we aim to verify that $(W_u,W_v)(\xi)=(U_*-R_u,1+R_v)(\xi)$, with
\beaa
(R_u,R_v)(\xi)=(\varepsilon_1\sigma(\xi)e^{-\lambda_u\xi},-\eta_1 e^{-\lambda_1\xi}),
\eeaa
satisfies \eqref{super solution b=0} by setting  $\delta_0$ sufficiently small.

Similar as the construction of $R_w(\xi)$ for scalar equation problem, we define
$$\sigma(\xi):=\frac{4}{\lambda_1^2}e^{-\frac{\lambda_1}{2}(\xi-\xi_1)}-\frac{4}{\lambda_1^2}+\frac{4}{\lambda_1}\xi-\frac{4}{\lambda_1}\xi_1$$
which satisfies
$$\sigma(\xi_1)=0,\ \sigma'(\xi)=\frac{4}{\lambda_1}-\frac{2}{\lambda_1}e^{-\frac{\lambda_1}{2}(\xi-\xi_1)},\ \sigma''(\xi)=e^{-\frac{\lambda_1}{2}(\xi-\xi_1)},$$
and $\sigma(\xi)= O(\xi)$ as $\xi\to\infty$.
From \eqref{N3 b=0}, we have
\beaa
N_1[W_u,W_v]\le -e^{-\frac{\lambda_1}{2}(\xi-\xi_1)}R_u+R_u(2W_*-R_u+aR_v)-aW_*R_v.
\eeaa
Since $W_*(\xi)= O(\xi e^{-\lambda_u\xi})$ as $\xi\to\infty$ and  $0<\lambda_1<\lambda_u$, we obtain $N_1[W_u,W_v]\le 0$ for all $\xi\in[\xi_1+\delta_1,\infty)$ up to enlarging $M$ if necessary.

Next, we deal with the inequality of $N_2[W_u, W_v]$.
From \eqref{N4 b=0} and \eqref{condition on lambda 1 b=0}, we have
\begin{equation*}
N_2[W_u,W_v]\ge -C_0R_v-rR_v^2-\delta_0r(1+R_v)W_*.
\end{equation*}
Since $0<\lambda_1<\lambda_u$ and $R_v<0$, by setting $\delta_0\ll \eta_1$ sufficiently small, then we have $N_2[W_u,W_v]\geq0$ for all $\xi\in[\xi_1+\delta_1,\infty)$.

\medskip

\noindent{\bf{Step 2}:} We consider $\xi\in[\xi_{2}+\delta_2,\xi_1+\delta_1]$ with $\xi_1+\delta_1$ fixed by Step 1.
In this case, we have
$(R_u,R_v)(\xi)=(\varepsilon_2e^{\lambda_2\xi},-\delta_v)$
with $\lambda_2$ satisfying \eqref{condition on lambda 2 b=0} and $\delta_v$ defined as \eqref{delta u v b=0}. It is easy to see that $R_v(\xi)$ is continuous at $\xi=\xi_1+\delta_1$, and $\angle \alpha_2<180^{\circ}$ since
$$R'_v((\xi_1+\delta_1)^+)>0=R'_v((\xi_1+\delta_1)^-).$$

On the other hand, we set
$$\varepsilon_2=\varepsilon_2(\varepsilon_1,\lambda_2)=\frac{\varepsilon_1\sigma(\xi_1+\delta_1)e^{-\lambda_u(\xi_1+\delta_1)}}{e^{\lambda_2(\xi_1+\delta_1)}}$$
such that $R_u(\xi)$ is continuous at $\xi=\xi_1+\delta_1$. Then, by some straightforward computations, we have
\beaa
&&R_u'((\xi_1+\delta_1)^+)=\varepsilon_1\sigma'(\xi_1+\delta_1)e^{-\lambda_u(\xi_1+\delta_1)}-\varepsilon_1\lambda_u\sigma(\xi_1+\delta_1)e^{-\lambda_u(\xi_1+\delta_1)},\\
&&R_u'((\xi_1+\delta_1)^-)=\lambda_2R_u(\xi_1+\delta_1).
\eeaa
Thus, $R_u'((\xi_1+\delta_1)^+)>R_u'((\xi_1+\delta_1)^-)$ 
is
equivalent to
$$(\lambda_2+\lambda_u)\sigma(\xi_1+\delta_1)<\sigma'(\xi_1+\delta_1),$$
which holds 
since
\eqref{cond delta 1 b=0}. Hereafter, $\delta_1$ is fixed.

From \eqref{N3 b=0}, \eqref{condition on lambda 2 b=0}, and $R_v<0$, we have
\beaa
N_1[W_u,W_v]\le-C_1R_u+aW_*\delta_v.
\eeaa
Notice that, we can set $\eta_1\ll \varepsilon_1$ such that $\delta_v\ll|R_u|$ for all $\xi\in[\xi_2+\delta_2,\xi_1+\delta_1]$. Therefore, we have $N_1[W_u,W_v]\le 0$ for all $\xi\in[\xi_2+\delta_2,\xi_1+\delta_1]$.
On the other hand, from \eqref{N4 b=0} and $R_v<0$, we have
\beaa
N_2[W_u,W_v]=r\delta_v(1-\delta_v)-\delta_0r(1-\delta_v)(W_*-R_u).
\eeaa
Therefore, up to reducing $\delta_0\ll \eta_1$ if necessary, we have $N_2[W_u,W_v]\ge 0$ for all
$\xi\in[\xi_2+\delta_2,\xi_1+\delta_1]$. Moreover, it is easy to see that $N_2[W_u,W_v]\ge 0$ for all $\xi\in(-\infty,\xi_1+\delta_1]$ as long as $\delta_0\ll \eta_1$ is sufficiently small since $W_*-R_u\leq 1$ in $\mathbb{R}$. Therefore, hereafter it suffices to verify the inequality of $N_1[W_u,W_v]$.

\medskip

\noindent{\bf{Step 3}:} We consider $\xi\in[\xi_{2}-\delta_4,\xi_2+\delta_2]$ with $\xi_2+\delta_2$ fixed by Step 2 and
\bea\label{cond delta 2 b=0}
\delta_2>\frac{1}{\lambda_2}.
\eea
In this case, we have $(R_u,R_v)=(\varepsilon_3\sin(\delta_3(\xi-\xi_2)),-\delta_v)$. We first set
$$\varepsilon_3=\varepsilon_3(\varepsilon_1,\delta_2,\delta_3,\lambda_2)=\frac{\varepsilon_2e^{\lambda_2(\xi_2+\delta_2)}}{\sin(\delta_2\delta_3)}=\frac{\varepsilon_1\sigma(\xi_1+\delta_1)e^{\lambda_2(\xi_2+\delta_2)-\lambda_u(\xi_1+\delta_1)}}{\sin(\delta_2\delta_3)e^{\lambda_2(\xi_1+\delta_1)}}$$
such that $R_u(\xi)$ is continuous at $\xi=\xi_2+\delta_2$.
Then, by some straightforward computations, we have
$$R_u'((\xi_2+\delta_2)^+)=\lambda_2R_u(\xi_2+\delta_2)\quad\text{and}\quad R_u'((\xi_2+\delta_2)^-)=\varepsilon_3\delta_3\cos(\delta_2\delta_3).$$
Thus, from $\frac{x\cos x}{\sin x}\to 1$ as $x\to 0$,
$$R_u'((\xi_2+\delta_2)^+)>R_u'((\xi_2+\delta_2)^-)\ \ \text{and}\ \ \angle \alpha_3<180^{\circ},$$
follows by taking
$\delta_3$ sufficiently small and $\delta_2$ satisfying \eqref{cond delta 2 b=0}.

It suffices to only verify the inequality of $N_1[W_u,W_v]$. From \eqref{N3 b=0}, we have
\beaa
N_1[W_u,W_v]=\delta_3^2R_u-c^*\delta_3\varepsilon_3\cos(\delta_3(\xi-\xi_2))-R_u(1-a-2W_*+R_u-aR_v)-aW_*R_v.
\eeaa
For $\xi\in[\xi_2,\xi_2+\delta_2]$, we have
$$N_1[W_u,W_v]\le (\delta_3^2+1+2a)\varepsilon_3\sin(\delta_2\delta_3)-c^*\delta_3\varepsilon_3\cos(\delta_2\delta_3).$$
Note that, from $\frac{x\cos x}{\sin x}\to 1$ as $x\to 0$,
$$(\delta_3^2+1+2a)\varepsilon_3\sin(\delta_2\delta_3)-c^*\delta_3\varepsilon_3\cos(\delta_2\delta_3)\le 0$$
is equivalent to $\delta_2<\frac{\delta_3c^*}{\delta_3^2+1+2a}$ which holds since $\lambda_2$ in \eqref{cond delta 2 b=0} can be chosen arbitrarily large.
For $\xi\in[\xi_2-\delta_4,\xi_2]$, from $R_u\le 0$ and \eqref{xi 2 b=0}, up to enlarging $M_0$, we have
$$N_1[W_u,W_v]\le -c^*\delta_3\varepsilon_3\cos(\delta_2\delta_3)-aW_*R_v.$$
Then, by setting
\bea\label{delta 4 b=0}
0<\delta_4<\frac{1}{\lambda_2}<\delta_2<\frac{\delta_3c^*}{\delta_3^2+1+a},
\eea
we have $N_1[W_u,W_v]\le 0$ for all $\xi\in[\xi_2-\delta_4,\xi_2+\delta_2]$.

\medskip

\noindent{\bf{Step 4}:} We consider $\xi\in(-\infty,\xi_{2}-\delta_4]$ with $\xi_2-\delta_4$ fixed in Step 3. In this case, we have $(R_u,R_v)=(-\delta_u,-\delta_v)$.
From \eqref{delta u v b=0}, $R_u(\xi)$ is continuous at $\xi=\xi_2-\delta_4$. It is easy to see that
$$R_u'((\xi_2-\delta_4)^+)>0=R_u'((\xi_2-\delta_4)^-)\ \ \text{and}\ \ \angle \alpha_4<180^{\circ}.$$
 Moreover, from $\delta_v\ll R_u(\xi_2+\delta_2)$, we assert that $\delta_v\ll\delta_u$ up to reducing $\eta_1/\varepsilon_1$ if necessary.

From \eqref{N3 b=0} and \eqref{xi 2 b=0}, we have
\beaa
N_1[W_u,W_v]=\delta_u(1-a-2W_*-\delta_u+a\delta_v)+aW_*\delta_v\le 0
\eeaa
since $\delta_v\ll\delta_u$. The construction of $(R_u,R_v)(\xi)$ is complete.

\medskip

We are ready to complete the proof of Proposition~\ref{prop:0<b*<infty}.
From Step 1 to Step 4, we are equipped with a super-solution $(W_u,W_v)(\xi)$.
Next, we consider the Cauchy problem
\beaa
\left\{
\begin{aligned}
&u_t=u_{xx}+u(1-u-av),\\
&v_t=dv_{xx}+rv(1-v-\delta_0u),
\end{aligned}
\right.
\eeaa
with initial datum given by \eqref{initial datum}. By setting $x_0$ very large, the function $(\overline u,\underline v)(t,x):=(W_u,W_v)(x-2\sqrt{1-a}t-x_0)$ is a super-solution, propagating with the speed $2\sqrt{1-a}$. However, this contradicts the assumption $b_1=0$, which implies the actual propagation speed must be strictly greater than $2\sqrt{1-a}$ for all $b>0$. 
Hence, the case $b_1=0$ is impossible. This completes the proof of Proposition \ref{prop:0<b*<infty}.
\end{proof}

Together with Proposition \ref{prop:0<b*<infty},
we immediately obtain Lemma~\ref{lm: existence b*}.
\begin{lemma}\label{lm: existence b*}
For any $d>0$, $r>0$ and $a\in(0,1)$, there exists $0<b^*<\infty$  such that
\begin{equation*}
c_{LV}^{*}(b)=2\sqrt{1-a}\ \ \mbox{for}\ \ b\in(0,b^*]\ \  \mbox{and}\ \  c_{LV}^{*}(b)>2\sqrt{1-a} \ \ \mbox{for}\ \  b\in(b^*,+\infty).
\end{equation*}
\end{lemma}

\subsection{Construction of the super-solution}

\medskip

Now, we are ready to state the most important part of our argument. Let $(c_{LV}^{*}, U_*,V_*)$ be the minimal traveling wave of system \eqref{tw solution weak} with $b=b^*>0$
and $c_{LV}^{*}=c_{LV}^{*}(b^*)=2\sqrt{1-a}$.
Hereafter, for simplicity we denote
\beaa
\lambda_u:=\lambda_u^-(c_{LV}^{*}(b^*))>0,\ \lambda_v:=\lambda_v^-(c_{LV}^{*}(b^*))>0,\
\eeaa
 where
$\lambda_u^-$ and $\lambda_v^-$ are defined in the \S~\ref{subsec:Asymp behavior}.

The first and most involved step is to show $(i)\Rightarrow (ii)$, {\it i.e.}, if $b=b^*$, then $A=0$ in \eqref{1}. We shall use a contradiction argument to establish the following result.
\begin{proposition}\label{Prop-supersol-system}
Assume that {\bf(H)} holds. In addition, if
\bea\label{AS-U-infty-for-contradiction}
\lim_{\xi\rightarrow+\infty}\frac{U_{*}(\xi)}{\xi e^{-\lambda_u\xi}}=A_0\quad \mbox{for some $A_0>0$,}
\eea
then there exist two continuous functions $R_u(\xi)$ and $R_v(\xi)$
defined in $\mathbb{R}$ with
\bea\label{decay rate of Ru Rv}
R_u(\xi)=O(\xi e^{-\lambda_u\xi}) \quad \mbox{as $\xi\to\infty$},
\eea
such that
$$(W_u,W_v)(\xi):=\Big(\min\{(U_*-R_u)(\xi),1\},\max\{(V_*+R_v)(\xi),0\}\Big)$$
is a super-solution satisfying
\begin{equation}\label{tw super solution system}
\left\{
\begin{aligned}
&N_3[W_u,W_v]:=W_u''+2\sqrt{1-a}W_u'+W_u(1-W_u-aW_v)\le0,\ \mbox{a.e. in $\mathbb{R}$},\\
&N_4[W_u,W_v]:=dW_v''+2\sqrt{1-a}W_v'+rW_v(1-W_v-(b^*+\delta_0)W_u)\ge 0,\ \mbox{a.e. in $\mathbb{R}$},
\end{aligned}
\right.
\end{equation}
for some small $\delta_0>0$, where
$W'_u(\xi_0^{\pm})$
(resp. $W'_v(\xi_0^{\pm})$) exists and
\beaa
W'_u(\xi_0^+)\leq W'_u(\xi_0^-)\quad (resp.\, W'_v(\xi_0^+)\geq W'_v(\xi_0^-))
\eeaa
if  $W'_u$ (resp., $W'_v$) is not continuous at $\xi_0$.
\end{proposition}



In the following discussion, we divide the construction of
$(R_u,R_v)(\xi)$ into two subsections: $b^*\ge 1$ (the strong-weak competition case and the critical case); $0<b^*<1$ (the weak competition case).

\subsubsection{For the case $b^*\ge 1$}\label{sec:3-1}

\noindent

In this subsection, we always assume $b^*\ge 1$.
First, since
$(U_*,V_*)(-\infty)=(1,0)$ and
$(U_*,V_*)(+\infty)=(0,1)$, for any given small $\rho>0$, we can take $M_0>0$ sufficiently large such that
\bea\label{rho}
\begin{cases}
0< U_*(\xi)<\rho,\quad 1-\rho<V_*(\xi)<1&\quad\mbox{for all}\quad\xi\geq M_0,\\
0< V_*(\xi)<\rho,\quad 1-\rho<U_*(\xi)<1&\quad\mbox{for all}\quad\xi\leq -M_0.
\end{cases}
\eea

For $\xi$ being close to $\infty$, we have the following for later use. First,
due to \eqref{AS-U-infty-for-contradiction}, up to enlarging $M_0$ if necessary, we may assume that for some positive constant $A_0$,
\beaa
U_*(\xi) \leq 2A_0\xi e^{-\lambda_u\xi}\quad\mbox{for all}\quad\xi\geq M_0.
\eeaa
Moreover, due to and Lemma~\ref{lm: behavior around + infty}(ii), we may also assume there exists $C_0>0$ such that
\beaa
V_*(\xi) \geq 1-C_0\xi^2 e^{-\min\{\lambda_u,\lambda_v\}\xi}\quad
\mbox{for all}\quad\xi\geq M_0.
\eeaa

\medskip
We now define $(R_u,R_v)(\xi)$ as (see Figure~\ref{Figure:b>1})
\begin{equation*}
(R_u,R_v)(\xi):=\begin{cases}
(\varepsilon_1\sigma(\xi) e^{-\lambda_u\xi},\eta_1 (\xi-\xi_1)e^{-\lambda_u\xi}),&\ \ \mbox{for}\ \ \xi_1+\delta_1\le\xi,\\
(\varepsilon_2\sin(\delta_2(\xi-\xi_1+\delta_3)),\eta_2e^{\lambda_{1}\xi}),&\ \ \mbox{for}\ \ \xi_1-\delta_4\le\xi\le\xi_{1}+\delta_1,\\
(-\varepsilon_3,\eta_2e^{\lambda_{1}\xi}),&\ \ \mbox{for}\ \ \xi_2+\delta_5\le\xi\le\xi_1-\delta_4,\\
(-\varepsilon_4(-\xi)^{\theta}[1-U_*(\xi)],\eta_3\sin(\delta_6(\xi-\xi_2))),&\ \ \mbox{for}\ \ \xi_2-\delta_7\le\xi\le\xi_2+\delta_5,\\
(-\varepsilon_4(-\xi)^{\theta}[1-U_*(\xi)],-\eta_4(-\xi)^{\theta}V_*(\xi),&\ \ \mbox{for}\ \ \xi\le\xi_2-\delta_7,
\end{cases}
\end{equation*}
where
$0<\theta<1$, and $\lambda_{1}>0$ is very large such that
\begin{equation}\label{condition on lambda 1 system}
d\lambda_1^2+2\sqrt{1-a}\,\lambda_1-r(2+b^*)>0\ \ {\rm and}\ \ \lambda_1>\frac{2r(b^*+1)}{\sqrt{1-a}}.
\end{equation}
Here $\xi_1>M_0$, $\xi_2<-M_0$, $\varepsilon_{i=1,\cdots,4}>0$ and $\eta_{j=1,\cdots,4}>0$, $\delta_{k=1,\cdots,7}>0$,
and $\sigma(\xi)$ will be determined later.


\begin{figure}
\begin{center}
\begin{tikzpicture}[scale = 1.1]
\draw[thick](-6,0) -- (6,0) node[right] {$\xi$};
\draw [ultra thick] (-6,-0.7)to [out=0,in=140] (-3,-1.5) -- (-1.7,-1.5)  to [out=30,in=270]   (-1,0)  to [out=90,in=185]  (1,3) to [out=45,in=190] (1.5,3.3) to [out=0,in=160] (6,0.5);
\draw [semithick] (-6, -0.1) to [ out=0, in=150] (-4,-0.5) to [out=30, in=230] (-3.5,0)  to [out= 40, in=180] (-3,0.2) to [out=15, in=220] (1,1.6) to [out=70,in=180] (2,2)to [out=0,in=170] (6,0.2);
\node[below] at (1,0) {$\xi_{1}+\delta_1$};
\draw[dashed] [thick] (-3.5,0)-- (-3.5,-0.5);
\node[below] at (-3.5,-0.4) {$\xi_2$};
\draw[dashed] [thick] (-3,0.2)-- (-3,-0.5);
\node[below] at (-2.7,-0.4) {$\xi_2+\delta_5$};
\draw[dashed] [thick] (-4,0)-- (-4,1);
\node[above] at (-4,0.8) {$\xi_2-\delta_7$};
\draw[dashed] [thick] (1,0)-- (1,1.6);
\draw[dashed] [thick] (1,2)-- (1,3);
\draw[dashed] [thick] (-1,0)-- (-1,-1.1);
\draw[dashed] [thick] (-1.7,-0.9)-- (-1.7,1.7);
\node[above] at (-1.7,1.55) {$\xi_1-\delta_4$};
\node[below] at (-0.7,-1) {$\xi_1-\delta_3$};
\draw [thin] (-3.85,-0.41) arc [radius=0.2, start angle=40, end angle= 140];
\node[above] at (-4,-0.43) {$\alpha_6$};
\draw [thin] (-2.8,-1.5) arc [radius=0.2, start angle=0, end angle= 135];
\node[above] at (-2.8,-1.4) {$\alpha_5$};
\draw [thin] (-2.8,0.255) arc [radius=0.2, start angle=30, end angle= 175];
\node[above] at (-3,0.3) {$\alpha_4$};
\draw [thin] (-1.5,-1.35) arc [radius=0.2, start angle=20, end angle= 200];
\node[above] at (-1.7,-1.3) {$\alpha_3$};
\draw [thin] (1.1,1.8) arc [radius=0.2, start angle=70, end angle= 220];
\node[above] at (0.7,1.7) {$\alpha_1$};
\draw [thin] (1.1,3.13) arc [radius=0.2, start angle=50, end angle= 180];
\node[above] at (1,3.13) {$\alpha_2$};
\node[above] at (2,2.3) {\Large{$R_u$}};
\node[above] at (2,1) {\Large{$R_v$}};
\end{tikzpicture}
\caption{$(R_u,R_v)$ for the case $b^*\ge 1$.}\label{Figure:b>1}
\end{center}
\end{figure}

\medskip



Next, we divide the proof into several steps.
\medskip

\noindent{\bf{Step 1}:} We consider $\xi\in[\xi_{1}+\delta_1,\infty)$ with $\xi_1>M_0$ ($M_0$ is defined in \eqref{rho}) and  some small $\delta_1$ satisfying
\bea\label{cond delta 1}
0<\delta_1<\frac{1}{\lambda_u+\lambda_1}\quad\text{and}\quad\frac{1}{\lambda_u}+\frac{3}{\lambda_u}(1-e^{-\frac{\lambda_u\delta_1}{2}})-2\delta_1>0.
\eea
In Step 1, we aim to verify that $(W_u,W_v)(\xi)=(U_*-R_u,V_*+R_v)(\xi)$ with
\beaa
(R_u,R_v)(\xi)=(\varepsilon_1\sigma(\xi)e^{-\lambda_u\xi},\eta_1(\xi-\xi_{1}) e^{-\lambda_u\xi}),
\eeaa
satisfies \eqref{tw super solution system} by setting $\varepsilon_1\ll A_0$ and $\eta_1\ll \varepsilon_1$, and
\bea\label{cond delta 0 system case 1}
0<\delta_0=\delta_0(\varepsilon_1,\eta_1,\xi_1+\delta_1)<\frac{\varepsilon_1rb^*-2rC_1\lambda_u\eta_1-2|1-d|\lambda^2_u\eta_1}{rA_0\lambda_u},
\eea
where  
\bea\label{C1}
C_1:=\max_{\xi\in[\xi_1+\delta_1,\infty)}\Big|\frac{(d-2)(1-a)}{r}+1-2V_*(\xi)-(b^*+\delta_0)U_*(\xi)\Big|>0.
\eea
Note that, according to \eqref{cond delta 0 system case 1}, the choice of $\delta_0$ remains vaild regardless of enlarging $\xi_1+\delta_1$ or reducing $\eta_1$.

We define
$$\sigma(\xi):=\frac{4}{\lambda_u^2}e^{-\frac{\lambda_u}{2}(\xi-\xi_1)}-\frac{4}{\lambda_u^2}+\frac{4}{\lambda_u}\xi-\frac{4}{\lambda_u}\xi_1$$
which satisfies
$$\sigma(\xi_1)=0,\ \sigma'(\xi)=\frac{4}{\lambda_u}-\frac{2}{\lambda_u}e^{-\frac{\lambda_u}{2}(\xi-\xi_1)},\ \sigma''(\xi)=e^{-\frac{\lambda_u}{2}(\xi-\xi_1)},$$
and $\sigma(\xi)= O(\xi)$ as $\xi\to\infty$.
Therefore, $R_u$ satisfies the assumption \eqref{decay rate of Ru Rv}.
Moreover, by some  straightforward computations, we obtain $R'_u((\xi_1+\delta_1)^+)>0$ and
$R'_v((\xi_1+\delta_1)^+)>0$  from \eqref{cond delta 1}.

Recall that, $(U_*,V_*)$ is the minimal traveling wave satisfying \eqref{tw solution weak} with $c=2\sqrt{1-a}$.
By some  straightforward computations,  we have
$$N_3[W_u,W_v]=-\varepsilon_1\sigma''(\xi)e^{-\lambda_u\xi}-R_u(a-2U_*+R_u-aV_*-aR_v)-aR_vU_*,$$
and
\begin{equation*}
\begin{aligned}
N_4[W_u,W_v] = & rR_v\Big[\frac{(d-2)(1-a)}{r}+1-2V_*-R_v-(b^*+\delta_0)U_*+(b^*+\delta_0)R_u\Big]\\
&+2(1-d)\lambda_u\eta_1e^{-\lambda_u\xi}+rV_*[(b^*+\delta_0)R_u-\delta_0U_*].
\end{aligned}
\end{equation*}
 Then, from \eqref{AS-U-infty-for-contradiction}, by setting  $\varepsilon_1>0$ and $\eta_1>0$ relatively small to $A_0$,
for all $\xi\in[\xi_1+\delta_1,\infty)$, it holds
$$-2U_*+R_u-aR_v = o(e^{-\frac{\lambda_u}{2}\xi})\quad\text{and}\quad a-aV_*\ge 0.$$
Then, up to enlarging $\xi_1$ if necessary, since $aR_vU_*>0$, we obtain that $N_3[W_u,W_v]\leq0$
for all $\xi\in[\xi_1+\delta_1,\infty)$.

Next, we deal with the inequality of $N_4[W_u, W_v]$.
For $\xi\in[\xi_1+\delta_1,\infty)$, from \eqref{rho} and \eqref{C1}, we have
\begin{equation*}
N_4[W_u,W_v]\geq -rR_v(C_1+R_v)+2(1-d)\lambda_u\eta_1e^{-\lambda_u\xi}+r(1-\rho)(b^*+\delta_0)R_u-r\delta_0U_*.
\end{equation*}
From the definition of $\sigma(\xi)$, we can find a $M_1>\xi_1$ such that $\sigma(\xi)\sim \frac{4}{\lambda_u}\xi$,
 and
$$2(1-d)\lambda_u\eta_1e^{-\lambda_u\xi}=o(R_u)\ \ \text{for} \ \ \xi\ge M_1.$$
By further choosing ${\eta_1}/{\varepsilon_1}$ sufficiently small and $\delta_0$ satisfying \eqref{cond delta 0 system case 1}, we have $N_4[W_u,W_v]\geq0$
for $\xi\geq M_1$. For $\xi\in[\xi_1+\delta_1,M_1]$, by reducing ${\eta_1}/{\varepsilon_1}$ if necessary, we have
$$-rR_v(C_1+R_v)+2(1-d)\lambda_u\eta_1e^{-\lambda_u\xi}+r(1-\rho)(b^*+\delta_0)R_u> 0.$$
Since $U_*$ is bounded on $[\xi_1+\delta_1,M_1]$, by setting $\delta_0(\varepsilon_1,\eta_1,\xi_1+\delta_1)\ll \varepsilon_1$ sufficiently small, then we have $N_4[W_u,W_v]\geq0$ for all $\xi\geq \xi_1+\delta_1$.


\medskip

\noindent{\bf{Step 2:}} We consider $\xi\in[\xi_1-\delta_4,\xi_1+\delta_1]$ with $\xi_1+\delta_1$ fixed by Step 1. In this case, we have
$$(R_u,R_v)(\xi):=(\varepsilon_2\sin(\delta_2(\xi-\xi_1+\delta_3)),\eta_2e^{\lambda_{1}\xi})$$
with $0<\delta_2<1$ sufficiently small,
\bea\label{cond delta 4}
0<\delta_3<\delta_4\quad\text{satisfying}\quad|\delta_3-\delta_4|\ll 1,
\eea
 and
\bea\label{cond delta 3 4}
\delta_1+\delta_3=\frac{\sqrt{1-a}}{2+a}<\frac{\pi}{2\delta_2}.
\eea

We first verify the following claim: 
\begin{claim}\label{cl: alpha 1 aplha 2}
There exist 
$\varepsilon_2>0$ and $\eta_2>0$ sufficiently small
such that
\beaa
&&R_u((\xi_{1}+\delta_1)^+)=R_u((\xi_{1}+\delta_1)^-)\quad \text{and}\quad \angle\alpha_1<180^{\circ}, \\
&&R_v((\xi_{1}+\delta_1)^+)=R_v((\xi_{1}+\delta_1)^-)\quad\text{and}\quad \angle\alpha_2<180^{\circ},
\eeaa
provided that  $\delta_1,\delta_3$ satisfy \eqref{cond delta 1} and \eqref{cond delta 3 4},  and $\delta_2$ is sufficiently small.
\end{claim}
\begin{proof}
By some  straightforward computations, we have
\beaa
&&R_u((\xi_{1}+\delta_1)^+)=\varepsilon_1\sigma(\xi_1+\delta_1)e^{-\lambda_u(\xi_{1}+\delta_1)},\quad R_u((\xi_{1}+\delta_1)^-)=\varepsilon_2\sin(\delta_2(\delta_1+\delta_3)),\\
&&R'_u((\xi_{1}+\delta_1)^+)=\varepsilon_1\sigma'(\xi_1+\delta_1)e^{-\lambda_u(\xi_{1}+\delta_1)}-\lambda_uR_u(\xi_{1}+\delta_1),\\
&&R'_u((\xi_{1}+\delta_1)^-)=\varepsilon_2\delta_2\cos(\delta_2(\delta_1+\delta_3)).
\eeaa
We first choose 
$$\varepsilon_2=\varepsilon_2(\varepsilon_1,\delta_1,\delta_2,\delta_3)=\varepsilon_1\sigma(\xi_1+\delta_1)e^{-\lambda_u(\xi_{1}+\delta_1)}/\sin(\delta_2(\delta_1+\delta_3))$$
such that
\beaa
R_u((\xi_{1}+\delta_1)^+)=R_u((\xi_{1}+\delta_1)^-).
\eeaa
Then, by applying \eqref{cond delta 3 4} and the fact $\frac{x\cos x}{\sin x}\to 1$ as $x\to 0$, we have
$$R'_u((\xi_{1}+\delta_1)^+)-R'_u((\xi_{1}+\delta_1)^-)>0$$
 is equivalent to
\beaa
\frac{2}{\lambda_u}+\frac{2}{\lambda_u}(1-e^{-\frac{\lambda_u\delta_1}{2}})>(\frac{1}{\delta_1+\delta_3}+\lambda_u)\sigma(\xi_1+\delta_1),
\eeaa
which holds since $\sigma(\xi_1+\delta_1)\to 0$ as $\delta_1\to 0$ and \eqref{cond delta 3 4}. It follows that $\angle\alpha_1<180^{\circ}$.

On the other hand, by some  straightforward computations, we have
\beaa
&&R_v((\xi_{1}+\delta_1)^-)=\eta_2e^{\lambda_1(\xi_{1}+\delta_1)},\ R_v((\xi_{1}+\delta_1)^+)=\eta_1\delta_1e^{-\lambda_u(\xi_{1}+\delta_1)},\\
&&R'_v((\xi_{1}+\delta_1)^-)=\lambda_1\eta_2e^{\lambda_1(\xi_{1}+\delta_1)},\ R'_v((\xi_{1}+\delta_1)^+)=\eta_1(1-\delta_1\lambda_u)e^{-\lambda_u(\xi_{1}+\delta_1)},
\eeaa
where $\lambda_1$ satisfies \eqref{condition on lambda 1 system}.
We take
\bea\label{eta 2 system}
\eta_2=\eta_2(\eta_1,\delta_1,\lambda_1)=\eta_1\delta_1e^{-(\lambda_u+\lambda_1)(\xi_{1}+\delta_1)}>0,
\eea
which implies $R_v((\xi_{1}+\delta_1)^-)=R_v((\xi_{1}+\delta_1)^+)$. Then, from \eqref{cond delta 1}, we have
\beaa
R_v'((\xi_{1}+\delta_1)^+)-R_v'((\xi_{1}+\delta_1)^-)=\eta_1e^{-\lambda_u(\xi_1+\delta_1)}(1-\delta_1\lambda_u-\delta_1\lambda_1)>0,
\eeaa
which yields that $\angle\alpha_2<180^{\circ}$.
The proof of Claim \ref{cl: alpha 1 aplha 2} is complete.
\end{proof}

To finish Step 2, it suffices to take a small $\delta_2>0$ and suitable $0<\delta_3<\delta_4$ such that
\bea\label{goal: step 3}
N_3[W_u,W_v]\leq 0\quad\text{and}\quad N_4[W_u,W_v]\geq 0\quad \mbox{for}\quad \xi\in[\xi_1-\delta_4,\xi_1+\delta_1].
\eea
By some  straightforward computations, for $\xi\in[\xi_2-\delta_3,\xi_1+\delta_1]$ we have
\beaa
N_3[W_u,W_v]&=&-2\sqrt{1-a}\delta_2\varepsilon_2\cos(\delta_2(\xi-\xi_1+\delta_3))-a(U_*-R_u)R_v\\
&&-R_u(1-\delta_2^2-2U_*+R_u-aV_*),\\
N_4[W_u,W_v]&=&R_v\Big[d\lambda_1^2+2\sqrt{1-a}\lambda_1+r[1-2V_*-R_v-(b^*+\delta_0)(U_*-R_u)]\Big]
            \\&&+rV_*[(b^*+\delta_0)R_u-\delta_0U_*].
\eeaa
To estimate $N_3[W_u,W_v]$,
we consider $\xi\in[\xi_1-\delta_3,\xi_1+\delta_1]$ and $\xi\in[\xi_1-\delta_4,\xi_1-\delta_3]$ separately as follows:
\begin{itemize}
\item For $\xi\in[\xi_1-\delta_3,\xi_1+\delta_1]$, we have
$$0\le R_u(\xi)\le \varepsilon_2\sin(\delta_2(\delta_1+\delta_3))$$
 and $R_v(\xi)\geq0$.
Then, from \eqref{cond delta 3 4} and $U_*-R_u\ge 0$,  we have
\beaa
N_3[W_u,W_v]&\leq& -R_u({\xi})\Big(2\sqrt{1-a}\frac{\delta_2\cos(\delta_2(\delta_1+\delta_3))}{\sin(\delta_2(\delta_1+\delta_3))}+1-\delta_2^2-2U_*-aR_v-aV_*\Big)\\
&\le & -R_u({\xi})\Big(\frac{2\sqrt{1-a}}{\delta_1+\delta_3}-2-a\Big)\leq 0,
\eeaa
provided that $\delta_2$ is sufficiently small.

\item For $\xi\in[\xi_1-\delta_4,\xi_1-\delta_3]$, we have $R_u(\xi)\leq 0$, $R'_u(\xi)\geq 0$, and $R_v(\xi)\ge 0$.
Note that, we can set
$$|R_u(\xi)|\ll R_v(\xi)\quad\text{and}\quad |R_u(\xi)|\ll U_*(\xi)$$
for $\xi\in[\xi_1-\delta_4,\xi_1-\delta_3]$ since
\bea\label{delta3-4 to 0}
\max_{\xi\in[\xi_1-\delta_4,\xi_1-\delta_3]}|R_u(\xi)|\to 0\quad\text{as}\quad|\delta_3-\delta_4|\to 0.
\eea
Then, since $a(U_*-R_u)R_v> 0$, we have
\beaa
N_3[W_u,W_v]\leq -a(U_*-R_u)R_v-R_u(1-\delta_2^2-2U_*+R_u-aV_*)\leq 0,
\eeaa
provided $|\delta_3-\delta_4|>0$ is chosen sufficiently small.
\end{itemize}
From the above discussion, we asset that
$$N_3[W_u,W_v]\leq 0\ \ \text{for}\ \ \xi\in[\xi_1-\delta_4,\xi_1+\delta_1],$$
provided that $\delta_1, \delta_3, \delta_4$ satisfy \eqref{cond delta 4} and \eqref{cond delta 3 4}, and
$\delta_2$ is small enough.

From now on, $\delta_1$, $\delta_2$, and $\delta_3$ are fixed. On the other hand, thanks to \eqref{condition on lambda 1 system}, we have
\bea\label{N4-Step3}
N_4[W_u,W_v]&\geq&R_v\Big[r(2+b^*)+r[1-2V_*-R_v-(b^*+\delta_0)(U_*-R_u)]\Big]
            \\&&+rV_*[(b^*+\delta_0)R_u-\delta_0U_*]\notag\\
           &\geq&  \frac{r}{2}R_v+rV_*[(b^*+\delta_0)R_u-\delta_0U_*].\notag
\eea
Note that $R_u(\xi)\geq0$ for $\xi\in[\xi_1-\delta_3,\xi_1+\delta_1]$; $R_u(\xi)<0$ but satisfies \eqref{delta3-4 to 0} for $\xi\in[\xi_1-\delta_4,\xi_1-\delta_3)$.
Consequently, we assert that
$$N_4[W_u,W_v]\ge 0\ \ \text{for all}\ \ \xi\in[\xi_2-\delta_3,\xi_1+\delta_1]$$
 up to decreasing $|\delta_3-\delta_4|$ and $\delta_0(\varepsilon_1, \eta_1, |\delta_3-\delta_4|)$ if necessary.
This completes the proof of \eqref{goal: step 3}, and Step 2 is finished. 
Note that $|\delta_3-\delta_4|$ can be further reduced to get a smaller $\varepsilon_3$ in the following steps.

\medskip


\noindent{\bf{Step 3:}} We consider $\xi\in[\xi_2+\delta_5,\xi_1-\delta_4]$ with $\xi_2+\delta_5<-M_0$. From \eqref{condition on lambda 1 system}, we can set $\delta_5$ to satisfy
\bea\label{cond delta 5}
\frac{1}{\lambda_1}<\delta_5<\frac{\sqrt{1-a}}{2r(b^*+1)}\quad\text{and}\quad \Big|\delta_5-\frac{1}{\lambda_1}\Big|\ \text{is sufficiently small}.
\eea
In this case, we have
\beaa
(R_u,R_v)(\xi)=(-\varepsilon_3,\eta_2e^{\lambda_1\xi}).
\eeaa

First, we choose 
$$\varepsilon_3=\varepsilon_3(\varepsilon_1,\delta_3-\delta_4)=R_u(\xi_1-\delta_4)=\varepsilon_1\sigma(\xi_1+\delta_1)e^{-\lambda_u(\xi_{1}+\delta_1)}\frac{\sin(\delta_2(\delta_3-\delta_4))}{\sin(\delta_2(\delta_1+\delta_3))}$$
such that $R_u(\xi)$ is continuous at $\xi=\xi_1-\delta_4$.
Clearly, by setting $|\delta_3-\delta_4|$ very small as in Step 2, we have
$$R'_u((\xi_1-\delta_{4})^+)>0=R'_u((\xi_1-\delta_{4})^-),\  {\it i.e.},\  \angle\alpha_3<180^{\circ}.$$

By some straightforward computations, we have
\beaa
N_3[W_u,W_v]=-R_u(1-2U_*+R_u-a(V_*+R_v))-aU_*R_v,
\eeaa
and $N_4[W_u,W_v]$ satisfies \eqref{N4-Step3}.
Note that, $|\delta_3-\delta_4|\to0$ implies that $\varepsilon_3\to 0$, and $|R_v(\xi)|$
does not depend on $|\delta_3-\delta_4|$.
It follows that
$$|R_u(\xi)|\ll |R_v(\xi)|\ \ \text{for all}\ \ \xi\in[\xi_2+\delta_5,\xi_1-\delta_4]$$
up to decreasing $|\delta_3-\delta_4|$ if necessary. Also, we have $\min_{\xi\in(-\infty,\xi_1-\delta_4]}U_*(\xi)$ is positive and bounded from below on $\xi\in(-\infty,\xi_1-\delta_4]$. Therefore,
we see that
$$N_3[W_u,W_v]\leq 0\ \ \text{for}\ \ \xi\in[\xi_2+\delta_5,\xi_1-\delta_4]$$
by taking $|\delta_3-\delta_4|$ sufficiently small.
On the other hand, by the same argument in Step 2, we see that
$$N_4[W_u, W_v]\geq 0\ \ \text{for}\ \ \xi\in[\xi_2+\delta_5,\xi_1-\delta_4]$$
up to decreasing $|\delta_3-\delta_4|$ and $\delta_0(\varepsilon_1,\eta_1, |\delta_3-\delta_4|)$ is necessary. Moreover, the choice of $\xi_2+\delta_5$ remains vaild regardless of reducing $|\delta_3-\delta_4|$.

\medskip

\noindent{\bf{Step 4:}} We consider $\xi\in[\xi_2-\delta_7,\xi_2+\delta_5]$ with $\xi_2+\delta_5$ fixed by Step 3.
In this case, we have
\beaa
(R_u,R_v)(\xi)=\Big(-\varepsilon_4(-\xi)^{\theta}[1-U_*(\xi)],\eta_3\sin(\delta_6(\xi-\xi_2))\Big),
\eeaa
where $\theta\in(0,1)$ is fixed,
while $\varepsilon_4>0$, $\eta_3>0$,
$\delta_6>0$, and $\delta_7>0$ are determined below.



We first choose
\bea\label{epsilon4}
\varepsilon_4=\varepsilon_4(\varepsilon_1,\delta_3-\delta_4)=\frac{\varepsilon_1\sigma(\xi_1+\delta_1)e^{-\lambda_u(\xi_{1}+\delta_1)}}{(-\xi_2-\delta_5)^{\theta}[1-U_*(\xi_2+\delta_5)]}\frac{\sin(\delta_2(\delta_3-\delta_4))}{\sin(\delta_2(\delta_1+\delta_3))}
\eea
such that $R_u(\xi)$ is continuous at $\xi=\xi_2+\delta_5$.
Then, from the asymptotic behavior of $1-U_*$ for both cases $b^*>1$ and $b^*=1$ as $\xi\to-\infty$, $R'_u(\xi)<0$ for all $\xi\leq \xi_2+\delta_5$ by $-(\xi_2+\delta_5)>M_0$ vary large. In particular,
we have
\beaa
R'_u((\xi_2+\delta_5)^+)=0>R'_u((\xi_2+\delta_5)^-),
\eeaa
and thus $\angle\alpha_5<180^{\circ}$.
Next, we verify the continuity of $R_v$ at $\xi_2+\delta_5$ and the right angle of $\alpha_6$:
\begin{claim}\label{cl 5}
For any $\delta_5$ satisfying \eqref{cond delta 5}, there exist
$\eta_3>0$ and small $\delta_6>0$
such that $R_v(\xi)$ is continuous at $\xi=\xi_2+\delta_5$ and $\angle\alpha_6<180^{\circ}$.
\end{claim}
\begin{proof}
First, we  take
\bea\label{eta4}
\eta_3=\eta_3(\eta_1,\delta_5,\delta_6)= \eta_1\delta_1e^{-(\lambda_u+\lambda_1)(\xi_{1}+\delta_1)}\frac{e^{\lambda_1(\xi_2+\delta_5)}}{\sin(\delta_5\delta_6)}>0
\eea
such that $R_v((\xi_2+\delta_5)^+)=R_v((\xi_2+\delta_5)^-)$.

By some straightforward computations, we have
$$R'_v((\xi_2+\delta_5)^+)=\lambda_1\eta_2e^{\lambda_1(\xi_2+\delta_5)}.$$
Then from \eqref{eta4},
\beaa
R'_v((\xi_2+\delta_5)^-)=\eta_3\delta_6\cos(\delta_5\delta_6)=\eta_2e^{\lambda_1(\xi_2+\delta_5)}\frac{\delta_6\cos(\delta_5\delta_6)}{\sin(\delta_5\delta_6)},
\eeaa
which yields that
$$R'_v((\xi_2+\delta_5)^-)\rightarrow \eta_2e^{\lambda_1(\xi_2+\delta_5)}/\delta_5\quad\text{as}\quad \delta_6\to0.$$
Thus, $R'_v((\xi_2+\delta_5)^+)>R'_v((\xi_2+\delta_5)^-)$ is equivalent to $\delta_5>\frac{1}{\lambda_1}$ by setting $\delta_6$ sufficiently small.
This completes the proof of Claim~\ref{cl 5}.
\end{proof}

From now on, we fix $\delta_5$, which is unaffected by the reduction of $\delta_6$. The next claim shows how to determine $\delta_7$.
Note that the choice of $\delta_7$ is rather technical and crucial in verifying the differential inequalities later.

\begin{claim}\label{cl 6}
There exists $0<\delta_7\le \delta_5$ such that
$$R_v(\xi_2-\delta_7)=-\varepsilon_4(-\xi_2+\delta_7)^{\theta}V_*(\xi_2-\delta_7)$$
and
\bea\label{claim3.6}
-\varepsilon_4(-\xi)^{\theta}V_*(\xi)<R_v(\xi)<0\quad\text{for all}\quad \xi\in(\xi_2-\delta_7,\xi_2).
\eea
\end{claim}
\begin{proof}
Recall from Step 3 and \eqref{epsilon4}
that
$$R_v(\xi_2+\delta_5)\gg \varepsilon_3=\varepsilon_4(-\xi_2-\delta_5)^{\theta}[1-U_*(\xi_2+\delta_5)].$$
We also assume
\bea\label{Rv>V}
R_v(\xi_2+\delta_5)>\varepsilon_4(-\xi_2-\delta_5)^{\theta}V_*(\xi_2+\delta_5)
\eea
 by reducing $\varepsilon_3$ if necessary. This actually can be done by reducing $|\delta_3-\delta_4|$.
Furthermore, by the asymptotic behavior of $V_*(\xi)$ as $\xi\to -\infty$,  $-\varepsilon_4(-\xi)^{\theta}V_*(\xi)$ is strictly decreasing for all $\xi<\xi_2+\delta_5$ since $-(\xi_2+\delta_5)>M_0$ very large.  Together with \eqref{Rv>V}, we obtain that
\beaa
-\eta_3\sin(\delta_5\delta_6)=-R_v(\xi_2+\delta_5)<
-\varepsilon_4(-\xi_2-\delta_5)^{\theta}V_*(\xi_2+\delta_5)<-\varepsilon_4(-\xi_2+\delta_5)^{\theta}V_*(\xi_2-\delta_5).
\eeaa

Define $$F(\xi):=\eta_3\sin(\delta_6(\xi-\xi_2))+\varepsilon_4(-\xi)^{\theta}V_*(\xi).$$
Clearly, $F$ is continuous and strictly increasing for $\xi\in[\xi_2-\delta_5,\xi_2]$. Also,
we have $F(\xi_2)>0$ and $F(\xi_2-\delta_5)<0$.
Then, by the intermediate value theorem, there exists a unique $\delta_7\in(0,\delta_5)$ such that Claim~\ref{cl 6} holds.
\end{proof}

Let $\delta_5$ and $\delta_7$ be fixed by Claim \ref{cl 5} and Claim \ref{cl 5}. We now verify the differential inequalities.
Note that it suffices to assume $V_*+R_v\ge 0$.
By some straightforward computations, $N_3[W_u,W_v]$ satisfies
\begin{equation}\label{N3 inequalifty -infty}
\begin{aligned}
N_3[W_u,W_v]=&\varepsilon_4(-\xi)^{\theta}\Big(-U_*''-c^*U_*'-\theta(1-\theta)(-\xi)^{-2}(1-U_*)+2\theta(-\xi)^{-1}U_*'\\&-c^*\theta(-\xi)^{-1}(1-U_*)\Big)
-R_u(1-2U_*+R_u-a(V_*+R_v))-aU_*R_v\\
\leq&\varepsilon_4(-\xi)^{\theta}\Big(U_*(1-U_*-aV_*)-c^*\theta(-\xi)^{-1}(1-U_*)\Big)\\&-R_u(1-2U_*+R_u-a(V_*+R_v))-aU_*R_v.
\end{aligned}
\end{equation}
The last inequality holds due to $\theta\in(0,1)$ and $U'_*<0$.

We next divide our discussion into two parts: $\xi\in[\xi_2,\xi_2+\delta_5]$ and $\xi\in[\xi_2-\delta_7,\xi_2]$.
Notice that, $R_u(\xi)<0<R_v(\xi)$ and $(V_*+R_v)(\xi)\ge 0$ for $\xi\in[\xi_2,\xi_2+\delta_5]$.
Then, For $\xi\in[\xi_2,\xi_2+\delta_5]$,
\eqref{N3 inequalifty -infty} reduces to
\beaa
N_3[W_u,W_v]\leq \varepsilon_4(-\xi)^{\theta}\Big(U_*(1-U_*-aV_*)+(1-2U_*-c^*\theta(-\xi)^{-1})(1-U_*)\Big).
\eeaa
\begin{itemize}
    \item For $b^*=1$, we see from Lemma~\ref{lem:AS-infty:b=1} and Corolloary~\ref{lm: behavior around - infty b=1} that
    $$U_*(1-U_*-aV_*)=o((-\xi)^{-1})$$
     and $1-U_*\sim (-\xi)^{-1}$ as $\xi\to-\infty$. By \eqref{rho},
     $$1-2U_*-c^*\theta(-\xi)^{-1}<-\frac{1}{2}.$$
      Therefore
    we conclude that $N_3[W_u,W_v]\leq0$
for $\xi\in[\xi_2,\xi_2+\delta_5]$ as long as $M_0$ in \eqref{rho} is chosen sufficiently large at the beginning.
     \item For $b^*>1$,
we have
\beaa
N_3[W_u,W_v]\leq \varepsilon_4(-\xi)^{\theta}[(1-U_*)-c^*\theta(-\xi)^{-1}](1-U_*).
\eeaa
By Lemma \ref{lem:AS-infty:b>1}, since $1-U_*$ decays exponentially as $\xi\to-\infty$, we obtain $N_3[W_u,W_v]\leq0$
for $\xi\in[\xi_2,\xi_2+\delta_5]$ as long as $M_0$ is chosen sufficiently large at the beginning.
\end{itemize}



On the other hand, for $\xi\in[\xi_2-\delta_7,\xi_2]$, 
by using \eqref{claim3.6} and
$$\varepsilon_4(-\xi)^{\theta}U_*(1-U_*)=-R_uU_*,$$
from \eqref{N3 inequalifty -infty} we have
\beaa
N_3[W_u,W_v]&\le&-R_uU_*-a\varepsilon_4(-\xi)^{\theta}U_*V_*+c^*\theta(-\xi)^{-1}R_u-R_u(1-2U_*-aV_*)
\\
&&-R_u^2+aR_uR_v+a\varepsilon_4(-\xi)^{\theta}U_*V_*\\
&=&c^*\theta(-\xi)^{-1}R_u-R_u(1-U_*-aV_*)
-R_u^2+aR_uR_v.
\eeaa

Denote that
\beaa
I_1:=c^*\theta(-\xi)^{-1}R_u,\quad
I_2:=-R_u(1-U_*-aV_*),\quad
I_3:=-R_u^2+aR_uR_v.
\eeaa
\begin{itemize}
\item For the case $b^*=1$, by the equation satisfied by $U_*$ in \eqref{tw solution weak} and Lemma \ref{lm: behavior around - infty b=1},  $1-U_*-aV_*>0$ for all $\xi\leq -M_0$ (if necessary, we may choose $M_0$ larger). Therefore,
\beaa
I_3=-R_u^2+aR_uR_v\leq R_u\varepsilon_4(-\xi)^{\theta}(1-U_*-aV_*)(\xi)<0\quad \mbox{for}\quad \xi\in[\xi_2-\delta_7,\xi_2].
\eeaa
Moreover, in view of Corollary~\ref{lm: behavior around - infty b=1}, we have
$I_2=o(I_1)$ as $\xi\to-\infty$.
\item For the case $b^*>1$, since $1-U_*-aV_*\to 0$ exponentially (See Lemma \ref{lem:AS-infty:b>1}), we have $I_2,I_3\sim o((-\xi)^{-1})R_u$.
\end{itemize}
Then, as long as $M_0$ is chosen sufficiently large at the beginning, we have $N_3[W_u,W_v]\le 0$ for $\xi\in[\xi_2-\delta_7,\xi_2]$.

We next deal with the inequality of $N_4[W_u,W_v]$.
By some straightforward computations, we have
\begin{equation}\label{N4 inequality step 5}
\begin{aligned}
N_4[W_u,W_v]=& rR_v\Big(1-2V_*-R_v-(b^*+\delta_0)(U_*-R_u)
-\frac{d}{r}\delta_6^2\Big)\\&+2\sqrt{1-a}\delta_6\eta_3\cos(\delta_6(\xi-\xi_2))+rV_*((b^*+\delta_0)R_u-\delta_0U_*).
\end{aligned}
\end{equation}
For $\xi\in[\xi_2,\xi_2+\delta_5]$, \eqref{eta4} and the fact $\frac{x\cos x}{\sin x}\to 0$ as $x\to 0$ yield that
\beaa
\min_{\xi\in[\xi_2,\xi_2+\delta_5]}\delta_6\eta_3\cos(\delta_6(\xi-\xi_2))\to\frac{\eta_2e^{\lambda_1(\xi_2+\delta_5)}}{\delta_5}
=\frac{R_v(\xi_2+\delta_5)}{\delta_5}
\quad\text{as}\quad \delta_6\to0.
\eeaa
In view of \eqref{cond delta 5},
we have
\bea\label{Rv-estimate2}
\min_{\xi\in[\xi_2,\xi_2+\delta_5]}\delta_6\eta_3\cos(\delta_6(\xi-\xi_2))>\lambda_1R_v(\xi_2+\delta_5)>R_v(\xi_2+\delta_5)\frac{2r(b^*+1)}{\sqrt{1-a}}.
\eea

By applying \eqref{rho}, \eqref{N4 inequality step 5}, and \eqref{Rv-estimate2}, we have
\begin{equation*}
\begin{aligned}
N_4[W_u,W_v]\ge & -rR_v(\xi_2+\delta_5)\Big(1+R_v+(b^*+\delta_0)+\frac{d}{r}\delta_6^2\Big)+2r(b^*+1)R_v(\xi_2+\delta_5)\\&
+r\rho(b^*+\delta_0)R_u(\xi_2+\delta_5)-r\rho\delta_0.
\end{aligned}
\end{equation*}
Recall that,
$$|R_u(\xi)|\ll R_v(\xi_2+\delta_5)\ \ \text{for all}\ \ \xi\in[\xi_2,\xi_2+\delta_5]$$
up to decreasing $|\delta_3-\delta_4|$.
Therefore, we assert that $N_4[W_u,W_v]\geq 0$ for $\xi\in[\xi_2,\xi_2+\delta_5]$,
and for all small $\delta_0(\varepsilon_1,\eta_1, |\delta_3-\delta_4|,\delta_6)>0$.

For $\xi\in[\xi_2-\delta_7,\xi_2]$, since $R_v<0$ and $\xi_2<-M_0$, by applying \eqref{rho}, \eqref{N4 inequality step 5}, and a similar discussion as for
$[\xi_2,\xi_2+\delta_5]$, we have
\begin{equation*}
\begin{aligned}
N_4[W_u,W_v]\ge &rR_v(1-R_v)+2\sqrt{1-a}\delta_6\eta_3\cos(\delta_6(\xi-\xi_2))+rV_*((b^*+\delta_0)R_u-\delta_0U_*)\\
\ge &rR_v(\xi_2-\delta_7)\Big(1-R_v(\xi_2-\delta_7)\Big)+2r(b^*+1)R_v(\xi_2+\delta_5)\\
&+r\rho(b^*+\delta_0)R_u(\xi_2+\delta_5)-r\rho\delta_0.
\end{aligned}
\end{equation*}
Since $0<\delta_7\le \delta_5$ and $|R_v(\xi_2-\delta_7)|\le R_v(\xi_2+\delta_5)$, $N_4[W_u,W_v]\ge 0$ holds in $[\xi_2-\delta_7,\xi_2]$ for all small $\delta_0(\varepsilon_1,\eta_1, |\delta_3-\delta_4|,\delta_6)>0$. From the above discussion,
the construction for Step 4 is finished. Hereafter, we fix $\delta_4$ and $\delta_6$.

\medskip

\noindent{\bf{Step 5:}} We consider $\xi\in(-\infty,\xi_2-\delta_7]$ with $\xi_2-\delta_7$ fixed by Step 4.
In this case, we have
\beaa
(R_u,R_v)(\xi)=\Big(-\varepsilon_4(-\xi)^{\theta}[1-U_*(\xi)],-\eta_4(-\xi)^{\theta}V_*(\xi)\Big).
\eeaa
Let us take $$\eta_4=\eta_4(\eta_1)=\eta_1\delta_1e^{-(\lambda_u+\lambda_1)(\xi_{1}+\delta_1)}\frac{\sin(\delta_6\delta_7)e^{\lambda_1(\xi_2+\delta_5)}}{\sin(\delta_5\delta_6)(\delta_7-\xi_2)^{\theta}V_*(\xi_2-\delta_7)}$$
 such that
$R_v(\xi)$ is continuous at $\xi=\xi_2-\delta_7$.
Also, since $0<\delta_7\le \delta_5$ and $-\eta_4(-\xi)^{\theta}V_*(\xi)$ is decreasing on $\xi$ for $\xi<\xi_2$, we have
$$R'_v((\xi_2-\delta_7)^+)>0>R'_v((\xi_2-\delta_7)^-),$$ 
and hence $\angle\alpha_6<180^{\circ}$.

Finally, we verify the differentiable inequalities.
Since $\theta>0$, there exists $M_1>M_0$ sufficiently large such that $W_u=1$ and $W_v=0$ for all $\xi\in(-\infty,-M_1]$.
More precisely, Claim \ref{cl 6} implies $\eta_4=\varepsilon_4$. Then, from the definition of $(R_u,R_v)$, we may define $M_1$ satisfying $1-\eta_4(M_1)^{\theta}= 0$. Thus
$W_u(\xi)=1$, $W_v(\xi)=0$ for all $\xi\in(-\infty,-M_1]$, which implies that
$$N_3[W_u,W_v]\le 0\ \ \text{and}\ \ N_4[W_u,W_v]\ge 0\ \ \text{for}\ \ \xi\in(-\infty,-M_1].$$

It suffices to deal with the computation for $\xi\in[-M_1,\xi_2-\delta_7]$. Without loss of generality, we may assume $\xi_2-\delta_7<\xi_0$,
where $\xi_0$ 
is defined in Corollary~\ref{lm: behavior around - infty b=1}. Additionally, 
by the definition of $M_1$ and $\eta_4=\varepsilon_4$, we have
\bea\label{eta 4 epsilon 4}
1-\varepsilon_4(-\xi)^{\theta}=1-\eta_4(-\xi)^{\theta}> 0\quad\text{for all}\quad \xi\in(-M_1,\xi_2-\delta_7],
\eea
which yields $W_u<1$ and $W_v>0$ on $(-M_1,\xi_2-\delta_7]$.
Note that, $R_u,R_v<0$ in $[-M_1,\xi_2-\delta_7]$, and $N_3[W_u,W_v]$ satisfies \eqref{N3 inequalifty -infty}.
By applying the same argument as that in Step 4 for $\xi\in[\xi_2-\delta_7,\xi_2]$, we obtain
$N_3[W_u,W_v]\le 0$ for all $\xi\in[-M_1,\xi_2-\delta_7]$.

On the other hand, by some straight computations, we have
\beaa
N_4[W_u,W_v]&=&d\Big(V_*''+\theta(1-\theta)\eta_4(-\xi)^{\theta-2}V_*+2\theta\eta_4(-\xi)^{\theta-1}V'_*-\eta_4(-\xi)^{\theta}V''_*\Big)\\
&&+c^*\Big(V_*'+\theta\eta_4(-\xi)^{\theta-1}V_*-\eta_4(-\xi)^{\theta}V'_*\Big)\\&&
+r(V_*+R_v)(1-V_*-R_v-(b^*+\delta_0)(U_*-R_u)).
\eeaa

Then, by $V_*'>0$, $\varepsilon_4=\eta_4$, and $\theta\in(0,1)$, we further have
\bea\label{Step5-N4}
N_4[W_u,W_v]&\geq& r\eta_4(-\xi)^{\theta}V_*\Big(V_*-b^*(1-U_*)
+\frac{c^*\theta}{r}(-\xi)^{-1}+R_v-b^*R_u\Big)\notag\\
& &-r(U_*-R_u)(V_*+R_v)\delta_0.
\eea
\begin{itemize}
\item For the case $b^*>1$,  both $1-U_*\to 0$ and $V_*\to 0$ exponentially as $\xi\to-\infty$. Thus $1-U_*=o((-\xi)^{-1})$ and $R_v=o((-\xi)^{-1})$ for $\xi\in[-M_1,\xi_2-\delta_7]$.
\item For the case $b^*=1$, \eqref{Step5-N4} reduces to
\beaa
N_4[W_u,W_v]&\geq& r\eta_4(-\xi)^{\theta}V_*\Big((\eta_4(-\xi)^{\theta}-1)(1-U_*-V_*)
+\frac{c^*\theta}{r}(-\xi)^{-1}\Big)\\
&&-r(U_*-R_u)(V_*+R_v)\delta_0.
\eeaa
By Corollary \ref{lm: behavior around - infty b=1} and \eqref{eta 4 epsilon 4}, 
as long as $M_0$ is chosen large at the beginning, we have
$(\eta_4(-\xi)^{\theta}-1)(1-U_*-V_*)>0$ for $\xi\in[-M_1,\xi_2-\delta_7]$.
\end{itemize}
It follows that
$N_4[W_u,W_v]\ge 0$ for $\xi\in[-M_1,\xi_2-\delta_7]$
for very small $\delta_0(\varepsilon_1,\eta_1)>0$.
Therefore, the construction for Step 5 is finished.


\subsubsection{For the case $b^*<1$}\label{sec:3-2}

\noindent

In this subsection, we always assume $0<b^*<1$.
Let $(c_{LV}^*,U_*,V_*)$ be the minimal traveling wave of \eqref{tw solution weak} with $b=b^*$ and $c_{LV}^*=2\sqrt{1-a}$.
Different from the strong-weak competition case and the critical case, since $(U_*,V_*)(+\infty)=(0,1)$ and
$$(U_*,V_*)(-\infty)=(\frac{1-a}{1-ab^*},\frac{1-b^*}{1-ab^*}):=(\hat u,\hat v),$$
and $U'_*<0<V_*'$, for any given small $\rho>0$, we have
\bea\label{rho 2}
\begin{cases}
0< U_*(\xi)<\rho,\quad 1-\rho<V_*(\xi)<1&\quad\mbox{for all}\quad\xi\geq M_0,\\
\hat u-\rho< U_*(\xi)<\hat u,\quad \hat v<V_*(\xi)<\hat v+\rho&\quad\mbox{for all}\quad\xi\leq -M_0,
\end{cases}
\eea
up to enlarging $M_0>0$ if necessary.

\begin{figure}
\begin{center}
\begin{tikzpicture}[scale = 1.1]
\draw[thick](-6,0) -- (6,0) node[right] {$\xi$};
\draw [ultra thick] (-6,-1.5)--  (-1.7,-1.5)  to [out=30,in=270]   (-1,0)  to [out=90,in=185]  (1,3) to [out=45,in=190] (1.5,3.3) to [out=0,in=160] (6,0.5);
\draw [semithick] (-6, -0.5) -- (-4,-0.5) to [out=30, in=230] (-3.5,0)  to [out= 40, in=180] (-3,0.2) to [out=15, in=220] (1,1.6) to [out=70,in=180] (2,2)to [out=0,in=170] (6,0.2);
\node[below] at (1,0) {$\xi_{1}+\delta_1$};
\draw[dashed] [thick] (-3.5,0)-- (-3.5,-0.5);
\node[below] at (-3.5,-0.4) {$\xi_2$};
\draw[dashed] [thick] (-3,0.2)-- (-3,-0.5);
\node[below] at (-2.7,-0.4) {$\xi_2+\delta_5$};
\draw[dashed] [thick] (-4,0)-- (-4,1);
\node[above] at (-4,0.8) {$\xi_2-\delta_7$};
\draw[dashed] [thick] (1,0)-- (1,3);
\draw[dashed] [thick] (-1,0)-- (-1,-1.1);
\draw[dashed] [thick] (-1.7,-0.9)-- (-1.7,1.7);
\node[above] at (-1.7,1.55) {$\xi_1-\delta_4$};
\node[below] at (-0.7,-1) {$\xi_1-\delta_3$};
\draw [thin] (-3.85,-0.41) arc [radius=0.2, start angle=40, end angle= 170];
\node[above] at (-4,-0.43) {$\alpha_2$};
\draw [thin] (-2.8,0.255) arc [radius=0.2, start angle=30, end angle= 175];
\node[above] at (-3,0.3) {$\alpha_3$};
\draw [thin] (-1.5,-1.35) arc [radius=0.2, start angle=30, end angle= 190];
\node[above] at (-1.6,-1.3) {$\alpha_1$};
\node[above] at (2,2.3) {\Large{$R_u$}};
\node[above] at (2,1) {\Large{$R_v$}};
\end{tikzpicture}
\caption{$(R_u,R_v)$ for the case $0<b^*< 1$.}\label{Figure b<1}
\end{center}
\end{figure}

We consider $(R_u,R_v)(\xi)$ defined  as (see Figure \ref{Figure b<1})
\begin{equation*}
(R_u,R_v)(\xi):=\begin{cases}
(\varepsilon_1\sigma(\xi) e^{-\lambda_u\xi},\eta_1(\xi-\xi_1) e^{-\lambda_u\xi}),&\ \ \mbox{for}\ \ \xi\ge\xi_1+\delta_1,\\
(\varepsilon_2\sin(\delta_2(\xi-\xi_1+\delta_3)),\eta_2e^{\lambda_{1}\xi}),&\ \ \mbox{for}\ \ \xi_1-\delta_4\le\xi\le\xi_{1}+\delta_1,\\
(-\delta_u,\eta_2e^{\lambda_{1}\xi}),&\ \ \mbox{for}\ \ \xi_2+\delta_5\le\xi\le\xi_1-\delta_4,\\
(-\delta_u,\eta_3\sin(\delta_6(\xi-\xi_2))),&\ \ \mbox{for}\ \ \xi_2-\delta_7\le\xi\le\xi_2+\delta_5,\\
(-\delta_u,-\delta_v),&\ \ \mbox{for}\ \ \xi\le\xi_2-\delta_7,
\end{cases}
\end{equation*}
where $\xi_1>M_0$ and $\xi_2<-M_0$ are fixed points, and $\lambda_{1}$ satisfies
\begin{equation}\label{condition on lambda 3}
d\lambda_1^2+2\sqrt{1-a}\,\lambda_1-r(2+b^*)>0\ \ {\rm and}\ \ \lambda_1>\frac{r(\hat v+1)}{2\sqrt{1-a}}.
\end{equation}
Here
$\varepsilon_{1,2}>0$, $\eta_{1,2,3}>0$, and $\delta_{i=1,\cdots,7}>0$ are chosen as same as that in \S \ref{sec:3-1}.  Therefore, from \eqref{rho 2} and  $|R_u|,|R_v|\ll 1$, up to enlarging $M_0$, there exist $C_2>0$ and $C_3>0$ such that, for all $\xi\in(-\infty,\xi_2+\delta_5]$, it holds
\bea\label{xi 2 weak}
1-2U_*+R_u-aV_*-aR_v<-C_2<0,
\eea
and
\bea\label{xi 2 weak 2}
-(\hat v+\hat u\delta_0+C_3\rho)<1-2V_*-R_v-(b^*+\delta_0)U_*< C_3\rho.
\eea
Moreover, we set
\bea\label{delta u v}
\delta_u:=\varepsilon_2\sin(\delta_2(\delta_4-\delta_3))\quad\text{and}\quad \delta_v:=\eta_3\sin(\delta_6\delta_7),
\eea
which yield that $(R_u,R_v)(\xi)$ is continuous on $\mathbb{R}$. Furthermore, up to enlarging $M_0$ if necessary, we can set
\bea\label{rho 3}
\hat u-2\rho+a\delta_u\delta_v>\sqrt{a}b^*\hat u.
\eea

Note that, for the construction in \S \ref{sec:3-1} (see Step 2), we only set $|\delta_3-\delta_4|$  sufficiently small to obtain
$$|R_u(\xi)|\ll|R_v(\xi_2+\delta_5)|\ \ \text{in}\ \ [\xi_2+\delta_5,\xi_1-\delta_3].$$
 However, for the weak competition case,
we will subtly set $\delta_u$ and $\delta_v$ to satisfy
\bea\label{delta u v 1}
\delta_v=\frac{b^*}{\sqrt{a}}\delta_u,
\eea
which can be done by  adjusting $|\delta_3-\delta_4|$ and $|\delta_7|$.

Now, we define
\beaa
(W_u,W_v)(\xi):=\Big(\min\{(U_*-R_u)(\xi),1\},\max\{(V_*+R_v)(\xi),0\}\Big),
\eeaa
and
show that $(W_u,W_v)$
satisfies \eqref{tw super solution system}.
In fact, thanks to \eqref{rho 2} and the first condition in \eqref{condition on lambda 3},
for $\xi\in[\xi_1-\delta_4,\infty)$, $N_3[W_u,W_v]\le 0$ and $N_4[W_u, W_v]\ge 0$ follow from the same argument as that in \S \ref{sec:3-1}. Therefore, it suffices to deal with $\xi\in (-\infty,\xi_1-\delta_4]$.
Next, we divide the discussion into three steps as follows.

\noindent{\bf{Step 1:}} We consider $\xi\in[\xi_2+\delta_5,\xi_1-\delta_4]$ with $\xi_1>M_0$ fixed by the discussion similar to Step 1 and Step 2 in \S\ref{sec:3-1}.
In this case, we have
\beaa
(R_u,R_v)(\xi)=(-\delta_u,\eta_2e^{\lambda_1\xi}),
\eeaa
where $\lambda_1$ satisfies \eqref{condition on lambda 3},  $\delta_u=\delta_u(\varepsilon_1, |\delta_3-\delta_4|)$ is fixed as that in \eqref{delta u v}, and $\eta_2=\eta_2(\eta_1)$ is chosen like \eqref{eta 2 system}.
Note that $\delta_u\to 0$ as $|\delta_3-\delta_4|\to 0$, and thus
$$R'_u((\xi_1-\delta_{4})^+)>0=R'_u((\xi_1-\delta_{4})^-),\ {\it i.e.},\ \angle \alpha_1<180^{\circ}.$$

By some straightforward computations, since $R_v\ge 0$,
\bea\label{N3-step1 weak}
N_3[W_u,W_v]=\delta_u(1-2U_*-\delta_u-a(V_*+R_v))-aU_*R_v,
\eea
and $N_4[W_u,W_v]$ satisfies
\beaa
N_4[W_u,W_v]\geq \frac{r}{2}R_v+rV_*[(b^*+\delta_0)R_u-\delta_0U_*].
\eeaa

Note that $\lambda_1$ and $\eta_2$ have already been determined by the construction on $\xi\in[\xi_1-\delta_4,\infty)$. Since $\frac{\rho}{2}\le U_*\le \hat u$ and $\hat v\le V_*\le 1$ for $\xi\in[\xi_2+\delta_5,\xi_1-\delta_4]$, by setting $|\delta_3-\delta_4|$ small enough such that
\bea\label{delta u 2}
\delta_u<\min\Big\{\frac{a\rho}{2\hat u},\frac{1}{4b^*}\Big\}\eta_2e^{\lambda_1(\xi_2+\delta_5)},
\eea
we have
$$N_3[W_u,W_v]\le \delta_u(1-aV_*)-aU_*R_v\leq 0$$
 for $\xi\in[\xi_2+\delta_5,\xi_1-\delta_4]$.
Up to reducing $|\delta_3-\delta_4|$ and $\delta_0(\varepsilon_1,\eta_1,|\delta_3-\delta_4|)$ if necessary, $N_4[W_u,W_v]\geq 0$ follows immediately from  \eqref{delta u 2}. Thus, Step 1 is finished.

\medskip

\noindent{\bf{Step 2:}} We consider $\xi\in[\xi_2-\delta_7,\xi_2+\delta_5]$ with $\xi_2+\delta_5$ fixed by Step 1 and $\delta_7>0$ very small satisfying
\bea\label{cond delta 7 system}
\frac{2\sqrt{1-a}}{r\delta_7}-C_3\rho>2.
\eea
In this case, we have
\beaa
(R_u,R_v)(\xi)=\Big(-\delta_u,\eta_3\sin(\delta_6(\xi-\xi_2))\Big),
\eeaa
with
$\delta_6, \delta_7>0$ very small, and $\delta_5$ satisfying
\bea\label{cond delta 5 2}
\frac{r(\hat v+1))}{2\sqrt{1-a}}<\frac{1}{\delta_5}<\lambda_1.
\eea
It follows from the same argument as Claim \ref{cl 5} that, there exist $\eta_3=\eta_3(\eta_1,\delta_5,\delta_6)>0$ and small $\delta_6>0$
such that $R_v(\xi)$ is continuous at $\xi=\xi_2+\delta_5$ and $\angle\alpha_3<180^{\circ}$.

Note that, in this interval $N_3[W_u,W_v]$ still satisfies \eqref{N3-step1 weak}.
Then, from \eqref{xi 2 weak} and  $R_v(\xi)>0$ for $\xi\in[\xi_2,\xi_2+\delta_5]$, we have $N_3[W_u,W_v]\le 0$  for $\xi\in[\xi_2,\xi_2+\delta_5]$. 
We next deal with the inequality of $N_4[W_u,W_v]$. For $\xi\in[\xi_2,\xi_2+\delta_5]$, by applying the same argument as Step 5 in \S \ref{sec:3-1}, from \eqref{N4 inequality step 5}, \eqref{Rv-estimate2},  \eqref{xi 2 weak 2}, and \eqref{cond delta 5 2}, we have
\begin{equation*}
\begin{aligned}
N_4[W_u,W_v]\ge & rR_v\Big(1-2V_*-R_v-(b^*+\delta_0)U_*-\frac{d}{r}\delta_6^2\Big)+2\sqrt{1-a}\delta_6\eta_3\cos(\delta_6(\xi-\xi_2))\\
&+rV_*((b^*+\delta_0)R_u-\delta_0U_*)\\
\ge &r\Big(\frac{2\sqrt{1-a}}{r}\frac{1}{\delta_5}-\hat v-\hat u\delta_0-C_3\rho-\frac{d}{r}\delta_6^2\Big)R_v(\xi_2+\delta_5)\\
&-r(\hat v+\rho)(b^*+\delta_0)\delta_u-r\delta_0\\
\ge & r(1-\hat u\delta_0-C_3\rho-\frac{d}{r}\delta_6^2)R_v(\xi_2+\delta_5)-r(\hat v+\rho)(b^*+\delta_0)\delta_u-r\delta_0
\end{aligned}
\end{equation*}
Recall that,
$\delta_u\ll R_v(\xi_2+\delta_5)$ up to decreasing $|\delta_3-\delta_4|$, and $\delta_6$ can be chosen sufficiently small such that $\frac{d}{r}\delta_6^2<\frac{1}{4}$. Then, as long as  $M_0$ is chosen sufficiently large at the beginning, we have $N_4[W_u,W_v]\ge 0$ for $\xi\in[\xi_2,\xi_2+\delta_5]$, up to decreasing $\delta_0(\varepsilon_1,\eta_1,|\delta_3-\delta_4|)$ if necessary.

From now on, we fix $\delta_6$ and $|\delta_3-\delta_4|/\delta_7$ to get \eqref{delta u v 1}. For $\xi\in[\xi_2-\delta_7,\xi_2]$, since $-\delta_v\le R_v\le 0$,
from \eqref{rho 2}, \eqref{rho 3}, \eqref{delta u v 1}, \eqref{N3-step1 weak}, and $b^*<1$, we obtain that $N_3[W_u,W_v]\le 0$.
Since $R_v<0$ and $\xi_2<-M_0$, from \eqref{N4 inequality step 5}, \eqref{xi 2 weak 2}, and \eqref{cond delta 7 system}, we have
\begin{equation*}
\begin{aligned}
N_4[W_u,W_v]\ge & r(\frac{2\sqrt{1-a}}{r\delta_7}-C_3\rho)\delta_v-r(\hat v+\rho)(b^*+\delta_0)\delta_u-r\delta_0\hat u\\
\ge & 2r\delta_v-r(\hat v+\rho)(b^*+\delta_0)\delta_u-r\delta_0\hat u.
\end{aligned}
\end{equation*}
Then, from \eqref{delta u v 1}, as long as  $M_0$ is chosen sufficiently large at the beginning, we assert that $N_4[W_u,W_v]\ge 0$
up to decreasing $\delta_0(\varepsilon_1,\eta_1,|\delta_3-\delta_4|)$ if necessary.
Hence, the construction for Step 2 is finished. Hereafter, we fix $\delta_5$, and thus $\xi_2$.

\medskip

\noindent{\bf{Step 3:}} We consider $\xi\in(-\infty,\xi_2-\delta_7]$ with $\xi_2$ fixed by Step 2 and $\delta_7$ determined later.
In this case, we have
\beaa
(R_u,R_v)(\xi)=(-\delta_u,-\delta_v).
\eeaa
We first reduce $|\delta_3-\delta_4|$ and $\delta_7$ simultaneously to get $\delta_7\ll\frac{\pi}{2\delta_6}$, which implies
$$R'_v((\xi_2-\delta_7)^-)=0<R'_v((\xi_2-\delta_7)^+),\ {\it i.e.},\ \angle\alpha_2<180^{\circ}.$$

From now on, we fix $delta_7$. By applying the same argument as Step 2 above, $N_3[W_u,W_v]\le 0$ for $\xi\in(-\infty,\xi_2-\delta_7]$. Therefore, it suffices to verify the inequality of $N_4[W_u,W_v]$.  By some straightforward computations, from \eqref{rho 2}, we have
\begin{equation*}
\begin{aligned}
N_4[W_u,W_v]= & -rV_*(b^*+\delta_0)\delta_u-r\delta_v\Big[1-2V_*+\delta_v-(b^*+\delta_0)(U_*+\delta_u)\Big]-r\delta_0U_*V_*\\
\ge & -r(\hat v+\rho)(b^*+\delta_0)\delta_u-r\delta_v\Big[1-2\hat v+\delta_v-(b^*+\delta_0)(\hat u-\rho+\delta_u)\Big]-r\delta_0U_*V_*.
\end{aligned}
\end{equation*}
Then, from \eqref{delta u v 1} and $0<b^*<1$, we have $N_4[W_u,W_v]\ge 0$ up to decreasing $\delta_0(\varepsilon_1,\eta_1)$ if necessary. The construction for Step 3 is complete.

\bigskip
\subsection{Proof of Theorem~\ref{th:threshold}}

\noindent

We first prove Proposition~\ref{Prop-supersol-system}.

\begin{proof}[Proof of Proposition~\ref{Prop-supersol-system}]
Combining the construction of $(R_u,R_v)$ in \S~\ref{sec:3-1} and \S~\ref{sec:3-2},
we are now equipped with a super-solution
\beaa
(W_u,W_v)=(\min\{U_*-R_u,1\},\max\{V_*+R_v,0\})
\eeaa
satisfying \eqref{tw super solution system}.
Moreover, at the points of discontinuity of $W'_u$ and $W'_v$,
the corresponding one-sided derivatives have the right sign.
Therefore, we complete the proof of Proposition~\ref{Prop-supersol-system}.
\end{proof}

We are now ready to prove Theorem~\ref{th:threshold}.

\begin{proof}[Proof of Theorem~\ref{th:threshold}]
In view of Lemma~\ref{lm: existence b*}, it suffices to show that conditions 
(i), (ii), (iii) are equivalent.
We now deal with ${\rm(i)}\Leftrightarrow {\rm(ii)}$. To prove ${\rm(i)}\Rightarrow{\rm(ii)}$, we use the contradiction argument and assume that (ii) is not true, by Lemma~\ref{lm: behavior around + infty} (ii), we see that $U_*$ satisfies \eqref{AS-U-infty-for-contradiction} and thus Proposition~\ref{Prop-supersol-system} is available.

To reach a contradiction, we consider the Cauchy problem
\begin{equation}\label{system2}
\left\{
\begin{aligned}
&\partial_tu=u_{xx}+u(1-u-av), & t>0,\ x\in\mathbb{R},\\
&\partial_tv=dv_{xx}+rv(1-v-(b^*+\delta_0)u), & t>0,\ x\in \mathbb{R},\\
&u(0,x)=u_0(x),\quad  v(0,x)\equiv1, & \ x\in \mathbb{R},
\end{aligned}
\right.
\end{equation}
where $u_0(x)$ 
is the compactly supported continuous function. Additionally, we assume
$$\max_{x\in\mathbb{R}} |u_0(x)|<\frac{1-a}{1-ab^*}$$
and $\delta_0>0$ is sufficiently small such that
$b^*+\delta_0\neq 1$ if $b^*<1$.
By the definition of $b^*$, we see that the minimal traveling wave speed $c^{*}_{LV}(b^*+\delta_0)$ corresponding to the system \eqref{tw solution weak} with $b=b^*+\delta_0$ satisfies $c^{*}_{LV}(b^*+\delta_0)>2\sqrt{1-a}$.
Then, according to results from \cite{Lewis Li Weinberger 1, Lewis Li Weinberger 2},
the spreading speed of \eqref{system2} is
exactly $c^{*}_{LV}(b^*+\delta_0)$, strictly greater than $2\sqrt{1-a}$.

Let $(W_u, W_v)$ be constructed in Proposition~\ref{Prop-supersol-system}. Then, thanks to Proposition~\ref{Prop-supersol-system}, it is easy to see that
$(\overline{u},\underline{v})(t,x):=(W_u,W_v)(x-(2\sqrt{1-a})t-\eta)$, forms a super-solution for \eqref{system2} for all $t\ge 0$ and $x\in\mathbb{R}$, where $\eta\in\mathbb{R}$ is chosen large enough to have $\overline{u}(0,x)\geq u_0(x)$ and $\underline{v}(0,x)\leq v_0(x)$ for $x\in\mathbb{R}$.
By applying the comparison principle, we assert that the spreading speed of \eqref{system2} is smaller than or equal to $2\sqrt{1-a}$,
which reaches a contradiction.
The proof of ${\rm(i)}\Rightarrow {\rm(ii)}$ is finished.

Next, we show ${\rm(ii)}\Rightarrow {\rm(i)}$. Note that for $b>b^*$, the speed is nonlinearly selected, which together with Lemma~\ref{lm: behavior around + infty} implies that (ii) cannot hold. Therefore, it suffices to show that (ii) cannot happen with
$b<b^*$. We assume by contradiction that that there exists $b^\dag\in(0,b^*)$ such that
\beaa
U_{b^\dag}(\xi)=B^\dag e^{-\lambda_u \xi}+o(e^{-\lambda_u \xi})\quad \mbox{as $\xi\to+\infty$}
\eeaa
for some $B^\dag>0$.
In view of the asymptotic behavior of $(U_b,V_b)$ at $\pm\infty$ given in Section 2, we can define
\beaa
L^*:=\inf\{L\in\mathbb{R}|\, U_{*}(\xi-L)\geq U_{b^\dag}(\xi),\ V_{*}(\xi-L)\leq V_{b^\dag}(\xi),\ \forall \xi\in\mathbb{R}\}.
\eeaa
Note that, the discussion should be divided into several cases: $b^*>1$ and $b^\dag>1,=1,\text{or}<1$; $b^*=1$ and $b^\dag<1$; $b^*<1$ and $b^\dag<1$. But to define $L^*<\infty$ we only need $0<b^\dag<b^*$.

Next, we will apply the sliding method to reach a contradiction. By the continuity, we have
$$U_{b^*}(\xi-L^*)\geq U_{b^\dag}(\xi)\ \ \text{and}\ \ 1-V_{b^*}(\xi-L^*)\geq 1-V_{b^\dag}(\xi)\ \ \text{for all}\ \ \xi\in\mathbb{R}.$$
 If there exists $\xi^*\in\mathbb{R}$ such that
$$U_{b^*}(\xi^*-L^*)= U_{b^\dag}(\xi^*)\ \ \text{or}\ \ 1-V_{b^*}(\xi^*-L^*)=1- V_{b^\dag}(\xi^*),$$
by the strong maximum principle, we have $(U_{b^*},V_{b^*})(\xi-L^*)= (U_{b^\dag},V_{b^\dag})(\xi)$ for all $\xi\in\mathbb{R}$,
which is impossible since they satisfy different equations. Consequently,
\bea\label{U*>U+}
U_{b^*}(\xi-L^*)> U_{b^\dag}(\xi),\quad V_{b^*}(\xi-L^*)< V_{b^\dag}(\xi)\quad \mbox{ for all $\xi\in\mathbb{R}$}.
\eea

Furthermore, we claim that the touch point cannot happen at $-\infty$.
\begin{claim}\label{claim 1 to proof th 1.5}
It holds
\beaa
\text{{\rm(I)}}\lim_{\xi\to-\infty}\frac{1-U_{b^*}(\xi-L^*)}{1-U_{b^\dag}(\xi)}<1\quad\text{and}\quad\text{{\rm(II)}}\lim_{\xi\to-\infty}\frac{V_{b^*}(\xi-L^*)}{V_{b^\dag}(\xi)}<1.
\eeaa
\end{claim}
\begin{proof}
Without loss of generality, we only deal with the case $1<b^\dag<b^*$. The others, {\it i.e.}, $b^*>1$ and $b^\dag=1\ \text{or}<1$; $b^*=1$ and $b^\dag<1$; $b^*<1$ and $b^\dag<1$, can be proved by the same argument.
Recall that $\mu^+_u(c^*)>0$ and $\mu^+_v(c^*)>0$ defined as that in Lemma \ref{lem:AS-infty:b>1}.
Let us denote for simplicity that
\beaa
\mu_u=\mu^+_u(c^*),\quad  \mu_{v,1}=\mu^+_v(c^*,b^*),\quad  \mu_{v,2}=\mu^+_v(c^*,b^\dag).
\eeaa
Note that $\mu_u$ is independent on $b$.
Clearly, it follows from the definition of $\mu_{v,i}$, $i=1,2$, that $\mu_{v,1}>\mu_{v,2}$.
Then (II) immediately follows from Lemma \ref{lem:AS-infty:b>1}.

Next, we deal with (I).
First, we consider the case $\mu_u\ge \mu_{v,2}$. Since $\mu_{v,1}>\mu_{v,2}$,
(I) follows immediately from Lemma~\ref{lem:AS-infty:b>1} since $1-U_{b^*}(\xi)$ decays faster than
$1-U_{b^\dag}(\xi)$ as $\xi\to-\infty$.

For the case $\mu_{v,2}>\mu_u$,
we assume by the  contradiction that
$$\lim_{\xi\to-\infty}\frac{1-U_{b^*}(\xi-L^*)}{1-U_{b^\dag}(\xi)}=1.$$
Then from Lemma \ref{lem:AS-infty:b>1}, there exist $C_1,C_2>0$ satisfying $C_1=C_2e^{\mu_u L^*}$ such that
$$1-U_{b^*}(\xi)\sim C_1e^{\mu_u\xi}\ \ \text{and}\ \ 1-U_{b^\dag}(\xi)\sim C_2e^{\mu_u\xi}\ \ \text{as}\ \ \xi\to-\infty.$$
To reach a contradiction,
we set
\beaa
U_1(\xi)=(1-U_{b^\dag}(\xi))-(1-U_{b^*}(\xi-L^*)), \quad V_1(\xi):= V_{b^\dag}(\xi)-V_{b^*}(\xi-L^*).
\eeaa
Then, by \eqref{U*>U+},
$U_1(\xi)>0$ and $V_1(\xi)>0$ for all $\xi\in\mathbb{R}$, Moreover, $U_1$ satisfies
\bea\label{U1-eq}
U_1''+c^* U_1'-U_1+g_1+g_2=0,\quad \xi\in\mathbb{R}.
\eea
where
\beaa
&&g_1(\xi)=[2-U_{b^\dag}(\xi)-U_{b^*}(\xi-L^*)-aV_{b^*}(\xi-L^*)]U_1(\xi),\\
&&g_2(\xi)=aU_{b^\dag}(\xi)V_1(\xi).
\eeaa
It is clear that $g_1(\xi)=o(U_1(\xi))$ as $\xi\to-\infty$. Next, we show that  $g_2(\xi)=o(U_1(\xi))$ holds as $\xi\to-\infty$.

By using
$V_1(\xi)\sim C_3e^{\mu_{v,2}\xi}$ (for some $C_3>0$) as $\xi\to-\infty$,
there exist $\kappa_1,\kappa_2>0$ and $\mu_0\ge \mu_{v,2}$ such that
\bea\label{kappa 12}
\kappa_2e^{\mu_{0}\xi}\le g_2(\xi)\le \kappa_1e^{\mu_0\xi}.
\eea
We now assume by contradiction that there exists $\{\xi_n\}$ with $\xi_n\to-\infty$ as $n\to\infty$ such that for some $\kappa_3>0$,
\bea\label{kappa3 1}
\frac{g_2(\xi_n)}{U_1(\xi_n)}\ge \kappa_3\ \ \text{for all}\ \ n\in\mathbb{N}.
\eea
 Set $U_1(\xi)=\alpha(\xi)e^{\mu_0\xi}$, where $\alpha(\xi)>0$ for all $\xi$.
By substituting it into \eqref{U1-eq},
we have
\bea\label{alpha-eq 1}
L(\xi):=\Big(\alpha''(\xi)+(2\mu_0+c^*\mu_0)\alpha'(\xi)+(\mu_0^2+c^*\mu_0-1)\alpha(\xi)\Big)e^{\mu_0\xi}+g_1(\xi)+g_2(\xi)=0
\eea
for $\xi\approx-\infty$.
By \eqref{kappa 12} and \eqref{kappa3 1}, we have
\bea\label{alpha-bdd 1}
0<\alpha(\xi_n)\leq \frac{\kappa_1}{\kappa_3}\quad  \mbox{for all $n\in\mathbb{N}$.}
\eea
Now, we will reach a contradiction by dividing the behavior of $\alpha(\cdot)$ into two cases:
\begin{itemize}
    \item[(a)] $\alpha(\xi)$ oscillates for all large $\xi$;
    \item[(b)] $\alpha(\xi)$ is monotone for all large $\xi$.
\end{itemize}
For case (a), there exist local minimum points $\eta_n$ of $\alpha$ with $\eta_n\to\infty$ as $n\to\infty$ such that
\beaa
\alpha(\eta_n)>0,\quad \alpha'(\eta_n)=0,\quad \alpha''(\eta_n)\geq 0\quad  \mbox{for all $n\in\mathbb{N}$.}
\eeaa
Together with \eqref{kappa 12} and $g_1(\xi)=o(U_1(\xi))$,
from \eqref{alpha-eq 1} we see that
\beaa
0=L(\eta_n)\geq (\mu_0^2+c^*\mu_0-1)\alpha(\eta_n)e^{\mu_0\eta_n}+o(1)\alpha(\eta_n)e^{\mu_0\eta_n}+\kappa_2e^{\mu_0\eta_n}>0
\eeaa
for all large $n$, which reaches a contradiction since $\mu_0\ge \mu_{v,2}>\mu_u$.

For case (b),
due to \eqref{alpha-bdd 1}, there exists $\alpha_0\in[0, \kappa_1/\kappa_3]$
such that $\alpha(\xi)\to \alpha_0$ as $\xi\to\infty$. Hence, we can find subsequence $\{\eta_j\}$ that tends to $\infty$ such that $\alpha'(\eta_j)\to0$, $\alpha''(\eta_j)\to0$ and
$\alpha(\eta_j)\to \alpha_0$ as $n\to\infty$.
From \eqref{alpha-eq 1} we deduce that
\beaa
0=L(\eta_j)\geq (o(1)+(\mu_0^2+c^*\mu_0-1)\alpha(\eta_j)+ \kappa_2)e^{\mu_0\eta_j}>0
\eeaa
for all large $j$, which reaches a contradiction.
Therefore, we have proved that
$g_2(\xi)=o(U_1(\xi))$ as $\xi\to-\infty$.
Consequently, we have
\bea\label{J-small o 1}
g_1(\xi)+g_2(\xi)=o(U_1(\xi))\quad \mbox{as $\xi\to-\infty$.}
\eea

Thanks to \eqref{J-small o 1}, we can apply
\cite[Chapter 3, Theorem 8.1]{CoddingtonLevison} to assert that
 the asymptotic behavior of $U_1(\xi)$ at $\xi=-\infty$ satisfies $U_1(\xi)\sim e^{\mu_u \xi}$ which contradicts with $C_1=C_2e^{\mu_u L^*}$. 
 The proof of Claim \ref{claim 1 to proof th 1.5} is complete.
\end{proof}

Now, we are ready to prove that the touch point always happens on $U$-equation at $+\infty$.
\begin{claim}\label{claim 2 of proof of th 1.5}
It holds
\beaa
\lim_{\xi\to+\infty}\frac{U_{b^*}(\xi-L^*)}{U_{b^\dag}(\xi)}=1.
\eeaa
\end{claim}
\begin{proof}
Let $\lambda_v^-(c^*)<0$ be defined as in Lemma \ref{lm: behavior around + infty}. For the case $\lambda_v^-(c^*)\le -\sqrt{1-a}$,
we are going to prove
\bea\label{aaaaa}
\lim_{\xi\to+\infty}\frac{U_{b^*}(\xi-L^*)}{U_{b^\dag}(\xi)}>1\Longrightarrow \lim_{\xi\to+\infty}\frac{1-V_{b^*}(\xi-L^*)}{1-V_{b^\dag}(\xi)}>1.
\eea
We divide our discussion into three cases:
\begin{itemize}
    \item[(1)] if $\lambda_v^-(c^*)< -\sqrt{1-a}$, then by Lemma \ref{lm: behavior around + infty},
    we see that $U_b(\xi)$ and $1-V_b(\xi)$ have the same decay rate at $+\infty$ and there exists a positive constant $A_1$ such that
    \beaa
    \lim_{\xi\to+\infty}\frac{U_{b}(\xi)}{1-V_b(\xi)}=A_1.
    \eeaa
    Therefore, we have
    \beaa
     \lim_{\xi\to+\infty}\frac{1-V_{b^*}(\xi-L^*)}{1-V_{b^\dag}(\xi)}&=&
     \lim_{\xi\to+\infty}\Big[\frac{1-V_{b^*}(\xi-L^*)}{U_{b^*}(\xi-L^*)}\frac{U_{b^*}(\xi-L^*)}{U_{b^\dag}(\xi)}
     \frac{U_{b^\dag}(\xi)}{1-V_{b^\dag}(\xi)}\Big]\\
     &=&\frac{1}{A_1}\Big(\lim_{\xi\to+\infty}\frac{U_{b^*}(\xi-L^*)}{U_{b^\dag}(\xi)}\Big)A_1>1.
    \eeaa
    Hence \eqref{aaaaa} holds.

    \item[(2)] if $\lambda_v^-(c^*)= -\sqrt{1-a}$, then by Lemma \ref{lm: behavior around + infty}, there exists a positive constant $A_2$ such that
    \beaa
    \lim_{\xi\to+\infty}\frac{\xi U_{b}(\xi)}{1-V_b(\xi)}=A_2.
    \eeaa
     Therefore, we have
    \beaa
     \lim_{\xi\to+\infty}\frac{1-V_{b^*}(\xi-L^*)}{1-V_{b^\dag}(\xi)}&=&
     \lim_{\xi\to+\infty}\Big[\frac{1-V_{b^*}(\xi-L^*)}{\xi U_{b^*}(\xi-L^*)}\frac{ U_{b^*}(\xi-L^*)}{U_{b^\dag}(\xi)}
     \frac{\xi U_{b^\dag}(\xi)}{1-V_{b^\dag}(\xi)}\Big]\\
     &=&\frac{1}{A_2}\Big(\lim_{\xi\to+\infty}\frac{U_{b^*}(\xi-L^*)}{U_{b^\dag}(\xi)}\Big)A_2>1,
    \eeaa
    which yields \eqref{aaaaa}.

    \item[(3)] if $\lambda_v^-(c^*)> -\sqrt{1-a}$, we assume by contradiction that
\bea\label{case-iii-lim}
\lim_{\xi\to+\infty}\frac{1-V_{b^*}(\xi-L^*)}{1-V_{b^\dag}(\xi)}=1.
\eea
Then from Lemma \ref{lm: behavior around + infty} and \eqref{case-iii-lim}, there exist $C_1,C_2>0$ satisfying $C_1=C_2e^{\lambda_v^-(c^*) L^*}$ such that
$$1-V_{b^*}(\xi)\sim C_1e^{\lambda_v^-(c^*)\xi}\ \ \text{and}\ \ 1-V_{b^\dag}(\xi)\sim C_2e^{\lambda_v^-(c^*)\xi}.$$
To reach a contradiction, similar to the proof of Claim~\ref{claim 1 to proof th 1.5}, we set
\beaa
U_1(\xi):=U_{b^*}(\xi-L^*)-U_{b^\dag}(\xi), \quad V_1(\xi):= (1-V_{b^*}(\xi))-(1-V_{b^\dag}(\xi-L^*)).
\eeaa
Considering the equation satisfied by the positive function $V_1$:
\beaa
cV_1'+dV_1''-rV_1+h_1(\xi)+h_2(\xi)=0,\quad \xi\in\mathbb{R},
\eeaa
where
\beaa
&&h_1(\xi)=r[2-V_{b^*}(\xi)-V_{b^\dag}(\xi-L^*)]V_1(\xi),\\
&&h_2(\xi)=r bV_{b^\dag}(\xi-L^*)U_1(\xi).
\eeaa
Using a similar argument as in Claim \ref{claim 1 to proof th 1.5}, we can reach a contradiction, and thus \eqref{aaaaa} holds.
\end{itemize}



As a result, if Claim \ref{claim 2 of proof of th 1.5} is not true,
from Claim \ref{claim 1 to proof th 1.5} and \eqref{aaaaa}, 
it is easy to see that there exists
$\varepsilon>0$ sufficiently small such that
$U_{b^*}(\xi-(L^*+\varepsilon))> U_{b^{\dag}}(\xi)$ for $\xi\in\mathbb{R}$,
which contradicts the definition of $L^*$. Therefore,
the proof of Claim \ref{claim 2 of proof of th 1.5} is finished.
\end{proof}

Now, we are ready to finish the proof of ${\rm(ii)}\Rightarrow {\rm(i)}$
by the help of Claim \ref{claim 1 to proof th 1.5} and  Claim  \ref{claim 2 of proof of th 1.5}.
For this, we set
\beaa
U_2(\xi):=U_{b^*}(\xi-L^*)-U_{b^\dag}(\xi), \quad V_2(\xi):= (1-V_{b^*}(\xi))-(1-V_{b^\dag}(\xi-L^*)).
\eeaa
Then we focus on the equation satisfied by $U_2$ and
use a similar argument as in Claim \ref{claim 1 to proof th 1.5}, we can again reach a contradiction.
Consequently, we obtain
${\rm(ii)}\Rightarrow {\rm(i)}$.  

\medskip

Finally, we prove ${\rm(i)}\Leftrightarrow {\rm(iii)}$.
In view of Proposition~\ref{prop: implicit cond}, we have
\begin{itemize}
\item[(1)] $\int_{-\infty}^{\infty} e^{\lambda_u\xi} U_b(\xi)[a(1-V_b)-U_b](\xi)d\xi=0\quad \mbox{ for }\quad b=b^*$;
\item[(2)] $\int_{-\infty}^{\infty} e^{\lambda_u\xi} U_b(\xi)[a(1-V_b)-U_b](\xi)d\xi\neq0\quad \mbox{ for }\quad b<b^*$.
\end{itemize}
It suffices to prove that 
\bea\label{goal-iii}
\int_{-\infty}^{\infty} e^{\lambda_u\xi} U_b(\xi)[a(1-V_b)-U_b](\xi)d\xi\neq0\quad \text{ for }\quad b>b^*.
\eea
Since $b>b^*$, we have $c^*_{LV}(b)>2\sqrt{1-a}$. In this case, the minimal traveling wave $U_b(\xi)$ exhibits fast decay as $\xi\to+\infty$ (see \cite{Kan-On} or \cite[Lemma 2.3]{DuWangZhou}). Specifically, we have 
$U_b(\xi)\sim e^{-\lambda_u^+\xi}$ as $\xi\to\infty$, where 
\beaa
\lambda_u^+=\frac{c^*_{LV}(b)+\sqrt{(c^*_{LV}(b))^2-4(1-a)}}{2}>\sqrt{1-a}=\lambda_u.
\eeaa
Hence, the bilateral Laplace transform of $U_b$ is well-defined for $ -\lambda_u^+<{\rm Re} \lambda<0$, given by 
\beaa
\mathcal{L}(\lambda):=\int_{-\infty}^{+\infty}e^{-\lambda \xi}U_b(\xi) d\xi, \quad -\lambda_u^+<{\rm Re} \lambda<0.
\eeaa
By the equation satisfied by $U_b$ and integration by parts, we have
\bea\label{Phi-eq-2}
\int_{-\infty}^{\infty}e^{-\lambda \xi}U_b[a(1-V_b)-U_b](\xi)d\xi=- \Phi(\lambda)\mathcal{L}(\lambda),\quad -\lambda_u^+<{\rm Re} \lambda<0,
\eea
where
\beaa
\Phi(\lambda):=c_{LV}^*(b)\lambda+\lambda^2+1-a.
\eeaa
In particular, since $0<\lambda_u<\lambda_u^+$, we may substitute $\lambda=-\lambda_u$ into \eqref{Phi-eq-2} to obtain
 \beaa
\int_{-\infty}^{\infty}e^{\lambda_u\xi}U_b[a(1-V_b)-U_b](\xi)d\xi=- \Phi(-\lambda_u)\mathcal{L}(-\lambda_u)> 0,
\eeaa
since $\Phi(-\lambda_u)<0$ and $\mathcal{L}(-\lambda_u)>0$. Therefore, \eqref{goal-iii} holds.

This completes the proof of Theorem \ref{th:threshold}.
\end{proof}

\bigskip


\section{Classification of the traveling waves}

In this section, we conclude the main results of this paper and complete the proof of Theorem \ref{th: classification scalar nonlocal} and Theorem \ref{th: classification}.

In Section \ref{sec:threshold scalar}, we studied the process of how the linear selection on speed 
transitions to  nonlinear selection by considering the scalar local diffusion equation
$$w_t=w_{xx}+f(w;s)$$
with a family of continuously increasing nonlinearity $f(w;s)$ satisfying assumptions (A1)-(A5).
The characteristic equation 
$\lambda^2-c\lambda+f'(0;s)=0$, derived from the linearization of 
$$W''+cW'+f(W;s)=0$$ at the 
unstable state $W=0$, admits
\begin{itemize}
\item One double root $\lambda=\sqrt{f'(0)}$ if $c=c^*(s)=2\sqrt{f'(0)}$,
\item Two simple roots
$$\lambda_s^{\pm}=\frac{c\pm\sqrt{c^2-4f'(0)}}{2}\ \ \text{if}\ \ c\ge c^*(s)>2\sqrt{f'(0)}.$$
\end{itemize}

For the case $s\in (0,s^*]$, the spreading speed is linearly selected, namely $c^*(s)=2\sqrt{f'(0)}$.
By the classical ODE argument (see, e.g., \cite{Aronson Weinberger}),  the asymptotic behavior of the pulled front is given by the linear combination of 
$\xi e^{-\sqrt{f'(0)}\xi}$ and  $e^{-\sqrt{f'(0)}\xi}$. More importantly, we 
proved
that the decay rate of the minimal traveling wave changes to $e^{-\sqrt{f'(0)}\xi}$ if and only if $s=s^*$ which is the threshold between 
linear 
and nonlinear selection on speed. 
On the other hand, for $s>s^*$, the spreading speed is nonlinearly selected $c^*(s)>2\sqrt{f'(0)}$. It has been 
proved in \cite{Aronson Weinberger} by the basic phase plane analysis that the asymptotic behavior of the pushed front is given by the fast decay $e^{-\lambda_s^+\xi}$, {\it i.e.}. Furthermore, for $c>c^*(s)$, it follows from the basic sliding method that the asymptotic behavior is given by the slow decay $e^{-\lambda_s^-\xi}$, {\it i.e.}, (3) in Proposition \ref{prop: classification scalar}.

In Section \ref{sec:threshold-system}, we studied the transition between linear selection and nonlinear selection on speed for the Lotka-Volterra competition system
\begin{equation*}
\left\{
\begin{aligned}
&u_t=u_{xx}+u(1-u-av),\\
&v_t=dv_{xx}+rv(1-v-bu).
\end{aligned}
\right.
\end{equation*}
Note that,  in Remark \eqref{rm: c nonlinear selection}, by numerical simulation, we established two conditions under which the speed is nonlinearly selected for certain values of $0<a,b<1$. Therefore, to fully capture the entire process of how the speed transitions from linear selection to nonlinear selection, it is crucial to consider this problem within the extended parameter range of $0<a<1$ and $b>0$, rather than just $0<a<1$ and $b>1$ (the so-called strong-weak competition case).

Note that the characteristic equation $\lambda^2-c\lambda+1-a=0$, derived from the linearization 
$$U''+cU'+U(1-U-aV)=0$$
at the 
unstable state $(U,V)= (0,1)$, admits
\begin{itemize}
\item One double root $\lambda=\sqrt{1-a}$ if $c=c^*_{LV}(b)=2\sqrt{1-a}$,
\item Two simple roots
\bea\label{two roots LV}
\lambda_u^{\pm}=\frac{c\pm\sqrt{c^2-4(1-a)}}{2}\ \ \text{if}\ \ c\ge c_{LV}^*(b)>2\sqrt{1-a}.
\eea
\end{itemize}

For the case $b\in (0,b^*]$,  we have $c^*_{LV}(b)=2\sqrt{1-a}$, {\it i.e.}, the spreading speed is linearly selected. The asymptotic behavior of the pulled front is given by the linear combination of 
$\xi e^{-\sqrt{1-a}\xi}$ and  $e^{-\sqrt{1-a}\xi}$. Importantly, we proved that the decay rate of the minimal traveling wave transits from $\xi e^{-\sqrt{1-a}\xi}$
to $e^{-\sqrt{1-a}\xi}$ 
as $b$ evolves to $b^*$ from the left hand side,
which is the threshold between 
linear speed selection and nonlinear speed selection.
On the other hand, for $b>b^*$, the spreading speed is nonlinearly selected.
Namely, $c_{LV}^*(b)>2\sqrt{1-a}$.
By super and sub-solution argument, we will show in \S\ref{sec:pushed front LV} that the asymptotic behavior of the pushed front is given by the fast decay $e^{-\lambda_u^+\xi}$ (see (2) in Theorem \ref{th: classification}). Furthermore, for $c>c^*_{LV}(b)$, it follows from the standard 
sliding method that the asymptotic behavior is given by the slow decay $e^{-\lambda_u^-\xi}$, {\it i.e.}, (3) in Theorem \ref{th: classification}. The proof will also be given in \S\ref{sec:pushed front LV}.

In Section \ref{sec: threshold scalar nonlocal}, we extended our observation to the integro-differential equation which has a nonlocal diffusion kernel
$$w_t=J\ast w-w+f(w;q)$$
with a family of continuously increasing nonlinearity $f(w;q)$ satisfying assumptions (A1)-(A3) and (A6)-(A7). Different with the local diffusion equation, the linearly selected speed is given by a variational formula
$$c_0^*:=\min\frac{1}{\lambda}\Big(\int_{\mathbb{R}}J(x)e^{\lambda x}dx+f'(0;q)-1\Big),$$
which is also derived from the linearization of 
$$J\ast \mathcal{W}+c\mathcal{W}'+f(\mathcal{W};q)=0$$ at the 
unstable state $\mathcal{W}= 0$. Furthermore, since the function
$$h(\lambda):=\int_{\mathbb{R}}J(x)e^{\lambda x}dx+f'(0;q)-1$$
is positive and strictly convex,  the characteristic equation
$c\lambda=\int_{\mathbb{R}}J(x)e^{\lambda x}dx+f'(0;q)-1$ admits
\begin{itemize}
\item One double root $\lambda=\lambda_0$ if $c=c^*_{NL}( q)=c_0^*$,
\item Two simple roots $\lambda^{\pm}_q(c)$ satisfying
\begin{equation}\label{two roots}
0<\lambda^-_q(c)<\lambda_0<\lambda^+_q(c)\ \ \text{if}\ \ c\ge c^*_{NL}(q)>c_0^*.
\end{equation}
\end{itemize}

For the case $q\in (0,q^*]$, the spreading speed is linearly selected $c^*_{NL}(q)=c_0^*$.
By Ikehara's Theorem,  the asymptotic behavior of the pulled front is given by the linear combination of 
$\xi e^{-\lambda_0\xi}$ and  $e^{-\lambda_0\xi}$. 
We established results parallel to those of the scalar reaction-diffusion equation. More precisely,
we found that the decay rate of the minimal traveling wave changes to $e^{-\lambda_0\xi}$ if and only if $q=q^*$, which establishes the difference between the pulled front and the pulled-to-pushed front. Furthermore, if $c>c^*_{NL}$, Coville et al. showed in \cite{Coville} that the asymptotic behavior is given by the slow decay $e^{-\lambda_q^-\xi}$, {\it i.e.}, (3) in Theorem \ref{th: classification scalar nonlocal}. However, when the spreading speed is nonlinearly selected $c_{NL}^*(s)>c_0^*$,
the asymptotic behavior of the pushed front 
remains an open problem in the literature. We will prove that pushed front always decays with the fast rate $e^{-\lambda_q^+\xi}$ (see (2) in Theorem \ref{th: classification scalar nonlocal}) in \S\ref{sec: pushed front nonlocal}.

\subsection{The asymptotic behavior of the pushed front of the nonlocal diffusion equation}\label{sec: pushed front nonlocal}
This subsection is devoted to completing the proof of Theorem \ref{th: classification scalar nonlocal}. We show that the asymptotic behavior of the pushed front is also given by the fast decay $e^{-\lambda^+\xi}$. As a matter of fact, if the pushed front decays with the slow rate $e^{-\lambda^-\xi}$, then we can always construct a traveling wave solution with speed $c<c^*_{NL}$, which contradicts the definition of the minimal speed $c^*_{NL}$.

Hereafter, we always assume $c^*_{NL}>c^*_0$, and denote the pushed front by $\mathcal{W}_*(\xi)$ and $c^*=c^*_{NL}$ for simplicity.
Then by assuming
\bea\label{assume}
\mathcal{W}_*(\xi)\sim A_0e^{-\lambda^-\xi},
\eea
in which $\lambda^-$ is the smaller root of \eqref{two roots} with $c=c^*$,
we can find a sup-solution $\overline{\mathcal{W}}(\xi)$ of
\begin{equation}\label{N 1}
\mathcal N_1[\mathcal W]:=J\ast \mathcal W+(c^*-\delta_0)\mathcal W'-\mathcal W+f(\mathcal W)=0.
\end{equation}
As a result, we can assert that the propagation speed of the corresponding Cauchy problem is at most $c^*-\delta_0$, and get the contradiction.
\begin{proposition}\label{prop: super sol nl}
Let $\mathcal W_*$ be the minimal traveling wave solution satisfying \eqref{scalar nonlocal tw-parameter s} with $c^*_{NL}>c^*_0$.
Assume that $\mathcal W_*(\xi)\sim A_0e^{-\lambda^-(c)(\xi)}$ as $\xi\to+\infty$. Then, there exists a small $\delta_0>0$, such that the propagation speed of 
$$w_t=J\ast w-w+f(w), \ \ t>0,\ x\in\mathbb{R},$$
starting from a compactly supported initial datum, is at most $c^*_{NL}-\delta_0$. This contradicts  the well-known result that the propagation speed is equal to the minimal traveling wave $c^*_{NL}$. As a result, $\mathcal W_*(\xi)\sim e^{-\lambda^+(c)(\xi)}$ as $\xi\to+\infty$.
\end{proposition}

\subsubsection{Construction of the super-solution}
We first construct the super-solution of \eqref{N 1} which satisfies $\mathcal N_1[\overline{\mathcal W}]\le 0$. The construction is similar to the super-solution \eqref{definition of Rw nonlocal} provided in \S \ref{subsec-3-1}.

Let $\xi_1,\xi_2$ be chosen like that in Lemma \ref{lm: divide R to 3 parts}.
We consider a super-solution in the form of (see Figure \ref{Figure sup solution})
\begin{equation*}
\overline{\mathcal W}(\xi)=\left\{
\begin{aligned}
&\overline{\mathcal W}_1(\xi):=\varepsilon_1e^{-\lambda_0\xi},&\ \ \text{for}\ \ \xi\ge \xi_1,\\
&\overline{\mathcal W}_2(\xi):=\mathcal{W}_*(\xi)-\mathcal{R}_w(\xi),&\ \ \text{for}\ \ \xi< \xi_1,
\end{aligned}
\right.
\end{equation*}
in which $\lambda_0\in(\lambda^-(c^*-\delta_0),\lambda^+(c^*-\delta_0))$ is the double root obtained in Remark \ref{rm:lambda_0}. $\mathcal{W}_*$ is the pushed front satisfying \eqref{assume},  and $\mathcal R_w(\xi)$ defined as

\begin{equation*}
\mathcal{R}_w(\xi)=\begin{cases}
\mathcal{R}_1(\xi):=-\varepsilon_2 \Psi(\xi-\xi_1+\frac{L^*}{2}),&\ \ \mbox{for}\ \ \xi_{2}\le\xi\le\xi_{1},\\
\mathcal{R}_2(\xi):=-\varepsilon_3e^{\lambda_1\xi},&\ \ \mbox{for}\ \ \xi\le\xi_2.
\end{cases}
\end{equation*}
Here $\Psi>0$ is the eigenfunction defined on $[-L^*,L^*]$ as \eqref{eigenvalue} with $c^*_0$ replaced by $c^*$. Since $\nu_0\to 0$ and $\Psi(\xi)\to -\mathcal W'_*(\xi)$ uniformly as $L^*\to\infty$, we choose sufficiently large $L^*$ such that 
\begin{equation}\label{est Psi' 1}
\Psi(\xi)\sim K_0 e^{-\lambda^-\xi}\ \ \text{and}\ \ \Psi'(\xi)\sim -\lambda^-K_0 e^{-\lambda^-\xi}\ \ \text{for}\ \ \xi\in[\frac{L^*}{4}-L,\frac{L^*}{2}+L],
\end{equation}
where $[-L,L]$ is the support of $J$.
Hereafter, we always fix $\xi_1-\xi_2=L^*/4$. Moreover, we should choose very small
$\varepsilon_{2,3}>0$ such that $\overline{\mathcal{W}}(\xi)$ is continuous for all $\xi\in\mathbb{R}$.

\begin{figure}
\begin{center}
\begin{tikzpicture}[scale = 1.25]
\draw[thick](-6,0) -- (4,0) node[right] {$\xi$};
\draw [thick] (-6, 2.5) to(-4,2.5) to [out=-30, in=170] (1.5,1)  to [out=-40,in=170] (4,0.2);
\node[below] at (1.5,0) {$\xi_{1}$};
\node[left] at (-6,2.5) {$1$};
\draw[dashed]  (-6,2.5) -- (4,2.5);
\draw[dashed] [thick] (1.5,0)-- (1.5,0.5);
\end{tikzpicture}
\caption{the super-solution $\overline{\mathcal{W}}(\xi)$.}\label{Figure sup solution}
\end{center}
\end{figure}

Since $f(\cdot)\in C^2$, there exists $K_1>0$  such that
\bea\label{K1K2 nonlocal pushed}
|f'({\mathcal{W}}_*(\xi))|< K_1\quad \mbox{for all}\quad\xi\in\mathbb{R}.
\eea
We set $\lambda_1>0$
large enough such that
\begin{equation}\label{condition on lambda 1 nonlocal pushed}
\lambda_1>\max\{\frac{4K_1}{c^*},\frac{K_1+1}{c^*}\}.
\end{equation}
Furthermore, there exists $K_2>0$ such that
\begin{equation*}
f'({\mathcal{W}}_*(\xi))\le -K_2<0\quad\text{for all}\quad\xi\le \xi_2.
\end{equation*}
Without loss of generality, we assume $J\ge0$ on $[-L,L]$, and $J=0$ for $x\in(-\infty,-L)]\cup[L,\infty)$.
Let $\mu_0$ be the unique positive root obtained from Proposition \ref{prop: asy tw - infty nonlocal} with $c=c^*$.
Additionally, we set
\beaa\label{lambda 2 nonlocal pushed}
0<\lambda_1<\mu_0\ \ \text{and}\ \ 1+K_2-e^{\lambda_1L}-c^*\lambda_1>0.
\eeaa

We now divide the proof into 3 steps as \S\ref{subsec-3-1}. For $\xi\in (-\infty, \xi_1]$, the construction of $\overline{\mathcal{W}}$ is absorutly same as that in \S\ref{subsec-3-1}. Therefore, in the rest, we only need to verify the super-solution for $\xi\in[\xi_1,\infty)$.
\medskip

We consider $\xi\in[\xi_1,\infty)$
In this case, we have
$\overline{\mathcal{W}}(\xi)=\varepsilon_1e^{-\lambda_0\xi}$
for some large $\varepsilon_1$ satisfying
\begin{equation}\label{varepsilon 1 pushed}
\varepsilon_1e^{-\lambda_0\xi_1}>\mathcal{W}_*(\xi_1)
\end{equation}
and $\lambda_0\in(\lambda^-(c^*-\delta_0),\lambda^+(c^*-\delta_0))$, where $\lambda^{\pm}(c^*-\delta_0)$ is defined in \eqref{two roots}.

By Lemma \ref{lm: divide R to 3 parts} and some  straightforward computations,
we have
\begin{equation*}
\begin{aligned}
\mathcal{N}_1[\overline{\mathcal{W}}_1]=&\int_{\mathbb{R}}J(y)\varepsilon_1e^{-\lambda_0(\xi-y)}dy-\varepsilon_1e^{-\lambda_0\xi}-\lambda_0(c^*-\delta_0)\varepsilon_1e^{-\lambda_0\xi}+f(\varepsilon_1e^{-\lambda_0\xi})\\
=&\varepsilon_1e^{-\lambda_0\xi}\Big(\int_{\mathbb{R}}J(y)e^{\lambda_0y}dy-1-\lambda_0(c^*-\delta_0)+f'(0)+o(1)\Big).
\end{aligned}
\end{equation*}
Since $c^*>c^*_0$, by setting $\delta_0<c^*-c^*_0$, we have
$$\lambda_0\in(\lambda^-(c^*-\delta_0),\lambda^+(c^*-\delta_0)),$$
which implies
$$\int_{\mathbb{R}}J(y)e^{\lambda_0y}dy-1-\lambda_0(c^*-\delta_0)+f'(0)+o(1)<0.$$
Therefore, $\mathcal{N}_1[\overline{\mathcal{W}}_1]\leq 0$ for $\xi\ge \xi_1$.

The rest of Step 1 devotes to the verification $\mathcal{N}_1[\overline{\mathcal{W}}]\le 0$ for $\xi\in [\xi_1,\xi_1+L]$, where $\overline{\mathcal{W}}_2$ defined on $(-\infty,\xi_1]$ is also involved in the computation. From Remark \ref{rm: glue}, it suffices to show that
$\overline{\mathcal{W}}_1\ge \overline{\mathcal{W}}_2$ for $\xi\in[\xi_1-L,\xi_1]$ and $\overline{\mathcal{W}}_1\le \overline{\mathcal{W}}_2$ for $\xi\in[\xi_1,\xi_1+L]$. 

From now on, we fix $\xi_1$, and choose a very large $\varepsilon_1$ such that \eqref{varepsilon 1 pushed} holds. To make sure that $\overline{\mathcal{W}}$ is continuous at $\xi_1$,  we set
\begin{equation}\label{epsilon 2 nl 1}
\varepsilon_2=\varepsilon_2(\varepsilon_1)=\frac{\varepsilon_1 e^{-\lambda_0\xi_1}-\mathcal W_*(\xi_1)}{\Psi(L^*/2)}
\end{equation}
where $\Psi(L^*/2)=K_0e^{-\frac{\lambda^-L^*}{2}}$.  Recall that
$\overline{\mathcal{W}}'_1=-\lambda_0 \overline{\mathcal{W}}_1$ and $\overline{\mathcal{W}}'_2=-\lambda^-\overline{\mathcal{W}}_2$ from \eqref{est Psi' 1} and \eqref{assume}. By $\overline{\mathcal{W}}_1=\overline{\mathcal{W}}_2$ at $\xi_1$ and $\lambda_0>\lambda^-$, we assert that $\overline{\mathcal{W}}_1\ge \overline{\mathcal{W}}_2$ for $\xi\in[\xi_1-L,\xi_1]$ and $\overline{\mathcal{W}}_1\le \overline{\mathcal{W}}_2$ for $\xi\in[\xi_1,\xi_1+L]$.

Follow the discussion in Remark \ref{rm: glue}, $\mathcal{N}_1[\overline{\mathcal{W}}]\le 0$ for $\xi\in [\xi_1,\xi_1+L]$. Consequently, we find some $\delta_0<c^*-c^*_0$, not depending on $\xi_1$, such that $\mathcal N_1[\overline{\mathcal{W}}]\leq 0$ for $\xi\ge \xi_1$.

In the end of this subsection, we complete the proof of Proposition \ref{prop: super sol nl}.
\begin{proof}[Proof of Proposition \ref{prop: super sol nl}]
From the discussion above, we have constructed a super-solution $\overline w(t,x)=\overline{\mathcal{W}}(x-(c^*-\delta_0)-x_0)$. Consider the Cauchy problem of
$$w_t=J\ast w-w+f(w), \ \ t>0,\ x\in\mathbb{R},$$
with a compactly supported initial datum. It is well-known that, the propagation speed is equal to the minimal traveling wave speed $c^*$. However, by setting $x_0$ sufficiently large, the comparison principle implies that $w(t,x)\le \overline w(t,x)$, which means the propagation speed of $w(t,x)$ is at most $c^*-\delta_0$. This contradiction completes the proof.
\end{proof}

\subsection{Proof of Theorem \ref{th: classification scalar nonlocal}}
In this subsection, we complete the proof of Theorem \ref{th: classification scalar nonlocal}, {\it i.e.}, the statement (3). The statements (1) and (2) follow from Proposition \ref{prop:correction-U-linear-decay} and Proposition \ref{prop: super sol nl}, respectively. Let $\hat{\mathcal{W}}$ be the traveling wave with speed $c>c^*_{NL}\ge c^*_0$. We will prove that the asymptotic behavior of $\hat{\mathcal{W}}$ is given by the slow decay, {\it i.e.}, $\hat{\mathcal{W}}(\xi)\sim e^{-\lambda_q^-\xi}$ as $\xi\to+\infty$. We assume by contradiction that 
\bea\label{assume hat w}
\hat{\mathcal{W}}(\xi)\sim e^{-\lambda_q^+\xi}\quad\text{as}\quad \xi\to+\infty.
\eea
With the assumption \eqref{assume hat w}, we claim that there exists a finite $h$ such that 
\begin{equation}\label{W*>hat W}
\mathcal W_*(\xi-h)\ge \hat{\mathcal{W}}(\xi)\quad\text{for all}\quad\xi\in\mathbb{R},
\end{equation}
where $\mathcal W_*(\xi)$ is the minimal traveling wave with $c=c^*_{NL}$.

With (1) and (2) in Theorem \ref{th: classification scalar nonlocal}, 
as $\xi\to+\infty$ we have
$$\mathcal W_*(\xi)\sim e^{-\lambda_q^+(c^*_{NL})\xi}\ \ \text{if}\ \ c^*_{NL}>c^*_0\quad\text{or}\quad \mathcal W_*(\xi)\sim A\xi e^{-\lambda_0\xi}+Be^{-\lambda_0\xi}\ \ \text{if}\ \ c^*_{NL}=c^*_0.$$
On the other hand, with the assumption \eqref{assume hat w},
we have
$$\hat{\mathcal{W}}(\xi)\sim e^{-\lambda_q^+(c)\xi}.$$
Since $\lambda_q^+(c)$ is strictly increasing on $c>0$, we can assert that 
\bea\label{ff5}
\hat{\mathcal{W}}(\xi)=o(\mathcal W_*(\xi))\quad\text{as}\quad\xi\to+\infty.
\eea

Define $\mu_q^+(c)$ as the positive root of
$$\int_{\mathbb{R}}J(x)e^{-\mu x}dx-1+f'(1)+c\mu=0,$$
which is decreasing on $c>0$. Then it holds 
$$1-\mathcal W_*(\xi)\sim e^{\mu_q^+(c^*_{NL})\xi}\ \ \text{and}\ \ 1-\hat{\mathcal{W}}(\xi)\sim e^{\mu_q^+(c)\xi}\ \ \text{as} \ \ \xi\to-\infty.$$
Thus, with \eqref{ff5}, there exists a finite $h$ such that \eqref{W*>hat W} holds.

However, this is impossible. To see this, we may consider the initial value problem to 
$$w_t=J\ast w-w+f(w), \ \ t>0,\ x\in\mathbb{R},$$
with initial datum 
$$w_1(0,x)=\mathcal W_{*}(x-h)\ \ \text{and}\ \ w_2(0,x)=\hat{\mathcal{W}}(x),$$
respectively. By \eqref{W*>hat W},
we have $w_1(t,x)>w_2(t,x)$ for all $t\geq0$ and $x\in\mathbb{R}$. However, $w_2(t,x)$ propagates to the right with speed $c$, which is strictly greater than
the speed $c_{NL}^*$ of $w_1(t,x)$. Consequently, it is impossible to have $w_1(t,x)>w_2(t,x)$ for all $t \geq 0$ and $x\in\mathbb{R}$.
Thus, we reach a contradiction, and hence $\hat{\mathcal{W}}(\xi)\sim e^{-\lambda_q^-\xi}$ as $\xi\to+\infty$.
This completes the proof of (3) in Theorem \ref{th: classification scalar nonlocal}.

\subsection{The asymptotic behavior of the pushed front of the Lotka-Volterra competition system}\label{sec:pushed front LV}
This subsection is devoted to completing the proof of Theorem \ref{th: classification}. We show that the asymptotic behavior of the pushed front is given by the fast decay $e^{-\lambda_u^+(c^*_{LV})\xi}$. As a matter of fact, if the pushed front decays with the slow rate $e^{-\lambda_u^-(c^*_{LV})\xi}$, then we can always construct a traveling wave solution with speed $c<c^*_{LV}$, which contradicts the definition of the minimal speed $c^*_{LV}$. 

Hereafter, we always assume $c^*_{LV}>2\sqrt{1-a}$, and denote the pushed front as $(U_*,V_*)(\xi)$ and $c^*=c^*_{LV}$, $\lambda_u^{\pm}=\lambda_u^{\pm}(c^*)$ for simplicity.
Let us assume that
\bea\label{assume u}
U_*(\xi)\sim A_0e^{-\lambda_u^-\xi},
\eea
in which $\lambda_u^-$ is the smaller root of \eqref{two roots LV} with $c=c^*$. Consequently, from Lemma \ref{lm: behavior around + infty},
\bea\label{assume v}
1-V_*(\xi)\sim A_1\xi^{p}e^{-\Lambda_v\xi},
\eea
where $\Lambda_v=\min\{\lambda_u^-,\lambda_v^+\}$,  $p=0$ if $\lambda_u^-\neq\lambda_v^+$, and $p=1$ if $\lambda_u^-=\lambda_v^+$.

With conditions \eqref{assume u} and \eqref{assume v}, we can construct a super-solution $(\overline U, \underline V)(\xi)$ of
\begin{equation*}
\left\{
\begin{aligned}
&N_5[U,V]:=U''+(c^*-\delta_0)U'+U(1-U-aV)=0,\\
&N_6[U,V]:=dV''+(c^*-\delta_0)V'+rV(1-V-bU)=0.
\end{aligned}
\right.
\end{equation*}
As a result, the spreading speed of the solution to \eqref{system} with initial datum \eqref{initial datum}
is at most $c^*-\delta_0$, which yields the contradiction. 
\begin{proposition}\label{lm 1}
Let $(c^*,U_*,V_*)$ be the traveling wave solution defined as \eqref{tw solution weak} with $c^*>2\sqrt{1-a}$.
Assume that $U_*(\xi)\sim e^{-\lambda_u^-\xi}$ as $\xi\to+\infty$. Let $(u,v)(t,x)$ be the solution to the Cauchy problem of \eqref{system} with initial datum \eqref{initial datum}.
Then, there exists a $\delta_0>0$ such that 
\begin{equation}\label{b4}
\lim_{t\to\infty}u(t,(c^*-\frac{\delta_0}{2})t)=0.
\end{equation}
This contradicts the fact that the propagation speed is equal to the minimal traveling wave speed $c^*$. As a result, $U_*(\xi)\sim e^{-\lambda_u^+\xi}$ as $\xi\to+\infty$. 
\end{proposition}

\subsubsection{Construction of the super-solution for $b>1$}\label{sec:classification super b>1}
Assume $b>1$. We look for continuous function $(R_u(\xi),R_v(\xi))$
defined in $\mathbb{R}$,
such that
\beaa
(\overline U,\underline V)(\xi):=\Big(\min\{(U_*-R_u)(\xi),1\},\max\{(V_*+R_v)(\xi),0\}\Big)
\eeaa
forms a super-solution satisfying $N_5[\overline U,\underline V]\le 0$ and $N_6[\overline U,\underline V]\ge 0$
for some sufficiently small $\delta_0>0$. 
By some straightforward computations, we have
\bea\label{N7}
N_5[\overline U,\underline V]=-\delta_0U_*'-R_u''-(c^*-\delta_0)R_u'-R_u(1-2U_*+R_u-a(V_*+R_v))-aU_*R_v,
\eea
\begin{equation}\label{N8}
\begin{aligned}
N_6[\overline U,\underline V]=&-\delta_0V_*'+dR_v''+(c^*-\delta_0)R_v'+rR_v(1-2V_*-R_v-b(U_*-R_u))\\
&+rbV_*R_u.
\end{aligned}
\end{equation}

\begin{figure}
\begin{center}
\begin{tikzpicture}[scale = 1.1]
\draw[thick](-6,0) -- (6,0) node[right] {$\xi$};
\draw [semithick] (-6,-0.5)--  (4,-0.5) to [out=10,in=180] (6,-0.1);
\draw [ultra thick] (-6, -1) -- (-4,-1) to [out=30, in=260] (-3.5,0)  to [out= 70, in=180] (-3,0.6) to [out=20, in=220] (1,2.5) to [out=70,in=180] (2,3)to [out=0,in=120] (4,1) to [out=-30, in=160] (6,0.2);
\node[below] at (1,1) {$\xi_{1}+\delta_1$};
\draw[dashed] [thick] (4,0)-- (4,0.25);
\node[below] at (4,0.7) {$\xi_*$};
\draw[dashed] [thick] (4,1)-- (4,0.65);
\draw[dashed] [thick] (-3.5,0)-- (-3.5,0.4);
\node[below] at (-3.5,0.9) {$\xi_2$};
\draw[dashed] [thick] (-3,0)-- (-3,0.6);
\node[above] at (-2.8,-0.5) {$\xi_2+\delta_2$};
\draw[dashed] [thick] (-4,-0.5)-- (-4,1);
\node[above] at (-4,0.8) {$\xi_2-\delta_4$};
\draw[dashed] [thick] (1,0)-- (1,0.55);
\draw[dashed] [thick] (1,0.9)-- (1,2.5);
\draw [thin] (4.2,-0.46) arc [radius=0.2, start angle=20, end angle= 170];
\node[above] at (4,-0.42) {$\alpha_2$};
\draw [thin] (4.2,0.9) arc [radius=0.2, start angle=-20, end angle= 120];
\node[above] at (4.2,1.2) {$\alpha_1$};
\draw [thin] (-3.82,-0.83) arc [radius=0.2, start angle=50, end angle= 185];
\node[above] at (-4,-0.87) {$\alpha_5$};
\draw [thin] (-2.8,0.655) arc [radius=0.2, start angle=30, end angle= 175];
\node[above] at (-3,0.7) {$\alpha_4$};
\draw [thin] (1.1,2.7) arc [radius=0.2, start angle=70, end angle= 220];
\node[above] at (0.6,2.5) {$\alpha_3$};
\node[above] at (2,1.8) {\Large{$R_u$}};
\node[above] at (4,-1.2) {\Large{$R_v$}};
\end{tikzpicture}
\caption{$(R_u,R_v)$ for $b>1$.}\label{Figure pushed front LV}
\end{center}
\end{figure}

We consider $(R_u,R_v)(\xi)$ defined  as (see Figure \ref{Figure pushed front LV})
\begin{equation*}
(R_u,R_v)(\xi):=\begin{cases}
(U_*-\varepsilon_1e^{-\lambda_1\xi},-\eta_1 e^{-\lambda_2\xi}),&\ \ \mbox{for}\ \ \xi>\xi_*,\\
(\varepsilon_2(\xi-\xi_1)e^{-\lambda_3\xi},-\delta_v),&\ \ \mbox{for}\ \ \xi_1+\delta_1<\xi\le\xi_*,\\
(\varepsilon_3e^{\lambda_4\xi},-\delta_v),&\ \ \mbox{for}\ \ \xi_2+\delta_2\le\xi\le\xi_1+\delta_1,\\
(\varepsilon_4\sin(\delta_3(\xi-\xi_2)),-\delta_v),&\ \ \mbox{for}\ \ \xi_2-\delta_4\le\xi\le\xi_{2}+\delta_2,\\
(-\delta_u,-\delta_v),&\ \ \mbox{for}\ \ \xi\le\xi_2-\delta_4,
\end{cases}
\end{equation*}
where $\xi_*>\xi_1>M_0$ and $\xi_2<-M_0$ are fixed points.
Since $a<1$ and $b>1$, up to enlarging $M_0$ if necessary, we can find $\rho>0$ such that
\bea\label{M0 aa}
1-2U_*-aV_*<-1+2\rho<0\ \text{and}\ 1-2V_*-bU_*<-(1-b)+b\rho<0\ \text{for all}\ \xi<-M_0.
\eea
We also set
$\lambda_1=\sqrt{1-a}$, $\lambda_2\in(0,\Lambda_v)$ with $\Lambda_v=\min\{\lambda_u^-,\lambda_v^+\}$, $\lambda_3$ to satisfy
\bea\label{condition on lambda 3 pf}
0<\lambda_3<\min\{\lambda_u^-,\frac{c^*-\delta_0}{2}\},
\eea
and $\lambda_{4}$ to satisfy
\begin{equation}\label{condition on lambda 4 pf}
\lambda_4^2+2\sqrt{1-a}\,\lambda_4-3=C_1>0.
\end{equation}
Here,
$\varepsilon_{i=1,\cdots,4}>0$, $\eta_1>0$, and 
\bea\label{delta u v pf}
\delta_u=\varepsilon_4\sin(\delta_3\delta_4)\quad\text{and}\quad \delta_v=\eta_1 e^{-\lambda_2\xi_*}
\eea
make $(R_u,R_v)$ continuous on $\mathbb{R}$, while $\delta_{j=1,\cdots,4}>0$ will be determined later.

Next, we will divide the construction into several steps.

\noindent{\bf{Step 1}:} We consider $\xi\in(\xi_*,+\infty)$ with $\xi_*>\xi_1+\delta_1>M_0$.
In this case, we have 
\beaa
(R_u,R_v)(\xi)=(U_*-\varepsilon_1e^{-\lambda_1\xi},-\eta_1 e^{-\lambda_2\xi}),
\eeaa
with $\lambda_1\in(\lambda_u^-,\lambda_u^+)$ and $\lambda_2\in(0,\Lambda_v)$.

By some straightforward computations, we have
\beaa
N_5[\overline U,\underline V]\le \Big((\lambda_1^2-\lambda_1(c^*-\delta_0)+1-a)+a(1-V_*-R_v)\Big)\varepsilon_1e^{-\lambda_1\xi}.
\eeaa
Since  $\lambda_1=\sqrt{1-a}\in(\lambda_u^-,\lambda_u^+)$, by setting $\delta_0<c^*-2\sqrt{1-a}$ very small, there exists $C_2>0$ such that 
$$\lambda_1^2-\lambda_1(c^*-\delta_0)+1-a<-C_2.$$ 
Then, from $1-V_*(\xi)=o(1)$ and $R_v(\xi)=o(1)$ as $\xi\to+\infty$ , we obtain $N_5[\overline U,\underline V]\le 0$ for all $\xi\in[\xi_*,+\infty)$ up to enlarging $\xi_*$ if necessary.

Next, we deal with the inequality of $N_6[\overline U, \underline V]$.
From \eqref{N8}, we have
\begin{equation*}
N_6[\overline U,\underline V]\ge-\delta_0V_*'-\eta_1e^{-\lambda_2\xi}\Big(d\lambda_2^2-\lambda_2(c^*-\delta_0)-r+r(2-2V_*-R_v) \Big).   
\end{equation*}
Since $0<\lambda_2<\Lambda_v$, by setting $\delta_0<c^*-2\sqrt{1-a}$, there exists $C_3>0$ such that 
$$d\lambda_2^2-\lambda_2(c^*-\delta_0)-r\le -C_3.$$
Note that $2-2V_*(\xi)-R_v(\xi)=o(1)$ as $\xi\to+\infty$.
Therefore, from \eqref{assume v},  we obtain $N_6[\overline U,\underline V]\ge 0$ for all $\xi\in[\xi_*,+\infty)$ up to enlarging $\xi_*$ if necessary. Note that reducing $\delta_0$ does not affect the choice of $\xi_*$. Hence, we fix $\xi_*$.

\medskip

\noindent{\bf{Step 2}:} We consider $\xi\in[\xi_1+\delta_1,\xi_*)$ with $\xi_*$ fixed by Step 1.
In this case, we have 
\beaa
(R_u,R_v)(\xi)=(\varepsilon_2(\xi-\xi_1)e^{-\lambda_3\xi},-\delta_v),
\eeaa
with  $\lambda_3$ satisfying \eqref{condition on lambda 3 pf} and $\delta_v$ defined as \eqref{delta u v pf}.

We first set 
\bea\label{condition aa}
\varepsilon_2=\varepsilon_2(\varepsilon_1,\xi_1)=\frac{U_*(\xi_*)-\varepsilon_1e^{-\lambda_1\xi_*}}{(\xi_*-\xi_1)e^{-\lambda_3\xi_*}}
\eea
which implies $R_u(\xi)$ is continuous at $\xi=\xi_*$. By some straightforward computations, 
$$R_u'(\xi_*^+)=U_*'(\xi_*)+\lambda_1\varepsilon_1e^{-\lambda_1\xi_*},$$
$$R_u'(\xi_*^-)=\varepsilon_2(1-\lambda_3(\xi_*-\xi_1))e^{-\lambda_3\xi_*}.$$
With \eqref{assume u} and the condition \eqref{condition aa}, $R_u'(\xi_*^+)>R_u'(\xi_*^-)$ is equivalent to 
\beaa
(\lambda_1-\lambda_u^-)U_*(\xi_*)>\varepsilon_2e^{-\lambda_3\xi_*}(1+(\lambda_1-\lambda_3)(\xi_*-\xi_1))e^{-\lambda_3\xi_*}.
\eeaa
Note that, from \eqref{condition aa}, $\varepsilon_2$ can be set enough small by reducing $|U_*(\xi_*)-\varepsilon_1 e^{-\lambda_1\xi_*}|$.
Thus, this condition is admissable since $\lambda_1>\lambda_u^-$.  Consequently, we verified $\angle \alpha_1<180^{\circ}$.
 $\angle \alpha_2<180^{\circ}$ follows immediately from 
$R'_v(\xi_*^+)>0=R'_v(\xi_*^-)$.

\medskip
From \eqref{N7}, we have
\beaa
N_5[\overline U,\underline V]=-\delta_0 U_*'-(\lambda_3^2-\lambda_3(c^*-\delta_0)+1-a)R_u-\varepsilon_2(c^*-\delta_0-2\lambda_3)e^{-\lambda_3\xi}+o(R_u).
\eeaa
By \eqref{condition on lambda 3 pf}, $\lambda_3^2-\lambda_3(c^*-\delta_0)+1-a>0$ and $c^*-\delta_0-2\lambda_3>0$. Therefore, up to reducing $\delta_0(\varepsilon_1,\eta_1,\xi_1+\delta_1)$ if necessary, we have $N_5[\overline U,\underline V]\le 0$ for all $\xi\in[\xi_1+\delta_1,\xi_*)$.   

Next, we deal with the inequality of $N_6[\overline U, \underline V]$.
From \eqref{N8}, we have
\begin{equation*}
N_6[\overline U,\underline V]=-\delta_0V_*'-r\delta_v(1-2V_*-R_v-b(U_*-R_u))+rbV_*R_u.
\end{equation*}
Since $R_u>0$, by setting $\eta_1\ll \varepsilon_2$ such that $\delta_v\ll |R_u|$ for all $\xi\in[\xi_1+\delta_1,\xi_*)$, we have $N_6[\overline U,\underline V]\ge 0$, up to reducing $\delta_0(\varepsilon_1,\eta_1,\xi_1+\delta_1)$ if necessary.

\medskip

\noindent{\bf{Step 3}:} We consider $\xi\in[\xi_{2}+\delta_2,\xi_1+\delta_1)$ with $\xi_1+\delta_1$ fixed by Step 2 and $\delta_1$ satisfying
\bea\label{condition delta 1 pf}
\delta_1<\frac{1}{\lambda_3+\lambda_4}.
\eea
In this case, we have
$(R_u,R_v)(\xi)=(\varepsilon_3e^{\lambda_4\xi},-\delta_v)$
with $\lambda_4$ satisfying \eqref{condition on lambda 4 pf}. 

We first set
$$\varepsilon_3=\varepsilon_3(\varepsilon_1,\xi_1)=\frac{U_*(\xi_*)-\varepsilon_1e^{-\lambda_1\xi_*}}{(\xi_*-\xi_1)e^{-\lambda_3\xi_*}}\frac{\delta_1e^{-\lambda_3(\xi_1+\delta_1)}}{e^{\lambda_4(\xi_1+\delta_1)}}$$
such that $R_u(\xi)$ is continuous at $\xi=\xi_1+\delta_1$. Then, by some straightforward computations, we have
\beaa
&&R_u'((\xi_1+\delta_1)^+)=\varepsilon_2e^{-\lambda_3(\xi_1+\delta_1)}-\varepsilon_2\lambda_3\delta_1e^{-\lambda_3(\xi_1+\delta_1)},\\
&&R_u'((\xi_1+\delta_1)^-)=\lambda_4R_u(\xi_1+\delta_1).
\eeaa
Thus, $R_u'((\xi_1+\delta_1)^+)>R_u'((\xi_1+\delta_1)^-)$ 
is
equivalent to \eqref{condition delta 1 pf}. 

From \eqref{N7} and \eqref{condition on lambda 4 pf}, we have
\beaa
N_5[\overline U,\underline V]\le-\delta_0U_*'-C_1R_u+aU_*\delta_v.
\eeaa
Notice that, we can set $\eta_1\ll \varepsilon_2$ such that $\delta_v\ll|R_u|$ for all $\xi\in[\xi_2+\delta_2,\xi_1+\delta_1]$. Therefore, we have $N_5[\overline U,\underline V]\le 0$ for $\xi\in[\xi_2+\delta_2,\xi_1+\delta_1]$ up to decreasing $\delta_0(\varepsilon_1,\eta_1,\xi_2+\delta_2)$ if necessary.  $N_6[\overline U,\underline V]\ge 0$ is easy to verify using the same argument as in Step 2.

\medskip

\noindent{\bf{Step 4}:} We consider $\xi\in[\xi_{2}-\delta_4,\xi_2+\delta_2)$ with $\xi_2+\delta_2$ fixed by Step 4, and $\delta_2$ satisfying
\bea\label{cond delta 2 pf}
\frac{1}{\lambda_4}<\delta_2<  \frac{c^*-\delta_0}{\delta_3^2+1+2a}.
\eea
This condition is admissible since we can reduce $\delta_1$ in \eqref{condition delta 1 pf}.
In this case, we have $(R_u,R_v)=(\varepsilon_4\sin(\delta_3(\xi-\xi_2)),-\delta_v)$. 

To make $R_u(\xi)$ be continuous at $\xi=\xi_2+\delta_2$, we set
$$\varepsilon_4=\varepsilon_4(\varepsilon_1,\xi_1,\delta_1,\delta_2,\lambda_3,\lambda_4)=\frac{U_*(\xi_*)-\varepsilon_1e^{-\lambda_1\xi_*}}{(\xi_*-\xi_1)e^{-\lambda_3\xi_*}}\frac{\delta_1e^{-\lambda_3(\xi_1+\delta_1)}}{e^{\lambda_4(\xi_1+\delta_1)}}\frac{e^{\lambda_4(\xi_2+\delta_2)}}{\sin(\delta_2\delta_3)}.$$
Then, by some straightforward computations, we have
$$R_u'((\xi_2+\delta_2)^+)=\lambda_4R_u(\xi_2+\delta_2)\quad\text{and}\quad R_u'((\xi_2+\delta_2)^-)=\varepsilon_4\delta_3\cos(\delta_2\delta_3).$$
Thus, from $\frac{x\cos x}{\sin x}\to 1$ as $x\to 0$, 
$$R_u'((\xi_2+\delta_2)^+)>R_u'((\xi_2+\delta_2)^-)\ \ \text{and}\ \ \angle \alpha_4<180^{\circ}$$ 
follow by taking
$\delta_3$ sufficiently small and $\delta_2>1/\lambda_4$.

From now on, we fix $\delta_1$, $\delta_2$, $\delta_3$, $\lambda_3$, and $\lambda_4$ satisfying both \eqref{condition delta 1 pf} and  \eqref{cond delta 2 pf}.
We first verify the inequality of $N_5[\overline U,\underline V]$. From \eqref{N7}, we have
\beaa
N_5[\overline U,\underline V]\le\delta_3^2R_u-(c^*-\delta_0)\delta_3\varepsilon_4\cos(\delta_3(\xi-\xi_2))-R_u(1-aV_*-2U_*)+aU_*\delta_v.
\eeaa
For $\xi\in[\xi_2,\xi_2+\delta_2]$, we have
$$N_5[\overline U,\underline V]\le (\delta_3^2+1+2a)\varepsilon_4\sin(\delta_2\delta_3)-(c^*-\delta_0)\delta_3\varepsilon_4\cos(\delta_2\delta_3)+a\delta_v.$$
Note that, from $\frac{x\cos x}{\sin x}\to 1$ as $x\to 0$,
$$(\delta_3^2+1+2a)\sin(\delta_2\delta_3)-(c^*-\delta_0)\delta_3\cos(\delta_2\delta_3)< 0$$
is equivalent to \eqref{cond delta 2 pf}. $N_5[\overline U,\underline V]\le 0$ follows by setting $\delta_v\ll R_u(\xi_2+\delta_2)$.

For $\xi\in[\xi_2-\delta_4,\xi_2]$, from $R_u\le 0$ and \eqref{M0 aa}, up to reducing $\xi_2$, we have
$$N_5[\overline U,\underline V]\le -(c^*-\delta_0)\delta_3\varepsilon_4\cos(\delta_2\delta_3)+aU_*\delta_v.$$
Then, by setting
\bea\label{delta 4 pf}
0<\delta_4<\delta_2<\frac{c^*-\delta_0}{\delta_3^2+1+a},
\eea
we have $N_5[\overline U,\underline V]\le 0$ for all $\xi\in[\xi_2-\delta_4,\xi_2+\delta_2]$ up to decreasing $\delta_v(\eta_1)$ if necessary.

Next, we verify the inequality of $N_6[\overline U,\underline V]$. Since $R_u\ge 0$ for $\xi\in[\xi_2,\xi_2+\delta_2]$,   we have
\begin{equation*}
N_6[\overline U,\underline V]\ge-\delta_0V_*'-r\delta_v(1-2V_*-R_v-b(U_*-R_u)).
\end{equation*}
By \eqref{M0 aa}, we obtain $N_6[\overline U,\underline V]\ge 0$ for $\xi\in[\xi_2,\xi_2+\delta_2]$ up to reducing $\delta_0(\varepsilon_1,\eta_1)$ if necessary.

On the other hand, for $\xi\in[\xi_2-\delta_4,\xi_2]$, we have
\begin{equation}\label{ssss}
N_6[\overline U,\underline V]=-\delta_0V_*'-r\delta_v(1-2V_*-R_v-b(U_*-R_u))+rbV_*R_u.
\end{equation}
From \eqref{delta 4 pf}, by adjusting $\delta_4$, we can set 
\bea\label{condition delta u delta v}
a\delta_v<(1-2\rho-a\delta_v)\delta_u\quad\text{and}\quad b\rho\delta_u<(b-1-b\rho)\delta_v,
\eea
where $\rho$ is determined by $\xi_2$ as in \eqref{M0 aa}. From now on, we fix $\delta_4$.
Then, up to reducing $\delta_0(\varepsilon_1,\eta_1)$ if necessary,  $N_6[\overline U,\underline V]\ge 0$ follows from \eqref{ssss} and the second condition in \eqref{condition delta u delta v}.

\medskip

\noindent{\bf{Step 5}:} We consider $\xi\in(-\infty,\xi_{2}-\delta_4)$ with $\xi_{2}-\delta_4$ fixed by Step 4. In this case, we have $(R_u,R_v)=(-\delta_u,-\delta_v)$.
From \eqref{delta u v pf}, $R_u(\xi)$ is continuous at $\xi=\xi_2-\delta_4$. It is easy to see that 
$$R_u'((\xi_2-\delta_4)^+)>0=R_u'((\xi_2-\delta_4)^-)\ \ \text{and}\ \ \angle \alpha_5<180^{\circ}.$$
 
From \eqref{N7}, \eqref{M0 aa}, and the first condition in \eqref{condition delta u delta v}, we have
\beaa
N_5[\overline U,\underline V]\le -\delta_0U'_*+\delta_u(-1+2\rho+a\delta_v)+a\delta_v\le 0
\eeaa
provided $\delta_0(\varepsilon_1,\eta_1)$ is very small. $N_6[\overline U,\underline V]\ge 0$ follows by the same argument as that in Step 4.
The construction of $(R_u,R_v)(\xi)$ is complete.

\subsubsection{Construction of the super-solution for $b\le 1$}
The auxiliary function $(R_u,R_v)$ constructed in \S\ref{sec:classification super b>1} depends on the value $b>1$ (see the second condition of \eqref{condition delta u delta v}). For $b<1$, 
we consider $(R_u,R_v)(\xi)$ defined  as 
\begin{equation*}
(R_u,R_v)(\xi):=\begin{cases}
(U_*-\varepsilon_1e^{-\lambda_1\xi},-\eta_1 e^{-\lambda_2\xi}),&\ \ \mbox{for}\ \ \xi>\xi_*,\\
(\varepsilon_2(\xi-\xi_1)e^{-\lambda_3\xi},-\delta_v),&\ \ \mbox{for}\ \ \xi_1+\delta_1<\xi\le\xi_*,\\
(\varepsilon_3e^{\lambda_4\xi},-\delta_v),&\ \ \mbox{for}\ \ \xi_2+\delta_2\le\xi\le\xi_1+\delta_1,\\
(\varepsilon_4\sin(\delta_3(\xi-\xi_2)),-\delta_v),&\ \ \mbox{for}\ \ \xi_2-\delta_4\le\xi\le\xi_{2}+\delta_2,\\
(-\delta_u,-\delta_v),&\ \ \mbox{for}\ \ \xi\le\xi_2-\delta_4,
\end{cases}
\end{equation*}
in which $\xi_*>\xi_1+\delta_1>M_0$ and $\xi_2<-M_0$, with $M_0$ very large, are fixed points.
Since $a<1$ and $b<1$, up to enlarging $M_0$ if necessary, from Lemma \ref{lem:AS-infty:b<1}, we can find $\rho>0$ such that
\bea\label{M0 aaa}
1-2U_*-aV_*<\frac{a-1}{1-ab}+2\rho<0\ \text{and}\ 1-2V_*-bU_*<\frac{b-1}{1-ab}+b\rho<0\ \text{for}\ \xi<-M_0.
\eea
Similar to the construction in the case $b>1$, we  set
$\lambda_1=\sqrt{1-a}\in(\lambda_u^-,\lambda_u^+)$, $\lambda_2\in(0,\Lambda_v)$, $\lambda_3$ and $\lambda_4$ satisfying
\beaa
0<\lambda_3<\min\{\lambda_u^-,\frac{c^*-\delta_0}{2}\}\quad\text{and}\quad\lambda_4^2+2\sqrt{1-a}\,\lambda_4-3>0.
\eeaa
Moreover, we set
\beaa
\delta_u=\varepsilon_4\sin(\delta_3\delta_4)\quad\text{and}\quad \delta_v=\eta_1 e^{-\lambda_2\xi_*},
\eeaa
which yield $(R_u,R_v)(\xi)$ are continuous on $\mathbb{R}$.
We also set
$\varepsilon_{i=1,\cdots,4}>0$, $\eta_1>0$, and $\delta_{j=1,2,3}>0$ like that in \S\ref{sec:classification super b>1}.

However, different with the construction in \S\ref{sec:classification super b>1} (see \eqref{condition delta u delta v}), for any $\delta_v$,  by adjusting $\delta_4\in(0,\delta_2)$, we always set 
$$\delta_v=b\delta_u/a,$$ which yields
\begin{equation}\label{sss}
\delta_v(\frac{1-b}{1-ab}-b\rho)>b\delta_u(\frac{1-b}{1-ab}+\rho)\quad\text{and}\quad \delta_u(\frac{1-a}{1-ab}-2\rho)>a\delta_v\frac{1-a}{1-ab},
\end{equation}
up to enlarging $M_0$ if necessary. Moreover, in the proof below, we always set $|\delta_u|,|\delta_v|$ to be very small, but satisfy \eqref{sss}.

To prove both $N_5[\overline U,\underline V]\le 0$ and $N_6[\overline U,\underline V]\ge 0$ for $\xi\in(\xi_2-\delta_4,+\infty)$, we refer to the same verification as \S\ref{sec:classification super b>1}. The only difference is that, to verify $N_6[\overline U,\underline V]\ge 0$ for $\xi\in[\xi_2-\delta_4,\xi_2]$, we use \eqref{M0 aaa} and \eqref{sss}. 
More precisely, by some straightforward computations, we have
\beaa
N_6[\overline U,\underline V]\ge -\delta_0V_x'-r\delta_v(1-2V_*-bU_*+\delta_v)-b\delta_uV_*\ge 0,
\eeaa
up to reducing $\delta_0$ and $|\delta_v|$  (i.e. $\eta_1$) if necessary.
For the same reason, we also obtain  $N_6[\overline U,\underline V]\ge 0$ for $\xi\in(-\infty,\xi_2-\delta_4]$.
Therefore, to finish the construction, it suffices to verify $N_5[\overline U,\underline V]\le 0$  for $\xi\in(-\infty,\xi_2-\delta_4]$. 
By some straightforward computations, and thanks to \eqref{M0 aaa} and \eqref{sss} again, we have
\beaa
N_5[\overline U,\underline V]\le -\delta_0U_x'+\delta_u(1-2U_*-aV_*+a\delta_v)+a\delta_vU_*\le 0,
\eeaa
up to reducing $\delta_0$ and $|\delta_v|$ (i.e. $\eta_1$) if necessary.

\bigskip

For the critical case $b=1$, 
we consider $(R_u,R_v)(\xi)$ defined  as 
\begin{equation*}
(R_u,R_v)(\xi):=\begin{cases}
(U_*-\varepsilon_1e^{-\lambda_1\xi},-\eta_1 e^{-\lambda_2\xi}),&\ \ \mbox{for}\ \ \xi>\xi_*,\\
(\varepsilon_2(\xi-\xi_1)e^{-\lambda_3\xi},-\delta_v),&\ \ \mbox{for}\ \ \xi_1+\delta_1<\xi\le\xi_*,\\
(\varepsilon_3e^{\lambda_4\xi},-\delta_v),&\ \ \mbox{for}\ \ \xi_2+\delta_2\le\xi\le\xi_1+\delta_1,\\
(\varepsilon_4\sin(\delta_3(\xi-\xi_2)),-\eta_2(-\xi)^{\theta}V_*(\xi)),&\ \ \mbox{for}\ \ \xi_2-\delta_4\le\xi\le\xi_{2}+\delta_2,\\
(-\varepsilon_5(-\xi)^{\theta}(1-U_*(\xi)),-\eta_2(-\xi)^{\theta}V_*(\xi)),&\ \ \mbox{for}\ \ \xi\le\xi_2-\delta_4,
\end{cases}
\end{equation*}
in which $\theta\in(0,1)$, and $\xi_*>\xi_1>M_0$ and $\xi_2<-M_0$ are fixed points.

Like the construction for $b>1$ and $b<1$, we still set
$\lambda_1\in(\lambda_u^-,\lambda_u^+)$, $\lambda_2\in(0,\Lambda_v)$, $\lambda_3$ and $\lambda_4$ satisfying
\beaa
0<\lambda_3<\min\{\lambda_u^-,\frac{c^*-\delta_0}{2}\}\quad\text{and}\quad\lambda_4^2+2\sqrt{1-a}\,\lambda_4-3>0.
\eeaa
Moreover, we  set
$\varepsilon_{i=1,\cdots,4}>0$, $\eta_1>0$, and $\delta_{j=1,2,3}>0$ like that in \S\ref{sec:classification super b>1},
and set
\bea\label{b5}
\varepsilon_5=\frac{\varepsilon_4\sin(\delta_3\delta_4)}{(-\xi_2+\delta_4)^{\theta}(1-U_*(\xi_2-\delta_4))}\quad\text{and}\quad \eta_2=\frac{\eta_1 e^{-\lambda_2\xi_*}}{(-\xi_2-\delta_2)^{\theta}V_*(\xi_2+\delta_2)},
\eea
which yield $(R_u,R_v)(\xi)$ are continuous on $\mathbb{R}$. The inequalities $N_5[\overline U,\underline V]\le 0$ for $\xi\in(\xi_2-\delta_4,+\infty)$ and $N_6[\overline U,\underline V]\ge 0$ for $\xi\in(\xi_2+\delta_2,+\infty)$ follows by the same verification as \S\ref{sec:classification super b>1}. 

Without loss of generality, we may assume $\xi_2+\delta_2<\xi_0$,
where $\xi_0$ 
is defined in Corollary~\ref{lm: behavior around - infty b=1}.
The next claim shows how to choose $\delta_4$ such that $\varepsilon_5$ and $\eta_2$ determined in \eqref{b5} satisfy $\varepsilon_5=\eta_2$.
Note that the choice of $\delta_4$ is rather technical and crucial for the construction on  $\xi\in(-\infty,\xi_2-\delta_4)$.

\begin{claim}\label{cl 61}
There exists $0<\delta_4\le \delta_2$ such that 
$$R_u(\xi_2-\delta_4)=-\eta_2(-\xi_2+\delta_4)^{\theta}(1-U_*(\xi_2-\delta_4))$$ 
and
\bea\label{claim3.61}
-\eta_2(-\xi)^{\theta}(1-U_*(\xi))<R_u(\xi)<0\quad\text{for all}\quad \xi\in(\xi_2-\delta_4,\xi_2).
\eea
\end{claim}
\begin{proof}
Recall from Step 4 in \S \ref{sec:classification super b>1}, up to reducing $\eta_1$,
that 
$$R_u(\xi_2+\delta_2)\gg \delta_v=\eta_2(-\xi_2-\delta_2)^{\theta}V_*(\xi_2+\delta_2).$$
We also assume, up to reducing $\eta_1$ if necessary, that
\bea\label{Rv>V1}
R_u(\xi_2+\delta_2)>\eta_2(-\xi_2-\delta_2)^{\theta}[1-U_*(\xi_2+\delta_2)].
\eea

Furthermore, by the asymptotic behavior of $1-U_*(\xi)$ as $\xi\to -\infty$ and setting $\theta$ small,  
$$(-\xi)^{\theta}[1-U_*(\xi)]>0\quad\text{is strictly increasing for all}\quad\xi<\xi_2+\delta_2.$$ 
Together with \eqref{Rv>V1}, we obtain that
\beaa
-\varepsilon_4\sin(\delta_2\delta_3)=-R_u(\xi_2+\delta_2)&<
&-\eta_2(-\xi_2-\delta_2)^{\theta}[1-U_*(\xi_2+\delta_2)]\\
&<&-\eta_2(-\xi_2+\delta_2)^{\theta}[1-U_*(\xi_2-\delta_2)].
\eeaa
Define $$F(\xi):=\varepsilon_4\sin(\delta_3(\xi-\xi_2))+\eta_2(-\xi)^{\theta}[1-U_*(\xi)].$$
Clearly, from Corollary~\ref{lm: behavior around - infty b=1}, $F$ is continuous and strictly increasing for $\xi\in[\xi_2-\delta_2,\xi_2]$. Also,
we have $F(\xi_2)>0$ and $F(\xi_2-\delta_2)<0$.
Then, by the intermediate value theorem, there exists a unique $\delta_4\in(0,\delta_2)$ such that Claim~\ref{cl 61} holds.
\end{proof}

Since $\theta>0$ and $\varepsilon_5=\eta_2$, there exists $M_1>M_0$ sufficiently large such that $\overline U=1$ and $\underline V=0$ for all $\xi\in(-\infty,-M_1]$.
Then, from the definition of $(R_u,R_v)$, we may define $M_1$ satisfying $1-\eta_2(M_1)^{\theta}= 0$. Thus
$\overline U(\xi)=1$, $\underline V(\xi)=0$ for all $\xi\in(-\infty,-M_1]$, which implies that
$$N_5[\overline U,\underline V]\le 0\ \ \text{and}\ \ N_6[\overline U,\underline V]\ge 0\ \ \text{for}\ \ \xi\in(-\infty,-M_1].$$
Additionally,  we have
\bea\label{eta 4 epsilon 41}
1-\varepsilon_4(-\xi)^{\theta}=1-\eta_4(-\xi)^{\theta}> 0\quad\text{for all}\quad \xi\in(-M_1,\xi_2-\delta_4],
\eea
which yields $\overline U<1$ and $\underline V>0$ on $(-M_1,\xi_2-\delta_4]$.

We first verify the inequalities  $N_6[\overline U,\underline V]\ge 0$ for $\xi\in(-M_1,\xi_2+\delta_2)$. 
By some straight computations, we have
\beaa
N_6[\overline U,\underline V]&=&d\Big(V_*''+\theta(1-\theta)\eta_2(-\xi)^{\theta-2}V_*+2\theta\eta_2(-\xi)^{\theta-1}V'_*-\eta_2(-\xi)^{\theta}V''_*\Big)\\
&&+c^*\Big(V_*'+\theta\eta_2(-\xi)^{\theta-1}V_*-\eta_2(-\xi)^{\theta}V'_*\Big)-\delta_0(V'_*+R'_v)\\&&
+r(V_*+R_v)(1-V_*-R_v-(U_*-R_u)).
\eeaa
Notice that, in Claim \ref{cl 61}, we choose a suitable $\delta_4$ such that $\varepsilon_5=\eta_2$.
Then, from $V_*'>0$, $\theta\in(0,1)$, and $R_u(\xi)\ge -\eta_2(-\xi)^{\theta}[1-U_*(\xi)]$, we further have
\beaa
N_6[\overline U,\underline V]&\geq& r\eta_2(-\xi)^{\theta}V_*\Big(V_*-(1-U_*)
+\frac{c^*\theta}{r}(-\xi)^{-1}+R_v-R_u\Big)-\delta_0(V'_*+R'_v)\\
&\geq&r\eta_2(-\xi)^{\theta}V_*\Big((\eta_2(-\xi)^{\theta}-1)(1-U_*-V_*)
+\frac{c^*\theta}{r}(-\xi)^{-1}\Big)-\delta_0(V'_*+R'_v).
\eeaa
By Corollary \ref{lm: behavior around - infty b=1} and \eqref{eta 4 epsilon 41}, 
as long as $M_0$ is chosen very large at the beginnig, we have
$(\eta_2(-\xi)^{\theta}-1)(1-U_*-V_*)>0$ for $\xi\in[-M_1,\xi_2+\delta_2]$. It follows that
$N_6[\overline U,\underline V]\ge 0$ for $\xi\in[-M_1,\xi_2+\delta_2]$
for all small $\delta_0>0$.

\medskip
To complete the construction, we verify the inequalities  $N_5[\overline U,\underline V]\le 0$ for $\xi\in(-M_1,\xi_2-\delta_4)$. Due to $\theta\in(0,1)$ and $U'_*<0$, $N_5[W_u,W_v]$ satisfies
\begin{equation}\label{N3 inequalifty -infty1}
\begin{aligned}
N_5[\overline U,\underline V]
\leq&-\delta_0(U'_*-R'_u)+\varepsilon_5(-\xi)^{\theta}\Big(U_*(1-U_*-aV_*)-c^*\theta(-\xi)^{-1}(1-U_*)\Big)\\&-R_u(1-2U_*+R_u-a(V_*+R_v))-aU_*R_v.
\end{aligned}
\end{equation}
By using \eqref{claim3.61} and
$$\varepsilon_5(-\xi)^{\theta}U_*(1-U_*)=-R_uU_*,$$
from \eqref{N3 inequalifty -infty1} we have
\beaa
N_5[\overline U,\underline V]&\le&-R_uU_*-a\varepsilon_5(-\xi)^{\theta}U_*V_*+c^*\theta(-\xi)^{-1}R_u-R_u(1-2U_*-aV_*)
\\
&&-R_u^2+aR_uR_v+a\varepsilon_5(-\xi)^{\theta}U_*V_*-\delta_0(U'_*-R'_u)\\
&=&c^*\theta(-\xi)^{-1}R_u-R_u(1-U_*-aV_*)
-R_u^2+aR_uR_v-\delta_0(U'_*-R'_u).
\eeaa

Denote that
\beaa
I_1:=c^*\theta(-\xi)^{-1}R_u,\quad
I_2:=-R_u(1-U_*-aV_*),\quad
I_3:=-R_u^2+aR_uR_v.
\eeaa
By the equation satisfied by $U_*$ in \eqref{tw solution weak} and Lemma \ref{lem:AS-infty:b=1},  $1-U_*-aV_*>0$ for all $\xi\leq -M_0$. Therefore,
\beaa
I_3=-R_u^2+aR_uR_v\leq R_u\varepsilon_5(-\xi)^{\theta}(1-U_*-aV_*)(\xi)<0\quad \mbox{for}\quad \xi\in(-M_1,\xi_2-\delta_4].
\eeaa
Moreover, in view of Corollary~\ref{lm: behavior around - infty b=1}, we have
$I_2=o(I_1)$ as $\xi\to-\infty$.
Then, up to reducing $\delta_0$ if necessary, we have $N_5[\overline U,\underline V]\le 0$ for $\xi\in(-M_1,\xi_2-\delta_4]$.

The construction of the super-solution is complete.

\medskip

\begin{proof}[Proof of Proposition \ref{lm 1}]
We consider the solution $(u,v)(t,x)$ to be the Cauchy problem of \eqref{system} with the initial datum \eqref{initial datum}. Define $(\overline u,\underline v)(t,x)=(\overline U,\underline V)(x-(c^*-\delta_0)t-x_0)$ in which $(\overline U,\underline V)(\xi)$ is the super-solution constructed above. By setting $x_0$ very large, we have $\overline u(0,x)\ge u(0,x)$ and $\underline v(0,x)\le v(0,x)$. Then, by the comparison principle, we obtain $\overline u(t,x)\ge u(t,x)$ and $\underline v(t,x)\le v(t,x)$ for all $t> 0$ and $x\in\mathbb{R}$. Thus, we can conclude that
$$\lim_{t\to\infty}u(t,(c^*-\frac{\delta_0}{2}))\le \lim_{t\to\infty}\overline u(t,(c^*-\frac{\delta_0}{2}))=0.$$
This finishes the proof of Proposition \ref{lm 1}.
\end{proof}

\subsection{Proof of Theorem \ref{th: classification}}
In this subsection, we complete the proof of Theorem \ref{th: classification}, {\it i.e.}, the statement (3). Let $(\hat U,\hat V)$ be the traveling wave satisfying \eqref{tw solution weak} with speed $c>c^*_{LV}\ge 2\sqrt{1-a}$. We will prove that the asymptotic behavior of $\hat U$ is given by the slow decay, {\it i.e.}, $\hat U(\xi)\sim e^{-\lambda_u^-\xi}$ as $\xi\to+\infty$. We assume by contradiction that
\bea\label{assume hat u}
\hat U(\xi)\sim e^{-\lambda_u^+\xi}\quad\text{as}\quad \xi\to+\infty.
\eea
With the assumption \eqref{assume hat u}, we can find finite $h$ such that
\begin{equation}\label{U*>hat U}
U_*(\xi-h)\ge \hat U(\xi)\quad\text{and}\quad V_*(\xi-h)\le \hat V(\xi)\quad\text{for all}\quad\xi\in\mathbb{R}.
\end{equation}
To verify \eqref{U*>hat U}, it suffices to compare the decay rate of $(U_*,V_*)$ and $(\hat U,\hat V)$ at $\xi=\pm\infty$.

With (2) in Theorem \ref{th: classification} and Lemma \ref{lm: behavior around + infty},
as $\xi\to+\infty$ we have
$$U_*(\xi)\sim e^{-\lambda_u^+(c^*_{LV})\xi}\quad\text{or}\quad U_*(\xi)\sim \xi e^{-\lambda_u\xi},$$
$$1-V_*(\xi)\sim \xi^{p}e^{-\Lambda_v(c^*_{LV})\xi}\ \text{with}\ p\in\{0,1,2\},$$
in which $\Lambda_v(c)$ is defined \eqref{assume v}. Note that, $\lambda_u^+(c^*_{LV})=\lambda_u$ if $c^*_{LV}=2\sqrt{1-a}$. On the other hand, with the assumption \eqref{assume hat u} and Lemma \ref{lm: behavior around + infty},
we have
$$\hat U(\xi)\sim e^{-\lambda_u^+(c)\xi}\ \text{and}\ 1-\hat V(\xi)\sim \xi^{p}e^{-\Lambda_v(c)\xi}\ \text{with}\ p\in\{0,1\}.$$
Since $\lambda_u^+(c)$ and $\Lambda_v(c)$ are strictly increasing on $c>0$, we can assert that
\bea\label{ff1}
\hat U(\xi)=o(U_*(\xi))\quad{and}\quad 1-\hat V(\xi)=o(1-V_*(\xi))\quad\text{as}\quad\xi\to+\infty.
\eea

Next, we compare the decay rate of $(U_*,V_*)$ and $(\hat U,\hat V)$ at $-\infty$.
\begin{itemize}
\item for $b>1$, from Lemma \ref{lem:AS-infty:b>1}, since $\mu_u^+(c)$ and $\mu_v^+(c)$ are strictly decreasing on $c$, as $\xi\to-\infty$ we have
\bea\label{ff2}
1-U_*(\xi)\sim o(1-\hat U(\xi))\quad\text{and}\quad V_*(\xi)\sim o(\hat V(\xi)).
\eea
\item for $b=1$, from Lemma \ref{lem:AS-infty:b=1}, as $\xi\to-\infty$ we have
\bea\label{ff3}
1-U_*(\xi)\sim O(1-\hat U(\xi))\quad\text{and}\quad V_*(\xi)\sim O(\hat V(\xi)).
\eea
\item for $b<1$, from Lemma \ref{lem:AS-infty:b<1}, since $\mu_u^+(c)$ and $\mu_v^+(c)$ are strictly decreasing on $c$, as $\xi\to-\infty$ we have
\bea\label{ff4}
u^*-U_*(\xi)\sim o(u^*-\hat U(\xi))\quad\text{and}\quad V_*(\xi)-v*\sim o(\hat V(\xi)-v*).
\eea
\end{itemize}

In conclusion, from \eqref{ff1}, \eqref{ff2}, \eqref{ff3}, and \eqref{ff4},
there exists a finite $h$ such that \eqref{U*>hat U} holds.
However, this is impossible. To see this, we may consider the initial value problem to \eqref{system} with initial datum
$$(u_1,v_1)(0,x)=(U_{*},V_{*})(x-h)\ \ \text{and}\ \ (u_2,v_2)(0,x)=(\hat U,\hat V)(x),$$
respectively. By \eqref{U*>hat U},
we have $u_1(t,x)>u_2(t,x)$ and $v_1(t,x)<v_2(t,x)$ for all $t\geq0$ and $x\in\mathbb{R}$. However, $(u_2,v_2)(t,x)$ propagates to the right with speed $c$, which is strictly greater than
the speed $c_{LV}^*$ of $(u_1,v_1)(t,x)$. Consequently, it is impossible to have $u_1(t,x)>u_2(t,x)$ for all $t \geq 0$ and $x\in\mathbb{R}$.
Thus, we reach a contradiction, and hence $\hat U(\xi)\sim e^{-\lambda_u^-\xi}$ as $\xi\to+\infty$.
This completes the proof of (3) in Theorem \ref{th: classification}.





\bigskip

\noindent{\bf Acknowledgement.}
Maolin Zhou is supported by the National Key Research and Development Program of China (2021YFA1002400).
Chang-Hong Wu is supported by the Ministry of Science and Technology of Taiwan.
Dongyuan Xiao is supported by the Japan Society for the Promotion of Science P-23314.


\begin{thebibliography}{20}
%
%
\bibitem{Alfaro Xiao}
{M. Alfaro and D. Xiao},
Lotka-Volterra competition-diffusion system: the critical case, Commun. Partial Differ. Equ., 48:2(2023), 182-208.

\bibitem{Alfaro Giletti Xiao}
{M. Alfaro, T. Giletti, and D. Xiao}
The Bramson correction in the Fisher-KPP equation: from
delay to advance, preprint.

\bibitem{Alhasanat Ou2019-1}
{A. Alhasanat and C. Ou}, On a conjecture raised by Yuzo Hosono,
 J. Dyn. Diff. Equat., 31 (2019), 287-304.

\bibitem{Alhasanat Ou2019}
{A. Alhasanat and C. Ou}, Minimal-speed selection of traveling waves to the Lotka-Volterra
competition model, J. Differ. Equ. 266 (2019), 7357-7378.


\bibitem{An etal2023-a}
{J. An, C. Henderson, and L. Ryzhik}, Quantiative steepness, semi-FKPP reactions, and
pushmi-pullyu fronts. Arch. Ration. Mech. Anal. 247, 88 (2023).

\bibitem{An etal2023-b}
{J. An, C. Henderson, and L. Ryzhik}, Front location determines convergence rate to
traveling waves. Annales de l'Institut Henri Poincar\'e C, Analyse non lin\'eaire, (2024).

\bibitem{nonlocal book}
{F. Andreu-Vaillo, J. M. Maz{\'o}n,  J. D. Rossi, and J. J. Toledo-Melero}, Nonlocal diffusion problems (Vol. 165). American Mathematical Society, (2010).






\bibitem{Aronson Weinberger}
{D. G. Aronson and H. F. Weinberger},
Multidimensional nonlinear diffusion arising
in population genetics,
Adv. Math., 30 (1978), 33-76.




\bibitem{Avery Scheel}
{M. Avery and A. Scheel}, Universal selection of pulled fronts. Commun. Am. Math. Soc. 2 (2022), 172-231.

\bibitem{nl 1}
{P.W. Bates, P.C. Fife, X. Ren, and  X.Wang}, Traveling waves in a convolution model for phase transitions, Arch. Ration.
Mech. Anal. 138 (2) (1997) 105–136.


\bibitem{Benguria Depassier1994}
{R. D. Benguria and M. C. Depassier},
Validity of the linear speed selection mechanism for fronts of the nonlinear diffusion equation,
Phys. Rev. Lett. 73 (1994), 2272-2274.

\bibitem{Benguria Depassier1996}
{R. D. Benguria and M. C. Depassier},
Speed of fronts of the reaction-diffusion equation,
Phys. Rev. Lett. 77 (1996), 1171-1173.


\bibitem{Berestycki etal}
{J. Berestycki, E. Brunet, S.C. Harris, and M. Roberts}
Vanishing corrections for the position in a linear model of FKPP fronts, 
Comm. Math. Phys., 349 (2017), 857–893.

\bibitem{Berestycki Hamel2012}
{H. Berestycki and F. Hamel}, Generalized transition waves and their properties,
Commun. Pure Appl. Math. 65 (2012), 592-648.

\bibitem{Bonnefon Coville Garnier Roques2014}
{O. Bonnefon, J. Coville, J. Garnier, and L. Roques}, Inside dynamics of solutions of integro-differential equations, Discrete Contin. Dyn. Syst. B 19 (2014), 3057-3085.


\bibitem{Bosch Metz Diekmann}
{F. van den Bosch, J.A.J. Metz, and O. Diekmann},
The velocity of spatial population expansion, J. Math. Biol., 28 (1990), 529–565.


\bibitem{Bramson}
{M. Bramson}, Convergence of solutions of the Kolmogorov equation to travelling waves, vol. 285, American Mathematical Soc., 1983.

\bibitem{nl 6}
{K.J. Brown and J. Carr}, Deterministic epidemic waves of critical velocity, Math. Proc. Cambridge Philos. Soc. 81 (1977), 431-433.

 \bibitem{Carr Chmaj}
{J. Carr and A. Chmaj},
Uniqueness of travelling waves for nonlocal monostable equations,
Proc. Amer. Math. Soc. 132 (2004), 2433-2439.


\bibitem{Carrere}
{C. Carrere}, Spreading speeds for a two-species competition-diffusion system. J. Differential Equations, 264
(2018), 2133-2156.

\bibitem{ChangChenWang2023}
{M.-S. Chang, C.-C. Chen, and S.-C. Wang},
Propagating direction in the two species Lotka-Volterra competition-diffusion system,
Discrete and Continuous Dynamical Systems - Series B
28 (2023), 5998-6014.


 \bibitem{Chen Fu Guo}
 {X. Chen, S.-C. Fu, and J.-S. Guo},
 Uniqueness and asymptotics of traveling waves of monostable dynamics on lattices, SIAM J. Math. Anal. 38 (2006), 233–258.

 \bibitem{CoddingtonLevison} {E. A. Coddington and N. Levinson}, Theory
of Ordinary Differential Equations, McGraw-Hill, New York, 1955.

\bibitem{Coville}
{J. Coville, J. Dávila, and S. Martínez},
Nonlocal anisotropic dispersal with monostable nonlinearity,
J. Diff. Eqns.,
12 (2008), 3080-3118.

\bibitem{Coville Dupaigne}
J. Coville and L. Dupaigne,
Propagation speed of travelling fronts in nonlocal reaction–diffusion equations,
Nonlinear Analysis: Theory, Methods \& Applications, 60 (2005), 797-819.

\bibitem{nl 2}
{A. De Masi, T. Gobron, and E. Presutti}, Travelling fronts in nonlocal evolution equations, Arch. Ration. Mech.
Anal. 132 (2) (1995) 143–205.

\bibitem{DuWangZhou} {Y. Du, M.X. Wang, and M. Zhou}, Semi-wave and spreading speed for the diffusive competition model with a free boundary, J. Math. Pures Appl. 107 (2017), 253-287.

\bibitem{Du Li Zhou}
{Y. Du, F. Li, and  M. Zhou},
Semi-wave and spreading speed of the nonlocal Fisher-KPP equation with free boundaries,
Journal de Mathématiques Pures et Appliquées, 154 (2021), 30-66.










\bibitem{Ebert van Saarloos}
{U. Ebert and W. van Saarloos}, Front propagation into unstable states: universal algebraic convergence
towards uniformly translating pulled fronts, Phys. D 146 (2000), 1-99.


\bibitem{nl 3}
{P.C. Fife}, An integrodifferential analog of semilinear parabolic PDEs, in: Partial Differential Equations and Applications,
in: Lect. Notes Pure Appl. Math., vol. 177, Dekker, New York, (1996),  137–145.

\bibitem{Fife McLeod1977}
{P. C. Fife and J. B. McLeod}, The approach of solutions of nonlinear diffusion equations to travelling front solutions,
 Arch. Ration. Mech. Anal., 65 (1977), 335-361.

\bibitem{Fisher}
{R. A. Fisher},
The wave of advance of advantageous genes,
Ann. Eugen., 7 (1937), 335-369.



\bibitem{Gardner et al 2012}
{J. Garnier, T. Giletti, F. Hamel, and L. Roques}, Inside dynamics of pulled and pushed fronts,
J Math Pures Appl, 98 (2012), 428-449.


\bibitem{Giletti}
{T. Giletti}, Monostable pulled fronts and logarithmic drifts. NoDEA Nonlin. Differ. Eqs. Appl. 29, (2022).



\bibitem{Girardin}
{L. Girardin},
The effect of random dispersal on competitive exclusion-a review, Mathematical Biosciences, 318 (2019), 108271.

\bibitem{Girardin Lam}
{L. Girardin and K.-Y. Lam},
Invasion of an empty habitat by two competitors: spreading properties of monostable
two-species competition-diffusion systems, Proc. Lond. Math. Soc., 119 (2019), 1279-1335.

\bibitem{Girardin Nadin}
{L. Girardin and G. Nadin},
Travelling waves for diffusive and strongly competitive systems: relative motility and
invasion speed, European J. Appl. Math., 26 (2015), 521-534.

\bibitem{Guo}
{H. Guo},
Pushed fronts of monostable reaction-diffusion-advection equations
J. Diff. Equations, 356 (2023), 127-162.


\bibitem{Guo Liang}
{J.-S. Guo and X. Liang},
The minimal speed of traveling fronts for the Lotka-Volterra competition system,
J. Dynamics Diff. Equations, 23 (2011), 353-363.

\bibitem{Guo Lin}
{J.-S. Guo and Y.-C. Lin}, The sign of the wave speed for the Lotka-Volterra competition-diffusion system, Comm. Pure Appl. Anal., 12 (2013), 2083-2090.


\bibitem{Guo Wu}
{J.-S. Guo and C.-H. Wu},
Traveling wave front for a two-component lattice dynamical system
arising in competition models, J. Diff. Eqns.,  252 (2012), 4357-4391.

\bibitem{Hadeler Rothe}
{K. Hadeler and F. Rothe},
Travelling fronts in nonlinear diffusion equations, J. Math.
Biol. 2 (1975), 251-263.

\bibitem{Hamel}
{F. Hamel},
Qualitative properties of monostable pulsating fronts: exponential decay and monotonicity,
J. Math. Pures Appl. 89 (2008) 355–399.


\bibitem{Hamel etal}
{F. Hamel, J. Nolen, J.-M. Roquejoffre, and L. Ryzhik}, A short proof of the
logarithmic bramson correction in Fisher-Kpp equations, Networks and Heterogeneous
Media, 8 (2013) 275-279.




\bibitem{Holzer Scheel 2012}
{M. Holzer and A. Scheel}, A slow pushed front in a Lotka-Volterra competition model, Nonlinearity 25 (2012), 2151.

\bibitem{Hosono 1995}
{Y. Hosono}, Traveling waves for diffusive Lotka-Volterra competition model ii: a geometric approach.
Forma 10 (1995), 235-257.

\bibitem{Hosono 1998}
{Y. Hosono}, The minimal speed of traveling fronts for a diffusive Lotka Volterra competition
model, Bull. Math. Biol., 60 (1998), 435-448.

\bibitem{Hosono 2003}
{Y. Hosono}, Traveling waves for a diffusive Lotka-Volterra competition model I: singular perturbations.
Disc Cont Dyn Systems B, 3 (2003), 79-95.

\bibitem{Hou Li Meyer}
{X. Hou, Y. Li, and K. R. Meyer}, Traveling wave solutions for a reaction diffusion equation with double degenerate nonlinearities, Discrete and Continuous Dynamical Systems, 26(2010), 265-290.

\bibitem{Huang2010}
{W. Huang}, Problem on minimum wave speed for a Lotka-Volterra reaction-diffusion competition model,
J. Dyn. Diff. Equat., 22 (2010), 285-297.

\bibitem{HuangHan2011}
{W. Huang and M. Han}, Non-linear determinacy of minimum wave speed for a Lotka-Volterra competition
model, J. Diff. Eqns., 251 (2011), 1549-1561.


\bibitem{nl 4}
V. Hutson, S. Martinez, K. Mischaikow, and G.T. Vickers, The evolution of dispersal, J. Math. Biol. 47 (6) (2003), 483–517.

\bibitem{Kan-On}
{Y. Kan-On},
Parameter dependence of propagation speed of travelling waves for competition-diffusion equations,
SIAM J. Math. Anal., 26 (1995), 340-363.



\bibitem{Kan-on1997}
{Y. Kan-on}, Fisher wave fronts for the Lotka-Volterra competition model with diffusion, Nonlinear Anal., 28 (1997), 145-164.



\bibitem{Lau}
{K.-S. Lau}, On the nonlinear diffusion equation of Kolmogorov, Petrovskii and Piskunov, J. Diff. Eqs. 59 (1985), 44-70.

\bibitem{KPP}
{A. N. Kolmogorov, I. G. Petrovskii, and N. S. Piskunov},
A study of the equation of diffusion with increase in the quantity of matter, and its application to a biological problem,
Bull. Moscow State Univ. Ser. A: Math. and Mech., 1 (1937), 1-25.

\bibitem{Lewis Li Weinberger 1}
{M.A. Lewis, B. Li, and H.F. Weinberger}, Spreading speeds and the linear conjecture for two-species competition models, J. Math. Biol., 45 (2002), 219-233.

\bibitem{Lewis Li Weinberger 2}
{M.A. Lewis, B. Li, and H.F. Weinberger},  Spreading speeds as slowest wave speeds for cooperative systems, Math. Biosci. 196 (2005), 82–98.



\bibitem{LCW}
{F. Li, J. Coville, and X. Wang}, On eigenvalue problems arising from nonlocal diffusion models. Discrete and Continuous Dynamical Systems, 2017, 37(2): 879-903. 


\bibitem{Liang Zhao 2007}
{X. Liang and X. Zhao}, Asymptotic speeds of spread and traveling waves for monotone
semiflows with applications, Comm. Pure Appl. Math., 60 (2007), 1-40.






\bibitem{Lucia Muratov Novaga}
{M. Lucia, C.B. Muratov, and M. Novaga},
Linear vs. nonlinear selection for the propagation speed of the solutions of scalar reaction–diffusion
equations invading an unstable equilibrium, Comm. Pure Appl. Math. 57 (2004), 616–636.



\bibitem{Ma Huang Ou}
{M. Ma, Z. Huang, and C. Ou},
Speed of the traveling wave for the bistable Lotka-Volterra competition model, Nonlinearity, 32 (2019), 3143-3162.


\bibitem{Ma Ou}
{M. Ma and C. Ou},
Linear and nonlinear speed selection for mono-stable wave propagations, SIAM J. Math. Anal., 51 (2019), 321-345.

\bibitem{Mollison}
{D. Mollison},
Dependence of epidemic and population velocities on basic parameters,
Mathematical Biosciences,
107: 2 (1991), 255-287.

\bibitem{Morita_etal2023}
{Y. Morita, K. -I. Nakamura, and T. Ogiwara},
Front propagation and blocking for the competition-diffusion system in a domain of half-lines with a junction, Discrete and Continuous Dynamical Systems-B, 28 (2023), 6345-6361.

\bibitem{MoritaTachibana2009} {Y. Morita and K. Tachibana},
An entire solution for wave fronts to the Lotka-Volterra
competition-diffusion equations,
SIAM J. Math. Anal., 40 (2009), 2217-2240.




\bibitem{Murray1993}
{J. D. Murray}, Mathematical Biology, Berlin, Springer, 1993.

\bibitem{Okubo etal}
{A. Okubo, P.K. Maini, M.H. Williamson, and J.D. Murray,}
On the spatial spread of the grey squirrel in Britain,
Proc. R. Soc. Lond. B, 238 (1989), 113-125.

\bibitem{Peng Wu Zhou}
{R. Peng, C.-H. Wu, and M. Zhou},
Sharp estimates for the spreading speeds of the Lotka-Volterra diffusion system with strong competition, Annales de l'Institut Henri Poincar\'e C, Analyse non lin\'eaire, 38 (2021), 507-547.

\bibitem{Roquejoffre}
{J-M Roquejoffre},
Eventual monotonicity and convergence to travelling fronts for the solutions of parabolic equations in cylinders, Annales de l'Institut Henri Poincaré C, Analyse non linéaire, 14 (1997), 499-552.



\bibitem{Rodrigo Mimura}
{M. Rodrigo and M. Mimura}, Exact solutions of a competition-diffusion system, Hiroshima Math. J., 30 (2000),
257-270.

\bibitem{Roques et al 2012}
{L. Roques, J. Garnier, F. Hamel, and E. K. Klein},
Allee effect promotes diversity in traveling waves of colonization. Proc Natl Acad Sci USA., 109 (2012), 8828-8833.

\bibitem{Roques et al 2015}
{L. Roques, Y. Hosono, O. Bonnefon, and T. Boivin}, The effect of competition on the neutral intraspecific diversity of invasive species, J. Math. Biol. 71 (2015), 465-489.


\bibitem{Rothe1981}
{F. Rothe},
Convergence to pushed fronts, Rocky Mountain J. Math. 11 (1981), 617-634.

\bibitem{nl 7}
{K. Schumacher}, Traveling-front solutions for integro-differential equations. I, J. Reine Angew. Math. 316 (1980), 54-70.


\bibitem{Stokes1976}
{A.N. Stokes}, On two types of moving fronts in quasilinear diffusion, Math. Biosc.
31 (1976), 307-315.




\bibitem{Uchiyama}
{K. Uchiyama}, The behavior of solutions of some nonlinear diffusion equations for large time,
J. Math. Kyoto Univ. 18 (1978), 453-508.



\bibitem{van Saarloos1988}
{W. van Saarloos},
Front propagation into unstable states: marginal stability as a dynamical mechanism
for velocity selection. Phys. Rev. A (3) 37 (1988), 211–229.

\bibitem{van Saarloos1989}
{W. van Saarloos},
Front propagation into unstable states. II. Linear versus nonlinear marginal
stability and rate of convergence. Phys. Rev. A (3) 39 (1989), 6367–6390.

\bibitem{van Saarloos2003}
{W. van Saarloos},
Front propagation into unstable states, Phys. Rep. 386 (2003) 29–222.


\bibitem{nl 5}
{K. Schumacher}, Travelling-front solutions for integro-differential equations. I, J. Reine Angew. Math. 316 (1980), 54–70.

\bibitem{Volpert}
{A. I. Volpert, V. A. Volpert, and V. A. Volpert},
Traveling wave solutions of parabolic systems,
Amer. Math. Soc., Providence, 1994.



\bibitem{Wiggins}
{S. Wiggins},
Introduction to Applied Nonlinear Dynamical Systems and Chaos, 2nd ed, Springer-Verlag, Berlin, 2003.



\bibitem{Wu Xiao Zhou}
{C.-H. Wu, D. Xiao, and M. Zhou},
Sharp estimates for the spreading speeds of the Lotka-Volterra competition-diffusion system: the strong-weak type with pushed front,
J. Math. Pure Appl., 172 (2023), 236-264.

\bibitem{Zhang_etal2012}
{G. Zhang, W. Li, and Z. Wang},
Spreading speeds and traveling waves for nonlocal dispersal equations with degenerate monostable nonlinearity,
J. Diff. Eqns., 252 (2012), 5096-5124.

\bibitem{nl 8}
{B. Zinner, G. Harris, and W. Hudson}, Traveling wavefronts for the discrete Fisher’s equation, J. Differential Equations 105 (1993), 46-62.

\end{thebibliography}
\end{document}